%% file: main.tex
\documentclass[11pt]{article}
\pdfoutput=1  
\usepackage{etoolbox}

\usepackage{booktabs}       %

\newtoggle{colt}
\togglefalse{colt} 
\newtoggle{arxiv}
\toggletrue{arxiv} 

\usepackage{preamble/kz_style} 

\input{preamble/header}

\input{preamble/general_macros}

\input{preamble/specific_macros}
\input{preamble/jacros}

\input{preamble/figs_preamble_arxiv}

\title{ Globally Convergent Policy Search over Dynamic Filters for Output Estimation}
\author{
Jack Umenberger\\
 \texttt{umnbrgr@mit.edu}\\
 MIT  \and
 Max Simchowitz\\\texttt{msimchow@mit.edu}\\ MIT \and 
 Juan C. Perdomo \\ \texttt{jcperdomo@berkeley.edu}\\ University of California, Berkeley  \and
 Kaiqing Zhang \\ \texttt{kaiqing@mit.edu} \\ MIT \and
 Russ Tedrake \\ \texttt{russt@mit.edu} \\ MIT
   }

\date{\today}

\begin{document}

\maketitle

\begin{abstract}
\input{body/abstract}
\end{abstract}

\input{body/newest_intro}

\input{body/newest_techniques}
\input{arxiv_only/organization_arxiv}
\section{Preliminaries}
\label{sec:prelim}
\input{body/preliminaries}

\section{Main Results}
\label{sec:main_results}
\input{body/algorithm_two}

\section{Numerical Experiments}
\label{sec:simulations}
\input{arxiv_only/simulations_arxiv}

\section{Analysis Framework}
\label{sec:analysis_framework}
\input{body/newest_analysis_overview}

\section{Proof of \Cref{thm:main_rate,thm:informative_optimal}}
\label{sec:proof_main_thms}
\input{body/proof_of_main_thm}

\subsection{\DCL{} for Output Estimation (\Cref{prop:DCL_for_Kalman})}
\label{sec:step2_weak_PL_verify}
\input{body/weak_pl}

\section{Conclusion}
\label{sec:conclusion}
\input{body/conclusion}

\newpage
\bibliographystyle{plainnat}
\bibliography{refs}
\newpage 
\appendix

\input{arxiv_only/app_organization_arxiv}

\input{appendix/alg_details_more}
\input{appendix/subsampling_oracle}

\part{General Control-Theoretic Proofs}\label{app:control_proofs}
\input{appendix/ctrb_assumption}
\input{appendix/control_proofs}
\input{arxiv_only/app_examples_arxiv}

\part{Proofs for Convergence Guarantee}
\label{part:convergence}
\section{Supporting Lemmas in Proof of \Cref{thm:main_rate,thm:informative_optimal}}
\label{app:proof_main_thms}
\input{appendix/concluding_proof_app}

\input{appendix/optimization_proofs}

\section{Supporting Proof of \OE{} Convex Reformulation (\Cref{prop:DCL_for_Kalman})} \label{sec:kal_dcl_proofs}
\input{appendix/app_scherer_dcl}

\input{appendix/lyapunov}

\input{appendix/smoothness_max}

\end{document}

%% file: preamble/header.tex
\usepackage{boxedminipage}
\usepackage{multirow,nicefrac}
\usepackage{makecell,upgreek}
\usepackage{footnote}
\usepackage{longtable}
\usepackage[shortlabels]{enumitem}
\usepackage{tablefootnote}
\usepackage[T1]{fontenc}
\usepackage{verbatim} 
\usepackage[utf8]{inputenc}

\usepackage{booktabs}   
\usepackage{float}

\usepackage{xcolor}              
\usepackage{preamble/color-edits}
\addauthor{kz}{brown}
\addauthor{ms}{magenta}
\addauthor{ju}{cyan}
\addauthor{jp}{purple}

\usepackage{amsthm}

\iftoggle{arxiv}
{
\usepackage{subfig}
  \usepackage[breaklinks=true]{hyperref}

  \usepackage{fullpage}
  \usepackage[noend]{algpseudocode}
  \usepackage{algorithm}
}
{
  
}

\usepackage{amsfonts,amssymb,mathrsfs}
\usepackage{mathtools}
\DeclareMathSymbol{\shortminus}{\mathbin}{AMSa}{"39}

	\usepackage{natbib}

\usepackage{prettyref}

\usepackage[capitalise,nameinlink]{cleveref}
\Crefname{equation}{Eq.}{Eqs.}
\Crefname{assumption}{Assumption}{Assumptions}
\Crefname{condition}{Condition}{Conditions}
\Crefname{claim}{Claim}{Claims}

\newcommand{\coltpar}[1]{\iftoggle{arxiv}{\paragraph{#1}}{\textbf{#1}}}

\usepackage{breakcites,bm}

\hypersetup{colorlinks,citecolor=blue,linkcolor = blue}
\usepackage{latexsym}
\usepackage{relsize}
\usepackage{thm-restate}
\usepackage{appendix}

\usepackage{xcolor}
\usepackage{dsfont}

\newcommand{\N}{\mathbb{N}}

\newcommand{\R}{\mathbb{R}}

\newcommand{\defeq}{:=}

\numberwithin{equation}{section}

\newcommand{\ftil}{\tilde{f}}

\def\grad{\nabla}
\newcommand{\nablatwo}{\nabla^{\,2}}

\newcommand{\rmd}{\mathrm{d}}

\newcommand{\bzero}{\ensuremath{\mathbf 0}}

\def\bB{\mathbf{B}}
\def\bC{\mathbf{C}}

\def\by{\mathbf{y}}
\def\bz{\mathbf{z}}

\def\bv{\mathbf{v}}

\def\bw{\mathbf{w}}
\def\bA{\mathbf{A}}
\def\bS{\mathbf{S}}
\def\bG{\mathbf{G}}
\def\bI{\mathbf{I}}

\def\bP{\mathbf{P}}

\def\bV{\mathbf{V}}
\def\bone{\mathbf{1}}

\def\xhat{\hat{\mathbf{x}}}

\DeclareFontFamily{U}{mathx}{\hyphenchar\font45}
\DeclareFontShape{U}{mathx}{m}{n}{
      <5> <6> <7> <8> <9> <10>
      <10.95> <12> <14.4> <17.28> <20.74> <24.88>
      mathx10
      }{}
\DeclareSymbolFont{mathx}{U}{mathx}{m}{n}
\DeclareFontSubstitution{U}{mathx}{m}{n}
\DeclareMathAccent{\widecheck}{0}{mathx}{"71}
\DeclareMathAccent{\wideparen}{0}{mathx}{"75}

\newcommand{\ignore}[1]{}

\DeclareMathOperator{\BigOm}{\mathcal{O}}

\newcommand{\BigOh}[1]{\BigOm\left({#1}\right)}
\DeclareMathOperator{\BigOmtil}{\widetilde{\mathcal{O}}}
\newcommand{\BigOhTil}[1]{\BigOmtil\left({#1}\right)}

\newcommand{\matM}{\mathbf{M}}

\newcommand{\calN}{\mathcal{N}}
\newcommand{\matx}{\mathbf{x}}

\newcommand{\matw}{\mathbf{w}}
\newcommand{\maty}{\mathbf{y}}

\newcommand{\iidsim}{\overset{\mathrm{i.i.d}}{\sim}}

\newcommand{\matz}{\mathbf{z}}

\newcommand{\op}{\mathrm{op}}

\newcommand{\fro}{\mathrm{F}}
\newcommand{\calK}{\mathcal{K}}

\newcommand{\algcomment}[1]{\texttt{\color{blue} #1}}

	\theoremstyle{plain}
	\newtheorem{theorem}{Theorem}

	\newtheorem{lemma}{Lemma}[section]
	\newtheorem{claim}[lemma]{Claim}
	\newtheorem{corollary}{Corollary}[section]
	\newtheorem{proposition}[lemma]{Proposition}

	\theoremstyle{definition}

	\newtheorem{definition}{Definition}[section]
	\newtheorem{example}{Example}[section]

	\newtheorem{remark}{Remark}[section]
    \newtheorem{observation}[lemma]{Observation}

  \newtheorem{fact}{Fact}[section]	
  \newtheorem{assumption}{Assumption}[section]
  \newtheorem{condition}{Condition}[section]

\makeatletter
\newcommand{\neutralize}[1]{\expandafter\let\csname c@#1\endcsname\count@}
\makeatother

\newenvironment{thmmod}[2]
  {\renewcommand{\thetheorem}{\ref*{#1}#2}%
   \neutralize{theorem}\phantomsection
   \begin{theorem}}
  {\end{theorem}}

\newtheorem*{theorem*}{Theorem}
\newtheorem*{lemma*}{Lemma}
\newtheorem*{corollary*}{Corollary}
\newtheorem*{proposition*}{Proposition}
\newtheorem*{claim*}{Claim}
\newtheorem*{fact*}{Fact}
\newtheorem*{observation*}{Observation}

\newtheorem*{definition*}{Definition}
\newtheorem*{remark*}{Remark}
\newtheorem*{example*}{Example}

\newtheoremstyle{named}{}{}{\itshape}{}{\bfseries}{}{.5em}{\Cref{#3} {\normalfont (informal)} }
{}
\theoremstyle{named}

\theoremstyle{plain}

\DeclareMathAlphabet{\mathbfsf}{\encodingdefault}{\sfdefault}{bx}{n}

\DeclareMathOperator*{\argmin}{arg\,min}

\let\Pr\relax
\DeclareMathOperator{\Pr}{\mathbb{P}}

\newcommand{\norm}[1]{\|#1\|}

\newcommand{\trace}{\mathrm{tr}}

\newcommand{\poly}{\mathrm{poly}}

\renewcommand{\leq}{~\le~}
\renewcommand{\geq}{~\ge~}

\let\oldtfrac\tfrac
\renewcommand{\tfrac}[2]{\smash{\oldtfrac{#1}{#2}}}

\let\nablaold\nabla
\renewcommand{\nabla}{\nablaold\mkern-2.5mu}

\newcommand{\Exp}{\mathbb{E}}


%% file: preamble/general_macros.tex
\newcommand{\I}{\mathbb{I}_{\infty}}
\newcommand{\Sigkone}{\bSigma_{11,\sfK}}
\newcommand{\Sigktwo}{\bSigma_{22,\sfK}}
\newcommand{\Sigkonetwo}{\bSigma_{12,\sfK}}

\newcommand{\Sigktwoof}[1]{\bSigma_{22,#1}}
\newcommand{\Sigkonetwoof}[1]{\bSigma_{12,#1}}
\newcommand{\Sigonesys}{\bSigma_{11,\mathrm{sys}}}

\newcommand{\Cpl}{C_{\mathtt{PL}}}
\newcommand{\Csys}{C_{\mathtt{sys}}}

\newcommand{\Hurn}[1][n]{\mathsf{Hur}_{#1}}
\newcommand{\algname}{\textsc{IR}\text{\normalfont -}\textsc{PG}}
\newcommand{\recond}{\mathsf{recond}}
\newcommand{\Zst}{\matZ_{\star}}

\newcommand{\regexp}{\mathcal{R}_{\mathtt{info}}}
\newcommand{\Llam}{\mathcal{L}_{\lambda}}
\newcommand{\Sback}{\mathcal{S}_{\mathrm{bkt}}}

\renewcommand{\OE}{{\normalfont \texttt{OE}}}
\newcommand{\LQG}{{\normalfont \texttt{LQG}}}

\newcommand{\LQR}{\normalfont{\texttt{LQR}}}
\newcommand{\Loe}{\mathcal{L}_{\mathtt{OE}}}
\newcommand{\Pst}{\bP_{\star}}

\newcommand{\Grad}{\mathsf{Grad}}
\newcommand{\Eval}{\mathsf{Eval}}

\newcommand{\Gradcov}{\mathsf{Grad}_{\mathrm{cov}}}
\newcommand{\Evalcov}{\mathsf{Eval}_{\mathrm{cov}}}

\newcommand{\DelA}{\bDelta_{\bA}}
\newcommand{\DelB}{\bDelta_{\bB}}
\newcommand{\DelC}{\bDelta_{\bC}}
\newcommand{\DelK}{\Delta_{\sfK}}
\newcommand{\ddtsq}{\frac{\rmd^2}{\rmd t^2}}
\newcommand{\ddt}{\frac{\rmd}{\rmd t}}

\newcommand{\Aclkdel}{\bA_{\mathrm{cl},\sfK + t\DelK}}
\newcommand{\Sigkonetwodel}{\bSigma_{12,\sfK + t\DelK}}
\newcommand{\Sigktwodel}{\bSigma_{22,\sfK + t\DelK}}
\newcommand{\Cgradone}{C_{\mathtt{grad},1}}
\newcommand{\Cgradtwo}{C_{\mathtt{grad},2}}
\newcommand{\Csigone}{C_{\Sigma,1}}
\newcommand{\nuc}{\mathrm{nuc}}
\newcommand{\ltwonorm}[1]{\|#1\|_{\ell_2}}
\newcommand{\fronorm}[1]{\|#1\|_{\fro}}
\newcommand{\nucnorm}[1]{\|#1\|_{\nuc}}
\newcommand{\circnorm}[1]{\|#1\|_{\circ}}

\newcommand{\conslyapK}{C_{\mathtt{lyap}}(\sfK)}
\newcommand{\clyap}{\mathsf{clyap}}

\newcommand{\Sigone}{\bSigma_{11}}
\newcommand{\Lamone}{\bLambda_{11}}
\newcommand{\Lamtwo}{\bLambda_{22}}
\newcommand{\Lamonetwo}{\bLambda_{12}}
\newcommand{\sigst}{\sigma_{\star}}
\newcommand{\lossreg}{\cL_{\lambda}}
\newcommand{\polyop}{\poly_{\op}}

\newcommand{\flift}{f_{\subfont{lft}}}
\newcommand{\fcvx}{f_{\subfont{cvx}}}
\newcommand{\subfont}[1]{\mathtt{#1}}
\newcommand{\calKfull}{\mathcal{K}_{\subfont{ctrb}}}
\newcommand{\calKstab}{\mathcal{K}_{\subfont{stab}}}
\newcommand{\calKopt}{\mathcal{K}_{\subfont{opt}}}
\newcommand{\calKexp}{\calK_{\subfont{info}}}

\newcommand{\cC}{\mathcal{C}}

\newcommand{\mb}[1]{\mathbf{#1}}

%% file: preamble/specific_macros.tex
\newcommand{\Abad}{\sA_{\mathrm{bad}}}
\newcommand{\Bbad}{\sB_{\mathrm{bad}}}
\newcommand{\Cbad}{\sC_{\mathrm{bad}}}

\newcommand{\LfK}{L_{f,\farg_0}}
\newcommand{\Lreco}{L_{\mathrm{cond}}}
\newcommand{\LrecoK}{L_{\mathrm{cond},\farg_0}}
\newcommand{\betaK}{\beta_{\farg_0}}
\newcommand{\alphaK}{\alpha_{\farg_0}}

\newcommand{\alphadcl}{\alpha_{\textsc{dcl}}}
\newcommand{\dcltriple}{(\fcvx,\flift,\Phi)}

\newcommand{\boldt}{\mathbf{t}}

\newcommand{\Kbad}{\sfK_{\mathrm{bad}}}

\newcommand{\rhok}{\rho_{K}}

\newcommand{\matv}{\mb{v}}

\newcommand{\matX}{\mb{X}}

\newcommand{\Acl}{A_{\mathrm{cl}}}

\renewcommand{\epsilon}{\varepsilon}

\newcommand{\orac}{\mathsf{orac}}
\newcommand{\oraceval}{\orac_{\mathrm{eval}}}
\newcommand{\oracgrad}{\orac_{\mathrm{grad}}}

\newcommand{\ObsGram}{\mathcal{G}_{\mathrm{obs}}}
\renewcommand{\by}{\bm{y}}
\renewcommand{\bz}{\bm{z}}
\newcommand{\pd}[1]{\mathbb{S}^{#1}_{++}}
\newcommand{\sphered}[1][d]{\mathcal{S}^{#1-1}}
\newcommand{\psd}[1]{\mathbb{S}^{#1}_{+}}
\newcommand{\eye}{\mathbf{I}}

\newcommand{\flam}{f_{\lambda}}

\newcommand{\xstar}{x^\star}
\newcommand{\dom}{\mathsf{dom}}
\newcommand{\domsub}{\mathsf{dom}_{>}}

\newcommand{\gtil}{\tilde{\gvec}}

\newcommand{\Rebar}{\bar{\R}}

\newcommand{\farg}{\boldsymbol{x}}
\newcommand{\xiarg}{\boldsymbol{\xi}}
\newcommand{\dxi}{d_{\xi}}
\newcommand{\gvec}{\boldsymbol{g}}
\newcommand{\tilarg}{\boldsymbol{y}}
\newcommand{\cvxarg}{\boldsymbol{z}}
\newcommand{\fargst}{\farg^\star}

\newcommand{\cvxargst}{\cvxarg^\star}
\newcommand{\dcvx}{d_z}
\newcommand{\dfarg}{d}
\newcommand{\dtilarg}{d_y}
\newcommand{\tilset}{\mathcal{Y}}
\newcommand{\cvxset}{\mathcal{Z}}
\newcommand{\cctwo}{\mathscr{C}^2}
\newcommand{\ccone}{\mathscr{C}^1}

\newcommand{\sfK}{\mathsf{K}}

\newcommand{\GL}[1]{\mathbb{GL}(#1)}
\renewcommand{\Acl}{\bA_{\mathrm{cl}}}

\newcommand{\Ak}{\bA_{\sfK}}
\newcommand{\Bk}{\bB_{\sfK}}
\newcommand{\Ck}{\bC_{\sfK}}
\newcommand{\Similar}{\mathsf{Sim}}
\newcommand{\Sigk}{\bSigma_{\sfK}}
\newcommand{\Zk}{\bZ_{\sfK}}

\newcommand{\xhatk}{\hat{\matx}_{\sfK}}

\newcommand{\Aclk}{\bA_{\mathrm{cl},\sfK}}
\newcommand{\Wclk}{\bW_{\mathrm{cl},\sfK}}
\newcommand{\matY}{\mathbf{Y}}
\newcommand{\matZ}{\mathbf{Z}}
\newcommand{\sym}[1]{\mathbb{S}^{#1}}
\newcommand{\Cclass}{\mathscr{C}}



%% file: preamble/jacros.tex
\newcommand{\transp}{^\top}

\newcommand{\Ccl}{\bC_\mathrm{cl}}

\newcommand{\ml}{\prec}

\newcommand{\mg}{\succ}
\newcommand{\mge}{\succeq}

\newcommand{\hinf}{\mathcal{H}_{\infty}}
\newcommand{\htwo}{\mathcal{H}_{2}}

\newcommand{\fA}{\bA_\mathsf{K}}
\newcommand{\fB}{\bB_\mathsf{K}}
\newcommand{\fC}{\bC_\mathsf{K}}

\newcommand{\fx}{\hat{\mathbf{x}}}
\newcommand{\fo}{\hat{\po}}

\newcommand{\fparams}{\mathsf{K}}

\newcommand{\sW}{\bW_1}
\newcommand{\sV}{\bW_2}

\newcommand{\jcov}{\bSigma}

\newcommand{\sA}{\bA}
\newcommand{\sB}{\bB}
\newcommand{\sC}{\bC}
\newcommand{\sO}{\bG}

\newcommand{\sx}{\matx}
\newcommand{\sy}{\maty}
\newcommand{\sw}{\matw}
\newcommand{\sv}{\matv}

\newcommand{\opt}{\star}

\newcommand{\po}{\matz}
\newcommand{\mo}{\maty}

\newcommand{\jf}{f}
\newcommand{\jx}{{\bm{x}}}
\newcommand{\jnx}{{n_x}}
\newcommand{\jch}{\Psi}
\newcommand{\jn}{\bm{\nu}}
\newcommand{\jnn}{{n_{\nu}}}

\newcommand{\jh}{f_{\mathrm{cvx}}}

\newcommand{\nx}{n}
\newcommand{\ny}{m}

\newcommand{\tmpW}{\bM}
\newcommand{\obsv}{\mathcal{O}}
\newcommand{\kinit}{\fparams_0}
\newcommand{\kopt}{\fparams_\opt}

\newcommand{\ogram}{\bY}

\newcommand{\Aclkst}{\bA_{\mathrm{cl}, \kopt}}

\newcommand{\Akst}{\bA_{\kopt}}
\newcommand{\Bkst}{\bB_{\kopt}}
\newcommand{\Ckst}{\bC_{\kopt}}

\newcommand{\jacsec}{\S}

\newcommand{\sbopt}{\bar{\jx}}

\newcommand{\Lst}{\bL_\opt}

\newcommand{\stol}{\epsilon}

\newcommand{\gram}{\bY}
\newcommand{\regctrb}{\mathcal{R}_\mathrm{ctr}}
\newcommand{\regobsv}{\mathcal{R}_\mathrm{obs}}

\newcommand{\schZ}{\bS}
\newcommand{\schX}{{\bX}}
\newcommand{\schY}{{\bY}}
\newcommand{\schXb}{\bar{\bX}}
\newcommand{\schA}{\bar{\bA}}
\newcommand{\schB}{\bar{\bB}}
\newcommand{\schC}{\bar{\bC}}
\newcommand{\schK}{{\bK}}
\newcommand{\schL}{{\bL}}
\newcommand{\schM}{{\bM}}

\newcounter{relctr} 
\everydisplay\expandafter{\the\everydisplay\setcounter{relctr}{0}} 

\newcommand\labelrel[2]{%
	\begingroup
	\refstepcounter{relctr}%
	\stackrel{\textnormal{(\alph{relctr})}}{\mathstrut{#1}}%
	\originallabel{#2}%
	\endgroup
}
\AtBeginDocument{\let\originallabel\label} 


%% file: preamble/figs_preamble_arxiv.tex
\usepackage{float}
\usepackage{subfig}

%% file: body/abstract.tex

We introduce the first direct policy search algorithm which provably converges to the globally optimal $\textit{dynamic}$ filter for the classical problem of predicting the outputs of a linear dynamical system, given noisy, partial observations. Despite the ubiquity of partial observability in practice,  
theoretical guarantees for direct policy search algorithms, one of the backbones  of modern reinforcement learning,  have proven difficult to achieve. This is primarily due to the degeneracies which arise when optimizing over filters that maintain internal state. 
 
 In this paper, we provide a new perspective on this challenging  problem based on the notion of $\textit{informativity}$, which intuitively requires that all components of a filter's internal state are  representative of the true state of the underlying dynamical system. We show that  informativity overcomes the aforementioned degeneracy. Specifically, we propose a $\textit{regularizer}$ which explicitly enforces  informativity, and establish that gradient descent on this regularized objective - combined with a  ``reconditioning step'' - converges to the globally optimal cost a $\mathcal{O}(1/T)$ rate. Our analysis relies on several new  results which may be of independent interest, including a new framework for analyzing non-convex gradient descent via convex reformulation, and novel bounds on the solution to linear Lyapunov equations in terms of (our quantitative measure of) informativity.


%% file: body/newest_intro.tex

\renewcommand{\OE}{{\normalfont \texttt{OE}}}
\newcommand{\DCL}{{\normalfont\texttt{DCL}}}

\section{Introduction}

Data used for prediction and control of real world dynamical systems is almost always noisy and 
incomplete (partially observed).
Sensors and other measurement procedures inevitably introduce errors into the
datasets, so designing reliable learning algorithms for these noisy or partially observed domains requires confronting fundamental questions of disturbance filtering and state estimation. 
Despite the ubiquity of partial observation in practice, these concerns are often underexplored in
modern 
analyses of learning for control that assume perfect observations of the underlying dynamics. 

In this work, we study the output estimation (\OE) problem or learning to predict in partially observed linear dynamical systems. 
The output estimation problem is one of the most fundamental problems in theoretical statistics and learning theory. 
Both in theory and in practice, advances in predicting partially observed linear systems have led to 
successes in a variety of areas from controls to biology and economics, (c.f. e.g. \cite{athans1974importance,lillacci2010parameter,gautier2001extended}). 
We revisit this classical problem from a modern optimization perspective, and study the possibility of learning the optimal predictor via model-free 
procedures and 
direct policy search.

Relative to 
model-based procedures, which 
first
estimate the underlying dynamics and then return a 
policy 
by solving an 
optimization problem using the estimated model, model-free methods offer several 
potential
advantages. 
First, 
direct policy search allows one to easily specify the complexity of the policy class over which one searches. 
For example, the vast majority of industrial control systems are built upon proportional-integral-derivative (PID) controllers. 
Each PID controller comprises three scalar variables (gains)
yet successfully regulates complex feedback loops in high-dimensional systems (e.g chemical plants). 
In addition, 
model-free policy search optimizes for performance directly on the true system,
rather than an approximate model. 
As such, there is no gap between the model used for synthesis and the system on which the controller is deployed.
Such gaps are typically covered by robust control techniques \citep{zhou1996robust}, which may introduce conservatism.


In light of these advantages, there has recently been significant interest from both theoreticians and practitioners in understanding the foundations of model-free control. 
However, so far, this attention has been mostly focused on problems with {\it full-state}  observation such as the linear quadratic regulator (LQR) or fully-observed Markov Decision Processes (MDPs) which admit \emph{static} policies. 
Progress in dealing with partially observed problems has 
been complicated by
the difficulties associated with 
optimizing over \emph{dynamic} policies that maintain internal state to summarize past observations.
In this paper, we provide the first policy search algorithm which provably converges to the globally optimal filter for the \OE\ problem,  and shed new light on the intricacies of the underlying optimization landscape.   


\paragraph{The Output Estimation problem.} We study one of the simplest and most basic problems with partial observability: the  \emph{output estimation} (\OE) problem.
In brief, the goal is to search for a predictor of the output $\po(t)$ of a linear dynamical system given partial measurements $\mo(t)$. 
For the \emph{true system} with states $\sx(t)$ and dynamics that evolve according to,
\begin{equation}\label{eq:true_system}
\begin{aligned}
&\tfrac{\rmd}{\rmd t} \sx(t) = \sA\sx(t) + \sw(t), \quad \sy(t)=\sC\sx(t) + \sv(t), \quad \po(t) = \sO \sx(t), \quad \sx(0) = 0, \\
&\quad \sw(t) \iidsim \calN(0,\sW), \quad \sv(t) \iidsim \calN(0,\sV),
\end{aligned}
\end{equation}
the goal is to find the parameters $\fparams = (\fA,\fB,\fC)$ of the \emph{filter} (interchangably, \emph{policy}),
\begin{equation}\label{eq:filter_system}
\tfrac{\rmd}{\rmd t} \fx(t) = \fA\fx(t) + \fB\sy(t), \quad \fo(t) = \fC\fx(t),
\end{equation}
that minimizes the steady-state prediction error,
\begin{align}\label{eq:loe}
\mathcal{L}_{\OE}(\fparams) := \lim_{T \to \infty} \frac{1}{T}\left(\mathbb{E} \int_{0}^{T} \| \po(t) - \fo(t) \|^2 \rmd t\right) = \lim_{t \to \infty} \mathbb{E} \| \po(t) - \fo(t) \|^2.
\end{align}
In this paper, 
we study solving the \OE\ problem via model-free methods, where the goal is to search for the optimal filter parameters $\fparams = (\fA,\fB,\fC)$ using direct policy search
without knowledge or estimation of the true system parameters $\sA, \sC, \sO, \sW, \sV$;
\iftoggle{arxiv}
{
	cf. \Cref{sec:prelim} for a detailed problem description. 
}
{
	see \Cref{sec:prelim} for further details.
}

\iftoggle{colt}
{
	While the \OE\ problem is not a control problem per se, it is an attractive stepping stone for studying model-free control. 
It is a special case of the partially observed linear quadratic regulator problem (\LQG).
In fact, it is one of the two subproblems solved to obtain the solution to \LQG\ via the separation principle \citep{aastrom1971introduction,dgkf89}. 
Furthermore, like \LQG, optimal policies for the \OE\ problem are \emph{dynamic}, in the sense that they maintain an internal state. 
It therefore serves as a valuable testbed  to develop the theoretical foundations of direct policy search for this important class of policies which has thus far remained poorly understood.
}
{}

\subsection{Contributions}\label{ssec:contributions}

In this paper, we propose a novel policy search method 
which provably converges to a globally optimal $\mathcal{L}_{\OE}$ cost. Despite extensive prior work on \emph{static} policy search (e.g. \cite{fazel18lqr,agarwal2021theory}, our result  constitutes the first rigorous guarantee for policy search over \emph{dynamic} policies.

A key concept underpinning our results is the notion of \emph{informativity},
which requires each component of the internal filter state $\fx$ 
to capture
some information about the true state $\sx$ of the system.
More precisely, a policy is said to be \emph{informative}
if the steady-state correlation matrix 
$\bSigma_{12} := \lim_{T \to \infty} \frac{1}{T}[\int_{0}^T\sx(t)\fx(t)^\top \rmd t]$
is full-rank.\footnote{For simplicity, we assume knowledge of the \emph{dimension} of the true system state, and policies are parameterized so that $\sx$ and $\fx$ are of the same dimension.}  
Our contributions are summarized as follows:

	\coltpar{Limitations of direct policy search.} 
	Through simulations and counterexamples,  we show that gradient descent on the prediction loss $\Loe(\cdot)$ can fail to recover the optimal filter for $\OE$ problem. 
While consistent with prior work \citep{tang2021analysis}, the failure of gradient descent remains puzzling, as the \OE{} problem admits a convex reformulation \citep{scherer97lmi}, a fact which at first glance seems to rule out suboptimal stationary points. 


	\coltpar{Structure of the \OE\ optimization landscape.} 
	We reconcile this apparent contradiction - the existence of convex reformulations and the failure of policy search - by studying cases in which
	the former
	breaks down. 
	We show that suboptimal stationary points can arise when the internal state $\fx$ of the filter is non-informative  about the state $\sx$ of the true system, in the sense described above. 
	These are precisely the points at which the convex reformulation breaks down. We also establish the converse: 
	when $\fx$ is ``uniformly informative'' about $\sx$, all stationary points are globally optimal.

	\coltpar{A provably convergent policy search algorithm.} 
	Building on this insight, 
	we propose a regularizer $\cR$ that ensures the
	internal state of the learned policy remains  ``uniformly informative'' about the state of the true system. We prove that gradient descent on the {\it regularized} objective $\Loe(\cdot) + \lambda \cR(\cdot)$ converges at a $\BigOh{1/T}$ rate to an optimal policy. 
	

%% file: body/newest_techniques.tex

\subsection{Our techniques}\label{sec:techniques}

Searching over dynamic policies introduces two key challenges: 
(1) spurious critical points can arise when one or more factors become ``degenerate'' in a certain way; 
(2) changes of basis produce a continuum of ``equivalent realizations'' of the filter $\sfK$, some of which are poorly conditioned.  
Similar challenges have been observed in problems with rotational symmetries, e.g. nonconvex matrix factorization.
Neither challenge arises when searching over static policies \citep{fazel18lqr}. 


Departing from prior literature on nonconvex factorization problems, which often either leverages closed-form gradient computations and/or the presence of {\it strict}-saddles,  
our approach is centered around the idea of \emph{convex reformulations} of control synthesis problems, and the following fact regarding functions which admit these reformulations, cf. \cref{app:cvx_fact} for proof.
\vspace{-.25em}
\begin{restatable}{fact}{cvxfact}\label{fact:no_suboptimal_sp}
Let $\jf:\R^{\jnx}\to \R$  be a differentiable, possibly nonconvex function such that $\min_\jx \ \jf(\jx)$ is finite. There exists a differentiable function $\jch:\R^\jnn\to\R^\jnx$ satisfying the following two properties:
	(i) the mapping $\jch$ is surjective, i.e. for all $\jx\in\R^\jnx$ there exists $\jn\in\R^\jnn$ such that $\jx=\jch(\jn)$,
	(ii) under the change of variables the function $\jh(\jn)\defeq\jf(\jch(\jn))$ is differentiable and \emph{convex}. Then all first-order stationary points, $\jx$ s.t $\grad \jf(\jx)=0$, are globally optimal.	
\end{restatable}
\vspace{-.25em}
The \OE{} problem, LQG, and many other related control tasks admit convex reformulations \citep{scherer97lmi}.
Given that gradient descent (under additional mild regularity assumptions) converges to stationary points, we might hope that \cref{fact:no_suboptimal_sp} guarantees that direct policy search on the \OE{} filter will succeed at finding an optimal policy, when applied to loss functions admitting such convex reformulations.  Somewhat surprisingly, we find that this is emphatically not the case: 
gradient descent on the $\Loe$ objective fails to reliably converge to optimal solutions (see \Cref{sec:suboptimal_point}).\footnote{Failure modes for the \LQG{} problem were presented by \cite{tang2021analysis}. } 

To resolve this paradox, we show that the surjectivity condition of \Cref{fact:no_suboptimal_sp} may fail for the convex reparametrization of \OE{}: there are filters $\fparams$ with finite cost $\Loe(\fparams)$, which are not in the image of the reformulation map $\jch(\cdot)$. We find that degeneracy occurs precisely when  \emph{informativity}, defined in \Cref{ssec:contributions} as $\Sigkonetwo$ having full rank, fails to hold. Conversely, when  $\Sigkonetwo$ is full-rank, the conditions of \Cref{fact:no_suboptimal_sp} are met and the parametrization behaves as needed. Thus, we identify \emph{non-informativity} - rank deficiency of $\Sigkonetwo$ - as the fundamental notion of degeneracy corresponding to challenge (1). Motivated by this observation, we introduce a novel ``informativity regularizer'' $\regexp(\cdot)$ which enforces that $\Sigkonetwo$ is full rank. Our proposed algorithm, \algname{} alternates between gradient updates on the regularized loss $\cL_{\lambda}(\cdot) := \Loe(\cdot) + \lambda \regexp(\cdot) $,  and  ``reconditioning'' steps to ensure well-conditioned realizations of the filters $\sfK$, thereby addressing challenge (2) above. We stress that our notion of informativity differs  from the \emph{minimality} criterion emphasized in \cite{tang2021analysis}, whose limitations we discuss in \Cref{sec:suboptimal_point}. 

In order to achieve our quantitative converge guarantees, we establish numerous results which may be of independent interest, including a quantitative analysis of the \OE{} convex reformulation due to \cite{scherer1995mixed}, and novel bounds on the magnitude of solution to Lyapunov equations under the closed-loop \OE{} filter dynamics. Both arguments appeal to a (quantitative measure of) informativity, suggesting informativity is somehow natural for the \OE{} landscape. 

We also develop a quantitative analogue of \Cref{fact:no_suboptimal_sp} via a paradigm we call \emph{differentiable convex liftings}, or \DCL s. Informally, a \DCL{} ``lifts'' the nonconvex $f$ to a possibly nonconvex, non-smooth function $\flift$ with $n_y \ge n_x$ parameters, which may take values in the extended reals so as incorporate constraints. We stress that both the `lifting' and accommodation of constraints are essential to
 capture
 the \OE{} convex reformulation.  A \DCL{} further requires existence of a map $\Phi$ and a convex function $\jh$ with $n_{\nu} \le n_y$ parameters such that $\flift(\cdot) := \jh(\Phi(\cdot))$. 
Intuitively, $\Phi$ corresponds to the inverse of the map $\Psi$ in \Cref{fact:no_suboptimal_sp}, though this parameterization allows more flexibility because $n_y > n_{\nu}$ may be permitted. 
For such liftings, the following result strengthens and refines \Cref{fact:no_suboptimal_sp}:
\begin{theorem}[Informal]\label{thm:informal_dcl}
Let $\jf$ be a {smooth} nonconvex function which admits a \emph{differentiable convex lifting} $(\flift, \jh,\Phi)$ such that $\sigma_{d_z}(\nabla \Phi(\cdot))$ is bounded away from zero. 
If $\jh(\Phi(\cdot))$ has compact level sets, then any $\jx$ such that $\|\nabla \jf(\jx)\|_2 \le \epsilon$, satisfies $\jf(\jx) - f_{\star} \le \BigOh{\epsilon}$.  Here, $\flift$ and $\jh$ need not be differentiable, and may only be defined (or finite) on restricted domains. 
\end{theorem}

\subsection{Related work}

\coltpar{Solution of the \OE\ problem\ \& Convex reformulation. }The solution to the \OE\ problem\footnote{\cite{kalman60new} addressed  the discrete-time problem, with $\sO=\sC$.} is given by the celebrated \emph{Kalman filter} 
\citep{kalman60new}.
The problem is also a special case of 
\LQG, cf. 
\cite{dgkf89}. 
Solution methods based on linear matrix inequalities (LMI) for \OE\ - 
as well as many other control problems, including $\htwo$, $\hinf$, and mixed $\htwo/\hinf$ synthesis - were developed, concurrently and independently, by \cite{scherer97lmi} and \cite{masubuchi1998lmi}.
The methods were based on convex reformulations of the variety described in \cref{sec:techniques},
and represent non-trivial generalizations of the well-known change of variables used to obtain LMI formulations of static state feedback problems, cf. \cite{thorp1981guaranteed,bernussou1989robust}. 
It has long been appreciated that most cost functions optimized for controller synthesis are \emph{nonconvex}. 
Such problems are usually solved indirectly, e.g. by reconstructing policies from the 
solutions of Riccati equations 
\citep{dgkf89}
or LMIs 
\citep{scherer97lmi},
or by using (model-based) policy parametrizations that make the cost function convex, e.g. 
\cite{youla1976parti,youla1976partii,kuvcera1975stability}. 

\coltpar{Direct policy search for  control with full state observation.}  Recent years have witnessed a resurgence of interest in direct policy search, driven perhaps in part by the success of such approaches in reinforcement learning, e.g. \cite{schulman2015trust,schulman2017proximal,openai2020manipulation}. 
Specifically, 
\cite{fazel18lqr} established global convergence of policy gradient methods on the discrete-time linear quadratic regulator (\LQR) problem, the simplest continuous state-action optimal control problem. 
Subsequent work has sharpened rates \citep{malik2019derivative},
analyzed convergence under more general frameworks 
 \citep{bu2019lqr}, and 
extended the analysis to work in continuous-time  \citep{mohammadi2021convergence}. 
\iftoggle{colt}
{
	
}
{}Beyond \LQR,
\cite{zhang20mixed} analyzed global convergence 
of policy search for mixed $\htwo/\hinf$ control and risk-sensitive control. \cite{furieri2020learning} and \cite{li2021distributed} established the  convergence of policy search for certain distributed control problems.  
\cite{sun2021learning} also considered analysis via convex reformulations, for state feedback problems. 
For discrete state-action (discounted) Markov decision processes (MDPs),
\cite{agarwal2021theory} established convergence rates for a variety of policy gradient methods with both tabular and parametric policies, cf. also \cite{bhandari2019global}.  
All of these works considered \emph{static} state-feedback policies, with perfect state information. 

\coltpar{Problems with partial observation.} 
For problems with partial observation,
optimal policies are typically dynamic, so as to incorporate information from the entire past history of observations.
%
The most relevant related work is \cite{tang2021analysis}, which studied the optimization landscape of the \LQG\ problem.  
They establish that all stationary points corresponding to controllable and observable controllers are globally optimal.
They also show (both empirically and via theoretical counterexamples) that gradient descent may fail to converge to globally optimal policies; a finding that, as we shall show, remains valid even for the simpler \OE\ problem. 
\cite{fatkhullin2021optimizing} also considered  linear quadratic control in the partially observed setting, but restrict their attention to \emph{static} output feedback policies, and provide conditions under which gradient descent converges to (possibly suboptimal) stationary points. 

%% file: arxiv_only/organization_arxiv.tex
\subsection{Organization}
This paper is organized as follows. \Cref{sec:prelim} provides the relevant preliminaries and assumptions for our setting, as well as details for our interaction protocol. \Cref{sec:main_results} provides our main results: first, a number of counterexample examples explaining the challenges of policy search over dynamic filters, and limitations of past work; second, a detailed distribution of our algorithm, \algname; third, a rigorous convergence guarantee. \Cref{sec:simulations} presents the numerical examples that illustrate the performance of our algorithm. \Cref{sec:analysis_framework} describes our main technical hammer - \DCL{}s - and how they afford quantitative convergence guarantees for gradient descent. Finally, \Cref{sec:proof_main_thms} provides the skeleton of the proofs of our main theorems. We provide concluding remarks in \Cref{sec:conclusion}, and detail the organization of the appendix in \Cref{app:organization_summary}.

%% file: body/preliminaries.tex

Before presenting our main results in \Cref{sec:main_results}, we first introduce some of the relevant definitions, and provide the reader with some the relevant background on prediction in partially-observed dynamical systems.

\iftoggle{arxiv}{\input{body/notation_primer}}{We adopt standard notation wherever possible, and for brevity, defer details to \Cref{app:notation_summary}.}  
 \iftoggle{arxiv}{\subsection{Output Estimation (\OE)}\label{sec:oe}}{}
As outlined in the introduction, we consider the problem of predicting the outputs of a partially observed linear dynamical system. We refer to the dynamical system defined in \Cref{eq:true_system} as the \emph{true system}, with states $\matx(t) \in \R^n$, observations $\maty(t) \in \R^m$, and performance outputs $\matz(t) \in \R^p$. To ensure the dynamics have a well-defined steady-state, we assume that $\bA$ is stable.


\begin{assumption}\label{asm:stability} The matrix $\bA$ is \emph{Hurwitz stable}. That is, the real components of all its eigenvalues are strictly negative: $\Re[\lambda_i(\bA)] < 0$ for $i \in [n]$. 
\end{assumption}
Because these policies only access the system outputs, and are only evaluated in relation to system outputs, we assume that the true system state is \emph{observable}. Further, we assume that dynamics are subject to sufficiently rich noise excitations.
\begin{assumption}\label{asm:observability} The pair $(\bA,\bC)$ in \Cref{eq:true_system} is observable. That is, the observability Gramian defined as 
\iftoggle{arxiv}
{
  \begin{align*}
\ObsGram \defeq \int_{0}^{\infty} \exp(s \bA)^\top \bC^\top\bC \exp(s\bA) \rmd s
\end{align*}
}
{
  $\ObsGram \defeq \int_{0}^{\infty} \exp(s \bA)^\top \bC^\top\bC \exp(s\bA) \rmd s$
}
is strictly positive definite. 
\end{assumption}
\begin{assumption}\label{asm:pd} We assume that the noise matrices $\sW$ and $\sV$ are strictly positive definite.\footnote{We may relax this assumption to $(\sA,\sW)$ controllable.}  
\end{assumption}
As stated previously, we restrict our attention to finding the best dynamic filter within the parametric family described in \Cref{eq:filter_system}. Note that this family contains the \emph{Bayes optimal predictor} for the $\Loe$ objective, as we will later describe in more detail. We review a number of basic facts:
\paragraph{A. Steady state distributions.} We define $\calKstab  := \left\{\sfK: \bA_{\sfK} \text{ is Hurwitz-stable}\right\}$ to be the set of filters such that $\Ak$ is stable. Under \Cref{asm:stability}, \Cref{sec:controllable_nonsingular} shows this equivalent to stability of the closed-loop matrix $\Aclk$.
\begin{align} 
\Aclk := \begin{bmatrix} \bA & 0\\
\Bk \bC & \Ak
\end{bmatrix}, \quad \calKstab = \left\{\sfK: \Aclk \text{ is Hurwitz-stable}\right\}.  \label{eq:aclk}
\end{align}
Stability of $\Aclk$ is a sufficient condition for $\Loe(\sfK)$ to be finite, and for the following limiting covariance to be well defined,
\iftoggle{arxiv}
{
\begin{align*}
\Sigk  =  \lim_{t \to \infty} \Exp \left[\begin{bmatrix}\matx(t)\\ \xhatk(t)\end{bmatrix}\begin{bmatrix}\matx(t)\\ \xhatk(t)\end{bmatrix}^\top\right] \in \psd{2n}. 
\end{align*}
}
{
  $\Sigk  =  \lim_{t \to \infty} \Exp \left[\begin{smallmatrix}\matx(t)\\ \xhatk(t)\end{smallmatrix}\right]\left[\begin{smallmatrix}\matx(t)\\ \xhatk(t)\end{smallmatrix}\right]^\top \in \psd{2n}. $
}
This steady-state covariance is given by the solution to the continuous-time  Lyapunov equation,
\iftoggle{arxiv}
{
  \begin{equation}\label{eq:lyapunov_sigma}
 \begin{aligned}
 \Aclk  \bSigma + \bSigma\Aclk^\top + \Wclk = 0, \quad \text{where } \Wclk := \begin{bmatrix} \bW_1 & 0\\
0 & \Bk\bW_2\Bk^\top
\end{bmatrix}.  
\end{aligned}
\end{equation}
}
{
  \begin{equation}\label{eq:lyapunov_sigma}
 \begin{aligned}
 \Aclk  \bSigma + \bSigma\Aclk^\top + \Wclk = 0, \quad \text{where } \Wclk := \left[\begin{smallmatrix} \bW_1 & 0\\
0 & \Bk\bW_2\Bk^\top
\end{smallmatrix}\right].   
\end{aligned}
\end{equation}
}
Notice that $\Sigk$ depends only on $(\Ak,\Bk)$, but not on $\Ck$, and that the first $n \times n$ block of $\Sigk$ does not depend on \iftoggle{arxiv}{the choice of filter}{} $\sfK$ at all. To highlight these distinctions, we partition \iftoggle{arxiv}{matrices $\bSigma \in \R^{2n \times 2n}$ as}{ }
\begin{align*}
\bSigma &= \begin{bmatrix} \bSigma_{11} & \bSigma_{12}\\
\bSigma_{12}^\top & \bSigma_{22}
\end{bmatrix}, \quad \bSigma_{\sfK}  =\begin{bmatrix} \Sigonesys & \Sigkonetwo\\
\Sigkonetwo^\top & \Sigktwo
\end{bmatrix} ,
\end{align*}
and define  $\calKfull := \{\sfK \in \calKstab: \bSigma_{22,\sfK} \succ 0 \}$ as the set of filters whose internal state covariance is full rank. We refer to these as the \emph{controllable} policies, as these are precise the policies for which the pair $(\Ak,\Bk)$ is controllable.\footnote{Controllability is the ``dual'' of observabilitity, and is equivalent to observability of $(\Ak^\top, \Bk^\top)$, cf. \Cref{sec:controllable_nonsingular}.} 
\paragraph{B. Equivalent realizations.} \iftoggle{arxiv}{Contrary to static feedback policies, such as LQR,there }{There}
are many different ways of parametrizing a given \emph{dynamic} feedback policy, all of which have exactly the same input-output behavior. 
In particular, given an invertible matrix $\bS \in \GL{n}$, the \OE{} loss of a filter $\sfK$ is invariant under the following class of similarity transforms: 
\iftoggle{arxiv}
{
\begin{align}
\label{eq:similarity_transforms}
\Similar_{\bS}(\sfK): (\Ak,\Bk,\Ck) \mapsto (\bS \Ak \bS^{-1},\bS \Bk,\Ck \bS^{-1}).
\end{align}
}
{
  $\Similar_{\bS}(\sfK): (\Ak,\Bk,\Ck) \mapsto (\bS \Ak \bS^{-1},\bS \Bk,\Ck \bS^{-1})$.
}Formally, for any $\sfK \in \calKstab$ and any $\bS \in \GL{n}$, $\Loe(\sfK) = \Loe(\Similar_{\bS}(\sfK))$. 
We say that $\sfK$ and $\sfK'$ are \emph{equivalent realizations} if they are related by a similarity transformation $\Similar_{\bS}(\sfK) = \sfK'$ for some $\bS \in \GL{n}$.\iftoggle{arxiv}{\footnote{This is a symmetric relationship, since then $\Similar_{\bS^{-1}}(\sfK') = \sfK$, and hence equivalent policies form an equivalence class.}}{}
Note that the set $\calKfull$ is also preserved under similarity transformation.



\paragraph{C. Optimal policies.} The landmark result by Kalman shows that for the system defined by $(\bA,\bC,\bW_1,\bW_2)$ the Kalman filter $\sfK_\star =(\bA - \bL_{\star} \bC, \,\bL_{\star},\,\sO)$ achieves minimal $\Loe$ loss. Here, $\bL_\star$  is the Kalman gain which is defined in terms of the solution of the following Riccati equation:
\begin{align}
\bA \Pst + \Pst \bA^\top - \Pst \bC^\top \bW_2^{-1}\bC\Pst+\bW_1=0,\qquad \quad \bL_\star = \Pst \bC^\top \bW_2^{-1}. \label{eq:Pst_Lst}
\end{align}
We define the set of optimal filters $\calKopt$ to be those which are equivalent to the Kalman filter:
\begin{align}
\calKopt := \bigcup_{\bS \in \GL{n}}\big\{\Similar_{\bS}(\bA - \bL_{\star} \bC, \,\bL_{\star},\,\sO)\big\}.
\end{align}

%

\iftoggle{arxiv}
{
  \subsection{Restricted problem setting}\label{sec:ctrb_assumption}
}
{
  \paragraph{D. Restricted problem setting.}
}
The problem description outlined \iftoggle{arxiv}{in \cref{sec:oe}}{above},
including \Cref{asm:stability,asm:observability,asm:pd},
constitutes the standard \OE\ problem, well-known in control theory, cf. \cite[\jacsec IV.D]{dgkf89}.
In this paper, we will make the following additional assumption that restricts the class of \OE\ problems we consider. Further discussion on the utility and necessity of this assumption (for our analysis) is provided in \cref{sec:algorithm} and \cref{app:ctrb_assumption};  the latter also shows that \Cref{asm:ctrb_of_opt} holds for ``generic'' problem instances.
\begin{assumption}\label{asm:ctrb_of_opt}
The optimal policy is itself controllable, i.e.  for all $(\Akst,\Bkst,\Ckst)\in\calKopt$ we have that $(\Akst,\Bkst)$ is controllable.
\end{assumption} 
\iftoggle{arxiv}
{
Here, we simply remark that \cref{asm:ctrb_of_opt} ensures that the regularizer $\regexp(\sfK)$, responsible for maintaining informativity of the policy $\sfK$, is well-defined at the optimal policy. 
}
{
}

\iftoggle{arxiv}
{
\subsection{Interaction protocol}\label{sec:interaction}
}{
  \paragraph{E. Interaction protocol.}
}
In the spirit of model-free methods, we introduce algorithms which work only assuming access to cost and gradient evaluation oracles. We abstract away the particular implementation of these oracles to simplify our presentation and assume that they are exact, in order to focus on the overall optimization landscape of the $\OE$ problem. More formally, for any filter $\sfK\in\calKstab$, 
\iftoggle{arxiv}
{
\begin{itemize}
    \item $\Eval(\sfK,\Loe)$ returns the \OE{} cost, $\Loe(\sfK)$.
    \item $\Grad(\sfK,\Loe)$ return the gradient of the \OE{} cost, $\nabla \Loe(\sfK)$.
\end{itemize} 
}
{
  $\Eval(\sfK,\Loe)$ returns the \OE{} cost, $\Loe(\sfK)$ and $\Grad(\sfK,\Loe)$ return the gradient of the \OE{} cost, $\nabla \Loe(\sfK)$.
}
 Despite this simplification, we would like to again emphasize that these can be efficiently approximated in finite samples, and purely on the basis of \emph{observations} $\maty_t$ subsampled in in discrete intervals. For further discussion, please see \Cref{app:oracle_details}. 
 
 Lastly, in addition to standard cost and gradient evaluations of the $\Loe$ loss, as part of our algorithm, we further require access to gradient and cost evaluations of smooth functions of the stationary-state covariance. Specifically, if $f: \psd{2n} \to \R$ is a function of the covariance matrix $\sfK$, we assume we can compute $\Evalcov(\sfK,f)$ which returns the $f(\Sigk)$ and $\Gradcov(\sfK,f)$ which returns $\nabla_{\sfK}f(\Sigk)$. In \Cref{app:subsample}, we show that these oracles can be implemented without direct state access by ``subsampling'' multiple observations at different time increments.



%% file: body/notation_primer.tex
\subsection{Notation}

We let lower case variables in script font $(\bx,\by,\bz)$ denote abstract parameters for optimization; standard vectors $(\matx,\maty,\matz)$ are reserved for random variables and/or dynamical quantities. Matrices are denoted in bold, e.g $\matX,\matY,\matZ$. For vectors, $\|\matx\|$ denotes the Euclidean norm, $\|\matX\|$ denotes the matrix operator norm and $\|\matX\|_{\fro}$, the Frobenius norm.

We let $\sphered[n]$ denote the unit sphere in $\R^{n}$. We denote the set of symmetric $n\times n$ matrices as $\sym{n}$; the set of nonstrictly positive semidefinite (PSD) matrices as  $\psd{n}$, strictly positive definite (PD) matrices as $\pd{n}$, and invertible matrices as $\GL{n}$. Given $\matX_1,\matX_2 \in \sym{n}$, we let $\matX_1 \preceq \matX_2$ denote nonstrict PSD inequality, with $\matX_1 \prec \matX_2$ denoting strict inequality. Given a square matrix $\bA \in \R^{n \times n}$, $\exp(\bA)$ denotes the matrix exponential. For $\bA$ with real eigenvalues, $\lambda_i(\bA), i = 1,\dots,n$ denotes its eigenvalues in descending order, with $\lambda_{\max}(\bA) = \lambda_1(\bA)$ and $\lambda_{\min}(\bA) = \lambda_n(\bA)$; when $\bA$ has complex eigenvalues, $\lambda_i(\bA)$ are arranged in an arbitrary order. For general rectangular matrices $\bA \in \R^{m \times n}$, $\sigma_{i}(\bA), i = 1,\dots,n$ denotes its singular values in descending order. We use $\eye_n$ to denote the identity matrix with dimension $n\times n$, and omit $n$ when the dimension is clear from context. 

We use parentheses to denote parameter concatenation: e.g. $\bar{\matX} = (\matX_{1},\matX_2,\matX_3) \in \R^{n_1\times m_1} \times \R^{n_2\times m_2} \times \R^{n_3\times m_3} $ for $\matX_i \in \R^{n_i \times m_i}$, and we define Euclidean norms of concatenation in the natural way 
 (e.g. $\|\bar{\matX}\|_{\ell_2} = \sqrt{\sum_{i} \|\matX_i\|_{\fro}^2}$ for the previous example $\bar{\matX} = (\matX_{1},\matX_2,\matX_3)$). 

%% file: body/algorithm_two.tex

In this section, we present the main contributions of our work.
After demonstrating that the \OE\ cost function
contains stationary points that are not globally optimal,
we present {\it informativity-regularized policy gradient}  (\algname), a direct policy search algorithm based on a novel regularization strategy to preserve \emph{informativity}, 
introduced in \cref{sec:techniques}.
We state a formal convergence result showing that \algname{} converges to a globally optimal filter at a $\BigOh{1/T}$ rate.

\subsection{Existence of suboptimal stationary points}\label{sec:suboptimal_point}

Perhaps the simplest model-free approach to the  \OE\ problem is to run
gradient descent \iftoggle{arxiv}{ on the loss function:
\begin{align}
\label{eq:naive_gd}
\sfK_{t+1} = \sfK_t - \eta_t \grad \Loe(\sfK_t),
\end{align}
for some stepsize(s) $\eta_t > 0$.
}{on $\Loe(\cdot)$.}Under mild assumptions on the loss function $\Loe$, gradient descent will converge to a first-order stationary point of $\Loe$. 
Unfortunately, despite the existence of 
a convex reformulation \citep{scherer97lmi} and \cref{fact:no_suboptimal_sp},
the $\Loe$ loss function contains suboptimal stationary points:

\begin{example}\label{ex:stationary}
	Consider the \OE\ instance given by
	$\sA=-\eye_2$,
	$\sC = \eye_2$,
	$\sW = 3\times \eye_2$,
	$\sV = \eye_2$,
	and the filter $\Kbad$ given by 
	$\Abad = -\stol \times \eye_{2}$, $\stol>0$,
	$\Bbad = \zero_2$,
	$\Cbad = \zero_2$.
	$\Kbad$ constitutes a suboptimal stationary point of $\Loe$ for this \OE\ instance.
\end{example}

\iftoggle{arxiv}
{
A formal proof of this claim is given in \cref{app:stationary},
however, one can easily verify that the cost is invariant under perturbations to any single parameter of the filter.
Specifically:
(i) perturbations to $\Abad$ and $\Bbad$ do not change the cost, because $\Cbad=\zero$ and thus the filter output $\po(t)$ is always zero, and
(ii) perturbations to $\Cbad$ do not change the cost either, because $\Bbad=\zero$ means that the internal state $\fx$ of the filter is always zero, which again implies $\po(t)\equiv 0$. 
Consequently, the gradient at $\Kbad$ is zero. 
Suboptimality can be seen by noticing that $\Kbad$ cannot be transformed to the non-zero $\kopt$ under any similarity transformation. 
}
{
	The example is similar in spirit to \citet[Theorem 4.1]{tang2021analysis}; details are  given in \cref{app:stationary}. Importantly, the cost is invariant under perturbations to any single parameter of the filter and $\Kbad$, being equivalent to the zero-filter, is suboptimal. 
}The same is true for any \OE\ instance:
every filter with $\Bbad=\zero$, $\Cbad=\zero$, and $\Abad$ being stable 
is a suboptimal stationary point of $\Loe$.  
\iftoggle{arxiv}
{
The existence of such suboptimal stationary points 
cautions against running 
simple
gradient descent, and
motivates our proposed regularization strategy.
}
{
	
}

\coltpar{The perils of enforcing minimality.} A filter $\sfK$ is \emph{minimal} if   $(\Ak,\Bk)$ is controllable, and $(\Ak,\Ck)$ is observable. \Cref{ex:stationary} is the extreme case of a  \emph{non-minimal} filter, since $\Bbad = \Cbad = \bzero_2$. Conversely, as a special case of \LQG, the \OE\ problem inherits the property 
that all stationary points corresponding to \emph{minimal} filters 
are globally optimal \cite[Theorem 4.3]{tang2021analysis}. Therefore, it may be natural to ask: \emph{can a local search algorithm enforce minimality to avoid suboptimal stationary  points?}
\iftoggle{arxiv}
{
	
}
{}A classical result due to \cite{brockett76geometry} suggests not: the set of minimal $n$-th order single-input-single-output transfer functions (e.g. filters)
is the disjoint union of $n+1$ open sets.
Thus it is impossible for a continuous path to pass from one of these open sets to another without entering a region corresponding to a non-minimal filter, suggesting that a local search algorithm regularized to ensure minimality at every iteration may never converge to the optimal solution. See \Cref{app:minimality_perils} for further discussion and supporting numerical experiments. 

\coltpar{Suboptimal controllable stationary points.}
Given the drawbacks of enforcing minimality,
one may wonder whether it is sufficient to enforce controllability alone.
In \cref{ex:stationary} above - and indeed, for all the examples in \cite{tang2021analysis} of suboptimal stationary points in the \LQG\ landscape - there is a loss of {\it both}  observability ($\Cbad=\zero$) and controllability ($\Bbad=\zero$).
Do suboptimal \emph{controllable} stationary points exist?
Unfortunately, the answer is affirmative, as the following example\iftoggle{arxiv}{}{(whose proof is given in \Cref{prop:OE_bad})} demonstrates:
\begin{example}
	\label{ex:stationary_ctrb}
	Consider the same \OE\ instance from \cref{ex:stationary}, 
	i.e.  
	$\sA=-\eye_2$,
	$\sC = \eye_2$,
	$\sW = 3\times \eye_2$,
	$\sV = \eye_2$.
	Consider the (family of) filter(s) $\Kbad$ given by
	\begin{equation}\label{eq:kbad_ctrb}
	\Abad = \begin{bmatrix}
	-2 & 0 \\  \gamma & - \gamma
	\end{bmatrix}, \quad
	\Bbad = \begin{bmatrix}
	1 & 0 \\ 0 & 0
	\end{bmatrix}, \quad
	\Cbad = \begin{bmatrix}
	1 & 0 \\ 0 & 0
	\end{bmatrix}.
	\end{equation}
	For any $\gamma>0$ the followings  are true:
	(i) $\Kbad$ is stable: $\Kbad \in \calKstab$,
	(ii) $\Kbad$ is controllable: $\Kbad\in \calKfull$, $\Sigma_{\Kbad,22} \succ 0$,
	(iii) $\Kbad$ is a first-order stationary  point:  $\nabla \Loe(\Kbad) = 0$,
	(iv) $\Kbad$ is strictly suboptimal: $\Kbad \notin \calKopt$, and
	(v) $\Kbad$ is not informative: $\Sigkonetwoof{\Kbad}$ is not full-rank. 
	\iftoggle{arxiv}{See 
	\cref{prop:OE_bad}
	for proof.}{}
\end{example} 
\iftoggle{arxiv}
{
Though $\Kbad$ in \cref{eq:kbad_ctrb} is controllable,
because $\Sigkonetwoof{\Kbad}$ is rank deficient
it corresponds to a suboptimal stationary point. This  
\algname\ 
circumvents such points by enforcing 
informativity ($\Sigkonetwoof{\fparams}$ being full-rank)
at all iterations, via the regularizer $\regexp$.
}
{
	
}

\coltpar{Difficulty of escaping suboptimal saddle points.}
The existence of suboptimal stationary points does not necessarily
rule out the efficacy of gradient descent. 
Recent work has established that variations of
gradient descent 
- e.g.
with appropriate random perturbations \citep{jin2017escape}
or acceleration \citep{jin2018accelerated}
- 
efficiently escape {\it strict}  saddle-points,
i.e. points at which the minimum eigenvalue of the Hessian is strictly negative.
In light of such results, it is reasonable to ask whether 
suboptimal stationary points of the kind identified in \cref{ex:stationary_ctrb}
are problematic for gradient descent.
Specifically, are such stationary points strict saddles?
It turns out that $\Kbad$ in \cref{ex:stationary_ctrb}
\emph{does} correspond to a strict saddle point;
however, the minimum eigenvalue of the Hessian can be made arbitrarily close to zero by making $\gamma$ sufficiently large. See \cref{app:stationary_ctrb} for details.



\subsection{A provably convergent algorithm}\label{sec:algorithm}
The previous discussion puts us in a bind: 
we cannot regularize to preserve minimality because of path-disconnectedness. Yet, controllability is not enough to rule out suboptimal stationary points. The construction of \Cref{ex:stationary_ctrb} hinges on point (v): the cross covariance $\Sigkonetwoof{\Kbad}$ between the true system state $\sx(t)$ and internal policy state $\xhat(t)$ is \emph{rank deficient}. We  call such filters \emph{non-informative}. Under \cref{asm:ctrb_of_opt}, however, all optimal policies $\sfK \in \calKopt$ must be \emph{informative}, that is, lie in the set $\calKexp := \{\sfK \in \calKstab: \rank(\Sigkonetwo) = n\}$ (see \Cref{app:opt_pol_nondegenerate} for proof). 
\begin{restatable}{lemma}{lemsigonetwo}\label{lem:Sigonetwo} Under \Cref{asm:observability,asm:stability,asm:pd,asm:ctrb_of_opt}, then  $\calKopt \subset \calKexp \subset \calKfull$, and $\calKexp$ is an open set. 
\end{restatable}
Our key insight is that informativity is also \emph{sufficient} to ensure optimality of stationary points, but does not cause path-connectedness issues as it did for minimality (see \Cref{app:thm_informative_optimal} for the proof).
\begin{theorem}\label{thm:informative_optimal} Let $\sfK \in \calKexp$; then (i) there is a continuous path lying in $\calKexp$ connecting $\sfK$ to some $\sfK_{\star} \in \calKopt$ and (ii) if $\nabla \Loe(\sfK) = \bzero$, then $\sfK \in \calKopt$.
\end{theorem}
\Cref{thm:informative_optimal} suggests that gradient descent with enforced informativity should converge to optimal filters. It is not, however, implied by the landscape analysis of \cite{tang2021analysis}, which focuses solely on \emph{minimal} stationary points. But numerous challenges remain: (1) how can one enforce informativity in a smooth fashion? (2) what quantitative measure of informativity provides quantitative suboptimality guarantees on approximate first-order stationary points? (3) Given that $\Loe(\sfK)$ need not have compact level sets (see \Cref{app:stationary_ctrb}), how does one ensure that the iterates of policy search do not escape to infinity, or reach regions where the loss of smoothness is arbitrarily poor? 

\paragraph{The explained covariance matrix.} In light of \Cref{thm:informative_optimal}, we design a policy search algorithm which ensures that $\Sigkonetwo$ remains full-rank throughout the search, but does so in a quantitative  fashion. Our central object is the \emph{explained covariance matrix}, which measures how much of the covariance of the steady-state system $\sx(t)$ is explained by the internal filter state $\xhatk(t)$ in the large $t$ limit: $\Zk := \lim_{t \to \infty} \left(\Cov[\sx(t) ] - \Exp[ \Cov[\sx(t) \mid \xhatk(t)]]\right)$.  When $\sfK \in \calKfull$, $\Zk$ admits an elegant closed-form expression, which provides an alternative definition of $\calKexp$:
\begin{align*}
\Zk = \Sigkonetwo\Sigktwo^{-1}\Sigkonetwo^\top, \text{ so that } 
\calKexp = \{\sfK : \sfK \in \calKfull \text{ and } \Zk \succ 0\}.
\end{align*}

Since $\Zk$ is invariant under similarity transformations\iftoggle{arxiv}{, as per \Cref{eq:similarity_transforms}}{}, $\Zk$ can be interpreted as a normalized analogue of $\Sigkonetwo$. Informally, the  quadratic form $v^\top \Zk v$ is a sufficient statistic for how much information $\fx(t)$ contains about the ``$v$-direction'' of $\sx(t)$; see \Cref{info:theoretic} for a precise statement.

\paragraph{Explained-covariance regularization.} 

We preserve informativity by ensuring our iterates satisfy $\Zk \succ 0$.  To this end,  we run gradient descent on the regularized objective for some $\lambda>0$:
\begin{align}
\cL_{\lambda}(\sfK) &:= \Loe(\sfK) + \lambda\cdot\regexp(\sfK), \label{eq:cL_def} \quad \text{ where }
\regexp(\sfK) \defeq \begin{cases} \trace[\Zk^{-1}] & \sfK \in \calKexp\\
\infty &  \text{otherwise}. 
\end{cases}
\end{align}
This choice of regularizer has several important properties. First, $\regexp$ is always non-negative, and tends to $\infty$ as $\Zk$ approaches singularity. Furthermore, the value of the regularizer is invariant under similarity transformations\iftoggle{arxiv}{(as per \Cref{eq:similarity_transforms})}{}. Next, many of the essential quantities arising in our analysis can be bounded in terms of $\Zk^{-1}$, justifying $\Zk$ is a natural quantitative measure of informativity. Lastly, the set of global-minimizers of $\regexp(\cdot)$ are precisely the optimal filters for the \OE{} problem,  as per the following lemma (see \Cref{app:opt_pol_nondegenerate} for proof).
\begin{restatable}[Existence of maximal $\mathbf{Z}_{\sfK}$]{lemma}{lemZmaximal}\label{lem:Z_maximal} Under \Cref{asm:stability,asm:observability,asm:pd}, there exists a unique $\Zst \succ 0$ such that $\Zst = \bZ_{\sfK}$ if and only if $\sfK \in \calKopt$, and $\Zst \succeq \bZ_{\sfK}$ for all $\sfK \in \calKfull \setminus \calKopt$. Consequently, 
\iftoggle{arxiv}
{
	\begin{align*}
	\calKopt = \argmin_{\sfK \in \calKfull}\regexp(\sfK).	
	\end{align*}
}{
	$\calKopt = \argmin_{\sfK \in \calKfull}\regexp(\sfK)$.
} 
\end{restatable}

\Cref{lem:Z_maximal} directly implies that the suboptimality of $\cL_{\lambda}(\cdot)$ upper bounds the suboptimality in $\Loe(\cdot)$, so we can minimize $\cL_{\lambda}$ as a proxy for minimizing $\Loe$. 
\begin{corollary}\label{cor:subopt_ub} For any $\sfK$, we have  $\Loe(\sfK) - \min_{\sfK'} \Loe(\sfK') \le \cL_{\lambda}(\sfK) - \min_{\sfK'} \cL_{\lambda}(\sfK')$. 
\end{corollary}

\paragraph{Reconditioning. } In addition to regularization, we introduce an additional normalization step between policy updates to ensure the iterates produced by our algorithm have well-conditioned covariance matrices; this in turn ensures the iterates produced by our algorithm remain in a compact set, and that the smoothness of $\cL_{\lambda}$ is uniformly bounded.  For any filter $\sfK \in \calKfull$ such that $\Sigktwo \succ 0$, the reconditioning operator $\recond(\sfK)$ returns a filter $\sfK'$ which is equivalent to $\sfK$, but for which $\Sigktwoof{\sfK'} = \eye_n$. Formally\footnote{Our proposed algorithm also works with an approximate balancing $\tilde{\recond}(\cdot)$, where $\tilde{\recond}(\sfK)$ returns a $\sfK'$ which is equivalent to $\sfK$, and $\|\Sigktwoof{\sfK'} - \eye_n\| \le \epsilon$ for some tolerance $\epsilon > 0$ (e.g. $\epsilon = 1/8$).}, 
\begin{align}
\recond(\sfK) := \Similar_{\bS}(\Ak,\Bk,\Ck), \text{~~where } \bS = \Sigktwo^{-1/2}. \label{eq:recond}
\end{align}
Since $\sfK$ and $\sfK' = \recond(\sfK)$ are equivalent realizations, we have $\Loe(\sfK) = \Loe(\sfK')$, $\regexp(\sfK) = \regexp(\sfK')$, and thus $\cL_{\lambda}(\sfK) = \cL_{\lambda}(\sfK')$.

\paragraph{Statement of \algname.} We can now describe \algname{}, whose pseudocode is displayed in \Cref{alg:our_algorithm}. \algname{} applies gradient descent on the regularized $\Loe$ objective, with an additional balancing step between gradient updates. \Cref{app:backtracking} provides a variant where the step size is chosen by backtracking line-search, which enjoys the same rigorous convergence guarantees.  
\iftoggle{arxiv}{

}
{}To guarantee convergence to an optimal filter (and finiteness of $\cL_{\lambda}$), we need to initialize at a filter such that $\bZ_{\sfK_0} \succ 0$, i.e. $\sfK_0 \in \calKexp$. Fortunately, random initializations from a continuous distribution satisfy this condition with probability $1$ (see \Cref{app:initializations} for a formal statement and proof).
\begin{algorithm}[!t]
\begin{algorithmic}[1]
    \State{}\textbf{Input:} Initial $\sfK_{0} \in \calKexp$, step size $\eta > 0$, regularization parameter $\lambda > 0$
    \Statex{}\algcomment{\quad \% Define $\cL_{\lambda}(\sfK) := \Loe(\cdot) + \lambda \trace[\Zk^{-1}]$}
    \For{each iteration $s= 0,1,2,\dots$}
    \State{}\textbf{Recondition } $\tilde{\sfK}_t = \recond(\sfK_t)$, where $\recond(\cdot)$ is defined in \Cref{eq:recond}.
    \State{}\textbf{Compute} $\nabla_s = \nabla \cL_{\lambda}(\tilde{\sfK}_t)$.
    \State{}\textbf{Update} $\sfK_{t+1} \gets \tilde{\sfK}_t - \eta \nabla_t$.
    \EndFor
  \end{algorithmic}
  \caption{Informativity-regularized Policy Gradient (\algname)}
  \label{alg:our_algorithm}
\end{algorithm}

\subsection{Formal guarantees}
We conclude this section by stating the formal convergence guarantee for \algname{}. Our results depend on natural problem quantities, among which is the minimum singular value of $\bP_{\star}$ (as defined  in \Cref{eq:Pst_Lst}), which we show is always strictly positive.
\begin{restatable}{lemma}{lemPstSigst}\label{lem:sigst} Let $\bP_{\star}$ be the solution to the Riccati equation in \Cref{eq:Pst_Lst}. Then under \Cref{asm:stability,asm:observability,asm:pd}, $\sigst := \lambda_{\min}(\bP_{\star})$ is strictly positive. Moreover, $\bP_{\star} = \Sigonesys-\Zst$.
\end{restatable}
In other words, $\bP_{\star}$ is the limiting conditional covariance of the true state $\sx(t)$ given the policy-state $\xhatk(t)$ under (any) optimal policy. \Cref{lem:sigst} states that this covariance is nonsingular, i.e.   not even optimal policies contain perfect information about any mode of $\sx(t)$. Given an initialization $\sfK_0$, our convergence rate depends polynomially on the following problem parameters:
\begin{align}
\Csys := \max\left\{\|\sA\|,\|\sC\|,\|\sO\|, \|\bW_2\|,\|\bW^{-1}_2\|, \|\bW_1^{-1}\|, \|\Sigonesys\|, \,\sigst^{-1} \right\}. \label{eq:Csys}
\end{align}

\begin{theorem}\label{thm:main_rate} Fix $\lambda > 0$, $\sfK_0 \in \calKexp$. There are terms $\cC_1,\cC_2 \ge 1$, which are at most polynomial in $n,m,\Csys,\lambda,\lambda^{-1}$ and $\cL_{\lambda}(\sfK_0)$, such that  the iterates of \algname{} with any stepsize $\eta \le \frac{1}{\cC_1}$ satisfy
	\begin{align*}
	\quad \Loe(\sfK_s) - \min_{\sfK} \Loe(\sfK) \le \cL_{\lambda}(\sfK_s) - \min_{\sfK} \cL_{\lambda}(\sfK) \le \frac{\cC_2}{\eta} \cdot \frac{1}{s}, \quad \forall s \ge 1.
	\end{align*}
\end{theorem}
The formal guarantee for back-tracking stepsizes is nearly analogous, and given in \Cref{app:backtracking}. \iftoggle{arxiv}
{

\paragraph{Oracle complexity.} At each iteration, one can compute the derivative of $\cL_{\lambda}$ using one call to $\oraceval$ (which evaluates $\Loe(\sfK_{s})$ and $\Sigk$), and one call to $\oracgrad$, which computes the gradients of these quantities. This is true because $\nabla\, \trace[\Zk^{-1}]$ admits a closed form in terms of $\Sigk$ and its gradient. The balancing step also requires only evaluation $\Sigk$, and can use an evaluation query called for the  gradient. Lastly, the backtracking step requires an evaluation query for all $|\Sback|$ filters of the form $\tilde{\sfK}_s - \eta \nabla_s$. In total, therefore, each iteration uses $1$ call to $\oracgrad$, and $|\Sback|+2$ calls to $\oraceval$.
}
{The oracle complexity of \algname{} is described in \Cref{app:oracle_complexity}.}We sketch the highlights of the proof in the following section, after introducing our \DCL{} framework. A rigorous proof overview, with statements of the constituent results, is deferred to \Cref{ssec:main_rate_proof}.

%% file: arxiv_only/simulations_arxiv.tex

In this subsection we present the results of a number of additional numerical experiments illustrating the performance of \algname.

\paragraph{Random generation of true systems.}
Each experimental trial begins with the random generation of a \emph{true system} of the form \cref{eq:true_system}.
System parameters $\sA,\sC$ are randomly generated using \texttt{Matlab}'s \texttt{rss} function,
with state dimension $\nx = 2$ and output dimension $\ny = 1$. 
The matrix $\sO$ defining the mapping from state to performance output $\po$ is set to $\sO=I$.
The intensity of the system disturbances is randomly generated as $\sW = \tmpW^\top\tmpW$ with each entry of $\tmpW \in\R^{\nx\times\nx}$ sampled from $\calN(0,1)$. 
The intensity of the measurement noise is normalized to $\sV=1$.  
To select suitable systems, we then reject samples  
according to the following criteria:
(i) $\sA$ must be strictly stable, and the observability Gramian $\obsv$ corresponding to $(\sA,\sC)$ must satisfy 
$10^{-4} \leq  \lambda_{\min}(\obsv) \leq 10^{-2}$;
(ii) $\sW$ must satisfy $\lambda_{\max}(\sW) \leq 5$;
(iii)~the optimal cost must satisfy $\Loe(\fparams_\opt) \leq 10^3$. 
The first criterion regulates the observability of the true system, which sets the difficulty of the filtering problem; 
the second ensures that the ratio between the disturbances and measurement noise remains ``reasonable''; 
and the third ensures that the problem instance is not ``pathological'', as determined by excessively high cost of the optimal filter.

\begin{remark}[Choice of $\sO=\eye$] 
As detailed in \cref{sec:algorithm}, \algname\ makes use of the regularizer $\regexp$, defined in \cref{eq:cL_def},
the computation of which requires access to the true system states $\sx$, as described in \cref{sec:prelim}.
To facilitate a more fair comparison with direct minimization of $\Loe$, we selected $\sO=\eye$ to effectively give the optimizer of $\Loe$ access to the true system states $\sx$ as well. 
As a result, all algorithms compared in this section have access to the same information concerning the true system. 
\end{remark}

\paragraph{Random generation of initial filters.} Next we randomly generate a filter $\kinit$ from which to initialize gradient descent. 
To do so, we take the optimal (Kalman) filter $\kopt$, and randomly perturb each of the parameters; 
specifically, we set $(\kinit)_i = (\kopt)_i + \delta_i$ with $\delta_i\sim\calN(0,100)$ for the $i$th parameter. 
Before accepting this $\kinit$, we rejection sample based on the following criteria:
(i) $\bSigma_{\kinit}$ must satisfy $10^{-5} \leq \sigma_{\min}(\bSigma_{12,\kinit}) \leq 10^{-3}$;
(ii) $\bSigma_{\kinit}$ must satisfy $10^{-3} \leq \sigma_{\min}(\bSigma_{22,\kinit}) \leq 1$; 
(iii) the initial suboptimality must satisfy $\Loe(\kinit) \leq 100 \times \Loe(\kopt)$. 
The first criterion ensures that we do not begin from an initial guess for which the informativity is too low, nor a guess for which it is too high (which makes the search easier).
The second criterion ensures that the initial filter is sufficiently controllable, to avoid initializations that are too close to suboptimal stationary points.
The final criterion ensures that the initial guess is, in all other ways, ``reasonable'', as measured by suboptimality.  

\paragraph{Optimization methods compared.} Given a randomly generated true system, and random initial filter $\kinit$, we then apply the following three optimization algorithms:
(i) gradient descent on $\Loe(\fparams)$;
(ii) gradient descent on $\Loe(\fparams)$ with filter state normalization performed before each gradient step, cf. \cref{eq:recond};
(iii) \algname, as detailed in \cref{alg:our_algorithm}, with regularization parameter $\lambda=10^{-4}$. 
See below for further discussion on the selection of $\lambda$.   
All methods are initialized from the same $\kinit$, and make use of the same backtracking line search to select step sizes. 
Moreover, all algorithms have the same termination criteria.
Each algorithm terminates when either:
(i) the Frobenius norm of the gradient of the cost function being minimized (either $\Loe$ or $\cL_\lambda$) falls below a tolerance of $10^{-8}$;
(ii) the step size selected by the line search falls below a tolerance of $10^{-16}$ for more than three consecutive iterations; or
(iii) the number of iterations (gradient descent steps) exceeds $100,000$.  

\paragraph{Results.}
The results of 60 such experimental trials are depicted in \cref{fig:suboptimality}.
It is evident that simple ``unregularized'' gradient descent on $\Loe$ routinely fails to converge to the global optimum, in the allotted number of iterations.
In fact, the median (normalized) suboptimality gap $\frac{\Loe(\fparams) - \Loe(\kopt)}{\Loe(\kopt)}$ exceeds $10^{-4}$,
and only a single trial achieves suboptimality less than $10^{-7}$. 
Loss of informativity in these trials can be seen clearly in \cref{fig:conditioning}.
The addition of the 
filter state reconditioning procedure of \cref{eq:recond} offers only minimal improvement. 
In contrast, \algname\ converges reliably to high-quality solutions that are extremely close to the global optimum; the median normalized suboptimality gap was zero, to numerical precision.  
In fact, for one third of trials, the suboptimality gap was actually \emph{negative} (by very small margins, e.g. $10^{-17}$) indicating that \algname\ has reached the limits of numerical precision with which \texttt{Matlab}'s \texttt{icare} solves Riccati equations (used to compute $\kopt$).

\begin{figure}
	\centering
	\subfloat[Normalized suboptimality at the termination of each algorithm.]{
		\includegraphics[width=6.4cm]{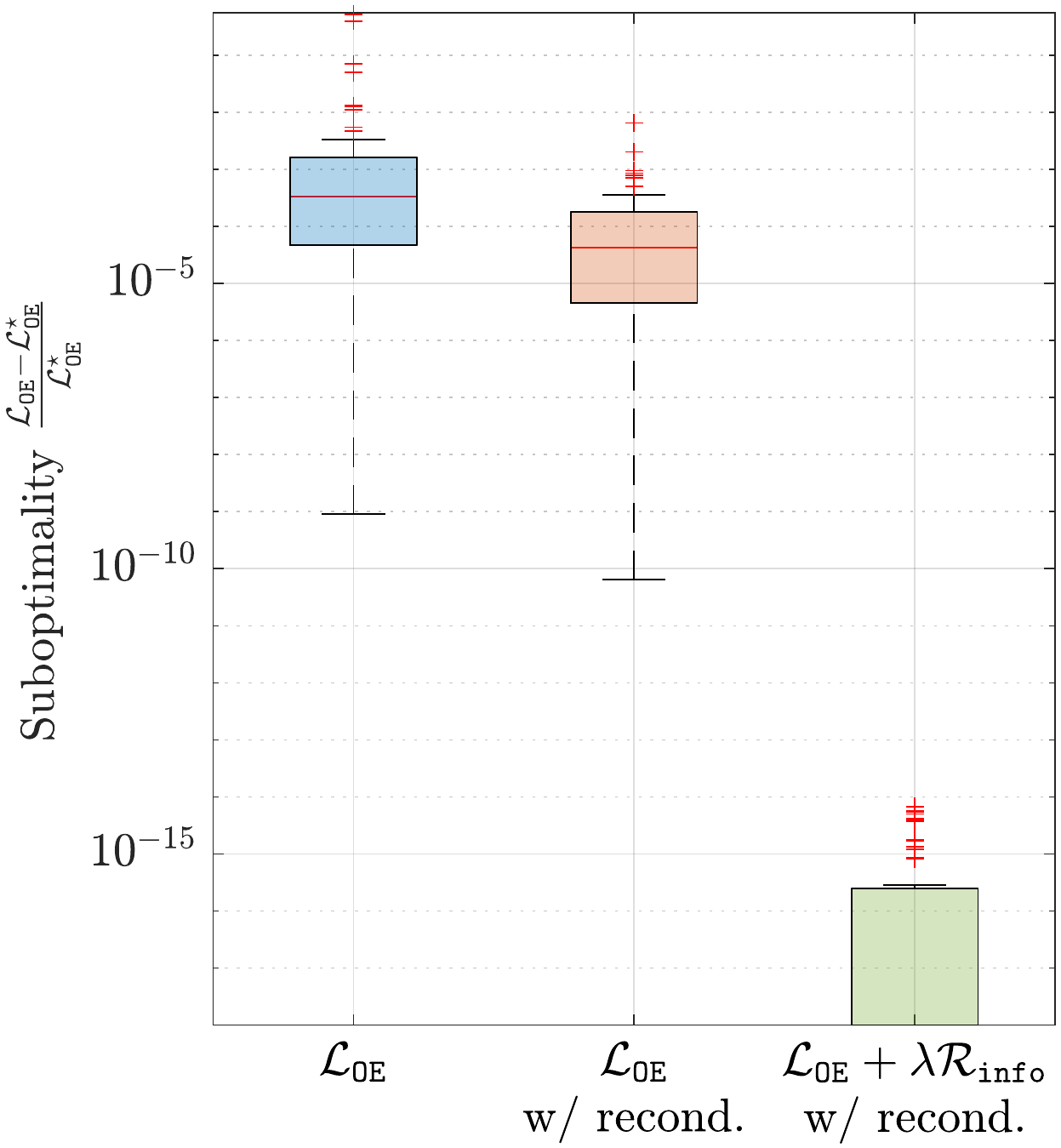}} \hspace{1em}
	\subfloat[Normalized suboptimality as a function of iteration for each algorithm. 
	]{
		\includegraphics[width=9.2cm]{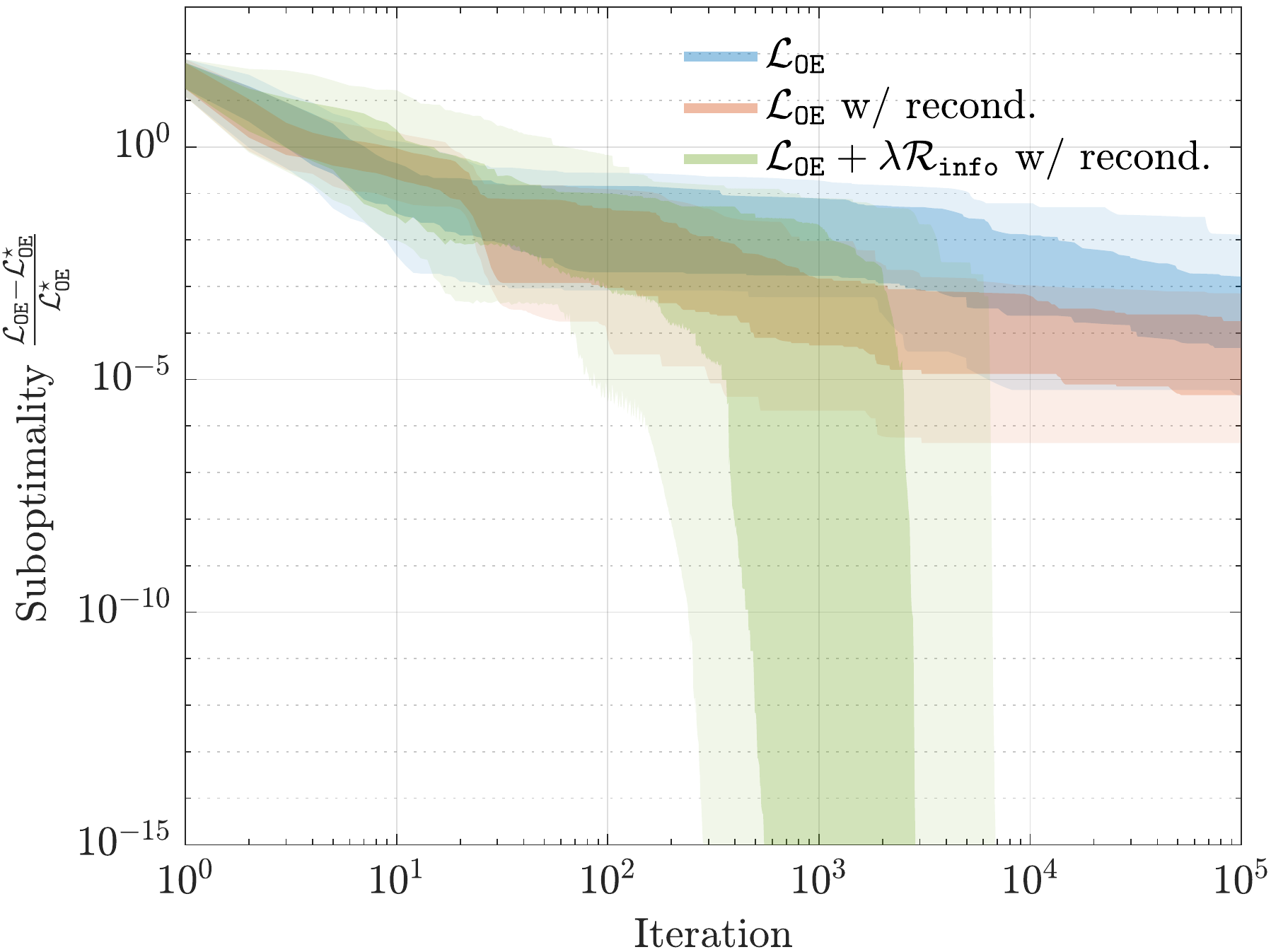}}
	\caption{Performance of each algorithm as measured by the normalized suboptimality of the output estimation cost, $\frac{\Loe(\fparams) - \Loe(\kopt)}{\Loe(\kopt)}$. 60 trials of the experimental procedure described in \cref{sec:simulations} are plotted.  
		In (b), the lightly shaded region covers the 10th to 90th percentiles, and the darker region covers the 25th to 75th percentiles.
	}
	\label{fig:suboptimality}
\end{figure}


\paragraph{Selection of regularization parameter $\lambda$.} 
Performance of \algname\ is in many instances insensitive to the value of $\lambda$ selected.
However, we observed that a handful of experimental trails required $\lambda$ to be chosen more judiciously,
in particular, when the spectral properties of $\grad^{~2}\regexp$ differ significantly from those of $\grad^{~2}\Loe$.
Very small stepsizes may be required when $\lambda_{\max}(\grad^{~2}\regexp)$ is very large,
which means the search may make slow progress in updating $\Ck$, as $\regexp$ is independent of $\Ck$.
We have observed good performance in practice by simply ``turning off'' the regularizer (i.e. setting $\lambda=0$)
when the stepsize becomes excessively small (e.g. drops below $10^{-16}$).

\begin{figure}
	\centering
	\subfloat[Informativity, as measured by $\sigma_{\min}(\jcov_{12})$. ]{
		\includegraphics[width=7.5cm]{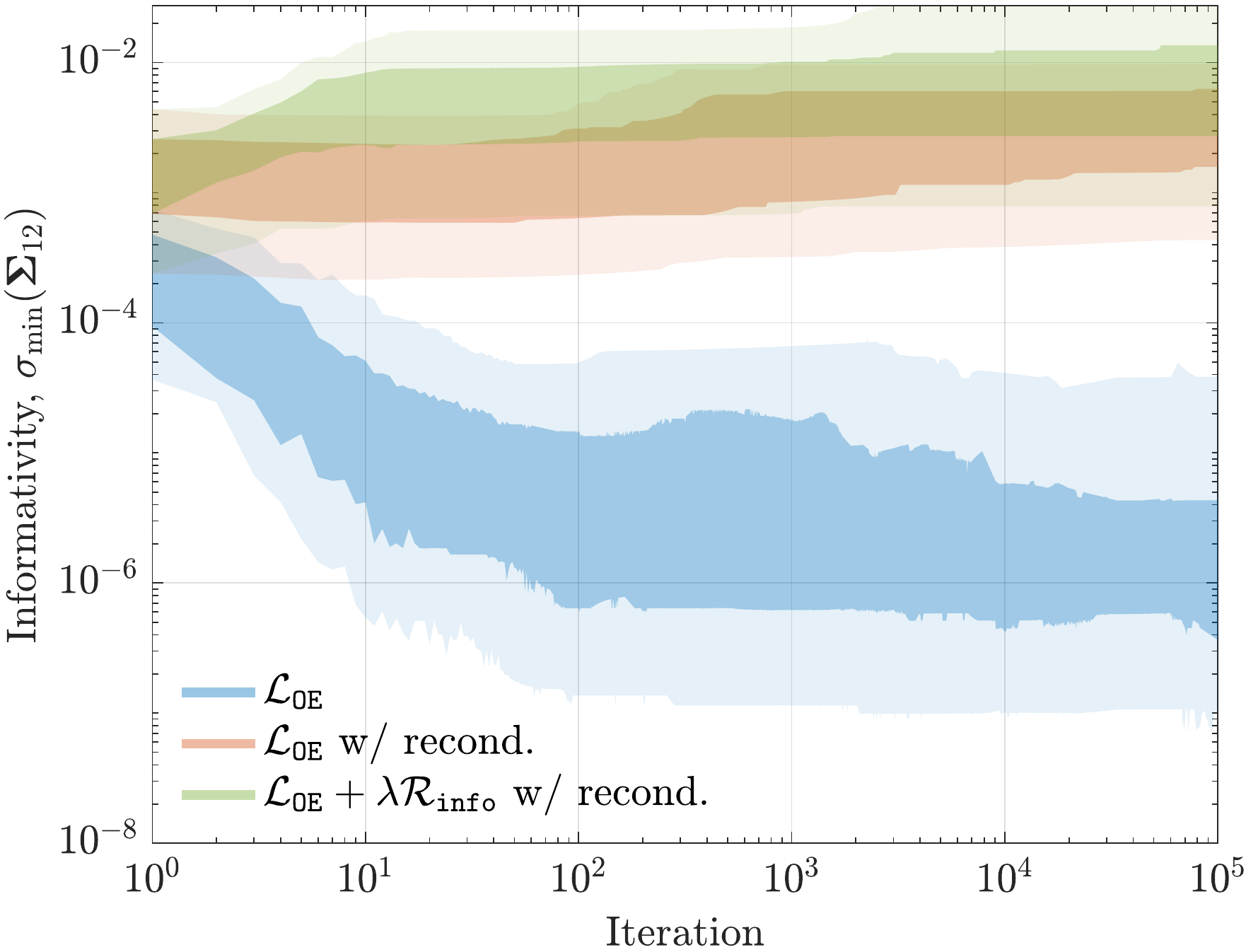}} \hspace{1em}
	\subfloat[Conditioning, as measured by $\sigma_{\min}(\jcov_{22})$.  ]{
		\includegraphics[width=7.5cm]{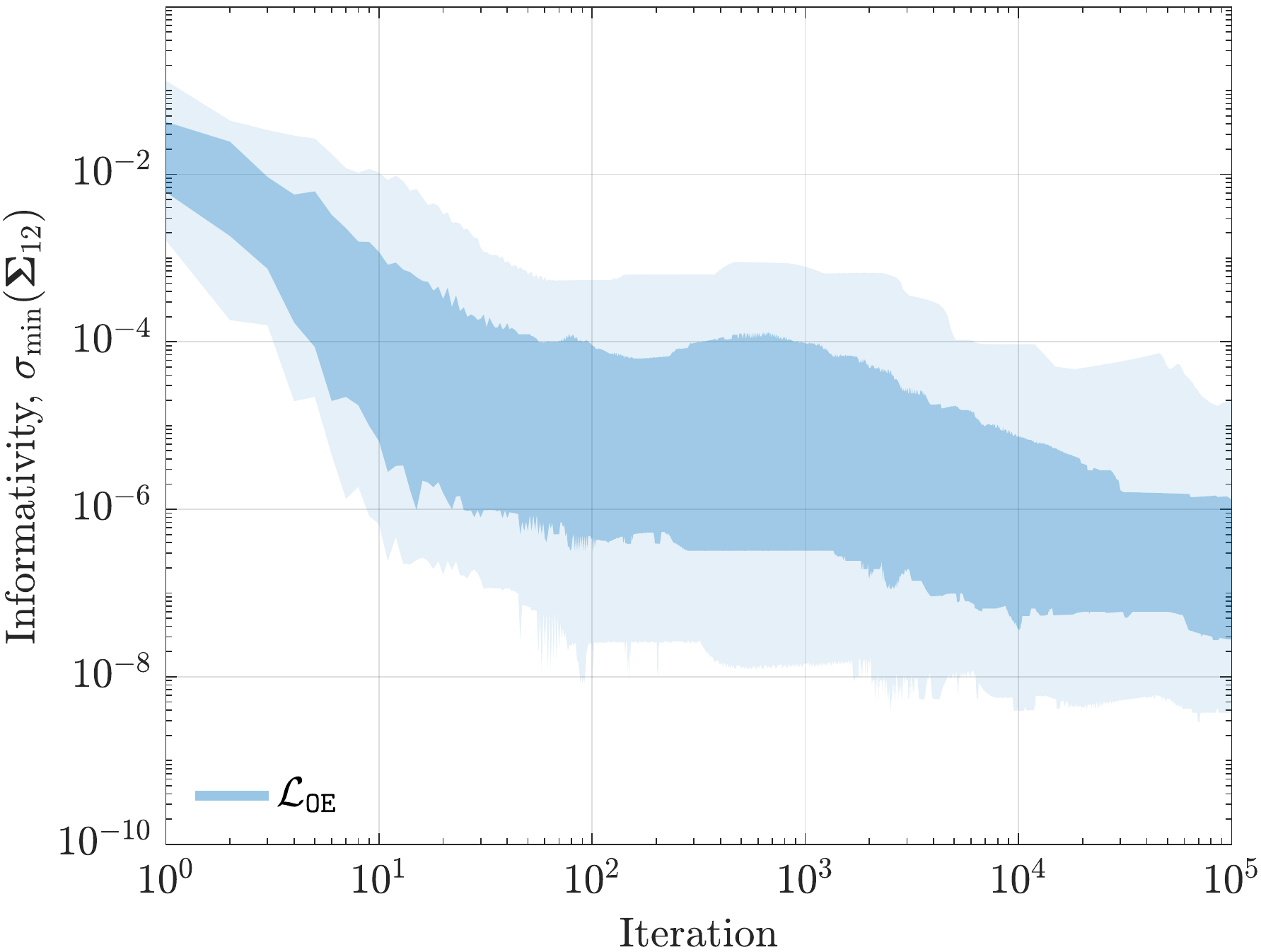}}
	\caption{	
		Properties of $\jcov$ for the same 60 trials plotted in \cref{fig:suboptimality}. 	
		The lightly shaded region covers the 10th to 90th percentiles, and the darker region covers the 25th to 75th percentiles. 
	}
	\label{fig:conditioning}
\end{figure}


%% file: body/newest_analysis_overview.tex

\iftoggle{arxiv}{
}
{
\subsection{Differentiable convex liftings (\textbf{\texttt{DCL}}s)}}
This section introduces \emph{differentiable convex liftings} (\textbf{\texttt{DCL}}s), a rigorous and flexible framework for operationalizing convex reformulations of nonconvex objectives.

\coltpar{Preliminaries. }To neatly accommodate optimization over constrained domains,  we express functions $f: \R^d \to \Rebar$ as taking values in the extended reals $\Rebar := \R \cup \{\infty\}$.\footnote{Because we solely consider minimizations, $\Rebar$ does not include $-\infty$} 
Given such an $f$, we denote its domain $\dom(f): \{x:f(\cvxarg)\ne \infty\}$ as the set on which $f$ is finite; we say $f$ is \emph{proper} if $\dom(f) \ne \emptyset$; we define its minimal value $\inf(f) = \inf_{\bx \in \R^d} f(\bx)$. \iftoggle{arxiv}{All functions are assumed to take extended real values, and are infinite outside their domain when not otherwise defined (for example, in the case of an ill-posed inverse).} We say $f \in \Cclass^k(\calK)$ on a if $f$ is $k$-times continuously differentiable (and finite) on some open set containing $\calK \subset \R^d$. 

\textbf{\texttt{DCL}s}.  The \DCL{} is a generic template for convex reformulation that significantly generalizes the setup in \Cref{fact:no_suboptimal_sp}. Rather than relating $f$ directly to a convex $\fcvx$, we ``lift'' $f$ to a function $\flift$ by appending auxiliary  variables. We then assume a reparametrization $\Phi$ mapping the domain of $\flift$ to that of $\fcvx$ in this higher dimensional space;  intuitively, $\Phi$ is the ``inverse''  of $\jch$ in \Cref{fact:no_suboptimal_sp}.
\begin{restatable}{definition}{dcldef}\label{defn:DCL} A triplet of functions $\dcltriple$ is a \DCL{} of a proper function $f: \R^{\dfarg} \to \Rebar$ if
\iftoggle{arxiv}
{
\begin{enumerate}
\item $\fcvx: \R^{\dcvx} \to \Rebar$ is an (extended-real valued) convex function whose minimum is attained by some $\cvxargst$, $\fcvx(\cvxargst) = \inf(\fcvx) > - \infty$.
\item For some additional number of parameters $\dxi \ge 0$, $\flift: \R^{\dfarg +\dxi} \to \Rebar$ is related to $f$ via partial minimization: $f(\farg) = \min_{\xiarg \in \R^{d_{\xi}}} \flift(\farg,\xiarg)$. 
\item For an open set $\tilset $ containing $\dom(\flift)$, $\Phi: \tilset \to \dom(\fcvx)$, is $\ccone(\tilset)$, and relates $\flift$ to $\fcvx$ via $\flift(\cdot) = \fcvx(\Phi(\cdot))$. 
\end{enumerate}
}
{
	\textbf{(1)} $\fcvx: \R^{\dcvx} \to \Rebar$ is a a proper  convex function whose minimum is attained by some $\cvxargst$.\\ 
	\textbf{(2)} For some additional number of parameters $\dxi \ge 0$, $\flift: \R^{\dfarg +\dxi} \to \Rebar$ is related to $f$; via partial minimization: $f(\farg) = \min_{\xiarg \in \R^{d_{\xi}}} \flift(\farg,\xiarg)$.\\
	\textbf{(3)} There is an open set $\tilset \supseteq \dom(\flift)$ for which $\Phi: \tilset \to \dom(\fcvx)$ is $\ccone$ and satisfies $\flift(\cdot) = \fcvx(\Phi(\cdot))$. 
}
\end{restatable}
The mere existence of a \DCL{} implies that approximate stationary points of $f$ are also approximate minimizers, under conditions elaborated on below:
\begin{theorem}\label{thm:DCL} Let $f: \R^{\dfarg} \to \Rebar$ be a proper function with \DCL{} $\dcltriple$. Then, for any $\farg \in \dom(f)$ at which $f$ is differentiable, $f$ satisfies the weak-PL condition:
\begin{align*}
\|\nabla f(\farg)\| \ge \alphadcl(\farg) \cdot\left(f(\farg) - \inf(f)\right), \quad \text{ where } \alphadcl(\farg) := \max_{\substack{\cvxargst \in \argmin \fcvx(\cdot) \\ \xiarg \in \argmin \flift(\farg,\cdot)}}  \frac{\sigma_{\dcvx}(\nabla \,\Phi(\farg,\xiarg))}{ \|\Phi(\farg,\xiarg) - \cvxargst\| }.
\end{align*}
\end{theorem}

\Cref{thm:DCL} strengthens \Cref{fact:no_suboptimal_sp} in two respects. For one, it does not impose any smoothness restrictions on $\flift$ or $\fcvx$; in particular $\fcvx$ can be highly non-smooth and, due to the extended-real function formulation, can also include constraints. And second, the lifting $\flift$ adds considerable flexibility, which we show is necessary to capture the convex reformulation of \OE{} (\Cref{sec:step2_weak_PL_verify}).

The factor $\alphadcl(\farg)$ depends on two quantities. The numerator is the $d_z$-th singular value of $\nabla \Phi(\farg,\xiarg)$ for any $\xiarg \in \argmin\flift(\farg,\cdot)$.  This captures how large perturbations of $\flift$'s arguments must be in order to achieve a desired perturbation of the arguments of $\fcvx$, under the reparameterization $\Phi$. The additional arguments in $\flift$ compared to $f$ adds additional columns to $\nabla\, \Phi$ thereby making it easier to ensure $\sigma_{d_z}(\nabla \Phi(\cdot)) > 0$.
On the other hand, the denominator measures the Euclidean distance between any minimizer of the convex function $\fcvx(\cdot)$ and image of $(\farg,\xiarg)$ under the reparameterization $\Phi$, and can be bounded under quite benign conditions.  
\iftoggle{arxiv}
{
	\begin{proof}[Proof Sketch of \Cref{thm:DCL}] 
	The formal proof of \Cref{thm:DCL} (given in \Cref{app:thm:DCL}) takes special care to handle  that allow $\flift$ and $\fcvx$ to be finite only on restricted domains, and possible non-smoothness; still, the main ideas behind are  intuitive.  

	For $\fcvx$  convex, $\fcvx(\cvxarg) - \inf(\fcvx) = \cO(\norm{\grad \fcvx(\cvxarg)})$. Using the \DCL{} definition and analyzing the inverse image of a point under $\Phi$, we can also establish a gradient domination  result for $\flift$. Weak-PL result for the original function $f$ follows  since $f$ is related to $\flift$ by partial minimization, so its gradients must be larger than those of $\flift$.
	\end{proof}
}
{	
\emph{Proof Sketch.} The formal proof of \Cref{thm:DCL} (given in \Cref{app:thm:DCL}) takes special care to handle  that allow $\flift$ and $\fcvx$ to be finite only on restricted domains, and possible non-smoothness; still, the main ideas behind are  intuitive.  For $\fcvx$  convex, $\fcvx(\cvxarg) - \inf(\fcvx) = \cO(\norm{\grad \fcvx(\cvxarg)})$. Using the \DCL{} definition and analyzing the inverse image of a point under $\Phi$, we can also establish a gradient domination  result for $\flift$. Weak-PL result for the original function $f$ follows  since $f$ is related to $\flift$ by partial minimization, so its gradients must be larger than those of $\flift$.
}


\iftoggle{arxiv}
{
	\subsection{Gradient descent with \texttt{DCL}s. }
}
{
\textbf{Gradient descent with \texttt{DCL}s}
} We now describe how \DCL s yield quantitative convergence guarantees for gradient descent. A more general guarantee accommodating  the reconditioning step in \algname{} is deferred to \Cref{app:gd_reco_main}, and encompasses the bound below as a special case. Given $\alpha > 0$, we say that proper $f:\R^d \to \Rebar$ satisfies $\alpha$-\emph{weak-PL} (named after the stronger Polyak-\L{}ojasiewicz condition) on a domain $\calK \subset \R^d$ if $f \in \ccone(\calK)$ and \iftoggle{arxiv}
{
	\begin{align*}
	\forall \farg \in \calK, \quad \|\nabla f(\farg)\| \ge \alpha (f(\farg) - \inf(f). \tag{$\alpha$-weak-PL}
	\end{align*}
}
{
$\|\nabla f(\farg)\| \ge \alpha (f(\farg) - \inf(f)$.
}  From \Cref{thm:DCL}, we see $f$ satisfies $\alpha_{\calK}$-weak PL on $\calK$ if $f$ has a \DCL{} and $\alpha_{\calK} := \inf_{\farg \in \calK} \alphadcl(\farg) > 0$. To analyze gradient descent, we also require smoothness: we say $f$ is \emph{$\beta$-upper-smooth} on $\calK$ if $f \in \cctwo(\calK)$ and for all $\farg \in \calK$, \iftoggle{arxiv}{
\begin{align*}
\nablatwo f(\farg) \preceq \beta \bI \tag{$\beta$-upper-smooth}. 
\end{align*}
}
{ $\nablatwo f(\farg) \preceq \beta \bI$.} The following follows from a standard descent lemma for smooth (though possibly nonconvex) functions.
\begin{proposition}\label{prop:weak_pl_gd} Let $\farg_0 \in \dom(f)$, and suppose that the level set $\calK(\farg_0) := \{\farg:f(\farg) \le f(\farg_0)\}$ is compact, and $f$ satisfies $\alpha_{\farg_0} > 0$-weak PL and $\beta_{\farg_0} > 0$-upper smoothness on $\cK (\farg)$.  Then, $\farg_0$ lies in the same path-connected component of some minimizer of $f$, and for any $\eta \le 1/\beta_{\farg_0}$ the updates $\farg_{k+1}  = \farg_k - \eta\nabla f( \farg_k)$ satisfy $f(\farg_k) - \inf(f) \le 2/(k\cdot\alpha_{\farg_0}^2 \eta )$.
\end{proposition}

\iftoggle{arxiv}
{
	\input{body/gd_reco_sec}
}
{
	\subsection{Proof overview of \Cref{thm:informative_optimal,thm:main_rate}}\label{sec:body_proof_sketches}
	\input{colt_only/short_proof_sketch_of_thms_colt}
}

%% file: body/gd_reco_sec.tex

\subsection{Gradient descent with reconditioning}\label{app:gd_reco_main}

\iftoggle{arxiv}
{
	 We now extend \Cref{prop:weak_pl_gd}  to accommodate  the reconditioning step in \algname{} (\Cref{alg:our_algorithm}). Here, we state guarantees which establish both quantitative convergence rates and, under slightly stronger conditions, path-connectedness to global minimizers. All proofs are deferred to \Cref{app:gd_reco}.
}
{
	Before outlining the formal steps of our main results, we provide analyze gradient descent under the weak-PL condition. This generalizes \Cref{prop:weak_pl_gd} to accomodate the reconditioning step in \algname{} (\Cref{alg:our_algorithm}). All proofs are deferred to \Cref{app:gd_reco}.
}  

\begin{restatable}[Reconditioning matrix]{definition}{defnreco}\label{defn:reco} Given $f: \R^d \to \Rebar$, we say that $\bLambda: \dom(f) \to  \psd{n}$ is a reconditioning matrix for $f$ if it is continuous on $\dom(f)$, and for every $\farg \in \dom(f)$ such that $\bLambda(\farg) \succ 0$, there exists an $\farg' \in \dom(f)$ such that $\bLambda(\farg') = \eye_{n}$ and $f(\farg') = f(\farg)$. We define the set $\recond_{\bLambda}(\bx) := \{\farg': f(\farg') = f(\farg), \quad \bLambda(\bx) = \eye_n\}$ as the set of such points. We say $\farg$ is reconditioned if $\bLambda(\farg) = \eye_n$. 
\end{restatable}
\begin{observation}  $\bLambda(\sfK) = \bSigma_{\sfK,22}$ is a reconditioning matrix for the loss $\cL_{\lambda(\cdot)}$.
\end{observation}
\begin{proof} Since $\dom(\cL_{\lambda})  = \calKexp \subset \calKfull$, $\Sigktwo \succ 0$ on $\dom(\cL_{\lambda})$. As observed in \Cref{eq:recond}, there is a similarity transformation mapping $\sfK \in \calKexp$ some $\sfK'$ with $\bSigma_{\sfK',22} = \eye_n$. Since $\cL_{\lambda}$ is invariant under similarity transformation, it follows $\cL_{\lambda}(\sfK') = \cL_{\lambda}(\sfK)$.
\end{proof}

 Reconditioning serves to ensure that $f$ need only be well-behaved (i.e. satisfy upper-smoothness and weak-PL for suitable constants) on a restricted set of approximately reconditioned parameters $\farg: \bLambda(\farg) \approx \eye_n$.

The following proposition is the guiding template for the overall convergence analysis. Its proof is given in \Cref{app:proof:reco_descent}.

\begin{proposition}\label{prop:reco_descent} Let $f:\R^d \to \Rebar$, $\farg_0 \in \dom(f)$, and let $\bLambda$ be a reconditioning matrix for $f$ such that $\bLambda(\farg_0) \succ 0$. Define $\calK(\farg_0)$ as the following reconditioned level set, which we assume is closed:
\begin{align}
\calK(\farg_0) := \left\{\farg \in \R^d: f(\farg) \le f(\farg_0) \,\text{ and  }\, \|\bLambda(\farg)- \eye_n\|_{\op} \le \frac{1}{2}\right\}. \label{eq:farg_0}
\end{align}
Assume that the function $\farg \mapsto \bLambda(\farg)$ is $\LrecoK$-Lipschitz as a mapping from $(\R^d,\|\cdot\|) \to (\psd{n},\|\cdot\|_{\op})$ and that $f$ is $\betaK$-upper-smooth, $\LfK$-Lipschitz, and satisfies the $\alphaK$-weak PL condition for points in $\calK(\farg_0)$. Lastly, let $\{ \eta_k\}_{k=0}^\infty$ be a series of step sizes such that $0 < \inf_{k} \eta_k \le \sup_k \eta_k \le \min\{\frac{1}{\betaK},\frac{1}{2\LfK\LrecoK}\}$. If iterates are chosen according to, 
\begin{align}
\tilde\farg_k \in \recond_{\bLambda}(\farg_k), \quad \farg_{k+1}  = \tilde \farg_k - \eta_k\nabla f(\tilde \farg_k) \label{eq:reco_updates_a},
\end{align}
or the more general condition,
\begin{align}
\tilde\farg_k \in \recond_{\bLambda}(\farg_k), \quad \farg_{k+1} \text{ satisfies } f(\farg_{k+1}) \le  f(\tilde \farg_k - \eta_k\nabla f(\tilde \farg_k)) \label{eq:reco_updates_b},
\end{align}
then for all $k \geq 1$ it holds that
\begin{align}
f(\farg_{k})\le  \frac{2}{\alphaK^2 \eta} \cdot \frac{1}{k}, \quad \text{ where } \eta := \inf_{k \ge 1}\eta_k. \label{eq:guaranteed_reco_descent}
\end{align}
\end{proposition}

\Cref{prop:reco_descent} can also be used to establish that every $\farg_0 \in \dom(f)$ is in the path-connected component of some $\fargst \in \argmin(f)$. To do so, we need the matrix operator to be connected in the following sense:
\begin{definition} We say that a reconditioning matrix $\bLambda: \dom(f) \to  \psd{n}$ is \emph{connected} if there exists a parametrized operator $\overline\recond_{\bLambda}(\cdot,\cdot): \dom(f) \times [0,1] \to \psd{n}$ such that (a) $\overline\recond_{\bLambda}(\cdot,\cdot)(\farg,0) = \farg$ (b) $\overline\recond(\farg,1) = \recond(\farg)$, and (c) for all $\farg \in \dom(f)$, $t \mapsto \overline\recond(\farg,t)$ is connected, and its image lies in $\dom(f)$.
\end{definition}
\begin{observation}The reconditioning matrix $\bLambda(\sfK) = \bSigma_{\sfK,22}$ for the loss $\cL_{\lambda(\cdot)}$ is connected.
\end{observation}
\begin{proof} Define $\overline\recond(\sfK,t) := \Similar_{\bS_t}(\Ak,\Bk,\Ck)$, where $\bS_t = \Sigktwo^{-t/2}$. Since similarity transforms preserve membership in $\calKexp$, and since $t\mapsto\overline\recond(\sfK,t)$ is continuous and coincides with $\sfK$ at $t = 0$ (resp. $\recond(\sfK)$ at $t = 1$), the observation follows. 
\end{proof}
The following proposition, proved in \Cref{app:proof:connected_descent}, establishes path-connectedness for connected reconditioning matrices. 
\begin{proposition}\label{prop:con_comp} Consider the set up of \Cref{prop:reco_descent} with $\farg_0 \in \dom(f)$, and in addition, suppose (a) that $\bLambda(\cdot)$ is continuous reconditioning matrix and (b) the set $\cK(\farg_0)$ is compact. Then, there exists an $\fargst \in \argmin(f)$ and a path $\gamma:[0,1] \to \dom(f)$ such that $\gamma(0) = \farg_0$ and $\gamma(1) = \fargst$. 
\end{proposition}
\Cref{prop:weak_pl_gd} can be recovered as the special case when the reconditioning matrix $\recond_{\bLambda}(\farg) \equiv \eye_d$ is always the identity. In this case, the reconditioning step is vacuous. Moreover $\LrecoK = 0$ ($\recond_{\bLambda}$ is constant), and it is straightforward to modify the proof of \Cref{prop:reco_descent} to dispense with the dependence on $\LfK$. 

%% file: body/proof_of_main_thm.tex

\iftoggle{arxiv}
{
    \input{arxiv_only/short_proof_sketch_of_thms_arxiv}
}
{
    \input{body/gd_reco_sec}
}

\subsection{Proof of \Cref{thm:main_rate}}\label{ssec:main_rate_proof}

With key ingredients of the analysis  in mind, we now finish the proof of  \Cref{thm:main_rate} by illustrating the existence of a \DCL{} for the regularized \OE{} problem, and establishing smoothness and Lipschitzness of the objective when restricted to the reconditioned set so as to apply \Cref{prop:reco_descent}. More specifically, we first establish the relevant properties ``locally'', in that they depend on the choice of the filter $\sfK$, and then prove a uniform bound over all $\sfK$ in the reconditioned set at the very end. A recurring theme is that both the weak-PL and the smoothness properties are controlled by the informativity, as measured by $\|\Zk^{-1}\|$. These are terms are also controlled by $\|\Sigk\|,\|\Sigk^{-1}\|$, which we show below are bounded in terms of $\|\Sigktwo\|,\|\Sigktwo^{-1}\|$, which are both bounded due to the reconditioning step. 

As shorthand, we let $\polyop(\bX_1,\bX_2,\dots,\kappa)$ denote a term which is at most a polynomial function of the operator norm of the matrix arguments $\|\bX_1\|,\|\bX_2\|_,\dots$, and a polynomial in the scalar argument $\kappa$; $\|\cdot\|_{\ell_2}$ denotes the Euclidean norm (e.g. on parameters $\sfK = (\Ak,\Bk,\Ck)$). 
All results below assume $\sfK \in \calKexp$, and that  $\Sigk$ is invertible (we verify this condition in \Cref{lem:Sigma_K_conditioned} below.) 

\paragraph{A \DCL{} for the regularized \OE{} objective.} While it is by now well-known within the controls community that the \OE{} problem admits a convex reformulation \citep{scherer97lmi}, we prove a stronger result showing that this reformulation is in fact a \DCL{}. We prove the following result in \Cref{sec:step2_weak_PL_verify}.

\begin{proposition}\label{prop:DCL_for_Kalman} For any $\lambda \ge 0$ (non-strict), the objective $\cL_{\lambda}(\sfK)$ admits a \DCL{} $\dcltriple$  where the lifted parameter takes the form $(\sfK,\Sigk) \in \calKexp \times \pd{2n}$,  $\cL_{\lambda}(\sfK) = \flift(\sfK,\Sigk) = \min_{\bSigma \in\psd{2n}} \flift(\sfK,\bSigma)$, and where 
\begin{align*}
\sigma_{d_z}(\nabla\, \Phi(\sfK,\Sigk)) &\ge 1/\polyop\left(\bA,\bC,\bW_2^{-1},\Sigk, \Sigk^{-1}, \Zk^{-1},\Loe(\sfK) \right)\\
\|\Phi(\sfK,\Sigk)\|_{\ell_2} &\le (\max\{n,\sqrt{mn}\}+\sqrt{\Loe(\sfK)})\cdot\polyop\left(\bA,\bC,\bW_2^{-1},\Sigk, \Sigk^{-1}, \Zk^{-1}\right).
\end{align*}
Furthermore, the norms of the parameters $\Ak,\Bk,\Ck$ satisfy the following bounds:
\begin{align}
\max\{\|\Ak\|_{\op},\|\Bk\|_{\op}\} \le\polyop\left(\bA,\bC,\bW_2^{-1},\Zk^{-1},\Sigk,\Sigk^{-1} \right),\quad \|\Ck\|_{\fro} \le \sqrt{\Loe(\sfK)/\|\Sigk^{-1}\|}. \label{eq:par_bound}
\end{align}
\end{proposition}
 Recall that the domain of $\cL_{\lambda}(\sfK)$ is the set $\calKexp$, on which $\Zk$ and (as noted above) $\Sigk$ are invertible. Hence, all quantities in the above lemma are well-defined. Having established the existence of a \DCL{}, a direct application of \Cref{thm:DCL} shows that this objective satisfies the weak-PL property.
\begin{corollary}[Weak-PL Property of $\cL_{\lambda}$]\label{cor:weak_PL} For any $\lambda \ge 0$ and $\sfK \in \calKexp$,
\begin{equation}
\begin{aligned}
&\|\nabla \cL_{\lambda}(\sfK)\| \ge \frac{1}{\Cpl(\sfK) \cdot \max\{n,\sqrt{mn}\}} \cdot\left(\cL_{\lambda}(\sfK) - \inf(\cL_{\lambda}) \right), \quad \text{where}\\
&\qquad  \Cpl(\sfK) = \polyop\left(\bA,\bC,\bW_2^{-1},\Zk^{-1},\Sigk,\Sigk^{-1},\Loe(\sfK) \right).
\end{aligned}
\end{equation}
\end{corollary}

\paragraph{Smoothness and Lipschitzness  of $\cL_{\lambda}(\sfK)$.} 
To verify these regularity conditions, we need to bound the norms of various quantities, which are themselves the solutions to Lyapunov equations involving the closed-loop system matrix $\Aclk$ (defined in \Cref{eq:aclk}). The main step is therefore to show that the solutions to these Lyapunov equations are uniformly bounded, as per the following lemma (proof in \Cref{app:Lyap}).

\begin{proposition}[Stability of $\Aclk$]\label{prop:clyap_compact_form} Suppose that $\sfK \in \calKexp$. Then, for any matrix $\bY \in \sym{2n}$, the solution $\bSigma_{\sfK,\bY}$ to the Lyapunov equation $\Aclk \bSigma_{\sfK,\bY} +\bSigma_{\sfK,\bY} \Aclk^\top + \bY = 0$ satisfies 
\begin{align*}
\circnorm{\bSigma_{\sfK,\bY}} \le \conslyapK \cdot \circnorm{\bY}, \quad \text{where } \conslyapK  = \polyop\left(\Sigk,\Sigk^{-1},\Zk^{-1},\bW_1^{-1}, \bW_2^{-1},\bC \right),
\end{align*}
and where $\circnorm{\cdot}$ denotes either the operator, Frobenius, or nuclear norm.
\end{proposition}
Using this intermediate result, we can bound the norms of the various derivatives which govern the smoothness and Lipschitz constants for the regularized $\OE{}$ problem. We present the proof of the following result in \Cref{app:max_smooth}, as well as formal explanations of the notation of the norms below.  
\begin{restatable}[Smoothness and Lipschitzness]{proposition}{dercomps}\label{prop:der_bounds}  For any $\sfK \in \calKexp$, $\cL_{\lambda}(\cdot)$ is $\cctwo$ in an open neighorhood containing $\cK$, and 
\begin{align}
\|\nablatwo \cL_{\lambda}(\sfK)\|_{\ell_2\to \ell_2} &\le \Cgradtwo(\sfK) \cdot \conslyapK^2 \cdot (1+\lambda)  \tag*{(Local smoothness)} \\
\|\nabla \cL_{\lambda}(\sfK)\|_{\ell_2} &\le  \Cgradone(\sfK)\cdot \conslyapK\cdot  (1+\lambda)\sqrt{n}   \tag*{(Lipschitz loss)} \\
\|\nabla \,\Sigktwo\|_{\ell_2 \to \op} &\le \Csigone(\sfK) \cdot \conslyapK  , \quad \tag*{(Lipschitz reconditioning)} 
\end{align}
where $\Csigone(\sfK) = \polyop(\Sigk,\Bk, \bC,\bW_2)$, where
\begin{align*}
\Cgradone(\sfK),\Cgradtwo(\sfK) = \polyop(\Zk^{-1}, \Sigktwo^{-1},\Sigk,\Bk, \Ck,\bC,\sO,\bW_2),
\end{align*}
where $\conslyapK$  is as in \Cref{prop:clyap_compact_form}, and where the gradient norms are in the Euclidean geometry.
\end{restatable}
\paragraph{Concluding the proof: uniform parameter bounds.} Note again that bounds above are local, in that they depend on the choice of filter $\sfK$. To finish the proof of \Cref{thm:main_rate}, we prove a uniform bound over all filters $\sfK$ which lie in the set considered by \Cref{prop:reco_descent}, namely.
\begin{align}
\cK_0 := \Big\{\sfK \in \calKexp: \cL_{\lambda}(\sfK) \le \cL_{\lambda}(\sfK_0) \text{ and } \frac{1}{2}\eye_{n} \preceq \Sigktwo \preceq 2\eye_n \Big\}.  \label{eq:cKnot}
\end{align} 
Immediately, we see that on this set $\|\Sigktwo^{-1}\| \le 2$, and that
\begin{align*}
\Loe(\sfK) \le \cL_{\lambda}(\sfK) \le \cL_{\lambda}(\sfK_0), \quad \|\Zk^{-1}\| \le \trace[\Zk^{-1}] = \regexp(\sfK) \le \frac{1}{\lambda}\cL_{\lambda}(\sfK) \le \frac{1}{\lambda}\cL_{\lambda}(\sfK_0).
\end{align*}
As a consequence, we can bound the terms appear in the bounds above as follows (see \Cref{app:intermediate_simplification}):
\begin{lemma}\label{lem:constant_intermediate_simplification.} The terms $\Cpl(\sfK),\conslyapK,\Csigone(\sfK),\Cgradone(\sfK),\Cgradtwo(\sfK)$ appearing above are all bounded by at  most $\polyop(\Sigk^{-1},\Sigk,\bA, \bC,\sO,\bW_2,\bW_2^{-1},\bW_1^{-1},\cL_{\lambda}(\sfK_0),\frac{1}{\lambda})$. 
\end{lemma}
Lastly, we control the dependence on $\Sigk$ and $\Sigk^{-1}$. The follow lemma is proven in \Cref{app:Sig_K_conditioned}. 
\begin{lemma}\label{lem:Sigma_K_conditioned} Let $\sigst > 0$ be as in we mean \Cref{lem:sigst}. Then, for any $\sfK \in \calKfull$, it holds that: (a) $\Sigk \succ 0$ is invertible, (b) $\|\Sigk^{-1}\|  \le 2\|\Sigktwo^{-1}\| + 2\sigst^{-1}\max\{1, \|\Sigktwo^{-1}\|\|\Sigonesys\|\}$, and (c) $\|\Sigk\| \le 2\max\{\|\Sigktwo\|,\|\Sigonesys\|\}$. 
\end{lemma}
In particular, on $\cK_0$, where $\|\Sigktwo^{-1}\|,\|\Sigktwo\|\le 2$, we have  $\|\Sigk\|,\|\Sigk^{-1}\| \le \polyop(\|\Sigonesys\|,\sigst^{-1})$, so that the terms $\Cpl(\sfK),\conslyapK,\Csigone(\sfK),\Cgradone(\sfK),\Cgradtwo(\sfK)$ are all at most polynomial in  
\#\label{def:Csys}
\Csys := \max\{\|\Sigonesys\|,\|\bA\|, \|\bC\|,\|\sO\|,\|\bW_2\|,\|\bW_2^{-1}\|,\|\bW_1^{-1}\|,\sigst^{-1}\}, 
\#
as well as in $\cL_{\lambda}(\sfK_0),\frac{1}{\lambda}$. Thus, from \Cref{cor:weak_PL,prop:der_bounds}, we verify the conditions of \Cref{prop:reco_descent} uniformly on the set $\cK_0$.
\begin{corollary}\label{cor:uniform_bounds} The loss function $\cL_{\lambda}$ satisfies $\alpha$-weak PL and $\beta$-upper smoothness  on $\cK_0$ with
\begin{align*}
\alpha^{-1} \le \max\{n,\sqrt{mn}\}\cdot \poly(\Csys,\cL_{\lambda}(\sfK_0),\tfrac{1}{\lambda}), \quad \beta \le \poly(\Csys,\cL_{\lambda}(\sfK_0),\tfrac{1}{\lambda},\lambda),
\end{align*}
where $\Csys$ is defined in \Cref{def:Csys}. 
In addition, on $\cK_0$, $\cL_{\lambda}$ is $L \le \sqrt{n}\poly(\Csys,\cL_{\lambda}(\sfK_0),\lambda,\tfrac{1}{\lambda})$ Lipschitz , and $\sfK \mapsto \Sigktwo$ is at most $L_{\Sigma} \le \poly(\Csys,\cL_{\lambda}(\sfK_0),\tfrac{1}{\lambda},\lambda)$ Lipschitz as a mapping from $(\calKexp,\|\cdot\|_{\ell_2}) \to (\sym{n},\|\cdot\|_{\op})$.
\end{corollary} 
Lastly, we establish compact level sets. The subtlely here is not only showing  that $\cK_0$ is bounded (this is rather direct from \Cref{prop:DCL_for_Kalman}), but also closed. 
\begin{lemma} Let set $\cK_0$ in \Cref{eq:cKnot} is compact. \label{lem:compact_sets}
\end{lemma}
The upper bound on $\cL_{\lambda}(\sfK_s) - \min_{\sfK} \cL_{\lambda}(\sfK)$ in \Cref{thm:main_rate} is now a direct consequence of instantiating \Cref{prop:reco_descent} $\eta = \eta_s$ with the bounds in the above \Cref{cor:uniform_bounds}, and noting that $\cK_0$ is closed by \Cref{lem:compact_sets}.

The inequality  $ \Loe(\sfK_s) - \min_{\sfK} \Loe(\sfK) \le \cL_{\lambda}(\sfK_s) - \min_{\sfK} \cL_{\lambda}(\sfK)$ is just a consequence of \Cref{cor:subopt_ub}. 
\qed

\subsection{Proof of \Cref{thm:informative_optimal}}\label{app:thm_informative_optimal}
Due to the \DCL{} exhbited by \Cref{prop:DCL_for_Kalman}, and in particular \Cref{cor:weak_PL}, we find that any $\lambda \ge 0$ and $\sfK \in \calKexp$ for which $\nabla \cL_{\lambda}(\sfK) = 0$ must be optimal (in applying the corollary, we again note that $\Zk$ is guaranteed to be invertible of $\sfK \in \calKexp$, and $\Sigk$ invertible by \Cref{lem:Sigma_K_conditioned}).  By taking $\lambda = 0$, we have $\nabla \cL_{\lambda}(\sfK) = \nabla \Loe(\sfK)$, proving the theorem. Path connectedness follows from \Cref{prop:con_comp}, again noting that $\cK_0$ is compact (\Cref{lem:compact_sets}). \qed

%% file: arxiv_only/short_proof_sketch_of_thms_arxiv.tex

Before providing the exact details, we begin with a high-level overview of \Cref{thm:main_rate} in \Cref{sec:sketch_main_rate} to emphasize the key ideas. The formal proof skeleton is given in the following section, \Cref{ssec:main_rate_proof}, and \Cref{app:thm_informative_optimal} proves \Cref{thm:informative_optimal}. Lastly, \Cref{sec:step2_weak_PL_verify} specifies the construction of the \DCL{} for \OE{}. The proof of all the constituent results are deferred to the appendix. 
\subsection{Proof sketch of \Cref{thm:main_rate}}\label{sec:sketch_main_rate}

In light of \Cref{prop:reco_descent}, it suffices to establish both the weak-PL and smoothness of the loss $\cL_{\lambda}(\cdot)$ on the well-conditioned sublevel set $\calK_0 := \{\sfK: \cL_{\lambda}(\sfK) \le \cL_{\lambda}(\sfK_0), \frac{1}{2}\eye_n \preceq \Sigktwo \preceq 2 \eye_n\}$ (as well as some Lipschitz bounds on $\cL_{\lambda}$ and $\sfK \mapsto \Sigktwo$). 

To establish weak-PL, we exhibit a \DCL{} for which $\alpha_{\DCL}(\sfK)$ depends only on the operator norms of $\Sigk,\Sigk^{-1},\Zk$, as well as other system-quantities (\Cref{prop:DCL_for_Kalman,cor:weak_PL}); smoothness and Lipschitz bounds ared established in (\Cref{prop:der_bounds}), which relies on a novel bound on the solutions to Lyapunov equations involving $\Aclk$ (\Cref{prop:clyap_compact_form}), also in terms of $\|\Sigk\|,\|\Sigk^{-1}\|,\|\Zk^{-1}\|$. 

The regularization $\regexp(\sfK)$ ensures $\|\Zk^{-1}\|$ remains bounded on the sublevel $\calK_0$; we further show (\Cref{lem:Sigma_K_conditioned}) that $\frac{1}{2}\eye_n \preceq \Sigktwo \preceq \frac{3}{2} \eye_n$  implies $\Sigk$ is invertible, and ensures $\|\Sigk\|,\|\Sigk^{-1}\|$ are uniformly bounded. Thus, the weak-PL constant and smoothness parameters are uniformly bounded on $\calK_0$, concluding the proof. We stress that the proofs of \Cref{prop:DCL_for_Kalman} and \Cref{prop:clyap_compact_form} require several novel technical arguments, which may be of independent interest. 

The fact that $\|\Sigk\|,\|\Sigk^{-1}\|,\|\Zk^{-1}\|$ appear throughout the analysis suggests that (a) informativity as measured by $\Zk^{-1}$, and (b) the  conditioning  of $\Sigktwo$ may be fundamental to the \OE{} landscape.

%% file: body/weak_pl.tex

\newcommand{\calA}{\mathcal{A}}
\newcommand{\calQ}{\mathcal{Q}}
\newcommand{\calW}{\mathcal{W}}
\newcommand{\Ccvx}{\mathcal{C}_{\subfont{cvx}}}

\newcommand{\calC}{\mathcal{C}}
\newcommand{\Clift}{\calC_{\subfont{lft}}}
\newcommand{\dstate}{n}
\newcommand{\dobs}{m}
\newcommand{\dout}{p}

In this section, we establish the weak-PL property of our 
 regularized loss function $\lossreg(\cdot) = \Loe(\cdot) + \lambda \regexp(\cdot)$. Our strategy is to show that $\lossreg(\cdot)$ admits a \DCL{}, which leads to a weak-PL constant $\alpha(\sfK)$ for each $\sfK$, whose parameters are themselves bounded in terms of $\lossreg(\cdot)$. Before continuing, we recall that $n$ denotes the dimension of the system state $\sx$ (and internal state $\fx$), $m$ of the observation $\sy$, and $p$ the output $\po$, and that   $\polyop(\bX_1,\bX_2,\dots,\kappa)$ denote a (universal) polynomial function of operator norm of matrix, arguments $\|\bX_1\|,\|\bX_2\|_,\dots$, and a polynomial in scalar argument $\kappa$. We use $\I$ to denote the $1$-$\infty$ indicator, i.e. for some event $\cE$, $\I\{\cE\}=1$ if $\cE$ is true, and $\I\{\cE\}=+\infty$ otherwise.

All proofs of the lemmas that follow are deferred to \Cref{sec:kal_dcl_proofs}. 
To proceed, we need to invoke \Cref{thm:DCL} by specifying  the \DCL{} of the function
\begin{align*}
\lossreg(\sfK) = \Loe(\sfK) + \lambda\regexp(\sfK) = \lim_{t \to \infty} \Exp[\|\po(t) - \fo(t)\|^2] + \lambda\trace[\Zk^{-1}]. 
\end{align*}
Throughout, given a matrix $\bSigma \succ 0$ partitioned in $2\times 2$ blocks, we more generally define
\begin{align}
\matZ(\bSigma) := \bSigma_{12} \bSigma_{22}^{-1}\bSigma_{12}^\top. 
\end{align}
With the above notation, we can express
\begin{align}
\cL_{\lambda}(\sfK) &=\lim_{t \to \infty} \Exp\left[ \left\|\sO \sx(t) - \Ck \fx(t)\right\|^2 \right] + \lambda\cdot\trace\left[\Zk^{-1}\right] \nonumber\\
&=\lim_{t \to \infty} \trace\left[\begin{bmatrix}
	\sO & -\Ck
\end{bmatrix}  \Exp\left[\begin{bmatrix} \sx(t) \\ \fx(t) \end{bmatrix}\begin{bmatrix} \sx(t) \\ \fx(t) \end{bmatrix}^\top\right] \begin{bmatrix}
	\sO^\top \\ -\Ck^\top
\end{bmatrix}\right]  + \lambda\trace\left[\matZ(\Sigk)^{-1}\right] \nonumber\\
&= \trace\left[\begin{bmatrix}
	\sO & -\Ck
\end{bmatrix}  \Sigk \begin{bmatrix}
	\sO^\top \\ -\Ck^\top
\end{bmatrix}\right]  + \lambda\trace\left[\matZ(\Sigk)^{-1}\right].  \label{eq:llam_simplify}
\end{align}
This leads to the following notion of the lifted function. 
\begin{definition}[The lifted function] We define the lifted function on the space of parameters $(\sfK, \bSigma) \in \calKexp \times \sym{2n}$  as follows
\begin{subequations}
\begin{align}
&\flift(\sfK, \bSigma)=\left( \trace\left[\begin{bmatrix}
	\sO & -\Ck
\end{bmatrix} \bSigma \begin{bmatrix}
	\sO^\top \\ -\Ck^\top
\end{bmatrix} + \lambda \cdot\trace\left[\matZ(\bSigma)^{-1}\right] \right]\right)\cdot\I\{(\sfK,\bSigma) \in \Clift,\}\label{equ:ftil_def_obj}\\
&\Clift := \left\{(\sfK,\bSigma) : \begin{matrix*}\\
&(i)~\bSigma\succ 0, \,\matZ(\bSigma) \succ 0 \quad (ii)~\bA\bSigma_{11} +\bSigma_{11}\bA^\top +\bW_1 = 0\\
&(iii)~\begin{pmatrix} \bA & 0\\
\Bk\bC & \Ak\end{pmatrix}\bSigma+\bSigma\begin{pmatrix} \bA & 0\\
\Bk\bC & \Ak\end{pmatrix}^\top +\begin{pmatrix} \bW_1 & 0\\
0 & \Bk\bW_2\Bk^\top \end{pmatrix} \preceq 0
\end{matrix*}\right\}.\label{equ:ftil_def_constrain}
\end{align}
\end{subequations}
We extend $\flift(\sfK, \bSigma)$ to the space of all (unconstrained, even possible unstable) filters $\sfK = (\Ak,\Bk,\Ck)$ by setting the lifted function to be infinte when $\sfK \notin \calKexp$:  $\flift(\sfK, \bSigma) = \flift(\sfK, \bSigma)\I\{\sfK \in \calKexp\}$.\footnote{This formalism is just to accomodate for the fact that we encode constraints on domains in the function in general \DCL{} framework.}
\end{definition}


\paragraph{Step 1. Verifying the lifting. }  We first verify that $\flift$ is indeed a lifted function of $\cL_{\lambda}$. 
\begin{lemma}\label{lem:minimizer} For any feasible $\sfK\in \calKexp$, 
\begin{align*}
\cL_{\lambda}(\sfK) = \min_{\bSigma\in\sym{2n}} \flift(\sfK,\bSigma),
\end{align*}
and this minimum is attained for $\bSigma = \Sigk$.
\end{lemma}

\paragraph{Step 2. Convex reparametrization.} Next, we introduce the transformation $\Phi$:
\begin{definition}\label{defn:phidef} We define the convex parameter $\bnu :=  (\bL_1, \bL_2,\bL_3, \bM_1, \bM_2)$ and the transformation 
\begin{subequations}
\begin{align}\label{eq:phidef}
\bnu^\top=\Phi(\sfK,\bSigma) &:= \begin{pmatrix} \bU(\Ak \bV^\top + \Bk \bC (\bSigma)_{11})+ (\bSigma^{-1})_{11} \bA (\bSigma)_{11} \\
	\bU\Bk\\
	\Ck\bV^\top \\
	(\bSigma^{-1})_{11} 
	\\ (\bSigma)_{11}
	\end{pmatrix}, \\
	&\text{where } \begin{pmatrix}
	\bU \\ \bV \end{pmatrix} := \begin{pmatrix}(\bSigma^{-1})_{12} \\ (\bSigma)_{12}\end{pmatrix}. \label{eq:UV_def}
\end{align} 
\end{subequations}
\end{definition}

We let $d_{\nu}$ denote the dimension of the parameter $\bnu$  and let $d_y$ denote the dimension of the parameters $(\sfK,\bSigma)$, both as Euclidean vectors. One can then verify that $d_{\nu} \le d_{y}$; that is, the lifted function indeed has more parameters than the convex one. The following shows that there exists a convex function $\fcvx$,  which completes the \DCL:
\begin{lemma}\label{lemma:convex_kalman_par} There exists a convex function $\fcvx: \R^{d_{\nu}} \to \Rebar$ such that
\begin{align}
\flift(\sfK,\bSigma) &= \fcvx(\Phi(\sfK,\bSigma)).
\end{align}

\end{lemma}
The transformation $\Phi$ and associated convex function $\fcvx$
was first developed by \cite{scherer97lmi}, cf. also \cite{masubuchi1998lmi} for contemporaneous independent work.

\newcommand{\Uk}{\bU_{\sfK}}
\newcommand{\Vk}{\bV_{\sfK}}

\paragraph{Step 3. Controlling the weak-PL constant.} Lastly, we show that the \DCL{} lends itself to a bounded PL constant by invoking \Cref{thm:DCL}. To do this, we need to show that the image of $\Phi(\sfK,\bSigma)$ is not too large, and that $\nabla \Phi(\cdot)$ has rank at least $d_{\nu}$. We establish both in sequence.  Let $\Uk$ and $\Vk$ be corresponding to \Cref{eq:UV_def} with $\bSigma = \Sigk$, i.e. 
\begin{align}
\Uk = (\Sigk^{-1})_{12}, \quad \Vk = (\Sigk)_{12}.
\end{align}

\begin{lemma}[Parameter compactness]\label{lem:compact_level_sets} Consider $(\sfK,\Sigk)$, where $\Sigk$ is the stationary covariance associated with $\sfK$. Then, 
\begin{align}
\|\Phi(\sfK,\bSigma)\|_{\ell_2} &\le (\max\{n,\sqrt{nm}\} + \sqrt{\Loe(\sfK)}) \cdot \polyop(\bA,\bC,\bW_2^{-1},\Sigk,\Sigk^{-1}), \label{eq:nu_bound}
\end{align} 
where $\|\bnu\|_{\ell_2} := \sqrt{\sum_{i=1}^3 \|\bL_i\|_{\fro}^2 + \sum_{j=1}^2 \|\bM_i\|_{\fro}^2}$ denotes the Euclidean norm of the parameter $\bnu$. Moreover, if $\Uk$ and $\Vk$ are invertible, then the filter parameters are bounded by 
\begin{align*}
\max\left\{\|\Ak\|,\|\Bk\|\right\} \le \polyop(\bA,\bC,\bW_2^{-1},\Sigk,\Sigk^{-1},\Uk^{-1},\Vk^{-1}), \quad 
\|\Ck\|_{\fro} \le \sqrt{\Loe(\sfK)/\|\Sigk^{-1}\|}. 
\end{align*}
\end{lemma}
\begin{lemma}[Conditioning of $\nabla\,\Phi$]\label{lem:phi_cond} Suppose that $\sfK \in \calKexp$. Then, $\Phi$ is differentiable in an open neighborhood of $(\sfK,\Sigk)$, and if  $\Uk$ and $\Vk$  are invertible, 
\begin{align*}
\frac{1}{\sigma_{d_\nu}(\nabla\, \Phi(\sfK,\Sigk))} &\le \polyop\left(\bA,\bC,\Sigk,\Sigk^{-1},\Ak,\Bk,\Ck,\Uk^{-1} ,\Vk^{-1} \right)\\
&\le \polyop\left(\bA,\bC,\bW_2^{-1},\Sigk,\Sigk^{-1},\Uk^{-1} ,\Vk^{-1},\Loe(\sfK) \right),
\end{align*}
where the last line is a consequence of \Cref{lem:compact_level_sets}.
\end{lemma}
To conclude, we eliminate dependencies on $\Uk$ and $\Vk$:
\begin{lemma} \label{lem:UV_inverse} If $\bZ = \bZ(\bSigma)$ is invertible, the matrices $\bU = (\bSigma^{-1})_{12}$ and $\bV = \bSigma_{12}$ are invertible, and their inverses are bounded in operator norm as
\begin{align*}
\|\bU^{-1}\| \le \sqrt{\|\bZ^{-1}\| \|\bSigma^{-1}\| }, \quad \|\bV^{-1}\| \le \|\bSigma\|\sqrt{\|\bSigma^{-1}\|^3\|\|\bZ^{-1}\|}.
\end{align*}
As a consequence of \Cref{lem:phi_cond,lem:compact_level_sets},
\begin{align}
\max\left\{\|\Ak\|,\|\Bk\|,\frac{1}{\sigma_{d_y}(\nabla\, \Phi(\sfK,\Sigk))} \right\} \le \polyop\left(\bA,\bC,\bW_2^{-1},\Sigk,\Sigk^{-1},\Zk^{-1},\Loe(\sfK) \right). \label{eq:final_bound}
\end{align}
\end{lemma}
The conclusion of \Cref{eq:final_bound} and the bound $\|\Ck\|_{\fro} \le \sqrt{\Loe(\sfK)/\|\Sigk^{-1}\|}$ from \Cref{lem:compact_level_sets} are precisely the conclusions of \Cref{prop:DCL_for_Kalman}. \qed

%% file: body/conclusion.tex
The work introduces  the first policy search algorithm which converges to the globally optimal \emph{dynamic} filter for the output estimation problem. We hope that our analysis serves as a valuable starting point to study direct policy search for reinforcement learning and control problems with partial observations, in which the relevant class of policies are dynamic and maintain internal state. We also hope that both our proposed principle of informativity, and our technical contributions around convex reformulations, continue to prove useful in future work. 

%% file: arxiv_only/app_organization_arxiv.tex

\section{Organization of the Appendix}\label{app:organization_summary}

After establishing notation below, the rest of the appendix is organized as follows. \Cref{app:alg_details} provides additional algorithm details, including extension to backtracking and discussion of oracle complexity. \Cref{app:oracle_details} sketches implementation via finite-sample, finite-horizon oracles; notably, \Cref{app:subsample} describes implementation with an oracle which does not require direct access to system states, but which rather ``subsamples'' outputs at various time steps. 

\Cref{app:ctrb_assumption} provides further discussion on the somewhat-nonstandard controllability assumption, \Cref{asm:ctrb_of_opt}, and demonstrates it holds generically.  \Cref{app:control_proofs} contains assorted results about our assumptions and various other control-theoretic considerations. 
\Cref{app:control_proofs} also contains the proofs of various other supporting results, mainly on the characterization of optimal policies and their informativity. It also shows that random (continuous) initializations are informative with probability one. Finally, \Cref{app:counterexamples} provides further details for the various counterexamples presented in \cref{sec:main_results}. 

\Cref{part:convergence} turns to the proof of our main result, \Cref{thm:main_rate}, as well as its more qualitative  statement, \Cref{thm:informative_optimal}. Recall that the high level proofs were given in \Cref{sec:proof_main_thms}, and additional details for some of the more minor steps are given in \Cref{app:proof_main_thms}. The following appendices establishing the main constituent results in the proof of our main theorems. Specifically, \Cref{sec:proof_DCL_GD} establishes the proofs for the \DCL{} framework and gradient descent for general objective functions. \Cref{sec:step2_weak_PL_verify} substantiates the framework, and exhibits a \DCL{} for our regularized loss for the \OE\ problem, using a convex reformulation due to \cite{scherer1995mixed}. \Cref{app:Lyap} then establishes that informativity translates into bounds on the norm of the solutions to Lyapunov equations involving the closed loop matrix $\Aclk$. This is one of our most technically innovative arguments. Finally, \Cref{app:max_smooth} upper bounds the norms of various first- and second-order derivatives, via somewhat standard arguments.


%% file: appendix/alg_details_more.tex
\section{Additional Algorithmic Details}\label{app:alg_details}

\subsection{Backtracking}\label{app:backtracking}

In general, the smoothness constants may be difficult to compute in a model free fashion. We show that simple modification of our algorithm based on backtracking line search also inherits provable convergence guarantees. To this end, let $\Sback$ be finite set of step sizes (to ensure the algorithm is always well defined, we assume that $\Sback$ contains $0$.) It is common practice to choosen $\Sback$ to contain geometrically decreasing sizes (see,e.g. \citet[Chapter 3]{wright1999numerical}). To choose the step sizes $\eta_t$, we search over $\Sback$ to find the step which minimizes the objective subject to the constraint  that $\Sigktwo$ remains well-conditioned, i.e. 
\begin{align}
\sfK_{s+1} &= \tilde{\sfK}_s - \eta_s \nabla_s, \text{where } \nabla_s = \nabla \cL_{\lambda}(\tilde{\sfK}_s) \text{ and } \\
\eta_s &\in \argmin_{\eta \in \Sback}\left\{\Llam(\sfK) : \tfrac{1}{2} \eye_n \preceq \Sigktwo \preceq \frac{3}{2} \eye_n, \quad \text{where } \sfK := \tilde{\sfK}_s - \eta \nabla_s \right\} \label{eq:backtracking}.
\end{align}
Note that since $0 \in \Sback$ and $\sfK = \tilde{\sfK}_s$ has  $\Sigktwo = \eye_n$, the backtracking condition is at the very least met with $\eta_s = 0$. The following modifies \Cref{thm:main_rate}, and is proven in \Cref{sec:thm_backtrack}.

\begin{thmmod}{thm:main_rate}{a}\label{thm:main_rate_backtrack} Fix $\lambda > 0$, $\sfK_0 \in \calKexp$. There are terms $\cC_1,\cC_2 \ge 1$, which are at most polynomial in $n,m,\Csys,\lambda,\lambda^{-1}$ and $\cL_{\lambda}(\sfK_0)$  such, if $\Sback$ contains a step size $\eta > 0$ satisfying stepsize $\eta \le \frac{1}{\cC_1}$, then the iterates produced by \Cref{alg:our_algorithm_backtrack} satisfy
    \begin{align*}
    \quad \Loe(\sfK_s) - \min_{\sfK} \Loe(\sfK) \le \cL_{\lambda}(\sfK_s) - \min_{\sfK} \cL_{\lambda}(\sfK) \le \frac{\cC_2}{\eta} \cdot \frac{1}{s}, \quad \forall s \ge 1.
    \end{align*}
\end{thmmod}

\begin{algorithm}[H]
\begin{algorithmic}[1]
    \State{}\textbf{Input:} Initial $\sfK_{0} \in \calKexp$, step size $\eta > 0$, regularization parameter $\lambda > 0$
    \Statex{}\algcomment{\quad \% Define $\cL_{\lambda}(\sfK) := \Loe(\cdot) + \lambda \trace[\Zk^{-1}]$}
    \For{each iteration $s= 0,1,2,\dots$}
    \State{}\textbf{Recondition } $\tilde{\sfK}_s = \recond(\sfK_s)$, where $\recond(\cdot)$ is defined in \Cref{eq:recond}.
    \State{}\textbf{Compute} $\nabla_s = \nabla \cL_{\lambda}(\tilde{\sfK}_s)$.
    \State{}\textbf{Update} $\sfK_{s+1} \gets \tilde{\sfK}_s - \eta_s \nabla_s$, where $\eta_s$ is the backtracking step from \Cref{eq:backtracking}.
    \EndFor
  \end{algorithmic}
  \caption{\algname{} with backtracking}
  \label{alg:our_algorithm_backtrack}
\end{algorithm}

\iftoggle{arxiv}{}{
\subsection{Oracle Complexity}\label{app:oracle_complexity}

At each iteration, one can compute the derivative of $\cL_{\lambda}$ using one call to $\oraceval$ (which evaluates $\Loe(\sfK_{s})$ and $\Sigk$), and one call to $\oracgrad$, which computes the gradients of these quantities. This is true because $\nabla\, \trace[\Zk^{-1}]$ admits a closed form in terms of $\Sigk$ and its gradient. The balancing step also requires only evaluation $\Sigk$, and can use an evaluation query called for the  gradient. Thus, gradient descent variant (\Cref{alg:our_algorithm}) calls one evaluation and one gradient oracle per iteration. With backtracking (\Cref{alg:our_algorithm_backtrack}), the backtracking step requires an evaluation query for all $|\Sback|$ filters of the form $\tilde{\sfK}_s - \eta \nabla_s$. In total, therefore, each iteration uses $1$ call to $\oracgrad$, and $|\Sback|+2$ calls to $\oraceval$.
}

%% file: appendix/subsampling_oracle.tex
\newcommand{\C}{\mathbb{C}}
\section{Further Details on Evaluation Oracle}
\label{app:oracle_details}

Given that our primary focus is on understanding the {\it landscape properties} of the $\OE$ problem, we leave the precise details of finite sample considerations to future work. In this section, we provide brief remarks on how one might  approximate the cost and gradients from finitely-many, finite-horizon samples. Subsequently, we describe how to implement cost and gradient evaluations without direct access to the state covariance matrix $\Sigk$ assumed in the body of the work.

\subsection{Finite-sample considerations}
\label{app:finite_sample}

\paragraph{Gradient descent with inexact gradients.} In the finite-sample regime, one uses statistical approximations to the gradients and, in the case where the stepsize is determined by line search, function evaluations. A straightforward modification of our generic analysis of gradient descent under weak-PL, \Cref{prop:reco_descent}, can establish robustness to these inexact queries. Robustness of gradient descent to error is well-known in the literature, even in generic problem settings (see \cite{scaman2020robustness}; this is also related to the stability properties established in \cite{hardt2016train}). 

\newcommand{\step}{\updelta}
\paragraph{Time discretization.}
Using digital controllers, one must implement the filter in discrete time. Given a discretization incremement $\step$,  the (Euler) discretized filter dynamics for filter $\sfK = (\Ak,\Bk,\Ck)$  is 
\begin{align}
 &\fo_{k;\updelta} = \Ck \fx_{k;\updelta}, \quad \fx_{k+1;\updelta} = (\eye_n + \step \Ak)\fx_{k;\updelta} + \Bk \sy(k\updelta) , \quad \fx_{0;\updelta} = 0\label{eq:discrete_dynamics}. 
\end{align}
\paragraph{Finite-horizon, finite-sample losses.}
Given independent trials indexed by $i = 1,2,\dots,N$, and $T$ such that $H = T/\step$ is integral, we set
\begin{align*}
\hat{\cL}_{\texttt{OE}}(\sfK) := \frac{1}{N}\sum_{i=1}^N \|\po^{(i)}(H\step) - \fo_{H;\step}^{(i)}\|^2\\
\hat{\bSigma}_{\sfK} = \frac{1}{N}\sum_{i=1}^N \begin{bmatrix} \sx^{(i)}(H\step) \\
\fx_{H;\step}^{(i)}
\end{bmatrix}\begin{bmatrix} \sx^{(i)}(H\step) \\
\fx_{H;\step}^{(i)}
\end{bmatrix}^\top.
\end{align*} 
Using stability of the filter and nominal system and well-known properties of the Euler discretization, 
\begin{equation}\label{eq:discretization_guarantees}
\begin{aligned}
|\Exp[\hat{\cL}_{\texttt{OE}}(\sfK)] - \Loe(\sfK)| &=   + \BigOh{\step + e^{-\Omega(\step H)}}\\
\|\Exp[\hat{\bSigma}_{\sfK}]  -  \Sigk\|&=  \BigOh{\step + e^{-\Omega(\step H)}},
\end{aligned}
\end{equation}
which can be made arbitrarily close to being unbiased as $\step \to 0$ and $\step H \to \infty$. Here, the term $\BigOh{\step}$ comes from a standard error analysis of the Euler discretization (c.f. e.g \citet[Theorem 1.1]{iserles_2008}), and the exponentially decaying term $e^{-\Omega(\step H)}$ from standard mixing time arguments \cite{yu1994rates}. Above, we surpress various problem dependent constants, including terms polynominal in dimension. By standard concentration inequalities (e.g. \cite{tropp2015introduction}), we can obtain finite-sample concentration with high probability:
\begin{equation}\label{eq:discretization_guarantees_concentration}
\begin{aligned}
|\hat{\cL}_{\texttt{OE}}(\sfK) - \Loe(\sfK)| &=    \BigOh{\step + e^{-\Omega(\step H)}} +\BigOhTil{\frac{1}{\sqrt{N}}}\\
\|\Exp[\hat{\bSigma}_{\sfK}]  -  \Sigk\|&=  \BigOh{\step + e^{-\Omega(\step H)}} +\BigOhTil{\frac{1}{\sqrt{N}}}. 
\end{aligned}
\end{equation}
In particular, for  $\step$ sufficiently small and $\step N$ sufficiently large, invertibility of $\Sigktwo$ (i.e. $\sfK \in \calKfull$) implies that $\hat{\bSigma}_{22,\sfK}$ is invertible with high probability. We may then define:
\newcommand{\Zhatk}{\hat{\bZ}_{\sfK}}
\newcommand{\Regexphat}{\hat{\cR}_{\texttt{info}}}
\begin{align*}
\Regexphat(\sfK) := \trace(\hat{\bSigma}_{12,\sfK}(\hat{\bSigma}_{22,\sfK})^{-1}\hat{\bSigma}_{12,\sfK}^\top),
\end{align*}
which yields the estimated regularized loss
\begin{align*}
\hat{\cL}_{\lambda}(\sfK) :=  \hat{\cL}_{\texttt{OE}}(\sfK)+ \lambda\cdot\Regexphat(\sfK).
\end{align*}
Note that, since $\sfK$ passes nonlinearly into $\Regexphat(\sfK)$ and $\hat{\cL}_\lambda(\sfK)$, these losses are in general biased estimates of $\regexp(\sfK)$ and $\cL_{\lambda}(\sfK)$. Invoking \Cref{eq:discretization_guarantees_concentration}, together with some standard matrix (and matrix-inverse) perturbation arguments, 
\begin{align*}
&|\hat{\cL}_{\lambda}(\sfK) - \cL_{\lambda}(\sfK)|\\
&\quad = |\hat{\cL}_{\texttt{OE}}(\sfK) - \Exp[\hat{\cL}_{\texttt{OE}}(\sfK)] + \lambda\cdot\big(\trace(\hat{\bSigma}_{12,\sfK}(\hat{\bSigma}_{22,\sfK})^{-1}\hat{\bSigma}_{12,\sfK}^\top)-\trace(\Sigkonetwo\Sigktwo^{-1}\Sigkonetwo^\top)\big)| 
\\
&\qquad\le \BigOh{\step + e^{-\Omega(\step H)}} +\BigOhTil{\frac{1}{\sqrt{N}}} + \text{(lower order terms)}.
\end{align*}
\paragraph{Cost evaluations.} In light of the above discussion, $\hat{\cL}_{\lambda}(\sfK)$ can be used to evaluate $\cL_{\lambda}(\sfK)$ provided the step size $\step$ is sufficiently small, horizon $H$ sufficiently large, and sample size $N$ sufficiently large. This minimics the findings of \cite{fazel18lqr, mohammadi2021convergence, malik2019derivative} in various related settings.

\newcommand{\sfU}{\mathsf{U}}
\paragraph{Gradient evaluations.} To estimate gradients of the $\cL_{\lambda}$, two strategies are possible. One can use the zeroth-order gradient estimator \citep{flaxman2005}, where one estimates the gradient by evaluating
\begin{align*}
\hat{\nabla} \cL_{\lambda}(\sfK) =  \frac{1}{M}\sum_{j=1}^M \frac{1}{\mathsf{N}(r)}\hat{\cL}_{\lambda}(\sfK + r \sfU^{(j)}) \sfU^{(j)},
\end{align*}
where $\sfU^{(j)} = (\sfU_{A}^{(j)},\sfU_{B}^{(j)},\sfU_{C}^{(j)})$ are i.i.d. parameter perturbations from a suitable, zero-mean  distribution parameters (e.g. uniform on perturbation on the unit-Frobenius ball $(\|\sfU_A\|_{\fro}^2 + \|\sfU_B\|_{\fro}^2 + \|\sfU_C\|_{\fro}^2 = 1$)), $r$ a user-defined smoothing parameter that scales the perturbation, and $\frac{1}{\mathsf{N}(r)}$ a normalization constant.  As in previous work, (\cite{fazel18lqr, malik2019derivative, mohammadi2021convergence}), one can argue that this yields an estimator of the gradient with polynomial sample complexity. As in prior work, $r$ must be chosen sufficiently small so that the perturbations do not render $\Ak$ unstable. 

Because we consider a filtering problem, rather than a control problem, it is possible to directly compute the gradients of $\hat{\cL}_{\lambda}(\sfK)$ by differentiating through the discretizated filter dynamics in \Cref{eq:discrete_dynamics} (provided $\hat{\bSigma}_{22,\sfK} \succ 0$, so that the loss is defined and differentiable). Similar concentration techniques can be deployed to establish the accuracy of this estimator as well.

\subsection{Implementation without access to system states}\label{app:subsample}
We now turn to the implementation of our algorithm without direct access to system states. For simplicity, this example considers continuous-time, infinite-horizon, and infinite-sample cost evaluations (and gradients). In essence, we provide a reduction to the oracle described in the main text.

\paragraph{Subsampled covariance matrix.} In the subsampling oracle, we have access to evaluations and gradients of the following subsampled covariance matrix:
\begin{align}
\bar{\bSigma}_{\sfK,\boldt} := \lim_{T \to \infty} \frac{1}{T}\Exp\left[\int_{0}^T\begin{bmatrix} \sy(t + t_1)\\ \sy(t + t_2) \\\dots\\\sy(t+t_k) \\ \xhatk(t)\end{bmatrix}\begin{bmatrix} \sy(t + t_1)\\ \sy(t + t_2) \\\dots\\\sy(t+t_k) \\ \xhatk(t)\end{bmatrix}^\top \rmd t\right]. \in \R^{(k+1)n \times (k+1)n} \label{eq:Sigma_subsample}
\end{align} 
Here, $\boldt = (t_1,t_2,\dots,t_k)$ is a vector of increasing sampling times $0 = t_1 < t_2 < \dots < t_k$. Introduce, for the sake of analysis, the observability matrix
\begin{align*}
\bV_{\boldt} := \begin{bmatrix} \bC\exp(t_1 \bA) \\
\bC \exp(t_2 \bA)\\
\dots\\
\bC \exp(t_k \bA)
\end{bmatrix},
\end{align*}
where $\exp(\cdot)$ denotes the matrix exponential. We make the following assumption. 
\begin{assumption}\label{asm:bold_t} We assume that $\boldt$ is selected so that the observability matrix is full-rank: $\rank(\bV_{\boldt}) = n$.
\end{assumption}
Importantly, \Cref{asm:bold_t} holds \emph{generically} when $(\bA,\bC)$ is observable, as per \Cref{asm:observability}. The following lemma makes this precise:
\begin{lemma}\label{lem:generic_full_rank} Suppose $(\bA,\bC)$ is observable, and that $k \ge n$. Then, the $\{\boldt \in \R^{k}:\rank(\bV_{\boldt}) < n\}$ has Lebesgue measure zero. In particular, if $\boldt$ are drawn from a distribution with density with respect to the Lebesgue measure (e.g., drawn $k$ points uniformly $[0,1]$, and order them in increasing order), then $\Pr[\rank(\bV_{\boldt}) = n] = 1$. 
\end{lemma}
We establish the lemma at the end of the section.

\paragraph{Subsampled losses.} One can compute that $\bar{\bSigma}_{\sfK,\boldt}$ can be partitioned in the following form 
\begin{align}
\bar{\bSigma}_{\sfK,\boldt} = \begin{bmatrix} \bar{\bSigma}_{\sfK,\boldt,11}  & \bar{\bSigma}_{\sfK,\boldt,12} \\
\bar{\bSigma}_{\sfK,\boldt,12}^\top & \bar{\bSigma}_{\sfK,\boldt,22}\end{bmatrix} = \begin{bmatrix} * & \bV_{\boldt}\bSigma_{\sfK,12} \\
\bSigma_{\sfK,12}^\top\bV_{\boldt}^\top & \bSigma_{\sfK,22}\end{bmatrix}, \label{eq:bar_sig}
\end{align} 
where $*$ is immaterial to the following discussion. We define
\begin{align}
\bar{\bZ}_{\sfK} := \bar{\bSigma}_{\sfK,\boldt,12}\bar{\bSigma}_{\sfK,\boldt,22}^{-1}\bar{\bSigma}_{\sfK,\boldt,12}^\top \in \psd{nk}. 
\end{align}
We define the subsample regularized loss as follows:
\newcommand{\Regsub}{\mathcal{R}_{\subfont{sub}}}
\begin{align}\label{eq:Clambda}
\cL_{\lambda,\subfont{sub}}(\sfK) = \Loe(\sfK) + \lambda \Regsub(\sfK), \quad \Regsub(\sfK) :=  \sum_{i=1}^n \lambda_i(\bar{\bZ}_{\sfK})^{-1}. 
\end{align}
Notice that $\Regsub(\sfK)$ is reminiscent  of the regularizer $\regexp(\sfK)$ uses the state covariance oracle. It is clear that $\Regsub(\sfK)$, and thus $\cL_{\lambda,\subfont{sub}}(\sfK)$ can be evaluated for any $\sfK$. These quantities to do need knowledge of $\bV_{\boldt}$ to be evaluated. 

\paragraph{Differentiability of $\Regsub(\sfK)$.} We now show that $\Regsub(\sfK) $ is $\cctwo$ for $\sfK \in \calKexp$. Introduce the matrix $\cP$ to be any orthogonal projection matrix from the space spanned by the image of $\bV_{\boldt}$ (which is rank $n$) to $\R^n$. Define $\tilde{\bV}$ and $\tilde{\bZ}_{\sfK}$ by 
\begin{align*}
\tilde{\bV} = \cP\bV_{\boldt}, \quad 
\tilde{\bZ}_{\sfK} = \cP\bar{\bZ}_{\sfK}\cP^\top. 
\end{align*}
Since the row (and hence column) space of the symmetric matrix $\bar{\bZ}_{\sfK}$ is equal to the column space of $\bV_{\boldt}$, which is precisely the row space of $\cP$, we see that 
\begin{align*}
\lambda_i(\tilde{\bZ}_{\sfK}) = \lambda_i(\bar{\bZ}_{\sfK}),~~ i \in [n],
\end{align*}
so that
\begin{align}
\Regsub(\sfK) = \trace[\tilde{\bZ}_{\sfK}^{-1}].
\end{align} 
From \Cref{eq:bar_sig}, we can compute that $\tilde{\bZ}_{\sfK}$ is related to $\bZ_{\sfK}$ via conjugation by $\tilde{\bV}$:
\begin{align*}
\tilde{\bZ}_{\sfK} = \tilde{\bV}\Zk \tilde{\bV}^\top,
\end{align*}
so that
\begin{align*}
\Regsub(\sfK)  = \trace[(\tilde{\bV}\Zk \tilde{\bV}^\top)^{-1}],
\end{align*}
showing that $\Regsub(\sfK) $ is $\cctwo$. Thus, the subsampled oracle model affords both evaluations and derivatives of $\cL_{\lambda,\subfont{sub}}(\sfK)$. Note that $\Regsub(\sfK) $ can be evaluated without knowledge of  $\cP$ and $\bV_{\boldt}$ by using the original definition in \Cref{eq:Clambda}. 
\begin{remark} A similar approach  to the computation above can be used to derive a closed-form expression for the derivative of $\Regsub(\sfK) $ in terms of the derivatives of \emph{only} $\bar{\bSigma}_{\sfK,\boldt}$, and \emph{not} in terms of the observation matrix $\bV_{\boldt}$ (which we do not have access to in this model). 
\end{remark}
In view of the identity $\Regsub(\sfK)  = \trace[\tilde{\bZ}_{\sfK}^{-1}]$ established above, we see that optimizing
\begin{align}
\cL_{\lambda,\subfont{sub}}(\sfK) = \Loe(\sfK) + \trace[\tilde{\bZ}_{\sfK}^{-1}]
\end{align}
is equivalent to optimizing the state-covariance oracle loss $\cL_{\lambda}(\sfK)$ on the following similarity-transformed realization of the dynamics
\begin{equation}\label{eq:true_system_til}
\begin{aligned}
&\tfrac{\rmd}{\rmd t} \tilde\sx(t) = \tilde\sA\tilde\sx(t) + \tilde\sw(t), \quad \sy(t)=\tilde\sC\tilde\sx(t) + \sv(t), \quad \po(t) = \tilde\sO \tilde\sx(t), \quad \tilde\sx(0) = 0, \\
&\quad \tilde\sw(t) \iidsim \calN(0,\tilde{\bW}_1), \quad \sv(t) \iidsim \calN(0,\bW_2),
\end{aligned}
\end{equation}
where $\tilde{\sA} = \tilde{\bV}\sA\tilde{\bV}^{-1}$, $\tilde{\sC} = \sC\tilde{\bV}^{-1}$, $\tilde{\sO} = \sO\tilde{\bV}^{-1}$, and $\tilde{\bW}_1 = \tilde{\bV} \bW_1$ (it follows from \Cref{asm:bold_t} and the definition of the projection $\cP$ that $\tilde{\bV}$ is nonsingular). Indeed, $\Loe(\sfK)$ is invariant under similarity transformation of the true system, and if $\tilde{\bSigma}_{\sfK}$ is the associated covariance matrix (partitioned in the standard way), then we can verify
\begin{align*}
\tilde{\bZ}_{\sfK} = \tilde{\bSigma}_{\sfK,12}\tilde{\bSigma}_{\sfK,22}^{-1}\tilde{\bSigma}_{\sfK,12}^\top.
\end{align*}
Therefore, via this similarity-transformation, optimizing $\cL_{\lambda,\subfont{sub}}(\sfK) $ on the dynamics \Cref{eq:true_system} inherits all the guarantees of optimizing the loss  $\cL_{\lambda}(\sfK) $  on the tilde-dynamics in \cref{eq:true_system_til}. 

We complete the section by providing the proof of \Cref{lem:generic_full_rank}. 

\begin{proof}[Proof of \Cref{lem:generic_full_rank}]
The proof is divided into two steps. First, we exhibit a $\boldt$ for which $\rank(\bV_{\boldt}) = n$; then we use an analytic continution argument to establish that, if such a $\boldt$ exists, then $\rank(\bV_{\boldt}) = n$ Lebesgue almost everywhere. 

\paragraph{Existence of a $\boldt$ for which $\rank(\bV_{\boldt}) = n$. } Without loss of generality, we may assume that $k = n$. Fix $\step > 0$, and consider $t_i = (i-1) \step$. Expanding the matrix exponential, we have 
\begin{align*}
\bV_{\boldt} &:= \bC\cdot\begin{bmatrix} \bI_m \\
\exp(\step \bA)\\
\exp(2\step \bA)\\
\dots\\
\exp((n-1)\step \bA)
\end{bmatrix} \\
&= \underbrace{\begin{bmatrix} \eye_m & \bzero & \bzero & \dots &  \bzero  \\
\eye_m & \step  \eye_m& \frac{\step^2}{2!}\eye_m & \dots & \frac{\step^{n-1}}{(n-1)!} \eye_m  \\
\eye_m & 2\step \eye_m &  \frac{(2\step)^2}{2!} \eye_m & \dots & \frac{(2\step)^{n-1}}{(n-1)!} \eye_m \\
\dots\\
\eye_m & (n-1)\step \eye_m & \frac{((n-1)\step)^{2}}{2!} \eye_m  & \dots & \frac{((n-1)\step)^{n-1}}{(n-1)!}  \eye_m 
\end{bmatrix}}_{\cT_{n,\step}} \underbrace{\begin{bmatrix} \bC \\
\bC\bA \\
\bC\bA^2 \\
\dots\\
\bC\bA^{n-1}
\end{bmatrix}}_{\cO_n} + \underbrace{\begin{bmatrix} 0 \\
\bC\sum_{i \ge n} \frac{(\step\bA)^{i}}{i!} \\
\bC\sum_{i \ge n} \frac{(2\step\bA)^{i}}{i!} \\
\dots \\
\bC\sum_{i \ge n} \frac{((n-1)\step\bA)^{i}}{i!}
\end{bmatrix}}_{\cR_{n,\step}}.
\end{align*}
We show below that $\cT_{n,\step}$ is invertible, so it suffices to show that for some $\step > 0$, 
\begin{align*}
\cT_{n,\delta}^{-1}\bV_{\boldt}  = \cO_n + \cT_{n,\step}^{-1}\cR_{n,\step} \text{ has rank } n. 
\end{align*}
Since $(\bA,\bC)$ is observable , $\rank(\cO) = n$ (c.f. \citet[Theorem 3.3]{zhou1996robust}). Therefore, since the set of full-rank matrices is an open set, it suffices to show that $\lim_{\step \to 0} \cT_{n,\step}^{-1}\cR_{n,\step}  = \mathbf{0}$. Since $\|\cR_{n,\step}\| = \BigOh{\step^n}$ as $\step \to 0$, it suffices to show that $\|\cT_{n,\step}^{-1}\| = \BigOh{\frac{1}{\step^{(n-1)}}}$. To this end, we factor
\begin{align*}
\cT_{n,\step} = \underbrace{\begin{bmatrix} \eye_m & \bzero & \bzero & \dots &  \bzero  \\
\eye_m & \eye_m & \frac{1}{2!} \eye_m & \dots & \frac{1}{(n-1)!}\eye_m  \\
\eye_m & 2\eye_m  &  \frac{2^2}{2!} \eye_m & \dots & \frac{2^{n-1}}{(n-1)!} \eye_m\\
\dots\\
\eye_m & (n-1)\eye_m & \frac{(n-1)^{2}}{2!}\eye_m  & \dots & \frac{(n-1)^{n-1}}{(n-1)!}\eye_m
\end{bmatrix}}_{\cU_{n}}\underbrace{\begin{bmatrix} \eye_m & \bzero & \bzero & \dots &  \bzero \\
\bzero & \step \eye_m & 0 & \dots & \bzero  \\
\bzero & \bzero &  \step^2 \eye_m & \dots & \bzero\\
\dots\\
\bzero & \bzero& \bzero  & \dots & \step^{n-1}\eye_m
\end{bmatrix}}_{\cD_{n,\step}}.
\end{align*}
Using row elimination, it is easy to observe that $\cU_n$ is invertible for any $\step >0$. In addition, $\cD_{n,\step}$ is invertible, with $\|\cD_{n,\step}^{-1}\| = \frac{1}{\step^{n-1}}$. Note that the invertibility  of $\cU_n$ and $\cD_{n,\step}$ establish the invertibility of $\cT_{n,\step}$, as promised. To conclude, we observe that since $\cU_n$ does not depend on $\step$, 
\begin{align*}
\|\cT_{n,\step}^{-1}\| \le \|\cU_n\|^{-1} \cdot\|\cD_{n,\step}^{-1}\| = \frac{1}{\step^{n-1}} \|\cU_n\|^{-1} = \BigOh{\step^{n-1}}.
\end{align*} 

\paragraph{Proof for Lebesgue-almost-every $\boldt$.} Having established the result for a fixed $\boldt$, define the function $f(\boldt) := \det(\bV_{\boldt}^\top \bV_{\boldt})$, with domain $\boldt \in \R^k$.\footnote{Observe that, while we only select strictly increasing $\boldt$, this lemma does not need such a restriction.} Then $f(\boldt)$ is defined and analytic on all of $\R^k$. Moreover, $f(\boldt) = 0$ if and only if $\rank(\bV_{\boldt}) \ne n$. Therefore, the previous part of the lemma establishes that there exists at least some $\boldt \in \R^k$ for which $f(\boldt) \ne 0$. The lemma is now a direct consequence of the identity theorem for analytic functions (\Cref{fact:identity_thm}).
\end{proof}


%% file: appendix/ctrb_assumption.tex
\section{Discussion of Controllability \cref{asm:ctrb_of_opt}}
\label{app:ctrb_assumption}

\subsection{Remarks of \Cref{asm:ctrb_of_opt}}
\begin{lemma}\label{lem:equiv_control_asm}The following conditions are equivalent to \Cref{asm:ctrb_of_opt}:
\begin{itemize}
	\item[(a)] There exists at least \emph{one} $\sfK_{\star} \in \calKopt$ for which $(\bA_{\sfK_{\star}},\bB_{\sfK_{\star}})$ is controllable.
	\item[(b)] $(\bA - \bL_{\star} \bC ,\bL_{\star})$ is controllable.
	\item[(c)] $(\bA,\bL_{\star})$ is controllable.
\end{itemize}
\end{lemma}
\begin{proof} Point (a) follows  since controllability is invariant under similarity transform; point (b) follows by taking $(\bA_{\sfK_{\star}},\bB_{\sfK_{\star}})$ to be the  the canonical realization of the optimal filter; point (c) follows since  maps of the form $(\tilde \bA,\tilde \bB) \to (\tilde \bA + \tilde \bK \tilde \bB, \tilde \bB)$ preserve controllability. 
\end{proof}
\begin{proposition}\label{prop:generic_control_asm} Fix any $n, m \ge 1$,  $\bW_1 \in \pd{n},\bW_2 \in \pd{m}$, and suppose that $(\bA,\bC)$ are drawn from a distribution with density with respect to the Lebesgue measure, such that with probability $1$, $\bA$ is Hurwitz stable. Then $\Pr[ \text{\Cref{asm:ctrb_of_opt} holds for }(\bA,\bC,\bW_1,\bW_2)] = 1$. 
\end{proposition}
\Cref{prop:generic_control_asm} is proven in \Cref{app:generic_control_proof}. 

\subsection{A strictly smaller problem set}

\cref{asm:ctrb_of_opt} states that any optimal $(\Ak,\Bk)$ must be controllable, which implies that $\bSigma_{22,\sfK} \mge 0$, cf. \cref{sec:controllable_nonsingular}.
This in turn ensures that $\Zst = \bSigma_{12,\sfK} \bSigma_{22,\sfK}^{-1} \bSigma_{12,\sfK}\transp$ and the regularizer $\trace [\Zst^{-1}]$ are well-defined at optimality.
Not all \OE\ instances satisfy this property, as the following example demonstrates:

\begin{example}\label{ex:opt_not_ctrb}
	Consider the \OE\ problem instance given by
	\begin{equation*}
	\sA = \begin{bmatrix}
	-1 & 0 \\ 0 & -2
	\end{bmatrix}, \quad 
	\sC = \begin{bmatrix}
	1 & 1
	\end{bmatrix}, \quad
	\sW = \begin{bmatrix}
	48 & -36 \\ -36 & 48
	\end{bmatrix}, \quad
	\sV = 1.
	\end{equation*}
	It is readily verified that this instance satisfies \cref{asm:stability,asm:observability,asm:pd}.
	The optimal policy (up to similarity transformations) is given by
	\begin{equation*}
	\Akst = \begin{bmatrix}
	-5 & -4 \\  0 & -2
	\end{bmatrix}, \quad
	\Bkst = \Lst = \begin{bmatrix}
	4 \\ 0
	\end{bmatrix}, \quad
	\Pst = \begin{bmatrix}
	16 & -12 \\ -12 & 12
	\end{bmatrix}.
	\end{equation*}
	Recall that the optimal policy is independent of $\sO$, the value of which is irrelevant for this example.
	Straightforward calculations reveal that 
	\begin{equation*}
	\begin{bmatrix}
	\Bkst & \Akst\Bkst
	\end{bmatrix}
	=
	\begin{bmatrix}
	4 & -20 \\ 0 & 0
	\end{bmatrix}, \quad
	\bSigma_{\kopt} = 
	\begin{bmatrix}
	24 & -12 & 8 & 0 \\-12 & 12 & 0 & 0\\
	8 & 0 & 8 & 0\\ 0 & 0 & 0 & 0
	\end{bmatrix},
	\end{equation*}
	confirming that $(\Akst,\Bkst)$ is not controllable and $\bSigma_{\kopt}$ is rank-deficient. 
\end{example}

\subsection{Implications for the convex reformulation}

In this subsection, we discuss the implications of \cref{asm:ctrb_of_opt}
and \cref{ex:opt_not_ctrb} for the convex reformulation of \OE\ developed in \cite{scherer97lmi}.
In particular, 
it is natural to wonder whether the breakdown of the change of variables at the optimal policy 
for problems such \cref{ex:opt_not_ctrb} pose a problem for the methods of \cite{scherer97lmi}.
Fortunately, they do not. 
The LMI formulations of \cite{scherer97lmi} circumvent these degeneracies in the landscape 
by employing strict inequalities. 
As we detail below, one can always perturb the decision variables to satisfy these strict inequalities,
even at points where $\bSigma_{12,\sfK}$ is rank deficient,
resulting in arbitrarily tight upper bounds on the true cost $\Loe(\fparams)$.

Specifically, 
given the decision variables 
$\schZ, \schX, \schY, \schK, \schL, \schM$, and defining 
\begin{equation*}
\schA := \begin{bmatrix}
\sA\schY + \sB\schM & \sA \\ 
\schK & \sA\schX + \schL\sC
\end{bmatrix}, 
\
\schB := \begin{bmatrix}
\sW^{1/2} & \zero \\ 
\schX\sW^{1/2}  & \schL\sV^{1/2}
\end{bmatrix}, 
\
\schC := \begin{bmatrix}
\sO\schY  - \schM  & \sO
\end{bmatrix},
\
\schXb = \begin{bmatrix}
\schY & \eye \\ \eye & \schX
\end{bmatrix},
\end{equation*}
the approach of \cite{scherer97lmi} proposes solving the following semidefinite program (SDP)
\begin{align}\label{eq:oe_lmi}
\min \quad & \trace(\schZ) \\
\textup{s.t.} \quad 
&\begin{bmatrix}
\schZ & \schC \\
\schC\transp & 
\schXb
\end{bmatrix} \mg \zero, \quad
\begin{bmatrix}
\schA + \schA\transp & \schB \\ \schB\transp & -\eye
\end{bmatrix} \ml \zero, \nonumber
\end{align}
which minimizes a convex upper bound on the \OE\ cost.
At optimality, 
to achieve $\trace(\schZ)=\Loe(\kopt)$,
the above linear matrix inequalities (LMIs) must be tight.
Moreover, 
$\schXb$ can then be interpreted 
as $\bSigma_{\kopt}^{-1}$, subject to a specific congruence transformation, cf. \cite{scherer97lmi}.
However, for problem instances such as \cref{ex:opt_not_ctrb},
$\bSigma_{\kopt}$ is rank deficient and thus $\bSigma_{\kopt}^{-1}$ does not exist. 
The convex reformulation circumvents this problem by through the use of strict LMIs: at optimality,
the above inequalities remain strict, and $\trace(\schZ)>\Loe(\kopt)$. 
In fact, for \cref{ex:opt_not_ctrb},
if one approximates the strict LMIs $\bF\mg\zero$, for generic $\bF$,
with non-strict $\bF\mge\epsilon\eye$ for $\epsilon=10^{-8}$,
then \cref{eq:oe_lmi} returns a solution satisfying 
$\trace(\schZ) - \Loe(\kopt) \approx 8\times 10^{-6}$.

\newcommand{\fasm}{f_{\mathrm{asm}}}
\newcommand{\Uasm}{\cU_{\mathrm{asm}}}

\subsection{Proof of \Cref{prop:generic_control_asm}\label{app:generic_control_proof}}
Our argument relies on the identity theorem for real-analytic functions.  \footnote{For a proof, see e.g. \url{https://math.stackexchange.com/questions/1322858/zeros-of-analytic-function-of-several-real-variables}.}
\begin{fact}\label{fact:identity_thm} Let $\cU$ be an open, connected subset of $\R^k$, and $F:\cU \to \R$ be an analytic function which is not identically zero. Then the set $\{\bx \in \cU: f(\bx) = 0\}$ has Lebesgue measure zero. 
\end{fact}
Give $\bW_1 \in \pd{n},\bW_2 \in \pd{m} $. Let $\Hurn:= \{\bA \in \R^{n \times n}:\lambda_i(\bA) < 0, ~\forall i \in [n]\}$  denote the set of Hurwitz matrices. We consider $(\bA,\bC) \in \Uasm := \Hurn \times \R^{m \times n}$. $\Uasm$ is open and connected as a consequence of the following claim, due to \cite{duan1998note}:
\begin{claim}\label{claim:Hurn}The set of Hurwitz matrices $\Hurn:= \{\bA:\lambda_i(\bA) < 0\}$ is a connected, open subset of $\R^{n \times n}$. 
\end{claim}
We define our candidate function $\fasm$ as follows. Given $(\bA,\bC) \in \Uasm$, let
\begin{align}
\fasm(\bA,\bC) = \det(\sum_{i=0}^{n-1}\bA^i \bL_{\star}\bL_{\star}\bA^i). \label{eq:fasm},
\end{align}
where $\bL_{\star}$ solves is the associated optimal gain for $(\bA,\bC,\bW_1,\bW_2)$ (this exists for all Hurwitz $\bA$). From \citet[Theorem 3.3]{zhou1996robust}), $(\bA,\bL_{\star})$ is controllable if and only if $\rank[\bL_{\star} \mid \bA \bL_{\star} \mid \bA^2 \bL_{\star} \mid \dots \bA^{n-1} \bL_{\star}] = n$, which holds if and only if $\fasm(\bA,\bC) \ne 0$. Hence, by \Cref{lem:equiv_control_asm}, we conclude
\begin{claim}\label{claim:fasm} $(\bA,\bC) \in \Hurn \times \R^{m \times n}$ satisfies \Cref{asm:ctrb_of_opt} if and only if $\fasm(\bA,\bC) \ne 0$, which holds if and only if $(\bA,\bL_{\star})$ is controllable. 
\end{claim}

To conclude, we must argue that (1) $\fasm$ is analytic on $\Uasm$, and (2) $\fasm$ is not identically zero on $\Uasm$; i.e. there exists \emph{some} $(\bA,\bC) \in \Uasm$ for which $(\bA,\bL_{\star})$ .

\paragraph{Analyticity of $\fasm$.} For the first point, we have the following claim. 
\begin{claim} Fix matrices $\bW_1 \succ 0,\bW_2 \succ 0$. Then, the mapping $F_P: (\bA,\bC) \mapsto \Pst$ to the solution $\Pst$ to the Riccati equation below, as well as the map $F_L: (\bA,\bC) \mapsto \Lst$ given below, are both real analytic on $\Uasm$.
\begin{align}
\bA \Pst + \Pst \bA^\top - \Pst \bC^\top \bW_2^{-1}\bC\Pst+\bW_1=0,\qquad \quad \bL_\star = \Pst \bC^\top \bW_2^{-1}. \label{eq:Pst_app}
\end{align}
As a consequence, $\fasm$ is real analytic on $\Uasm$
\end{claim}
\begin{proof} Since $\bW_1,\bW_2$ are fixed, the map $F_0 : (\Pst,\bC) \mapsto \Lst$ is polynomial, and thus analytic. Hence, $F_L = F_0 \circ F_P$ is analytic whenever $F_P$ is. Similarly, $\fasm$ is analytic whenever $F_L$ is analytic, and hence whenever $F_P$ is analytic. 

To see that $F_P$ is analytic, let us use the implicit function. $F_P(\bA,\bC)$ is define by the zero of the equation
\begin{align*}
G(\bA,\bC,\bP) = \bA \bP + \bP \bA^\top - \bP \bC^\top \bW_2^{-1}\bC\bP+\bW_1. 
\end{align*}
The total derivative of $G$ is then
\begin{align*}
&\rmd G(\bA,\bC,\bP)  \\
&= \rmd\bA \bP + \bP \rmd\bA^\top - \bP \rmd(\bC^\top \bW_2^{-1}\bC\bP)+\bW_1 + (\bA - \bC^\top \bW_2^{-1}\bC\bP) \rmd \bP + \rmd \bP (\bA - \bC^\top \bW_2^{-1}\bC\bP)^\top\\
&= \rmd\bA \bP + \bP \rmd\bA^\top - \bP \rmd(\bC^\top \bW_2^{-1}\bC\bP)+\bW_1 + (\bA - \Lst \bC) \rmd \bP + \rmd \bP (\bA - \Lst \bC)^\top. 
\end{align*}
We see that a solution to $\rmd G(\bA,\bC,\bP) = 0$ must have that $\rmd \bP$ satisfies the following Lyapunov equation for $\bY := \rmd\bA \bP + \bP \rmd\bA^\top - \bP \rmd(\bC^\top \bW_2^{-1}\bC\bP)+\bW_1$:
\begin{align}\label{eq:dP}
\tilde{\bA}\rmd \bP + \tilde{\bA} \rmd  \bP + \bY = 0.
\end{align} 
Since the Since $\tilde{\bA} := (\bA - \Lst \bC)$ is Hurwitz for a solution $\Lst$ to \Cref{eq:Pst_app}, the solution $\rmd \bP$ to \Cref{eq:dP} is unique. Hence, $\rmd G(\bA,\bC,\bP)$ satisfies the conditions of the implicit function theorem. In addition, $G$ is analytic. This means that, in a neighborhood around any $(\bA,\bC) \in \Hurn \times \R^{m \times n}$, there is an analytic function corresponding to $(\bA,\bC) \mapsto \bP_{\star}$. By definition, this function coincides with $F_P$ on that neighborhood, meaning $F_P$ is also analytic. 
\end{proof}

\paragraph{$\fasm$ is not identically zero.}
To conclude, it suffices to show the existence of \emph{some} $(\bA,\bC) \in \Uasm$ for which $\fasm$ doesn't vanish; i.e., some $(\bA,\bC) \in \Uasm$ for which $(\bA,\bL_{\star})$ is controllable. The following lemma is useful in our construction.
\begin{lemma}\label{lem:lim_of_P} Fix $\bW_1 \in \pd{n}, \bW_2 \in \pd{m}$, $(\bA,\bC) \in \Uasm$, and let $\bP_{\star,k}$ be the solution to the Lyapunov equation with $(\bA,\frac{1}{k}\bC,\bW_1,\bW_2)$. Then, $\lim_{k \to \infty}\bP_{\star,k}  = \bP_{\star,\infty}$, where $\bP_{\star,\infty}$ solves
\begin{align*}
\bA \bP + \bP \bA^\top + \bW_1 = 0.
\end{align*}
\end{lemma}
\begin{proof}  The sequence $\bP_{\star,k}$ are the solution to the Ricatti equation  $\cT_k(\bP) = 0$, where 
\begin{align*}
\cT_k(\bP) &:= \bA \bP + \bP \bA^\top - \bP \bC_k^\top \bW_2^{-1}\bC\bP+\bW_1
\end{align*}
Since $\bA$ is stable, $\bP_{\star,k}$ also the unique solution $\bP$ to the Lyapunov equation $\tilde{\cT}_k(\bP) = 0$ constructed by fixing $\bP = \bP_{\star,k}$ in the third term in $\cT_k(\bP)$:
\begin{align*}
\tilde{\cT}_k(\bP) := \bA \bP + \bP \bA^\top + \tilde{\bW}_{1,k}, \quad \tilde{\bW}_{1,k} := \left(\bW_1 -  \frac{1}{k}\bP_{\star,k} \bC_k^\top \bW_2^{-1}\bC\bP_{\star,k}\right). 
\end{align*} 
Since $\tilde{\bW}_{1,k} \preceq \bW_1 $, we have that $\bP_{\star,k} \preceq \bP_{\star,\infty}$.  In addition, $\bP_{\star,k} \succeq 0$ for all $k$. Thus, $\bP_{\star,k}$ lie in the compact set $\cP := \{\bP \in \sym{n}: 0 \succeq \bP \succeq \bP_{\star,\infty}\}$, and hence it suffices to show that for any  convergent subsequence $(\bP_{\star,k_i})$ which converges to a limit $\tilde{\bP} \in \cP$, $\tilde{\bP} = \bP_{\star,\infty}$. To show show this, observe $\cT_k(\cdot) \to \cT_{\infty}(\cdot)$ uniformly on the compact set $\cP$, and since $\cT_{\infty}$ is continuous, it follows that
\begin{align*}
0 = \lim_{i \to \infty} \cT_{k_i}(\bP_{\star,k_i}) = \lim_{i \to \infty}\cT_{\infty}(\bP_{\star,k_i}) = \cT_{\infty}(\tilde{\bP}).
\end{align*}
Since $\cT_{\infty}(\cdot)$ is a Lyapunov equation with $A$ stable,  the solution to $\cT_{\infty}(\cdot) = 0$ is unique, and hence $\tilde{P} = \bP_{\star,\infty}$, as needed.
\end{proof}
\begin{claim} Fix $\bW_1 \in \pd{n}, \bW_2 \in \pd{m}$. Then, there exists an $(\bA,\bC) \in \Uasm$ for which $(\bA,\bL_{\star})$ is controllable, where $\bL_{\star}$ is as in \Cref{eq:Pst_app}. In particular, for this $(\bA,\bC)$, $\fasm(\bA,\bC) \ne 0$.
\end{claim}
\begin{proof} By a change of basis of $\R^n$ and $\R^m$, we may assume without loss of generality that $\bW_1 = \eye_n$ and $\bW_2 = \eye_m$. Let $\bA = \mathrm{Diag}(1,2,\dots,n)$, and let $\bC_1 := \begin{bmatrix} \bone& \bzero_n &\dots &\bzero_n \end{bmatrix}^\top$, and set $\bC_{k}=\frac{1}{k}\bC_1$. It $\bP_{\star,k}$ (resp. $\bL_{\star,k}$) solve the Ricatti equation (resp. be the optimal gain) matrix for $(\bA,\bC_k)$. We show that for all $k$ sufficiently large, $(\bA,\bL_{\star,k})$ is controllable (indeed, this establishes existence.)

It suffices to show that, for all $k$ sufficiently large, $(\bA,\tilde\bL_{\star,k})$ is controllable where $\tilde{\bL}_{\star,k} := k \bL_{\star,k}$. From \Cref{eq:Pst_app}, the definition of $\bC_k$, and assumption $\bW_2 = \eye_m$,
\begin{align*}
\tilde{\bL}_{\star,k} := k \bL_{\star,k} = k \bP_{\star,k}\bW_2^{-1}\bC_k^\top = \bP_{\star,k}\bC_1^\top
\end{align*}
Since the set of controllable matrices is an open set, and since $\lim_{k \to \infty}\bP_{k,\star} = \bP_{\star,\infty}$ by \Cref{lem:lim_of_P}, we see that $(\bA,\tilde\bL_{\star,k})$ is controllable for all $k$ sufficiently large as long  $(\bA, \bP_{\star,\infty}\bC_1^\top)$ is controllable. Since $\bA$ is diagonal, one can verify that $\bP_{\star,\infty} = -\frac{1}{2}\bA^{-1}$. In particular, $\bP_{\star,\infty}\bC_1^\top = \begin{bmatrix} -\frac{1}{2}\bA^{-1}\bone& \bzero_n &\dots &\bzero_n \end{bmatrix}$; hence the first column of $\bP_{\star,\infty}\bC_1^\top$ does not lie in any $\bA$-invariant subspace, so $(\bA, \bP_{\star,\infty}\bC_1^\top) = (\bA,\tilde \bL_{\star,k})$  is controllable for all $k$ large. As noted above, this implies $(\bA,\bL_{\star,k})$ is controllable, so that by \Cref{claim:fasm}, $\fasm(\bA,\bC_k) \ne 0$.
\end{proof}
\paragraph{Conclusion.} Hence, we have established that $\fasm$ is analytic, but not identically zero, on the open and commented domain $\Uasm$. The proof follows. \qed

%% file: appendix/control_proofs.tex
\section{Control Proofs}
\label{app:control_proofs}
\newcommand{\Gcont}{\mathcal{G}_{\mathrm{cont},\sfK}}

\newcommand{\barsy}{\bar{\sy}} 

\subsection{Controllability, stability, and nonsingularity of internal-state covariance}\label{sec:controllable_nonsingular}

In \Cref{sec:prelim}, we restricted our attention to policies $\sfK \in \calKstab$, that is, where the filter transition matrix $\Ak$ was Hurwitz stable. This is equivalent to stability of $\Aclk$, as shown by the following lemma. 
\begin{lemma}\label{lem:stab_equiv}
$\Ak$ is stable if and only if $\Aclk$ is stable, and $\Sigk$ is given by the solution of the Lyapunov equation \Cref{eq:lyapunov_sigma}.
\end{lemma}
\begin{proof}
The equivalence of the stability of $\Ak$ and $\Aclk$ comes from the fact that, due to the block-triangular form of $\Aclk$ with blocks $\bA$ and $\Ak$, the eigenvalues of $\Aclk$ are just the union of those of $\Ak$ and those of $\bA$. All eigenvalues of $\bA$ have negative real part by \Cref{asm:stability}, so the non-negative real part of the eigenvalue of $\Aclk$ are equal to those of $\Ak$. Thus, stability of $\Ak$ and $\Aclk$ are equivalent. That $\Sigk$ is given by the solution of the Lyapunov equation is standard, 
cf. \citep[Theorem 3.18]{zhou1996robust}.
\end{proof}
Next, we show the equivalence   between  $\Sigktwo \succ 0$ and controllability of $(\Ak,\Bk)$. We define controllability for (possibly unstable) $\Ak$ as follows,
cf., e.g.,  \cite[Theorem 3.1]{zhou1996robust}. 
\begin{definition}\label{defn:controllability} The pair $(\Ak,\Bk)$ is controllable if and only if there exists some $t > 0$ such that 
\begin{align*}
	\Gcont^{[t]} \defeq \int_{0}^{t} \exp(s \bA) \Bk\Bk^\top \exp(s\bA^\top) \rmd s
\end{align*}
	is strictly positive definite.
\end{definition}
\begin{lemma} Suppose that \Cref{asm:stability,asm:pd} hold.  Then the following statements are equivalent. 
	\begin{itemize}
	\item[(a)] The limiting covariance $\Sigktwo$, defined below, exists, and has $\Sigktwo \succ 0$,
	\begin{align*}
	\Sigktwo  =  \lim_{t \to \infty} \Exp \left[\xhatk(t)\xhatk(t)^\top\right] \in \psd{2n}. 
	\end{align*}
	\item[(b)] $\Ak$ is stable, and $(\Ak,\Bk)$ is controllable. 
	\item[(c)] $\Ak$  is stable and $\Sigktwo \succ 0$. 
\end{itemize}
Moreover, these equivalent conditions imply the limiting covariance $\Sigk$ is well-defined and given by the solution to  \Cref{eq:lyapunov_sigma}. 
\end{lemma}
\begin{proof} The ``moreover'' statement is a consequence of \Cref{lem:stab_equiv}. We establish the equivalences of (a), (b), and (c). 

\vspace{7pt}
\noindent\textbf{(a) implies (b).} We compute that
\begin{align*}
\Sigktwo = \lim_{t \to \infty} \Sigktwo^{[t]}, \quad \text{which is the bottom-diagonal block of } \Sigk^{[t]} = \int_{0}^t \exp(s\Aclk)\Wclk \exp(s\Aclk)^\top\rmd s. 
\end{align*}
First, we show that $(\Ak,\Bk)$ are controllable. Indeed, since $\Sigktwo \succ 0$ and $\lim_{t \to \infty} \Sigktwo^{[t]}$ exists and is finite, we have that for this $\tau$, $\Sigktwo^{[\tau]} \succ 0$. Thus by \Cref{lem:covwtocovv},  it follows that for some finite $\tau$,
\begin{align}
\int_{0}^{\tau}\exp(s\Ak)\Bk\bW_2 \Bk^\top \exp(s\Ak)^\top\rmd s \succ 0 \label{eq:tau_ge_zero}. 
\end{align}
Since $\bW_2 \succ 0$ by \Cref{asm:pd}, it therefore follows that 
\begin{align*}
\Gcont^{[\tau]} = \int_{0}^{\tau}\exp(s\Ak)\Bk \Bk^\top \exp(s\Ak)^\top\rmd s \succ 0.
\end{align*}
Next, we show stability. Since $\exp(s\Aclk)\Wclk \exp(s\Aclk)^\top \succeq 0$, existence of the limiting $\Sigktwo$ implies that for any vector of the form $\bv = (0,\bv_2) \in \R^{2n}$ for $\bv_2 \in \R^n$,  
\begin{align*}
\int_{0}^{\infty} \|\bv^\top\exp(s\Aclk)\Wclk^{1/2}\|^2 \rmd s < \infty. 
\end{align*}
Note the $(2,2)$-bock of $\exp(s\Aclk)$ is $\exp(s\Ak)$ (see \Cref{lem:computation_expon}), and  that, since $\Wclk$ is block-diagonal, 
\begin{align*}
\Wclk  := \begin{bmatrix} \bW_1 & 0 \\ 0 & \Bk\bW_2 \Bk^\top \end{bmatrix} \succeq \begin{bmatrix} 0 & 0 \\ 0 & \Bk\bW_2 \Bk^\top \end{bmatrix}  
\end{align*}
Thus, considering a vector $\bv$ of the form $(0,\bv_2)$, for $\bv_2 \in \R^n$,
\begin{align*}
\lim_{t \to \infty} \int_{0}^{t} \bv_2^\top\exp(s\Ak)\Bk\bW_2 \Bk^\top \exp(s\Ak) \bv_2 \rmd s < \infty,
\end{align*}
which shows that the following limiting integral is well defined $ \int_{0}^{\infty}\exp(s\Ak)\Bk\bW_2 \Bk^\top \exp(s\Ak)^\top$. On the other hand, by \Cref{eq:tau_ge_zero}, we must have that the following limiting integral is well-defined and strictly positive definite
\begin{align*}
\int_{0}^{\infty} \exp(s\Ak)\Bk\bW_2 \Bk^\top \exp(s\Ak)  \rmd s \succ 0. 
\end{align*}
Thus, \Cref{lem:lyap_xy} implies that $\Ak$ is Hurwitz stable. This (together with stability of $\bA$) implies Hurwitz stability of  $\Aclk$ (see below),  and \Cref{lem:cont_lyap} therefore guarantees that $\Sigk$ is the solution of the appropriate Lyapunov  equation, given in \Cref{eq:lyapunov_sigma}.

\vspace{7pt}
\noindent\textbf{(b) implies (c).}  From the computation in part (a), one can check that 
\begin{align*}
\Sigktwo \succeq \lim_{t \to \infty} \int_{0}^{\infty} \exp(s\Ak)\Bk\bW_2 \Bk^\top \exp(s\Ak) \rmd s \succeq \lambda_{\min}(\bW_2)\lim_{t \to \infty} \Gcont^{[t]}.
\end{align*}
Thus, controllability of $(\Ak,\Bk)$ implies $\Gcont^{[t]} \succ 0$ for some finite $t$, which implies $\Sigktwo \succ 0$.

\vspace{7pt}
\noindent\textbf{(c) implies (a).} From \Cref{lem:cont_lyap}, stability of $\Ak$ implies stability of $\Aclk$, which implies that the limiting covariance $\Sigk$ exists. In particular, the limiting (2,2)-block covariance exists.
\end{proof}

\subsection{Characterization of optimal policies}\label{app:opt_polices}

We begin by reviewing some well-known properties of the optimal solution to the \OE\ problem. 
\begin{lemma}\label{lem:dgkf_optimal}
	Under \Cref{asm:stability,asm:observability,asm:pd},
	the \emph{unique} (up to similarity transformations) optimal solution to the \OE\ problem is given by the policy 
	\begin{equation}\label{eq:optimal_policy_params}
	\bA_{\kopt} = \sA - \Pst \bC^\top \sV^{-1} \bC, \quad
	\bB_{\kopt} =  \Pst \bC^\top \sV^{-1}, \quad
	\bC_{\kopt} = \sO,
	\end{equation}
	where $\Pst \mg 0$ is the solution to the algebraic Riccati equation \cref{eq:Pst_Lst}.
\end{lemma}

\begin{proof}
	A proof of this classical result can be found, e.g, in \cite[\jacsec IV.D]{dgkf89}.
	We note that strict positive definiteness of $\Pst$ is implied by the 
	controllability of $(\sA,\sW)$, cf. \cite[\jacsec II.B]{dgkf89}.
	Controllability of $(\sA,\sW)$ follows from $\sW\mg0$, cf. \cref{asm:pd}. 
\end{proof}


\begin{fact}\label{fact:optimal_independent}
	The optimal solution to the \OE\ problem is independent of $\sO$,
	and optimal for all values of $\sO$.
\end{fact}

\begin{proof}
	The optimal policy given in \cref{eq:optimal_policy_params} and the Riccati equation \cref{eq:Pst_Lst} are both independent of $\sO$. 
	Moreover, there are no restrictions placed on $\sO$ (beyond the requirement that the number of columns matches the dimension of the state of the true system). 
\end{proof}



\subsection{Informativity of optimal policies}\label{app:opt_pol_nondegenerate}

We begin with the following useful fact.

\begin{fact}\label{fact:opt_sigmas_equal}
Let $\kopt\in\calKopt$ denote the realization of the optimal policy given in \cref{eq:optimal_policy_params},
i.e. with $\bC_{\kopt} = \sO$.
Then, 
under \Cref{asm:stability,asm:observability,asm:pd},
$\bSigma_{12,\kopt} = \bSigma_{22,\kopt}$.
\end{fact}

\begin{proof}
All optimal policies $\sfK \in \calKopt$ must satisfy 
\begin{equation}
\frac{\partial \Loe(\sfK)}{\partial \Ck} = 2\Ck\bSigma_{22,\sfK} - 2\sO\bSigma_{12,\sfK} = 0.
\end{equation}	
In particular, for the realization of the optimal policy in \cref{eq:optimal_policy_params} with $\bC_{\kopt} = \sO$, 
this implies that $\sO(\bSigma_{22,\kopt} - \bSigma_{12,\kopt}) = 0$.
By \cref{fact:optimal_independent}, this must hold for all $\sO$, which implies that $\bSigma_{22,\kopt} = \bSigma_{12,\kopt}$.
\end{proof}

\lemsigonetwo*
\begin{proof} We prove each part in sequence.

\paragraph{Inclusion $\calKopt \subset \calKexp$.} Let $\kopt\in\calKopt$ denote the realization of the optimal policy given in \cref{eq:optimal_policy_params}.
By \cref{asm:ctrb_of_opt},
all optimal policies are controllable, and so $\bSigma_{22,\kopt} \mg 0$.
By \cref{fact:opt_sigmas_equal}, we have $\bSigma_{12,\kopt} = \bSigma_{22,\kopt} \mg 0$,
which implies that $\bSigma_{12,\kopt}$ is full-rank. 
The rank of $\bSigma_{12,\sfK}$ is invariant under similarity transformations of the policy;
hence, $\bSigma_{12,\sfK}$ is full-rank for all $\sfK \in \calKopt$.  

\paragraph{Inclusion $\calKexp \subset \calKfull$.} Recall that $\calKfull := \{\sfK  \in \calKstab: \Sigktwo \succ 0\}$. Hence, it suffices to show that if $\sfK \in \calKstab$ has $\rank(\Sigkonetwo) = n$, then $\Sigktwo \succ 0$. This follows since $\Sigk \succeq 0$.

\paragraph{Openness.} To see that $\calKexp$ is open, we observe that $\calKstab$ is open (this follows from \Cref{claim:Hurn}), and that $\sfK \mapsto \Sigk$ is continuous on $\calKstab$ (this is standard, and follows, for example, from arguments in\Cref{app:initializations}), Hence, the map $f:\sfK \mapsto \det(\Sigkonetwo)$ is continuous on $\calKstab$, and thus $\calKexp : \{\sfK \in \calKstab: \det(\Sigktwo) \ne 0\}$, being the inverse-image  of the open set $\R \setminus \{0\}$ under $f$,  is open.
\end{proof}

\subsection{Maximality of $\Zst$}

\lemZmaximal*
\begin{proof}
	We restrict our attention to $\sfK\in\calKfull$, otherwise $\bSigma_{22,\sfK}$ is not invertible and 
	$\Zk = \bSigma_{12,\sfK} \bSigma^{-1} _{22,\sfK}\bSigma^\top_{12,\sfK}$ is not well-defined.
	Recall that $\Zk$ is independent of the realization of $\sfK$,
	i.e. $\Zk$ is invariant under similarity transformations of $\sfK$.
	
	First, observe that the \OE\ cost can be written as
	\begin{equation}\label{eq:loe_expaned}
	\Loe(\sfK) = 
	\trace\left[
	\begin{bmatrix}
	\sO & -\Ck
	\end{bmatrix}
	\bSigma_{\sfK}
	\begin{bmatrix}
	\sO & -\Ck
	\end{bmatrix}^\top\right]
	=
	\trace[\sO\Sigonesys\sO\transp]
	-2\trace[\sO\Sigkonetwo\Ck\transp]
	+ \trace[\Ck\Sigktwo\Ck\transp],
	\end{equation}
	where $\bSigma_{\sfK}$ satisfies the Lyapunov equation in \cref{eq:lyapunov_sigma}. 
	Minimizing \cref{eq:loe_expaned} w.r.t. $\Ck$ (keeping $\Ak, \ \Bk$ fixed) gives
	\begin{equation}
	\trace[ \sO( \Sigonesys -  \underbrace{\bSigma_{12,\sfK} \bSigma^{-1} _{22,\sfK}\bSigma^\top_{12,\sfK}}_{=\Zk}  )\sO\transp ]
	=
	\min_{\Ck}~~ \Loe((\Ak,\Bk,\Ck)).
	\end{equation}
	Let $\kopt\in\calKopt$, and denote $\Zst = \bSigma_{12,\kopt} \bSigma^{-1} _{22,\kopt}\bSigma^\top_{12,\kopt}$.
	Then by optimality of $\kopt$ we have 
	\begin{equation}
	\trace[ \sO( \Sigonesys - \Zk )\sO\transp ]
	\geq
	\trace[ \sO( \Sigonesys - \Zst )\sO\transp ]
	\implies 
	\trace[ \sO( \Zst - \Zk )\sO\transp ] \geq 0,
	\end{equation}
	with equality if and only if $\sfK\in\calKopt$, due to uniqueness (of the transfer function) of the optimal policy, cf. \cref{lem:dgkf_optimal}.
	By \cref{fact:optimal_independent},  this holds for all $\sO$, 
	which implies that $ \Zst - \Zk \mge 0$, again with equality if and only if $\sfK\in\calKopt$.
	This completes the first part of the proof. 
	
	To show that $\kopt$ minimizes $\regexp(\sfK) = \trace[\Zk^{-1}]$, 
	we distinguish between two cases: those in which $\Zk$ is invertible, and those in which it is not.
	Consider the former, and assume $\sfK$ is such that $\Zk$ is invertible.
	Observe that $\Zst$ is always invertible:
	by \cref{asm:ctrb_of_opt} we have that $\bSigma_{22,\kopt}\mg0$,
	and by \cref{lem:Sigonetwo} we have that $\bSigma_{12,\kopt}$ is full-rank. 
	Therefore, $\Zst = \bSigma_{12,\kopt} \bSigma^{-1} _{22,\kopt}\bSigma^\top_{12,\kopt}$ is also full-rank. 
	We then have the following:
	\begin{equation}
	\Zst \mge \Zk \implies \Zk^{-1} \mge \Zst^{-1}
	\implies \trace[ \Zk^{-1}] \geq \trace[\Zst^{-1}]
	\implies \regexp(\sfK) \geq \regexp(\kopt),
	\end{equation}
	with equality if and only if $\sfK\in\calKopt$.
	This implies that $\kopt\in\calKopt$ minimizes $\regexp(\cdot)$ over all $\sfK\in\calKfull$ such that $\Zk$ is invertible.
	
	Next, we consider the case in which $\sfK$ is such that $\Zk$ is not invertible.
	In this case, $\regexp(\sfK) \defeq \infty$, and so $\regexp(\sfK)  \geq \regexp(\kopt)$ holds trivially. 
	This completes the proof that $\regexp(\sfK)  \geq \regexp(\kopt)$ for all $\sfK\in\calKfull$. 
\end{proof}

\subsection{Positivity and characterization  of $\sigst$}
\lemPstSigst*
\begin{proof}
Strict positivity of $\sigst := \lambda_{\min}(\bP_{\star})$ follows directly from \cref{lem:dgkf_optimal}, which states that $\Pst \mg 0$.

To show that $\Pst = \Sigonesys-\Zst$,
we will first show that $\Pst = \Sigonesys-\bSigma_{22,\kopt}$, where $\kopt$ denotes the realization given in \cref{eq:optimal_policy_params}.
Let $\bSigma_{\kopt}$ be given by the solution to the Lyapunov equation
$\Aclkst \bSigma_{\kopt} + \bSigma_{\kopt} \Aclkst\transp + \bW_{\text{cl},\kopt} = 0$,
as in \cref{eq:lyapunov_sigma}.
The (2,2) block of this Lyapunov equation is given by
\begin{equation}\label{eq:lyap_22}
\Akst  \bSigma_{22,\kopt} + \bSigma_{22,\kopt} \Akst\transp + \Bkst \sC \bSigma_{12,\kopt} 
+  \bSigma_{12,\kopt}\transp\sC\transp\Bkst\transp + \Bkst\sV\Bkst\transp = 0.
\end{equation}
Substituting 
$\bA_{\kopt} = \sA - \Pst \bC^\top \sV^{-1} \bC$ and  $\bB_{\kopt} =  \Pst \bC^\top \sV^{-1}$ into \cref{eq:lyap_22} gives 
\begin{align}\label{eq:lyap_22_opt}
& (\sA - \Pst \bC^\top \sV^{-1} \bC) \bSigma_{22,\kopt} + \bSigma_{22,\kopt} (\sA - \Pst \bC^\top \sV^{-1} \bC)\transp 
+ \Pst \bC^\top \sV^{-1}  \sC \bSigma_{12,\kopt} 
+ \bSigma_{12,\kopt}\transp\sC\transp \sV^{-1} \bC \Pst \nonumber \\
& + \Pst\sC\transp \sV^{-1}\sC\Pst = 0.
\end{align}
Subtracting the (1,1) block of the Lyapunov equation \cref{eq:lyapunov_sigma}, 
given by $\sA\Sigonesys + \Sigonesys\sA\transp + \sW=0$,
from \cref{eq:lyap_22_opt} and collecting terms leads to
\begin{align}\label{eq:lyap_22_sub_11}
& \sA(\bSigma_{22,\kopt} -\Sigonesys ) +  (\bSigma_{22,\kopt} -\Sigonesys )\sA\transp        
+ \Pst \bC^\top \sV^{-1}  \sC (\bSigma_{12,\kopt} - \bSigma_{22,\kopt}) \nonumber \\
& + (\bSigma_{12,\kopt} - \bSigma_{22,\kopt})\transp\sC\transp \sV^{-1} \bC \Pst 
+ \Pst\sC\transp \sV^{-1}\sC\Pst - \sW = 0.
\end{align}
Next, from \cref{fact:opt_sigmas_equal} we have $\bSigma_{12,\kopt} = \bSigma_{22,\kopt}$ for this particular realization of $\kopt$, given in \cref{eq:optimal_policy_params}.
Making this substitution, and adding the Riccati equation \cref{eq:Pst_Lst} to \cref{eq:lyap_22_sub_11} gives
\begin{equation}\label{eq:lyap_22_final}
\sA(\Pst + \bSigma_{22,\kopt} -\Sigonesys ) +  (\Pst + \bSigma_{22,\kopt} -\Sigonesys )\sA\transp = 0.
\end{equation} 
Clearly, $\Pst + \bSigma_{22,\kopt} - \Sigonesys = 0$ is a valid solution to \cref{eq:lyap_22_final}.
As $\sA$ is stable, the solution to the Lyapunov equation \cref{eq:lyap_22_final} is unique,
and hence $\Pst = \Sigonesys - \bSigma_{22,\kopt}$. 
Recall once more that due to \cref{fact:opt_sigmas_equal} we have 
$\bSigma_{12,\kopt} = \bSigma_{22,\kopt}$, with $\bSigma_{22,\kopt}$ being symmetric.
Therefore, 
\begin{equation*}
\bSigma_{22,\kopt} = \bSigma_{12,\kopt} = \bSigma_{12,\kopt} {\bSigma_{22,\kopt}^{-1}\bSigma_{12,\kopt}\transp} =: \Zst,
\end{equation*}
and so $\Pst = \Sigonesys - \bSigma_{22,\kopt} =  \Sigonesys - \Zst$.
Though we arrived at this conclusion via a specific realization \cref{eq:optimal_policy_params} of the optimal policy $\kopt$,
both $\Sigonesys$ and $\Zst$ are independent of the realization of the optimal policy. 
\end{proof}

\subsection{Information-theoretic interpretation of $\Zk$}\label{info:theoretic}
Recall that 
\begin{align*}
\Sigk = \lim_{t\to \infty}\Exp\left[\begin{bmatrix}\sx(t)\\\fx(t)\end{bmatrix}\begin{bmatrix}\sx(t)\\\fx(t)\end{bmatrix}\right]. 
\end{align*}
Since $(\sx(t),\fx(t))$ are jointly Gaussian with zero mean, $(\sx(t),\fx(t))$ converge in distribution to a limiting Gaussian distribution
\begin{align*}
\begin{bmatrix}\sx_{\infty}\\\fx_{\infty}\end{bmatrix} \sim \mathcal{N}(\mathbf{0},\Sigk),~~~~ \Sigk = \begin{bmatrix} \Sigonesys & \Sigkonetwo\\
\Sigkonetwo^\top & \Sigktwo\end{bmatrix}. 
\end{align*}
The conditional covariance of  $\sx_{\infty}$ given $\fx_{\infty}$ is then given by the formula
\begin{align*}
\mathrm{Cov}[\sx_{\infty} \mid \fx_{\infty}] = \Sigonesys - \Sigkonetwo \Sigktwo^{-1}\Sigkonetwo = \Sigonesys - \Zk.
\end{align*}
In other words, $\Zk$ describes the reduction in covariance of $\sx_{\infty}$ provided by the information in $\fx_{\infty}$.

\subsection{Random Stable Initializations Are Informative}\label{app:initializations}
\begin{lemma} Fix $\Ck$, and suppose that the $(\Ak,\Bk)$ is chosen from some probability distribution $\Pr$ with density with respect to the Lebesgue measure on $\R^{n \times n} \times \R^{n \times m}$ satisfying $\Pr[\Ak \text{ is Hurwitz}] = 1$. Then, $\Pr[ \sfK \in \calKexp] = 1$.
\end{lemma}
\begin{proof}Let $\Hurn$ denote the set of Hurwitz matrices in $\R^n \times n$.  Note that if $(\Ak,\Bk) \in \Hurn \times \R^{n \times m}$, then $\sfK \in \calKexp$ if and only if $\rank(\Sigktwo) = n$ and $\rank(\Sigkonetwo) = n$. In fact, since $\Sigk \succeq 0$, The Schur complement test implies that $\sfK \in \calKexp$ if and only if $\rank(\Sigkonetwo) = n$ (as this also implies $\rank(\Sigktwo) = n$). Thus, if $f(\Ak,\Bk)$ is the mapping from $(\Ak,\Bk)$ to $\det(\Sigkonetwo)$, then, given $(\Ak,\Bk) \in \Hurn \times \R^{n \times m}$, $\sfK \in \calKexp$ if and only if $f(\Ak,\Bk) \ne 0$.

As shown in \Cref{claim:Hurn}, the set $\Hurn $ is open and connected, so the $\cU := \Hurn \times \R^{n \times m}$. Moreover, $f$ does not identically vanish on $\cU$: indeed, for any $(\bA_{\sfK_\star},\bB_{\sfK_{\star}})$ corresponding to some $\sfK_{\star} \in \calKopt$, we have $\rank(\bSigma_{12,\sfK_{\star}}) = n$ by \Cref{lem:Sigonetwo}, so $f(\bA_{\sfK_\star},\bB_{\sfK_{\star}}) \ne 0$. 

Therefore, to prove our result, it suffices to show that $f$ is an analytic function of $(\Ak,\Bk)$, and apply the identity theorem (\Cref{fact:identity_thm}). In fact, we show $f(\Ak,\Bk)$ is an \emph{rational function}. The following claim is useful. 
\begin{claim} Let $\bar{F}:\Hurn[2n] \times \sym{2n} \to \sym{2n}$ be the map for which $\bar{F}(\bar{\bA},\bar{\bW})$ is the solution to the Lyapunov equation $\bar{\bA} \bGamma + \bGamma \bar{\bA} + \bar{\bW} = 0$. Then $\bar{F}$ is a rational function with no poles on $\mathsf{Hur}_{2n} \times \sym{2n}$.
\end{claim} 
\begin{proof}Since this solution to the Lyapunov equation is unique for $\bar{\bA} \in \Hurn[2n]$, we see that the map $\cT_{\bar{\bA}}: \bGamma \mapsto \bar{\bA} \bGamma + \bGamma \bar{\bA}$ is invertible, and hence $\bar{F}(\bar{\bA},\bar{\bW}) = \cT_{\bar{\bA}}^{-1}(\bar{\bW})$. It follows that $\bar{F}(\bar{\bA},\bar{\bW})$ is a rational function (notice the entries of $\cT_{\bar{\bA}}$ are linear in $\bar{\bA}$, and thus the inverse is a rational function of $\cT_{\bar{\bA}}$ using the adjugate formula for matrix inverses). It has no polls because $\cT_{\bar{\bA}}$ is invertible for $\bar{\bA} \in \Hurn[2n]$
\end{proof}
By composing the rational $\bar{F}(\cdot,\cdot)$ in the above claim with the polynomial-function $(\Ak,\Bk) \mapsto (\Aclk, \Wclk)$, we see that $(\Ak,\Bk) \mapsto \Sigk$ is a rational function function on $\Hurn \times \R^{n \times m}$. In particular,   $(\Ak,\Bk) \mapsto \Sigk$ is an analytic function. Thus, $f(\Ak,\Bk)$, being a polynomial in $\Sigk$, is also rational. This concludes the proof.
\end{proof}

%% file: arxiv_only/app_examples_arxiv.tex
\newcommand{\Sigonetwobad}{\bSigma_{12,\mathrm{bad}}}
\newcommand{\Sigtwobad}{\bSigma_{22,\mathrm{bad}}}
\newcommand{\Cktwo}{\bC_{\sfK,2}} 
\newcommand{\Delc}{\bDelta_{C}}
\newcommand{\Dela}{\bDelta_{A}}
\newcommand{\Delb}{\bDelta_{B}}
\newcommand{\Delonetwo}{\bDelta_{12}}
\newcommand{\Delk}{\bDelta_{\sfK}}

\section{Details for examples in \cref{sec:main_results}}
\label{app:counterexamples}
\subsection{Details for \cref{ex:stationary}}\label{app:stationary}

That $\Kbad$ is a suboptimal stationary point follows from \cite[Theorem 4.2]{tang2021analysis},
as \OE is a special case of \LQG.
Nonetheless, it is straightforward to verify that $\Kbad$ is indeed a stationary point. 
Specifically, one can readily verify that the controllability Gramian
\begin{equation*}
\bSigma = \begin{bmatrix}
\Sigonesys & \bzero \\
\bzero & \bzero
\end{bmatrix}
\end{equation*}
satisfies the Lyapunov equation 
\begin{equation*}
\begin{bmatrix}
\sA & \bzero \\ \bzero & \Abad
\end{bmatrix}
\begin{bmatrix}
\Sigonesys & \bzero \\
\bzero & \bzero
\end{bmatrix}
+ 
\begin{bmatrix}
\Sigonesys & \bzero \\
\bzero & \bzero
\end{bmatrix}
\begin{bmatrix}
\sA & \bzero \\ \bzero & \Abad
\end{bmatrix}\transp
+
\begin{bmatrix}
\sW & \bzero \\ \bzero & \bzero
\end{bmatrix}
= \bzero,
\end{equation*}
and that the observability Gramian
\begin{equation*}
\obsv = \begin{bmatrix}
\obsv_{11} & \bzero \\
\bzero & \bzero
\end{bmatrix}
\end{equation*}
satisfies the Lyapunov equation
\begin{equation*}
\begin{bmatrix}
\sA & \bzero \\ \bzero & \Abad
\end{bmatrix}\transp
\begin{bmatrix}
\obsv_{11} & \bzero \\
\bzero & \bzero
\end{bmatrix}
+ 
\begin{bmatrix}
\obsv_{11} & \bzero \\
\bzero & \bzero
\end{bmatrix}
\begin{bmatrix}
\sA & \bzero \\ \bzero & \Abad
\end{bmatrix}
+
\begin{bmatrix}
\sO\sO\transp & \bzero \\ \bzero & \bzero
\end{bmatrix}
= \bzero.
\end{equation*}
It is then straightforward to confirm that 
\begin{align*}
\frac{\partial\Loe(\Kbad)}{\partial \Abad}
& = 
2\obsv_{12}\transp\bSigma_{12} + 2\obsv_{22}\bSigma_{22} \\
& =
2\times\bzero\times\bzero + 2\times\bzero\times\bzero = \bzero, \\
\frac{\partial\Loe(\Kbad)}{\partial \Bbad}
& = 
2(\obsv_{12}\transp\Sigonesys\sC\transp + \obsv_{22}\bSigma_{12}\transp\sC\transp + \obsv_{22}\Bbad\sV ) \\
& = 
2(\bzero\transp\times\Sigonesys\sC\transp + \bzero\times\bzero\transp\times\sC\transp + \bzero\times\bzero\times\sV ) = \bzero, \\
\frac{\partial\Loe(\Kbad)}{\partial \Cbad}
& = 
2(\Cbad\bSigma_{22} - \sO\bSigma_{12}) \\
& = 
2(\bzero\times\bzero - \sO\times\bzero) = \bzero.
\end{align*}
Moreover, $\bT\Bbad=\bzero$ and $\Cbad\bT^{-1}=\bzero$ for all similarity transformations $\bT$.
Given that $\sB_{\opt}$ and $\sC_{\opt}$ are nonzero, 
it is clear that $\Kbad$ is not equivalent to $\kopt$ under any similarity transformation.
Hence, $\Kbad$ is suboptimal.

\subsection{The perils of enforcing minimality}\label{app:minimality_perils}

A classical result due to \cite{brockett76geometry} states that the set of minimal $n$-th order single input-single output transfer functions is the disjoint union of $n+1$ open sets.
Moreover, it is impossible for a continuous path through parameter space to pass from one of these open sets to another without entering a region corresponding to a non-minimal transfer function.
This implies that if one were to regularize so as to ensure minimality of the filter at every iteration,
the search will remain confined in the open set in which it is initialized,
unable to reach the set containing the optimal filter,
unless there is some mechanism (e.g. sufficiently large step size) by which to `hope' from one set two another. 
We now illustrate this phenomenon on a simple second-order ($\nx=2$) example.
We begin by characterizing the three open sets that partition the space of minimal second-order transfer functions;
cf. \cite[\S II]{brockett76geometry} for derivation.

\begin{fact}
Every strictly proper second-order transfer function with no pole-zero cancellations
belongs to exactly one of the following three open sets,
characterized as follows:
\begin{enumerate}
	\item both poles are real, and both residues are positive. This set is simply connected.
	\item poles are complex, or if both poles are real, then the residues have opposite signs. This set is not simply connected.
	\item both poles are real, and both residues are negative. This set is simply connected.
\end{enumerate}
\end{fact}
For the purpose of the following example, we shall refer to these sets as regions 1 to 3.

\begin{example}\label{ex:peril}
Consider \OE\ instance given by:
\begin{align*}
\sA &= \begin{bmatrix}
   -1.2901 &   -0.2626 \\
-0.2626 &  -0.2814
\end{bmatrix}, \quad
\sC = \begin{bmatrix}
0.5710 & -0.5093
\end{bmatrix}, \quad
\sO = \sC, \\
\sW &= \begin{bmatrix}
 3.0940 &  -1.5716 \\
-1.5716 &   1.2422
\end{bmatrix}, \quad
\sV = 1.
\end{align*}
It may be verified by straightforward calculations that the optimal filter $\kopt$ for this instance
belong to region 1.
Let $\kinit$ denote the filter from which policy search is initialized. $\kinit$ is given by:
\begin{equation*}
\sA_{\fparams_0} = \begin{bmatrix}
 -9.863 & -20.19 \\
17.4 & -4.143
\end{bmatrix}, \quad
\sB_{\fparams_0}  = \begin{bmatrix}
-1.499 \\
-16.44
\end{bmatrix}, \quad
\sC_{\fparams_0}  = \begin{bmatrix}
11.56  & -2.97
\end{bmatrix}.
\end{equation*}
Similarly, it may be readily verified that $\kinit$ belong to region 2.
\end{example}

We apply policy search to \cref{ex:peril}, using four different regularization strategies:
\begin{enumerate}[a.]
	\item No regularization, i.e. gradient descent on $\Loe(\fparams)$. 
	\item Regularization for controllability, i.e. gradient descent on $\Loe(\fparams) + \lambda\regctrb(\fparams)$,
	where $\regctrb(\fparams) := \| \gram_{\mathrm{ctr},\fparams} - \gram_{\mathrm{ctr},\fparams}^{-1} \|_\fro^2$
	and $\gram_{\mathrm{ctr},\fparams}$ is the controllability Gramian for $(\Ak,\Bk)$
	satisfying the Lyapunov equation $\Ak\gram_{\mathrm{ctr},\fparams} + \gram_{\mathrm{ctr},\fparams}\Ak\transp + \Bk\Bk\transp = 0$.
	\item Regularization for minimality, i.e. gradient descent on $\Loe(\fparams) + \lambda(\regctrb(\fparams) + \regobsv(\fparams))$,
where $\regobsv(\fparams) := \| \gram_{\mathrm{obs},\fparams} - \gram_{\mathrm{obs},\fparams}^{-1} \|_\fro^2$
and $\gram_{\mathrm{obs},\fparams}$ is the observability Gramian for $(\Ak,\Ck)$
satisfying the Lyapunov equation $\Ak\transp\gram_{\mathrm{obs},\fparams} + \gram_{\mathrm{obs},\fparams}\Ak + \Ck\transp\Ck = 0$.
\item The proposed algorithm \algname. 
\end{enumerate}
The results are presented in 
\cref{fig:perils}.
Observe that while all other methods eventually cross from region 2 (containing the initial $\kinit$) to region 1 (containing $\kopt$),
the method regularized to preserve minimality at each iteration remains `trapped' in region 2.

%

\begin{figure}[htbp]
	
	\centering
	\subfloat[No regularization.]{
		\includegraphics[width=7.5cm]{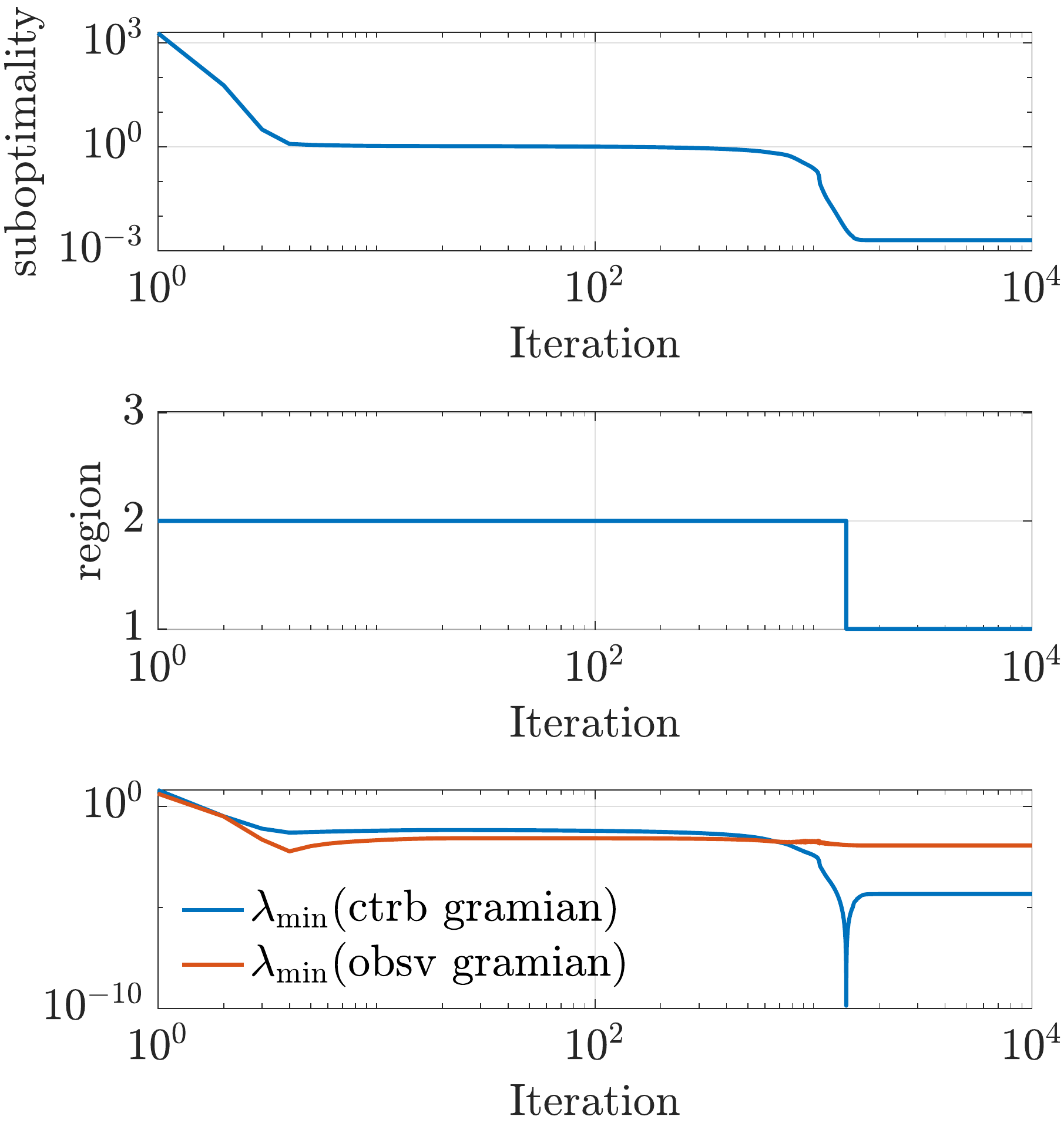}} \hspace{1em}
	\subfloat[Regularizing for controllability.]{
	\includegraphics[width=7.5cm]{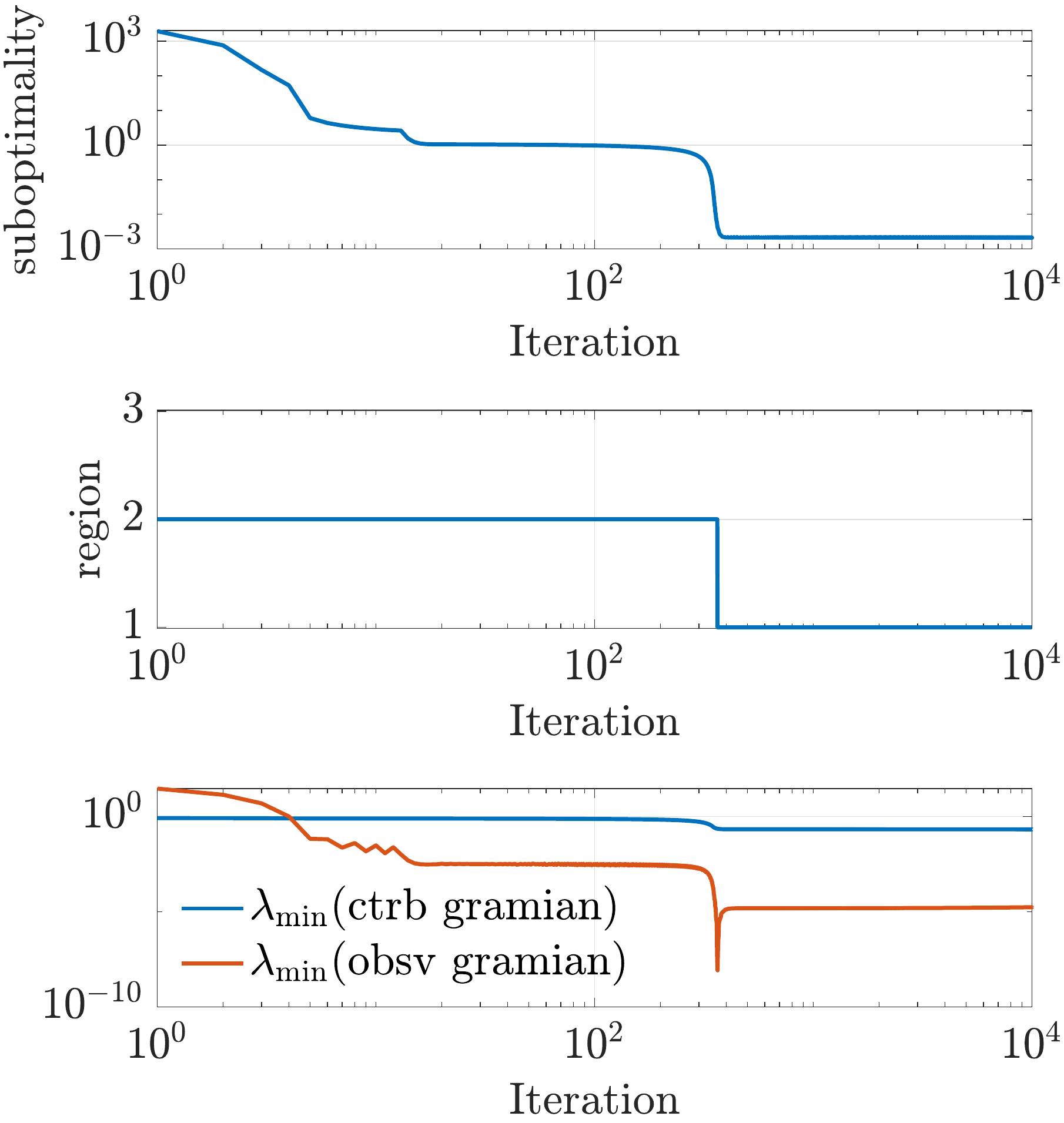}} \\
	\subfloat[Regularizing for minimality.]{
	\includegraphics[width=7.5cm]{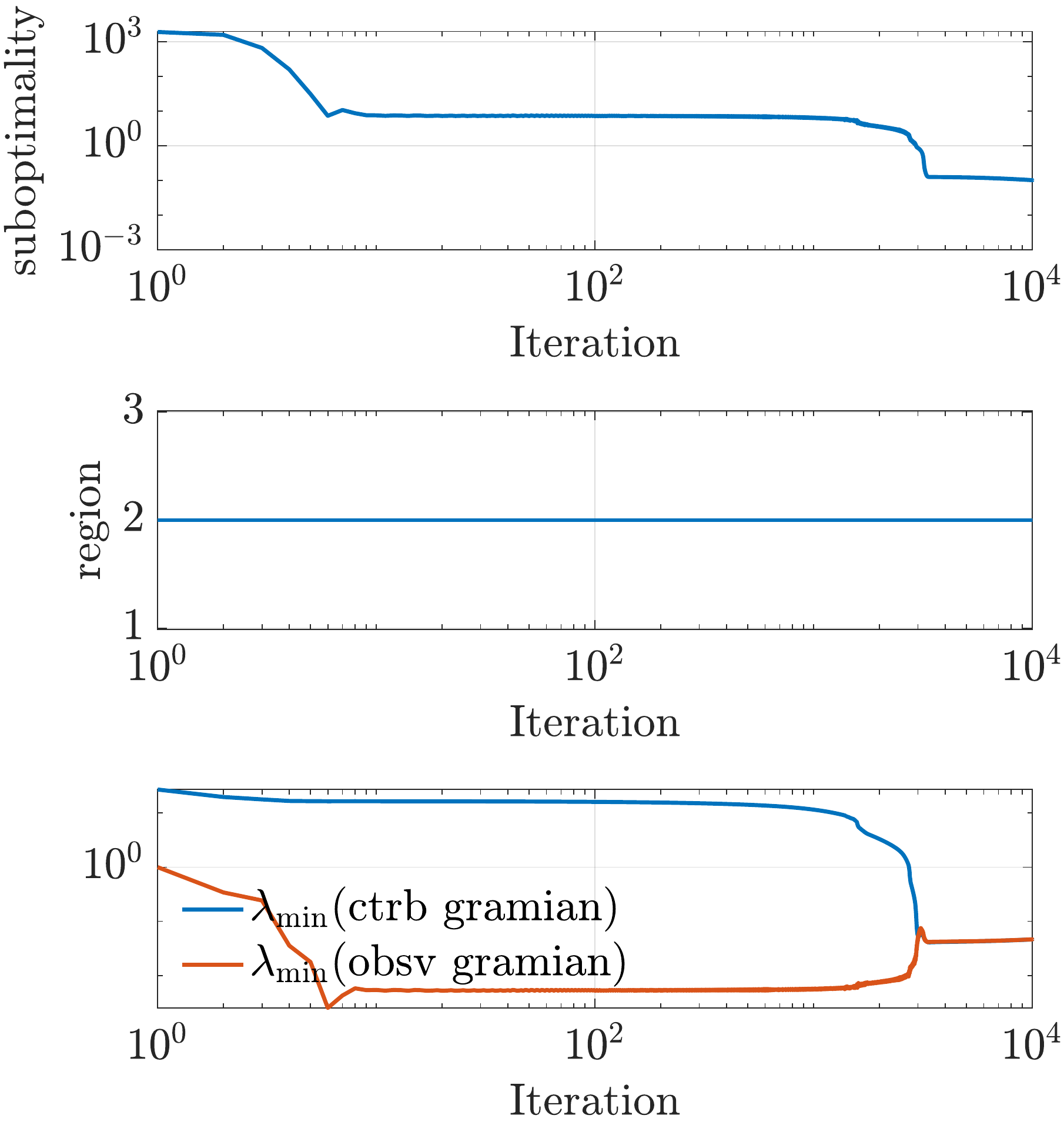}} \hspace{1em}
	\subfloat[\algname.]{
	\includegraphics[width=7.5cm]{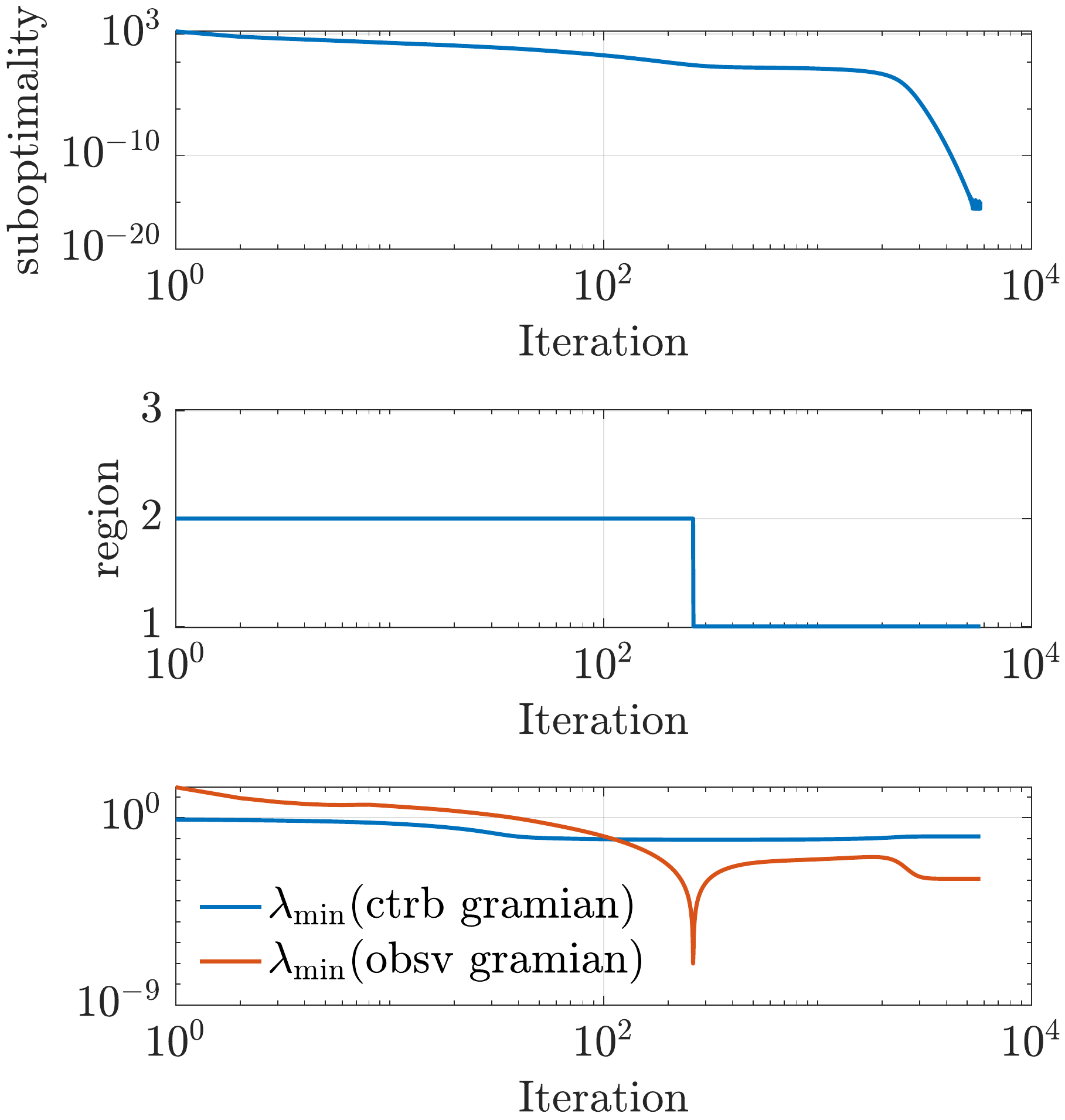}} 
	\caption{Suboptimality, region of parameter space, and controllability/observability as a function of iteration for \cref{ex:peril} and four different regularization strategies. 
		All searches are initialized at the same filter in region 2 of parameter space;
		the optimal filter is located in region 1.		
		(a) with no regularization, the iterate crosses from region 2 to region 1 with a loss of controllability.
		(b) regularizing for controllability, the iterate now crosses from region 2 to region 1 with a loss of observability instead.
		(c) regularizing for minimality, the iterate never crosses from region 2 to region 1.
		(d) under the proposed method, \algname, the iterate crosses from region 2 to region 1 with a loss of observability, and quickly converges to the global optimum.
	}
	\label{fig:perils}	
	
\end{figure}


\subsection{Details for \cref{ex:stationary_ctrb}}
\label{app:stationary_ctrb}

Consider the \OE\ instance given by
\begin{equation}\label{eq:example_oe}
\sA = \begin{bmatrix}
-1 & 0 \\ 0 & -1
\end{bmatrix}, \quad 
\sC = \eye_2, \quad
\sW = 3\times \eye_2, \quad
\sV = \eye_2,
\end{equation}
and the filter $\Kbad$ given by 
\begin{equation}\label{eq:example_filter}
\Abad = \begin{bmatrix}
-2 & 0 \\  \gamma & - \gamma
\end{bmatrix}, \quad
\Bbad = \begin{bmatrix}
1 & 0 \\ 0 & 0
\end{bmatrix}, \quad
\Cbad = \begin{bmatrix}
1 & 0 \\ 0 & 0
\end{bmatrix}.
\end{equation}
The following shows that the true system \cref{eq:example_oe} satisfies all our major assumptions, and that the filter \cref{eq:example_filter} is a critical point, but, because $\Sigkonetwoof{\Kbad}$ is not full rank, it is a strictly suboptimality, first-order critical point of $\Loe(\sfK)$. The following proposition is proven in \Cref{app:OE_bad_proof}. 
\begin{proposition}\label{prop:OE_bad}
	For the \OE\ instance \cref{eq:example_oe} and any $\gamma > 0$, and any filter $\Kbad$ of the form \cref{eq:example_filter}, the following are true:
	\begin{enumerate}[i.]
		\item \cref{eq:example_oe} satisfies \Cref{asm:stability,asm:pd,asm:observability,asm:ctrb_of_opt}.
		\item $\Kbad \in \calKstab$.
		\item $\Kbad$ is a first-order critical point:  $\nabla \Loe(\Kbad) = 0$. 
		\item The filter is strictly suboptimal: $\Kbad \notin \calKopt$.
		\item $\Kbad$ is controllable: $\Kbad\in \calKfull$, $\Sigma_{\Kbad,22} \succ 0$.	
		\item $\Sigkonetwoof{\Kbad}$ is not full rank.
	\end{enumerate} 
	Moreover, $\Loe(\sfK)$ does not depend on $\gamma$, showing that $\Loe$ does not have compact level sets.
\end{proposition}
Furthermore, as shown in \cref{fig:bad_Hessian} below, 
the minimum eigenvalue of the Hessian $\nablatwo\Loe(\Kbad)$ can be made arbitrarily close to zero
by taking $\gamma$ in \cref{eq:example_filter} to be arbitrarily large. 
Existing results suggest that first order methods may take take $\Omega(\poly(\epsilon))$-iterations to escape an approximate saddle point with minimum-Hessian eigenvalue $\epsilon$  \citep{jin2017escape,jin2018accelerated,carmon2018accelerated,agarwal2017finding}; hence, these large-$\gamma$ examples may prove challenging for first-order methods designed to escape approximate saddles. In addition, the non-compactness of the level sets for the \OE{} objective may also lead to a number of pathologies. 

\begin{figure}[H]
	\centering
	\includegraphics[width=13cm]{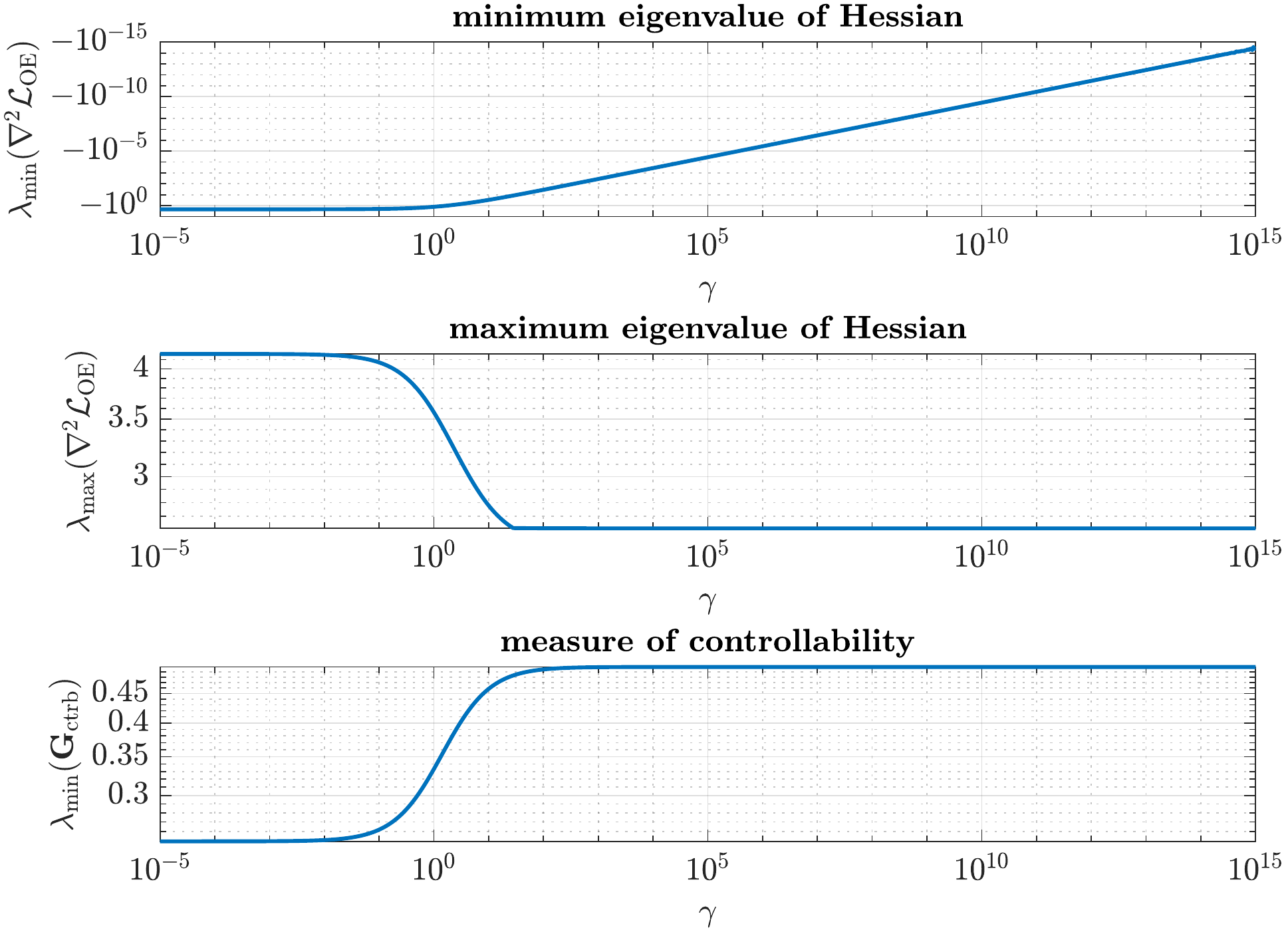}
	\caption{	
		Spectral properties of the Hessian $\nablatwo\Loe(\Kbad)$ in \cref{ex:stationary_ctrb}
		for various values $\gamma$, cf. $\Abad$ in \cref{eq:example_filter}.
		Here $\mathbf{G}_\mathrm{ctrb}$ denotes the controllability Gramian associated with $(\Abad,\Bbad)$.
	}
	\label{fig:bad_Hessian}
\end{figure}

\subsection{Proof of \Cref{prop:OE_bad}\label{app:OE_bad_proof}}
\paragraph{Part i. Assumptions.}
The matrix $\sA$ is Hurwitz stable, with eigenvalues $-1$ (repeated), meeting \Cref{asm:stability}. 
The pair $(\sA,\sC)$ is observable, as $\sC=\eye_2$, meeting \Cref{asm:observability}.
$\sW$ and $\sV$ are also clearly positive definite, meeting \Cref{asm:pd}. Lastly, one can show that  \begin{align*}
(\Ak,\Bk,\Ck) = (-2\eye_2, \eye_2,\eye_n)
\end{align*}
is an optimal filter. Clearly $(\Ak,\Bk)$ is controllable, so \Cref{asm:ctrb_of_opt} is met.

\paragraph{Part ii. Stability.}
As $\Abad$ is lower diagonal, the eigenvalues are easily seen to be $(-2,-\gamma)$. Hence $\Abad$ is Hurwitz stable.

\paragraph{Part iii. First-Order Critical Point.}
Decompose
\begin{align}
\Loe(\sfK) = \Exp[\|\sx - \fo_{\sfK}\|^2] &= \underbrace{\Exp[|\sx[1] - \fo_{\sfK}[1]|^2]}_{\cL_1(\sfK)} + \underbrace{\Exp[|\sx[2] - \fo_{\sfK}[2]|^2]}_{\cL_2(\sfK)}, \label{eq:LOE_L1_L2}
\end{align}
where $(\sx,\fo_{\sfK})$ are jointly distribution as $\cN(0, \Sigk)$.  We show $\sfK = \Kbad$ is a critical point of both $\cL_1(\sfK)$ and $\cL_2(\sfK)$. We start with $\cL_1(\sfK)$. 
\begin{claim}\label{claim:stat_point_l1} We have $\nabla_{\sfK} \cL_1(\sfK)\big{|}_{\sfK = \Kbad} = 0$. 
\end{claim}
\begin{proof} It suffices to show that $\sfK = \Kbad$ is global minimizer of $\cL_1(\cdot)$. This can be checked by showing that $(a_{\sfK},b_{\sfK},c_{\sfK}) = (-2,1,1)$ is the optimal solution to the one-dimensional scalar \OE{} problem with $(a,c,w_1,w_2) = (-1,1,3,1)$ and $z = 1$. Solving the scalar Continuous Algebraic Riccati Equation, we see that an optimal filter is of the form $(a_{\sfK},b_{\sfK},c_{\sfK}) = (a - \ell,1,1) $, where $\ell= w_{2}^{-1} c p = p$, and $p > 0$ solves the continuous Algebriac Ricatti Equation 
	\begin{align*}
	0 = ap + pa + p^2 b^2 w_2^{-1} + w_1   = -2p - p^2  + 3
	\end{align*}
	Taking the positive solution to the quadratic $0 = p^2 + 2p - 3 = (p+3)(p-1) $, we have $p = 1$. Hence, the optimal filter has $l = w_2^{-1} c p = 1$. Hence, $(a - \ell,1,1) = (-1-1,1,1) = (-2,1,1)$ is an optimal solution to the scalar \OE{} problem, as needed.
\end{proof}
Next, we address $\cL_2(\sfK)$. We begin with a lemma establishing the structure of $\Sigkonetwo$ for $\sfK = \Kbad$, proven in \Cref{app:sigonetwok}.
\begin{lemma}\label{lem:Sigkonetwo_comp} For $\sfK = \Kbad$, we have
	\begin{align*}
	\Sigkonetwo &= \begin{bmatrix} \frac{1}{2} & \frac{ \gamma}{2(1+\gamma)} \\
	0 & 0 \end{bmatrix}
	\end{align*}
\end{lemma}
We can now conclude by checking that $\Kbad$ is a criticial point of $\cL_2(\cdot)$. 
\begin{claim} We have $\nabla_{\sfK} \cL_2(\sfK)\big{|}_{\sfK = \Kbad} = 0$. 
\end{claim}
\begin{proof} For simplicity, we drop the subscripts involving $\sfK$.
	\begin{align*}
	\cL_2(\sfK) &= \Exp[|\sx[2] - \fo[2]|^2] = \Exp[|\sx[2] - \be_2^\top \Ck\fx|^2]\\
	&= \Exp[\sx[2]^2]  -  2\be_2^\top \Exp[\sx\fx^\top]\Ck^\top \be_2  + \be_2 ^\top\Cktwo\Exp[\fx \fx^\top ] \be_2^\top\Cktwo\\
	&= \Exp[\sx[2]^2]  -2 \be_2^\top \Sigkonetwo \Ck \be_2 + \be_2^\top\Ck\Sigktwo\Ck^\top \be_2\\
	&= \underbrace{\Exp[\sx[2]^2]}_{=\cL_2(\Kbad)}  - 2\be_2^\top (\Sigkonetwo - \Sigonetwobad) (\Ck-\Cbad)^\top \be_2 + \be_2^\top(\Ck-\Cbad)^\top\Sigktwo(\Ck-\Cbad)^\top\be_2,
	\end{align*}
	where above we use $\Cbad^\top \be_2 = 0$ and, as shown in in \Cref{lem:Sigkonetwo_comp}, $\be_2^\top \Sigonetwobad = 0$. In particular, for a perturbation  $\Delk = (\Dela,\Delb,\Delc)$,
	\begin{align*}
	&\cL_2(\Kbad + t\Delk) - \cL_2(\Kbad) \\
	&=  - t\be_2^\top (\bSigma_{12,\Kbad + t \Delk} - \Sigonetwobad)  \DelC^\top \be_2 + t^2\be_2^\top \Delc \bSigma_{22,\Kbad + t \Delk} \Delc^\top \be_2\\
	&=   - t\be_2^\top \Sigonetwobad  \DelC^\top \be_2  - t^2\be_2^\top \Delonetwo  \DelC^\top \be_2 + t^2\be_2^\top \Delc \Sigtwobad \Delc^\top \be_2 + O(t^3)\\
	&=    - t^2\be_2^\top \left(\ddt \bSigma_{\Kbad + t \Delk,12} \big{|}_{t=0}\right)  \DelC^\top \be_2 + t^2\be_2^\top \Delc \Sigtwobad \Delc^\top \be_2 + O(t^3),
	\end{align*}
	where again we use $\be_2^\top \Sigonetwobad  = 0$ by \Cref{lem:Sigkonetwo_comp}.  Thus, $\ddt \cL_2(\Kbad + t\Delk) = 0$, showing $\nabla \cL_2(\sfK) \big{|}_{\sfK = \Kbad} = 0$.
\end{proof}

\paragraph{Part vi. Suboptimality.} 
By solving the continuous algebraic Ricatti equation (in the spirit of \Cref{claim:stat_point_l1}), one can show that 
\begin{align*}
(\Ak,\Bk,\Ck) = (-2\eye_2, \eye_2,\eye_n)
\end{align*}
is \emph{an} optimal filter. It is clear that there is no similarity transformation which relates this filter to $\Kbad = (\Abad,\Bbad,\Cbad)$ (for one, the the rank of $\Bk$, $\Ck$ would be preserved under such a similarity transform). Since optimal filters are unique up to similarity transformation  (\Cref{lem:dgkf_optimal}), $\Kbad$ cannot be optimal. 

\paragraph{Part v. Controllability and rank of $\Sigktwo$ ( $\sfK \in \calKfull$) } 
As shown in \Cref{sec:controllable_nonsingular}, $\Sigktwo \succ 0$ provided that $(\Abad,\Bbad)$ is controllable. The latter can be verified  since $\Bbad = [e_1 \mid \mathbf{0}_2]$, and $e_1$ is not an eigenvector of $\Abad$.

\paragraph{Part vi. Rank of $\Sigkonetwo$} 
The computation in \Cref{lem:Sigkonetwo_comp} shows $\Sigkonetwo$ has rank $1$. 
\\
\\
This concludes the demonstration of points i-vi. To see uniform boundedness, we again decomposition $\Loe(\sfK) = \cL_1(\sfK) + \cL_2(\sfK)$ as in \Cref{eq:LOE_L1_L2}. Since $\cL_1(\sfK)$ is globally minimized at $\sfK = \Kbad$, and since $\hat{\mathbf{z}}[2] \equiv 0$ regardless of $\gamma$, we see $\Loe(\sfK)$ does not depend on $\gamma$. \qed

\subsection{Proof of \Cref{lem:Sigkonetwo_comp}}\label{app:sigonetwok}
\begin{proof}
	Writing out the Lyapunov equation (and using $*$ to ignore irrelevant blocks), 
	\begin{align*}
	&-\begin{bmatrix} 3 \eye_n & 0 \\
	* & * 
	\end{bmatrix} \\
	&= \begin{bmatrix} \bA & 0\\ 
	\Bk &  *\end{bmatrix} 
	\begin{bmatrix} \Sigonesys & \Sigkonetwo\\ 
	\Sigkonetwo^\top & * \end{bmatrix}
	+ \left(\begin{bmatrix} \bA & 0\\
	\Bbad &  \Abad
	\end{bmatrix} \begin{bmatrix} \Sigonesys & \Sigkonetwo\\
	\Sigkonetwo^\top & *\end{bmatrix}\right)^\top\\
	&=\begin{bmatrix} \bA \Sigonesys & \bA \Sigkonetwo \\
	\Bbad \Sigonesys + \Abad \Sigkonetwo^\top & *
	\end{bmatrix} + \left(\begin{bmatrix} \bA \Sigonesys & \bA \Sigkonetwo \\
	\Bbad \Sigonesys + \Abad \Sigkonetwo^\top & *
	\end{bmatrix}\right)^\top\\
	&= \begin{bmatrix} \bA \Sigonesys + \Sigonesys \bA^\top & \bA \Sigkonetwo +  (\Bbad \Sigonesys + \Abad \Sigkonetwo^\top)^\top\\
	*& *
	\end{bmatrix}
	\end{align*}
	Using $\bA = -\eye_2$, we have $-3\eye_{2} = -2\Sigonesys$, so $\Sigonesys = \frac{3}{2}\eye_2$. 
	Then, 
	\begin{align*}
	0 &= \bA \Sigkonetwo +  (\Bbad \Sigonesys + \Abad \Sigkonetwo^\top)^\top \\
	&= -\Sigkonetwo + \frac{1}{2} (3\eye_2)\Bbad^\top + \bSigma_{12} \Abad^\top \\
	&= \frac{3}{2}\Bbad^\top + \Sigkonetwo (\Abad - \eye_n)^\top,
	\end{align*}
	so that 
	\begin{align*}
	\Sigkonetwo &= -\frac{3}{2}\Bbad^\top (\Abad - \eye_n)^{-\top}
	\end{align*}
	Next,
	\begin{align*}
	(\Abad - \eye_n)^{-1} &= -\left(\begin{bmatrix} 1 + a_{\star} & 0 \\
	- \gamma & 1 + \gamma\end{bmatrix}\right)^{-1} \\
	&= -\begin{bmatrix} (1 + a_{\star})^{-1} & 0 \\
	\frac{ \gamma}{(1+a_{\star})(1+\gamma)} & (1 + \gamma)^{-1}\end{bmatrix}\\
	&= \begin{bmatrix} -\frac{1}{3} & 0 \\
	\frac{ -\gamma}{3(1+\gamma)} & \frac{-1}{1+\gamma}\end{bmatrix}
	\end{align*}
	So, substituing in the definition of $\Bbad$
	\begin{align*}
	\Sigkonetwo &= -\frac{3}{2}\begin{bmatrix} 1 & 0 \\ 0 & 0 \end{bmatrix}^\top\begin{bmatrix} -\frac{1}{3} & 0 \\
	\frac{ -\gamma}{3(1+\gamma)} & \frac{-1}{1+\gamma}\end{bmatrix}^\top= \begin{bmatrix} \frac{1}{2} & \frac{ \gamma}{2(1+\gamma)} \\
	0 & 0 \end{bmatrix}
	\end{align*}.
\end{proof}

%% file: appendix/concluding_proof_app.tex

\subsection{Proof of \Cref{thm:main_rate_backtrack}}\label{sec:thm_backtrack}
The proof is nearly identical to that of \Cref{thm:main_rate}. The only difference is that the step sizes are selected according to backtracking line search. We apply \Cref{prop:reco_descent} where $\eta_s$ (in the statement of the proposition) is set to any $ \eta \in \Sback$ for all $s$ satisfying the same upper bound $\eta \le \frac{1}{\cC_1}$ required in \Cref{thm:main_rate}. Since since backtracking line search selects the step which attains the greatest direction of descent, at each iteration, we have
\begin{align*}
\cL_\lambda(\sfK_{t+1}) \le \cL_\lambda(\tilde\sfK_{t} - \eta \nabla \cL_\lambda(\tilde\sfK_{t})).
\end{align*}
Hence, backtracking satisfies the descent condition \Cref{eq:reco_updates_b}, and the theorem follows. 

\subsection{Proof of \Cref{lem:constant_intermediate_simplification.}}\label{app:intermediate_simplification} Recall that, for $\calK \in \cK_0$, \begin{align*}
\Loe(\sfK) \le \cL_{\lambda}(\sfK_0), \quad \|\Zk^{-1}\| \le \frac{1}{\lambda}\cL_{\lambda}(\sfK_0).
\end{align*}
Hence, for $\calK \in \cK_0 $ and $\Cpl(\sfK)$ as in \Cref{cor:weak_PL}
\begin{align*}
\Cpl(\sfK) &= \polyop\left(\bA,\bC,\bW_2^{-1},\Zk^{-1},\Sigk,\Sigk^{-1},\Loe(\sfK) \right)\\
&\le  \polyop\left(\bA,\bC,\bW_2^{-1},\Sigk,\Sigk^{-1},\cL_{\lambda}(\sfK_0),\frac{1}{\lambda} \right).
\end{align*}
In addition, from \Cref{prop:clyap_compact_form},
\begin{align*}
\conslyapK  = \polyop\left(\Sigk,\Sigk^{-1},\Zk^{-1},\bC,\bW_1^{-1},\bW_2^{-1}\right) \le \polyop\left(\Sigk,\Sigk^{-1},\bC,\bW_1^{-1},\bW_2^{-1},\cL_{\lambda}(\sfK_0),\frac{1}{\lambda}\right).
\end{align*}

Moreover, from \Cref{prop:DCL_for_Kalman}
\begin{align*}
\max\{\|\Ak\|_{\op},\|\Bk\|_{\op},\|\Ck\|_{\fro}\} &\le\polyop\left(\bA,\bC,\bW_2^{-1},\Zk^{-1},\Sigk,\Sigk^{-1},\Loe(\sfK)\right)\\
&\le\polyop\left(\bA,\bC,\bW_2^{-1},\Sigk,\Sigk^{-1},\cL_{\lambda}(\sfK)\right).
\end{align*}
Finally, from \Cref{prop:der_bounds}, we have for $\sfK \in \cK_0$,
\begin{align*}
\Csigone(\sfK),\Cgradone(\sfK),\Cgradtwo(\sfK) &= \polyop(\Zk^{-1}, \Sigktwo^{-1},\Sigk,\Bk, \Ck,\bC,\sO,\bW_2)\\
&\le \polyop(\Sigk,\Bk,\Ck, \bC,\sO,\bW_2,\cL_{\lambda}(\sfK_0),\frac{1}{\lambda})\\
&\le \polyop(\Sigk^{-1},\Sigk,\bA, \bC,\sO,\bW_2,\bW_2^{-1},\cL_{\lambda}(\sfK_0),\frac{1}{\lambda})
\end{align*}
Hence, in summary, 
\begin{align*}
&\Cpl(\sfK),\conslyapK,\Csigone(\sfK),\Cgradone(\sfK),\Cgradtwo(\sfK) \\
&\qquad = \polyop(\Sigk^{-1},\Sigk,\bA, \bC,\sO,\bW_2,\bW_2^{-1},\bW_1^{-1},\cL_{\lambda}(\sfK_0),\frac{1}{\lambda}).
\end{align*}

\qed
\subsection{Conditioning of the stationary covariance (\Cref{lem:Sigma_K_conditioned})}\label{app:Sig_K_conditioned}
	\paragraph{Part (a).} Recall the block decomposition
	\begin{align*}
	\Sigk = \begin{bmatrix} \Sigonesys& \Sigkonetwo\\
	\Sigkonetwo^\top & \Sigktwo
	\end{bmatrix},
	\end{align*}
	where we note that $\Sigonesys$ does not depend on $\sfK$.
	From the Schur complement test, $\Sigk \succ 0$ if and only if both $\Sigktwo \succ 0$ and $\Sigonesys \succ \Sigkonetwo\Sigktwo^{-1}\Sigkonetwo^\top = \Zk$. The first of these holds for $\sfK \in \calKfull$, and since $\Zk \preceq \Zst$ (for $\Zst$ as in \Cref{lem:Z_maximal}), the second holds from \Cref{lem:sigst}. 

	\paragraph{Part (b).} We invoke  \Cref{lem:diag_norm_bound} below to bound
	\begin{align*}
	\|\Sigk^{-1}\|  &\le 2\|\Sigktwo^{-1}\| + 2\|\bX_{\sfK}^{-1}\|\max\{1, \|\Sigktwo^{-1}\|\|\Sigonesys\|\}, 
	\end{align*}
	where $\bX_{\sfK} = \Sigonesys - \Sigkonetwo\Sigktwo^{-1}\Sigkonetwo^\top = \Sigonesys - \Zk$ is the Schur complement term.  Moreover, since $\Zk \preceq \Zst$, $\bX_{\sfK}^{-1} \preceq  (\Sigonesys - \Zst)^{-1}$, so $\|\bX_{\sfK}^{-1}\| \le \|(\Sigonesys - \Zst)^{-1}\| = 1/\lambda_{\min}(\Sigonesys - \Zst)$. Hence, 
	\begin{align*}
	\|\Sigk^{-1}\|  &\le 2\|\Sigktwo^{-1}\| + 2[\lambda_{\min}(\Sigonesys - \Zst)]^{-1}\max\{1, \|\Sigktwo^{-1}\|\|\Sigonesys\|\}, 
	\end{align*}
	as needed. By \Cref{lem:sigst}, we have $\sigst = \lambda_{\min}(\Sigonesys - \Zst)$
	\paragraph{Part (c).} Invoking \Cref{lem:diag_norm_bound} part (a), we directly obtain $\|\Sigk\| \le 2\max\{\|\Sigonesys\|,\|\Sigktwo\|\}$. By 

Now the remaining part is to prove the following Lemma. 

	\begin{lemma}\label{lem:diag_norm_bound} Suppose that $\bLambda \succeq 0$ is positive semidefinite and has  block-diagonal decomposition with blocks diagonal blocks $\bLambda_{11},\bLambda_{22}$. Then, 
	\begin{itemize}
	\item[(a)] $\|\bLambda\| \le 2\max\{\|\bLambda_{11}\|,\|\bLambda_{22}\|\}$. 
	\item[(b)] If in addition $\bLambda \succ 0$, then defining the Schur complement $\bX := \Lamone - \Lamonetwo\Lamtwo^{-1}\Lamonetwo^\top$, we have
	\begin{align*}
	\|\bLambda^{-1}\|  &\le 2\|\Lamtwo^{-1}\| + 2\|\bX^{-1}\|\max\{1, \|\Lamtwo^{-1}\|\|\Lamone\|\}. 
	\end{align*}
	\end{itemize}
	\end{lemma}
	\begin{proof} We prove each part in sequence: 
	\paragraph{Part (a).} It suffices to prove that
	\begin{align*}
	\bLambda = \begin{bmatrix}\bLambda_{11} & \bLambda_{12}\\ \bLambda_{12}^\top & \bLambda_{22}\end{bmatrix}\preceq 2\bar{\bLambda}, \quad \text{where } \bar{\bLambda} := \begin{bmatrix}\bLambda_{11} & 0 \\ 0 & \bLambda_{22}\end{bmatrix}
	\end{align*}
	To show the above, consider any vector $\bv = (\bv_1,\bv_2)$. First, for the modified vector $\tilde{\bv} = (\bv_1,-\bv_2)$, we compute
	\begin{align*}
	0 \le \tilde\bv^\top \bLambda \tilde\bv &=\bv_1^\top\bLambda_{11}\bv_1 + \bv_2^\top\bLambda_{22}\bv_2 - 2 \bv_1^\top\bLambda_{12}\bv_2. 
	\end{align*}
	Hence, 
	\begin{align*}
	\bv^\top \bLambda \bv &=\bv_1^\top\bLambda_{11}\bv_1 + \bv_2^\top\bLambda_{22}\bv_2 + 2 \bv_1^\top\bLambda_{12}\bv_2 \\
	&\le 2\bv_1^\top\bLambda_{11}\bv_1 + 2\bv_2^\top\bLambda_{22}\bv_2 = 2\bv^\top \bar{\bLambda}\bv.
	\end{align*}
	\paragraph{Part (b).} Introduce the $\bX := \Lamone - \Lamonetwo\Lamtwo^{-1}\Lamonetwo^\top$ as the Schur-complement term. From the block-matrix inversion formula,
	\begin{align*}
	\bLambda^{-1}  = \begin{bmatrix} \bX^{-1} & *\\
	* & \Lamtwo^{-1}\left(\bI +  \Lamtwo^{-1/2}\Lamonetwo^\top\bX^{-1}\Lamonetwo\Lamtwo^{-1/2}\right)
	\end{bmatrix}.
	\end{align*}
	From part (a), we then bound
	\begin{align*}
	\|\bLambda^{-1}\|  &\le 2\max\left\{\|\bX^{-1}\|, \|\Lamtwo^{-1}\left( \bI +  \Lamtwo^{-1/2}\Lamonetwo^\top\bX^{-1}\Lamonetwo\Lamtwo^{-1/2}\right)\|\right\}\\
	&\le 2\max\left\{\|\bX^{-1}\|, \|\Lamtwo^{-1}\|\left( 1 +  \|\bX^{-1}\|\cdot\|\Lamonetwo\Lamtwo^{-1/2}\|^2\right)\|\right\}.
	\end{align*}
	The term $\|\Lamonetwo\Lamtwo^{-1/2}\|^2 = \|\Lamonetwo\Lamtwo^{-1}\Lamonetwo^\top\|  \le \|\Lamone\|$, where we used that $\Lamonetwo\Lamtwo^{-1}\Lamonetwo^\top \preceq \Lamone$ by the Schur complement test. Hence, we conclude

	\begin{align*}
	\|\bLambda^{-1}\|  &\le 2\max\left\{\|\bX^{-1}\|, \|\Lamtwo^{-1}\|\left( 1 +  \|\bX^{-1}\|\cdot\|\Lamone\|\right)\|\right\} \\
	&\le 2\|\Lamtwo^{-1}\| + 2\|\bX^{-1}\|\max\{1, \|\Lamtwo^{-1}\|\|\Lamone\|\},
	\end{align*}
	which completes the proof of \Cref{lem:diag_norm_bound}. 
	\end{proof}
	This finally completes the proof of \Cref{lem:Sigma_K_conditioned}. 	
	\qed

\subsection{Proof of \Cref{lem:compact_sets}}
To see that $\cK_0$ is bounded, we use that $\|\Ak\|,\|\Bk\|,\|\Ck\|$ are uniformly bounded on $\cK_0$. This is a consequence of the bounds on these parameters in \Cref{prop:DCL_for_Kalman}, as well as the fact that the various terms in those bounds are in terms of $\|\Sigk\|,\|\Sigk^{-1}\|,\|\Zk^{-1}\|$ and $\Loe(\sfK)$, all of which are shown to be uniformly bounded on $\cK_0$.

To show $\cK_0$ is closed, it suffices to show that for any convergent sequence of controllers $\sfK^{(i)}$ in $\cK_0$, its limit is in $\cK_0$. In light of the boundness discussion above, this follows directly from the following lemma. 
 \begin{lemma} Let $\sfK^{(i)} \in \calKexp$ be a sequence of controllers converge to some $\sfK$, such that $\|\bSigma_{\sfK^{(i)}}\|,\|\bSigma_{\sfK^{(i)}}^{-1}\|,\|\bZ_{\sfK^{(i)}}^{-1}\|$, as well as $\|\bA_{\sfK^{(i)}}\|,\|\bB_{\sfK^{(i)}}\|,\|\bC_{\sfK^{(i)}}\|$ remain uniformly bounded. Then, $\sfK \in \calKexp$. 
 \end{lemma}
 \begin{proof}We prove stability, $\Sigktwo \succ 0$, and $\Zk \succ 0$ in succession. 
\newcommand{\Aclki}{\bA_{\mathrm{cl},\sfK^{(i)}}}
\newcommand{\Wclki}{\bW_{\mathrm{cl},\sfK^{(i)}}}
\newcommand{\Sigki}{\bSigma_{\sfK^{(i)}}}
\newcommand{\Sigkij}{\bSigma_{\sfK^{(i_j)}}}
\newcommand{\Zki}{\bZ_{\sfK^{(i)}}}

 \paragraph{Stability.} Let $\bGamma^{(i)} := \circnorm{\clyap(\bA_{\mathrm{cl},\sfK^{(i)}},\eye_{2n})}$. Then, $\sup_{i} \|\bGamma^{(i)}\| \le M$ for some $M > 0$. Moreover, for any $\epsilon > 0$  and $i \ge i_0$ sufficiently large, we have $\|\Aclk - \bA_{\mathrm{cl},\sfK^{(i)}}\|\le \epsilon$. Thus, for such $i \ge i_0$,
 \begin{align}
 \Aclk \bGamma^{(i)} + \bGamma^{(i)}\Aclk^\top \preceq \Aclki\bGamma^{(i)} + \bGamma^{(i)} \Aclki^\top  + 2 M \epsilon \eye_{2n} = -\eye_{2n}(1-2M \epsilon). \label{eq:Lyap}
 \end{align}
 Hence, for $\epsilon = 1/4M$, $\Aclk \bGamma^{(i)} + \bGamma^{(i)}\Aclk^\top \preceq -\frac{1}{2}\eye_{2n}$. Since $\bGamma^{(i)} \succ 0$, this implies $\Aclk$ is stable. 

 \paragraph{Controllability.} Define the functions $F_i(\bSigma) := \Aclki \bSigma + \bSigma \Aclki + \Wclki $, so that $\Sigki$ is the unique PSD solution to $F_i(\Sigki) = 0$. By \Cref{prop:clyap_compact_form}, $0 \preceq \Sigki \preceq M\eye_{2n} $ for some $i \ge 0$. Hence, there is a  subsequence $i_j$ such that $\Sigkij$ converges to a limit $\bar{\bSigma}$ on the set $\cX := \{\bSigma: 0 \preceq \bSigma \preceq M\eye_{2n}\}$. Since $\|\Ak^{(i)}\|,\|\Bk^{(i)}\|,\|\Ck^{(i)}\|$ remain uniformly bounded, $F_i \to F(\bSigma) := \Aclk \bSigma + \bSigma \Aclk + \Wclk $ uniformly on this set $\cX$, and thus, $F(\bar{\bSigma}) = \lim_{j \to \infty} F_j(\Sigkij) = \bzero$. Hence, since $\Aclk$ is stable as established above, $\bar{\bSigma} = \Sigk$. Since this holds for all subsequences, we have $\lim_{i \to \infty} \Sigki = \Sigk$. Hence, $\Sigk \succ 0$, since by assumption $\|\Sigki^{-1}\|$ is uniformly bounded in $i$. Thus $\Sigktwo \succ 0$, and thus, $\sfK \in \calKfull$.

\paragraph{Informativity.} As established above, $\lim_{i \to \infty} \Sigki = \Sigk$. Since the transformation mapping $\Sigki \to \Zki$ is continuous for $\Sigki \succ 0$, we see that $\lim_{i \to \infty} \Zki = \Zk$. Hence, since $\Zki \succ 0$ and $\Zki^{-1}$ is uniformly bounded, $\Zk \succ 0$. Thus, $\sfK \in \calKexp$.
 \end{proof}

%% file: appendix/optimization_proofs.tex
\section{Proofs for \DCL s and Gradient Descent}\label{sec:proof_DCL_GD} 

\subsection{Proof of  \cref{fact:no_suboptimal_sp}}\label{app:cvx_fact}
Throughout, we use the notation $\domsub(f) := \{\farg \in \dom(f): f(\farg) > \inf(f)\}$.
\cvxfact*

\begin{proof}
Let us proceed by contradiction. 
Suppose that $\sbopt$ is a suboptimal stationary point,
i.e. $\grad \jf(\sbopt) =0$ but $f(\sbopt) \neq \min_\jx \ \jf(\jx)$.
Let $\bar{\jn}$ be such that $\sbopt=\jch(\bar{\jn})$.
By surjectivity of $\jch$, such a $\bar{\jn}$ always exists.
Next, by application of the chain rule to $\jh(\jn)\defeq\jf(\jch(\jn))$,
we have 
\begin{equation}\label{eq:chain_rule}
\grad \jh(\jn)|_{ \jn = \bar{\jn}} = \grad \jf(\jx)|_{\jx = \sbopt} \cdot \grad \jch(\jn)|_{\jn = \bar{\jn}}.
\end{equation}
Therefore, by \cref{eq:chain_rule}, $\grad \jf(\sbopt) =0$ implies $\grad \jh(\bar{\jn})=0$.
However, 
\begin{equation}\label{eq:cvx_contradiction}
\jh(\bar{\jn}) 
\labelrel={eq:cvx_a} 
\jf(\jch(\bar{\jn})) 
\labelrel={eq:cvx_b} 
\jf(\sbopt) 
\labelrel\neq{eq:cvx_c}
\min_\jx \jf(\jx) 
\labelrel={eq:cvx_d} 
\min_{\jn} \jh(\jn),
\end{equation}
where 
\eqref{eq:cvx_a} follows by definition of $\jh$,
\eqref{eq:cvx_b} follows from $\sbopt=\jch(\bar{\jn})$,
\eqref{eq:cvx_c} follows by suboptimality of $\sbopt$,
and 
\eqref{eq:cvx_d} follows by the definition of $\jh$ and surjectivity of $\jch$.
However, \cref{eq:cvx_contradiction} contradicts the fact that $\jh$ is a convex function, for which all stationary points must be globally optimal.
Therefore, no such suboptimal stationary point $\sbopt$ can exist.
\end{proof}

\subsection{Proof of \Cref{thm:DCL}}\label{app:thm:DCL}

We prove \Cref{thm:DCL}, which can be thought of as a (considerable) strengthening of \Cref{fact:no_suboptimal_sp}. The theorem pertains \DCL s, whose definition we recall below. 
\dcldef*

Before beginning the proof, we explain why the following ``trivializing'' reparametrization is inadequate. 
\begin{remark}[Failure of the trivializing reparametrization]\label{rem:trivial_par} 
\newcommand{\ftillift}{\tilde{f}_{\subfont{lft}}}
\newcommand{\Phitil}{\tilde{\Phi}}
Given a \DCL{}  $\dcltriple$, it may seem that one can avoid the dependence on $\sigma_{d_z}(\nabla \Phi)$ with the following \emph{trivializing reparametrization} obtained by (a) augmenting the lifted parameters $(\farg,\xiarg)$ with the convex parameter $\cvxarg$ and (b) defining a new candidate \DCL{} $(\fcvx,\ftillift,\Phitil)$ given by $\ftillift(\farg,\xiarg,\cvxarg) = \flift(\farg,\xiarg)$ and $\Phitil(\farg,\xiarg,\cvxarg) = \cvxarg$. Note then that $\sigma_{d_z}(\Phitil) = 1$ since $\Phitil$ just projects onto the $\cvxarg$-coordinates, so this would circumvent the dependence on $\sigma_{d_z}(\nabla \Phi)$. In addition, $(\fcvx,\ftillift,\Phitil)$ meets the first two \DCL{} two criteria: $\fcvx$ is convex and $f(\farg) = \min_{(\xiarg,\cvxarg)}\ftillift(\farg,\xiarg,\cvxarg)$. However, the candidate \DCL{} does not meet the third criterion of \Cref{defn:DCL}  since the value of $\ftillift(\farg,\xiarg,\cvxarg)$ does not depend on $\cvxarg$, so $\ftillift(\farg,\xiarg,\cvxarg) \ne \fcvx(\cvxarg) = \fcvx(\Phitil(\farg,\xiarg,\cvxarg))$ in general. 
\end{remark}

We now begin the proof. We first define a notion of descent direction for functions which strictly generalizes the gradient:
    \begin{definition}[Cauchy Directions] 
    Let $f: \R^d \to \Rebar$ be a proper function, and $\farg \in \dom(f)$.
    \begin{itemize}
    \item[(a)] We say $\gvec \in \R^d$ is an \emph{Cauchy direction} of $f$ at $\farg$ if there exists constants $\epsilon_0 > 0$ such, that for all $\epsilon \in [0,\epsilon_0]$, $\farg-\epsilon \gvec \in \dom(f)$ and $\lim_{\epsilon \to 0^+}\frac{f(\farg-\epsilon \gvec)-f(\farg)}{\epsilon} \le - \|\gvec\|^2$
    \item[(b)] We say $\gvec \in \R^d$ is a \emph{generalized Cauchy direction} of $f$ at $\farg$ if, for some $\epsilon_0 > 0$ the exists a $\ccone$ curve $\phi := [0,\epsilon_0] \to \dom(f)$ such that $\phi(0) = \farg$, $\phi'(0) = g$, and 
    $\lim_{\epsilon \to 0^+}\frac{f(\phi(\epsilon))-f(\farg)}{\epsilon} \le - \|\gvec\|^2$. 
    \end{itemize}
    \end{definition}
    Observe that if $f$ is $\ccone$ at $x$, then the standard gradient $\nabla f(\cvxarg)$ is a Cauchy direction at $x$; indeed, our nomencalture is a tribute to the 1847 article in which Augustin-Louis Cauchy first described gradient descent, justifying its use via the computation  $f(\farg - \eta \nabla f) =f(\farg) -  \eta\|\nabla f\|^2 + o(\eta)$ (for more in depth history, see e.g. \cite{lemarechal2012cauchy}).  The purpose of generalized Cauchy direction is to accommodate  functions whose domains may not contain the segment  $\{\farg - \epsilon \gvec\}$, but may contain a curve $\phi$ with the same slope.

\paragraph{Cauchy directions for convex functions. }  At all high level, we show weak-PL by first showing that $\fcvx$ has a Cauchy direction at $\cvxarg$ of magnitude $\approx \fcvx(\cvxarg) - \inf(\fcvx)$, and then subsequently showing similar Cauchy directions for $\flift$ and $f$. For convex functions $\fcvx$, we can usally construct Cauchy directions using the subgradient at $\cvxarg$, a vector $\gvec$ such that $\fcvx(\cvxarg)-\fcvx(\cvxarg') \le \gvec^\top (\cvxarg - \cvxarg')$ for all other $\cvxarg' \in \dom(\fcvx)$. However, in certain pathological cases, the subgradient may not exist. 

Hence, we take a more conservative approach by showing considering not the whole domain of $\fcvx$, but rather the line segment joining $\cvxarg$ to any minimizer $\cvxargst$, defining the function
\begin{align}
\psi(t) = \fcvx(\cvxarg + t(\cvxargst - \cvxarg)) \label{eq:psi_def}
\end{align}
This approach allows for pathological cases where the subgradient is ``infinite'' (in the sense of $h = \infty$, in the sense of the proof below. )
    \newcommand{\cvxargdist}{\|\cvxarg-\cvxargst\|}
    \begin{lemma}\label{lem:convex_Cauchy} Suppose that $\fcvx$ is a proper convex function, with $\cvxargst \in \argmin_{\cvxarg} \fcvx(\cvxarg)$ attained. Then, for any $\cvxarg \in \dom_{>}(\fcvx)$, $\fcvx$  admits a Cauchy direction $g$ satisfying
    \begin{align}
    \|\gvec\| \ge \frac{\fcvx(\cvxarg)-\inf(\fcvx)}{\cvxargdist} \label{eq:cauchy_direc_convex}
    \end{align}
    \end{lemma}
    \begin{proof} 

     Recall $\psi(t)$ from \Cref{eq:psi_def}, and define the secant-approximation function $\phi(t) := \frac{\psi(t) - \psi(0)}{t}$ for $t \in (0,1]$. From convexity, one can check that $\phi(t)$ is non-increasing on $t \in (0,1]$. Hence, the limit
     \begin{align*}
     h = \lim_{t \to 0^+} \phi(t) = \lim_{t \to 0^+} \frac{\psi(t) - \psi(0)}{t}
     \end{align*}
     exists, and has $h \in \{-\infty\} \cup (-\infty,\phi(1)]$, where again, since $\phi(t)$ is non-increasing, we note that
     \begin{align}
     h \le \phi(1) = -(\fcvx(\cvxarg) - \inf(\fcvx)) \label{eq:h_ub}
     \end{align}
     Let us first assume $h \ne - \infty$. We now claim that $\gvec = -|h|(\cvxargst - \cvxarg)/\cvxargdist^2$ is a Cauchy direction of $f$ at $x$; this will conclude the proof since by \Cref{eq:h_ub}
    \begin{align*}
    \|\gvec\|=\frac{|h|}{\cvxargdist} \ge \frac{|f(\cvxarg)-f(\cvxargst)|}{\cvxargdist}.
    \end{align*}
    Let us show that $\gvec$ is a Cauchy direction. First, since $f$ is convex, $\dom(f)$ is convex.  Thus, since $\cvxarg,\cvxargst \in \dom(f)$, the line seqment joining $\cvxarg,\cvxargst$ is contained in $\dom(f)$, and hence for $\epsilon$ sufficiently small, $\cvxarg-\epsilon \gvec$ lies on this line segment, and is therefore also contained in $\dom(f)$.

    Next, we compute 
    \begin{align*}
    h = \lim_{t \to 0^+} \frac{f(\cvxarg + t(\cvxargst - \cvxarg))-f(\cvxarg)}{t} &= \lim_{t \to 0^+} \frac{f(\cvxarg - t \cvxargdist^2 \frac{\gvec}{|h|})-f(\cvxarg)}{t}\\
    &= \frac{\cvxargdist^2}{|h|} \cdot  \lim_{t \to 0^+} \frac{f(\cvxarg - t \gvec )-f(\cvxarg)}{t}.
    \end{align*}
    Hence, 
    \begin{align*}
    \lim_{t\to 0^+} \frac{f(\cvxarg - t \gvec )-f(\cvxarg)}{t} &= \frac{h|h|}{\cvxargdist^2} = \frac{-h^2}{\cvxargdist^2} = -\|\gvec\|^2,
    \end{align*}
    as needed. Now, consider the case where $h = - \infty$. Then, for any $\eta > 0$, we see that $\lim_{t \to 0^+} \frac{f(\cvxarg + t\cdot \eta(\cvxargst - \cvxarg)) - f(\cvxarg))}{t} = - \infty$. Hence, $\gvec = \eta \cdot (\cvxargst - \cvxarg)$ is Cauchy direction for any $\eta > 0$. In particular, taking $\eta = \frac{f(\cvxarg) - f(\cvxargst)}{\cvxargdist}$ satisfies the conclusion of the lemma.
    \end{proof}
    
    \paragraph{Smooth transformations preserve Cauchy directions.} We show that the existence of a Cauchy direction is preserved under smooth transformations.
    \newcommand{\fbar}{\bar{f}}
    \begin{lemma}\label{lem:Cauchy_change_of_basis} Let $\fbar$ be a proper function, $\cvxarg\in \dom (f)$, and $\gvec$ a Cauchy direction of $\fbar$ at $\cvxarg$. Let $\Psi$  be a $\ccone$ mapping from a neighborhood $\cvxset$ containing $ \cvxarg$ into a domain $\tilset$ such that $\sigma_{\dcvx}(\nabla\, \Psi(\cvxarg)) > 0$, and let $\flift: \tilset \to \Rebar$ satisfy $\fbar(\cvxarg') = \flift(\Psi(\cvxarg'))$ for all $\cvxarg' \in \cvxset$. Then, $\flift$ has a generalized Cauchy direction $\gtil$ at $\tilarg = \Psi(\cvxarg)$ of norm
    \begin{align*}
    \|\gtil\| \ge \frac{\|\gvec\|}{\|\nabla\, \Psi(\cvxarg)\|_{\op}}.
    \end{align*} 
    In particular, we take $\fbar$ to be proper, convex function $\fcvx$ whose minimum is attained at some $\cvxargst$, we can take
    \begin{align*}
    \|\gtil\| \ge \max_{\cvxargst \in \argmin \fcvx}\frac{\fcvx(\cvxarg) - \inf(\fcvx)}{\cvxargdist \cdot\|\nabla\, \Psi(\cvxarg)\|_{\op}}. 
    \end{align*}

    \end{lemma}
    \begin{proof} We may assume without loss of generality that $\gvec \ne 0$, for otherwise $\gtil = 0$ and the constant curve $\phi(\epsilon) = \Psi(\cvxarg) = \tilarg$ satisfies the conclusion of the lemma. Fix a parameter $\eta > 0$ to be chosen at the end of the proof, and define the curve $\phi(\epsilon) := \Psi(\cvxarg - \frac{\epsilon}{\eta}\gvec)$. Then, for $\epsilon$ sufficiently small,
    \begin{align*}
    \flift(\phi(\epsilon)) = \fbar(\cvxarg - \frac{\epsilon}{\eta}\gvec) < \infty,
    \end{align*}
    since $\gvec$ is Cauchy direction of $\fbar$. Hence, $\phi(\epsilon) \in \dom( \flift)$ for $\epsilon$ sufficiently small. We compute
    \begin{align}
    \lim_{\epsilon \to 0^+}\frac{\flift(\phi(\epsilon))-\flift(\tilarg)}{\epsilon} &= \lim_{\epsilon \to 0^+}\frac{\flift(\Psi(\cvxarg - \frac{\epsilon}{\eta}\gvec))-\flift(\tilarg)}{\epsilon} \nonumber\\
    &= \lim_{\epsilon \to 0^+}\frac{\fbar(\cvxarg - \frac{\epsilon}{\eta}\gvec)-\fbar(\cvxarg)}{\epsilon} \nonumber\\
    &= \frac{1}{\eta}\lim_{\epsilon \to 0^+}\frac{f(\cvxarg - \epsilon \gvec)-\fbar(\cvxarg)}{\epsilon} \nonumber\\
    &\le \frac{-\|\gvec\|^2}{\eta}. \label{eq:lim_comp_psi}
    \end{align}
    Furthermore,
    \begin{align*}
    \phi'(0) = -\frac{1}{\eta}\nabla\, \Psi(\cvxarg)\gvec.
    \end{align*}
    Since we assume $\gvec \ne 0$ (see above), and since $\sigma_{\dcvx}(\nabla \,\Psi(\cvxarg)) \ne 0$ by assumption, we find that $\|\phi'(0)\| > 0$. Thus, continuing \Cref{eq:lim_comp_psi},
    \begin{align*}
    \lim_{\epsilon \to 0^+}\frac{\flift(\phi(\epsilon))-\flift(\tilarg)}{\epsilon} &\le \frac{-\|\gvec\|^2}{\eta} = -\|\phi'(0)\|^2 \cdot  \frac{\|\gvec\|^2}{\eta \cdot \|\phi'(0)\|^2} = -\|\phi'(0)\|^2 \cdot \eta \cdot \frac{ \|\gvec\|^2}{\|\nabla\, \Psi(\cvxarg) \gvec\|^2}. 
    \end{align*}
    In particular, if we set $\eta = \frac{\|\nabla\, \Psi(\cvxarg) \gvec\|^2}{ \|\gvec\|^2}$, we see that $\phi(\cdot)$ is valid for certifying that $\gtil = \phi'(0)$ is a generalized Cauchy direction. In this case, we have that
    \begin{align*}
    \|\phi'(0)\| = \frac{\|\nabla\, \Psi(\cvxarg) \gvec\|\cdot\|\gvec\|^2}{\|\nabla\, \Psi(\cvxarg) \gvec\|^2} = \|\gvec\| \cdot \frac{\|\gvec\|}{\|\nabla\,\Psi(\cvxarg)\gvec\|} \ge \frac{\|\gvec\|}{\|\nabla\,\Psi(\cvxarg)\|_{\op}}.
    \end{align*}
    \end{proof}
    
    \paragraph{Partial minimization preserves Cauchy directions. } For our final lemma, recall the set up of \DCL s. Let $\tilarg = (\farg,\xiarg)$, and let $f(\farg) := \min_{\xiarg} \flift(\farg,\xiarg)$. As shorthand, we say $\tilarg = (\farg,\xiarg)$   is \emph{admissible}  if $\farg \in \dom(f)$ and $\xiarg \in \argmin_{\xiarg'} \flift(\farg,\xiarg')$. We show that if $\flift$ has a (generalied) Cauchy direction $\gtil$ at an admissible $\tilarg$, then the norm of the gradient of $f$ must be at least as large as $\|\gtil\|$.
    \begin{lemma}\label{lem:partial_min_Cauchy} Suppose that $f$ is proper, and that $f(\farg)$ is $\cctwo$ at $\farg$ for some $\farg \in \dom (f)$. Let $\tilarg = (\farg,\xiarg)$ be $f$-admissible, and suppose that $\flift$ has a generalized Cauchy direction $\gtil$ at $\tilarg$. Then, 
    \begin{align*}
    \|\nabla f(\farg)\| \ge \|\gtil\|.
    \end{align*}
    \end{lemma}
    \begin{proof} Let $\phi$ be a curve which certifies $\gtil$ as a Cauchy direction of $\flift$; namely $\phi(0) = \tilarg$, $\phi'(0) = \gtil$, and 
    \begin{align*}
    \lim_{\epsilon \to 0^+}\frac{\flift(\phi(\epsilon))-\flift(\tilarg)}{\epsilon} \le - \|\gtil\|^2.
    \end{align*}
    We write $\phi(\epsilon) = (\phi_{1}(\epsilon),\phi_{2}(\epsilon))$ in its $(\farg,\xiarg)$ components. Then,
    \begin{align*}
    f(\phi_{1}(\epsilon)) = \min_{\xiarg'}\flift(\phi_1(\epsilon),\xiarg') \le\flift(\phi_1(\epsilon),\phi_{2}(\epsilon))  = \flift(\phi(\epsilon)).
    \end{align*}
    By admissibility of $\tilarg = (\farg,\xiarg)$, $f(\phi_1(0)) = f(\farg) = \flift(\tilarg)$, so that 
    \begin{align*}
    f(\phi_{1}(\epsilon))  - f(\farg) \le \flift(\phi(\epsilon))-\flift(\tilarg).
    \end{align*}
    Dividing by $\epsilon$ and taking limits,
    \begin{align*}
    \lim_{\epsilon \to 0^+}\frac{f(\phi_{1}(\epsilon))  - f(\farg)}{\epsilon} \le \lim_{\epsilon \to 0^+}\frac{\flift(\phi(\epsilon))-\flift(\tilarg)}{\epsilon} = -\|\gtil\|^2.
    \end{align*}
    On the other hand, since $f$ and $\phi_1$ are both differentiable, 
    \begin{align*}
    \lim_{\epsilon \to 0^+}\frac{f(\phi_{1}(\epsilon))  - f(\farg)}{\epsilon} = \langle \nabla f(\phi_1(0)), \phi_1'(0)\rangle =\langle (\nabla f(\farg), \mathbf{0}), \gtil\rangle,
    \end{align*}
    where above $(\nabla f(\farg), \mathbf{0}) \in \R^{\dfarg+\dtilarg}$ has a $0$ in the remaining $\dtilarg$ coordinates. Therefore,
    \begin{align*}
    \langle (\nabla f(\farg), \mathbf{0}), \gtil\rangle \le -\|\gtil\|^2,
    \end{align*}
    which requires $\|\nabla f(\farg) \| = \|(\nabla f(\farg), \mathbf{0})\| \ge \|\gtil\|$. 
    \end{proof}

    \paragraph{Concluding the proof of \Cref{thm:DCL}. }
    \begin{proof}[Proof of \Cref{thm:DCL}] Given $\farg \in \dom(f)$,  pick any $\cvxargst \in \argmin(\fcvx)$, and any $\xiarg \in \argmin_{\xiarg'} \flift(\farg,\xiarg')$. Set $\cvxarg = \Phi(\tilarg)$, and note that  $f(\farg) = \flift(\tilarg) = \fcvx(\Phi(\tilarg)) = \fcvx(\cvxarg)$, so $\cvxarg \in \dom(\fcvx)$ and $\tilarg \in \dom(\flift)$. Note that we cannot have $\cvxargst = \cvxarg$, since $\farg \in \domsub(f)$ implies
    \begin{align*}
    \fcvx(\cvxarg) = f(\farg) > \inf_{\farg'}f(\farg') = \inf_{\tilarg} \flift(\tilarg) = \inf_{\tilarg}\fcvx(\Phi(\tilarg)) \ge \fcvx(\cvxargst).
    \end{align*}
    By \Cref{lem:convex_Cauchy}, $\fcvx$ has a Cauchy direction $\gvec$ at $\cvxarg$ satisfying
    \begin{align*}
    \|g\| \ge \frac{\fcvx(\cvxarg) -\inf(\fcvx)}{\cvxargdist}.
    \end{align*}
    Next, from the \DCL, the mapping $\Phi: \tilset \to \cvxset$ is $\ccone$ on an open neighborhoods containing $\tilarg$. Hence, $\nabla \,\Phi(\tilarg)$ is defined. We now claim that $\flift$ has a generalized 
    generalized Cauchy direction $\gtil$ of norm
    \begin{align}
    \|\gtil\| \ge \frac{\fcvx(\cvxarg) -\inf(\fcvx)}{\cvxargdist}  \cdot \sigma_{\dcvx}(\nabla\, \Phi(\tilarg)). \label{eq:gtil_inter}
    \end{align}
    Indeed, if $\sigma_{\dcvx}(\nabla \,\Phi(\tilarg)) = 0$, $\gtil = 0$ suffices (the zero vector is always a generalized Cauchy direction). Otherwise, if $\sigma_{\dcvx}(\nabla\, \Phi(\tilarg)) > 0$, the fact that $\dtilarg \ge \dcvx$ and the implicit function theorem implies that $\Phi$ admits a $\ccone$ right inverse $\Psi$ satisfying $\Phi \circ \Psi(\cvxarg') = \cvxarg'$ and $\Psi(\cvxarg) = \tilarg$ on a neighborhood of $\cvxarg$. This mapping must satisfy $\nabla\, \Psi(\cvxarg) = \nabla\, \Phi(\tilarg)^\dagger$, so that in particular, $\|\nabla \,\Psi(\cvxarg)\|_{\op}^{-1} = \sigma_{\dcvx}(\nabla\, \Phi(\tilarg))$ and $\sigma_{\dcvx}(\nabla \,\Psi(\cvxarg)) > 0$. Hence, \Cref{lem:Cauchy_change_of_basis} implies that
    \begin{align*}
    \|\gtil\| \ge \frac{\fcvx(\cvxarg) -\inf(\fcvx)}{\cvxargdist} \cdot \frac{1}{\|\nabla\,\Psi(\cvxarg)\|_{\op}} =\frac{\fcvx(\cvxarg) -\inf(\fcvx)}{\cvxargdist}  \cdot \sigma_{\dcvx}(\nabla \,\Phi(\tilarg)),
    \end{align*}
    verifying \Cref{eq:gtil_inter}. Finally, by \Cref{lem:partial_min_Cauchy}, 
    \begin{align*}
    \|\nabla f(\farg)\| \ge \|\gtil\| \ge \frac{\fcvx(\cvxarg) -\inf(\fcvx)}{\cvxarg}  \cdot \sigma_{\dcvx}(\nabla\, \Phi(\tilarg)).
    \end{align*}
    Lastly, using the \DCL, we have $\fcvx(\cvxarg) = f(\farg)$, $\inf(\fcvx) = \inf (f)$. Substituting in $\cvxarg = \Phi(\tilarg)$ and $\tilarg = (\farg,\xiarg)$,
    \begin{align*}
    \|\nabla f(\farg)\| \ge \|\gtil\| \ge \frac{f(\farg) -\inf(f)}{\|\Phi(\farg,\xiarg) - \xstar\|}  \cdot \sigma_{\dcvx}(\nabla\, \Phi(\farg,\xiarg)).
    \end{align*}
    Since the above holds for any $\xstar \in \argmin(\fcvx)$ and any $\xiarg \in \argmin \flift(\farg,\cdot)$, \Cref{thm:DCL} follows.
    \end{proof}

\subsection{Analysis of gradient descent and reconditioning under weak-PL}\label{app:gd_reco}

\subsubsection{Proof of \Cref{prop:reco_descent}}\label{app:proof:reco_descent}
The first step of the proof is to ensure sufficiently small step sizes remain in the set $\calK$ for which our regularity conditions.
\begin{claim}\label{claim:descent_claim_reco} Suppose that $\tilde\farg_k \in \calK$. Then,
\begin{align}
f(\tilde\farg_k  -  \eta_k \nabla f(\tilde\farg_k)) \le f(\tilde\farg_k) -\frac{\eta_k}{2}\|\nabla f(\tilde \farg_k)\|^2 = f(\farg_k) -\frac{\eta_k}{2}\|\nabla f(\tilde \farg_k)\|^2. \label{eq:descent_eq_reco.}
\end{align}
In addition, for all $t \in [0,\eta_k]$, $\tilde\farg_k  -  t \nabla f(\tilde\farg_k) \in \calK$.
\end{claim}
The proof of \Cref{claim:descent_claim_reco} is somewhat elementary, and deferred to the end of the broader argument. We now argue recursively that $\tilde\farg_{k} \in \calK$ for all $k$. We argue inductively, noting $f(\tilde \farg_0) = f(\farg_0)$ and $\bLambda(\tilde\farg_0) = \eye_n$ ensures the base case $\tilde \farg_0 \in \calK$. Now, if $\tilde \farg_k \in \calK$,
\begin{align}
f(\farg_{k+1}) \overset{(i)}{\le} f(\tilde \farg_k - \eta_k \nabla f(\tilde \farg_k)) \le  f(\farg_k) -\frac{\eta_k}{2}\|\nabla f(\tilde \farg_k)\|^2, \label{eq:descent_eq}
\end{align}
where $(i)$ is an equality under \Cref{eq:reco_updates_a}, but may be an inequality under \Cref{eq:reco_updates_b}. Hence, $f(\tilde \farg_{k+1}) = f(\farg_{k+1}) \le f(\farg_k) = f(\tilde\farg_k) \le f(\farg_0)$ (since $\tilde \farg_k \in \calK$). Hence, since $\tilde \farg_{k+1}$ is reconditioned, $\tilde \farg_{k+1} \in \calK$ as well.

Subtracting $\inf(f)$ from both sides of \Cref{eq:descent_eq} and invokingthe $\alphaK$-weak PL property of $f$, the suboptimality gaps $\delta_k := f(\farg_k) - \inf(f)$ and minimal step $\eta := \min_k \{\eta_k\} $satisfy
\begin{align}
\delta_{k+1} \le \delta_k - \frac{\eta_k\alphaK^2}{2} \delta_k^2 \le \delta_k - \frac{\eta\alphaK^2}{2} \delta_k^2. 
\end{align}
We solve this recursion following an argument described in Section 3.2 of \cite{bubeck2014convex}. Setting $\omega = \eta\cdot\alphaK^2/2$, we have $\delta_k \ge \omega \delta_{k}^2 + \delta_{k+1}$, or equivalently, $\frac{1}{\delta_{k+1}} \ge \omega \frac{\delta_{k}}{\delta_{k+1}} + \frac{1}{\delta_k}$. Since $\delta_k \ge \delta_{k+1}$, this implies that $\frac{1}{\delta_{k+1}} \ge \omega + \frac{1}{\delta_k}$. Hence, we find
\begin{align*}
\frac{1}{\delta_{k+1}} - \frac{1}{\delta_{k}} \ge \omega.
\end{align*}
Telescoping, we conclude that $\frac{1}{\delta_{k+1}} \ge \omega (k+1)$, whence 
\begin{align*}
 \flam(\farg_{k}) - \inf(f) = \delta_{k} \le \frac{1}{\omega k} = \frac{2}{\alphaK^2 \eta k}.
\end{align*}
\qed

\begin{proof}[Proof of \Cref{claim:descent_claim_reco}] Define $\bar\farg(\tau) := \tilde\farg_k -  \tau \nabla f(\tilde\farg_k)$, noting $\bar\farg(0) = \farg_k$. The key subtlety in proving the claim is ensuring that the entire line segment  $\{\bar\farg(\tau): \tau \in [0,\eta_k]\}$ lies in the set $\calK$ under which relevant regularity conditions on $f$ hold. To start, we may assume without loss of generality that $\nabla f(\tilde\farg_k) \ne 0$, for otherwise the bound follows trivially. 
We make two observations
\begin{enumerate}
\item since $f$ is $\betaK$-upper-smooth on $\calK$, there is an open set containing $\tilde\farg_k$ on which $f$ is $\cctwo$. Then, there  exist some $\tau_1$ such that, for all $\tau \in [0,\tau_1]$, $f(\bar\farg(\tau)) = f(\farg_k) - \tau \|\nabla f(\tilde\farg_k)\|^2 + o(\tau)  < \phi(0)$. Note further that $ = f(\recond_{\bLambda}(\farg_k)) = f(\farg_k) \le f(\farg_0)$ (since $\tilde\farg_k \in \calK$ by assumption.)
\item Since $\bLambda(\tilde\farg) = I$ and $\bLambda$ is $L$-Lipschitz on $\calK$, there exist some $\tau_2$ such that, for all $\tau \in [0,\tau_2]$, $\|\bLambda(\bar\farg(\tau))- \eye_n\|_{\op} \le \frac{1}{2}$. 
\end{enumerate} 
Let us choose $\tau_0$ as the largest real satisfying the above two constraints:
\begin{align*}
\tau_0 := \sup \left\{\tau \le \eta_k: \forall \tau' \in [0,\tau),~ f(\bar\farg(\tau')) \le f(\tilde\farg_k)\text{ and } \|\bLambda(\bar\farg(\tau'))- \eye_n\|_{\op} \le \frac{1}{2}\right\},
\end{align*}
and observe that $\bar\farg(\tau) \in \calK$ for all $\tau \in [0,\tau_0]$ by construction. 

First, we show that $\tau_0 > 0$. Indeed, by assumption, there is any open set containing $\calK$ on which $f$ is $\cctwo$, and hence, on this open set $f$ is finite. In particular, there is an open set $\cU \subset \dom(f)$ with $\tilde{\farg}_k \in \cU$, $f$ is $\cctwo$ on $\cU$. Since $f$ is $\cctwo$ on $\cU$ and $\nabla f(\tilde{\farg}_k) \ne 0$, there exists some $\tau_0 > 0$ for such that, for all $\tau' \in [0,\tau_0)$, $f(\bar{\farg}_k(\tau')) = f(\tilde{\farg}_k - \tau'\nabla f(\tilde{\farg}_k)) < f(\tilde{\farg}_k)$, and since $\bLambda$ is continuous on $\dom(f) \supset \cU$ and $\bLambda(\bar{\farg}_k) = \eye_n$, we can shrink $\tau_0$ if necessary to ensure that $\|\bLambda(\bar\farg(\tau'))- \eye_n\|_{\op} = \|\bLambda(\tilde{\farg}_k - \tau'\nabla f(\tilde{\farg}_k)))- \eye_n\|_{\op} \le \frac{1}{2}$.

Further, observe that by $\betaK$-smoothness of $f$ on $\calK$, a Taylor expansion along the segment  parameterized by $\bar\farg(\tau)$ yields
\begin{align}
f(\bar\farg(\tau)) &\le f(\tilde \farg_k) - \left(\tau - \frac{\tau^2\betaK}{2}\right) \cdot \|\nabla f(\tilde \farg_{k})\|^2 \quad \forall \tau \in [0,\tau_0] \nonumber\\
&\le f(\tilde \farg_k) - \frac{\tau}{2} \cdot \|\nabla f(\tilde \farg_{k})\|^2 \quad \forall \tau \in \left[0,\min\left\{\tau_0,\frac{1}{\betaK}\right\}\right]. \label{eq:second_fbar_ineq}
\end{align}
To conclude, it suffices to show $\tau_0 \ge \eta_k$. 
For the sake of contradiction, suppose instead that $\tau_0 < \eta_k \le \min\{\frac{1}{\betaK},\frac{1}{2\LfK\LrecoK}\}$.  By (a) continuity of $f$ and $\bLambda$ on $\calK$ (b) continuity of $\tau \mapsto \bar\farg(\tau) \in \calK$, and (c) the assumption that $\calK$ is closed, it must be the case that, either (a) $f(\bar\farg(\tau_0)) = f(\tilde \farg_{k})$ or (b) $\|\bLambda(\bar\farg(\tau))- \eye_n\|_{\op} = \frac{1}{2}$. To see that (a) cannot hold, we have that $\tau_0 \le \eta_k \le \frac{1}{\beta}$ and \Cref{eq:second_fbar_ineq} implies that $f(\bar\farg(\tau)) < f(\tilde \farg_k)$. To see that (b) cannot hold, we use $\bLambda(\tilde \farg_k) = \eye_n$ and $\Lreco$-Lipschitzness of $\bLambda$ in the $\|\cdot\|_2 \to \|\cdot\|_{\op}$ norm, $\LfK$ Lipschitzness of $f$, and the bound $\tau_0 \le \eta_k \le \frac{1}{2\LfK\LrecoK}$ to attain
\begin{align*}
\|\bLambda(\bar \farg(\tau_0)) - \eye_n\|_{\op} &= \|\bLambda(\bar \farg(\tau_0)) - \bLambda(\tilde \farg_k)\|_{\op} \\
&\le \LrecoK\|\bar \farg(\tau_0) - \tilde \farg_k\|  \le \LrecoK \LfK \tau_0 \le \frac{1}{2}.
\end{align*}
\end{proof}

\subsubsection{Proof of \Cref{prop:con_comp}}\label{app:proof:connected_descent}
We assume without loss of generality that $\farg_0 \notin \argmin(f)$. 
\begin{claim}\label{claim:path_one}Fix $\eta > 0$  consider the iterates $\farg_k$ and $\tilde\farg_k$ produced by the updates in \Cref{eq:reco_updates_a} with $\eta_k = \eta$,  $\eta$ satisfies the step-size conditions of \Cref{prop:reco_descent}. Then $\farg_0$ is in the same connected component of $\dom(f)$ as $\tilde{\farg}_k$ for all $k$.
\end{claim}
\begin{proof}
Since $\bLambda(\cdot)$ is a connected reconditiong operator, each $\farg_k$ and $\tilde\farg_k$ lie in the same path-connected component of $\dom(f)$ for all $k$. Moreover, by \Cref{claim:descent_claim_reco}, the line segment $\tilde \farg_k - t \nabla f(\tilde \farg_k), t \in [0,\eta]$ lies entirely in $\calK$, so $\tilde \farg_k$ and $\farg_{k+1} =\tilde \farg_k - \eta \nabla f(\tilde \farg_k)$ lie in the same connected component of $\dom(f)$. Since path-connectedness is an equivalence relation, the result follows.
\end{proof}
Now, since $\calK(\farg_0)$ is compact, and $\tilde \farg_k \in \calK(\farg_0)$ for all $k \ge 0$, there exists a convergent subsequence $\tilde \farg_{k_i} \to \bar{\farg} \in \calK(\farg_0)$. Since $f$ is continuous on $\calK(\farg_0)$, $\lim_{i \to \infty} f(\tilde \farg_{k_i}) = f(\bar{\farg})$, so by \Cref{prop:reco_descent}, $f(\bar{\farg}) = \inf(f)$, i.e. $\bar{\farg} \in \argmin(f) \cap \calK(\farg_0)$. Since $\bar\farg \in \calK(\farg_0)$ is contained in an open set $\cU$, which is in turn contained in $\dom(f)$, there is an open ball of radius $r$, $\cB_r(\bar \farg)$, contained in $\dom(f)$. Since for some $i_{\star}$ sufficiently large, $\tilde \farg_{k_{\star}} \in \cB_r(\bar \farg)$, $\tilde \farg_{k_{\star}}$ is in the same path-connected component as $\bar \farg$. But by \Cref{claim:path_one}, it is also in the same path-connected component as $\farg_0$. Since path-connectedness is an equivalence relation, the result follows.

%% file: appendix/app_scherer_dcl.tex

\subsection{Proof of \Cref{lem:minimizer}}  Recall from \Cref{eq:llam_simplify} that
\begin{align*}
\cL_{\lambda}(\sfK) &= \begin{bmatrix}
	\sO & -\Ck
\end{bmatrix}  \Sigk \begin{bmatrix}
	\sO^\top \\ -\Ck^\top
\end{bmatrix}  + \lambda\trace\left[\matZ(\Sigk)^{-1}\right].
\end{align*}
Since $\Sigk$ satisfies the constraint in \Cref{equ:ftil_def_constrain} with equality, and since $\Sigk \succ 0$ and $\Zk \succ 0$ for any $\Sigk \in \calKexp$ by \Cref{lem:Sigma_K_conditioned}, we see that $(\sfK,\Sigk) \in \Clift$, and therefore
\begin{align*}
\cL_{\lambda}(\sfK) = \left\{\begin{bmatrix}
	\sO & -\Ck
\end{bmatrix}  \Sigk \begin{bmatrix}
	\sO^\top \\ -\Ck^\top
\end{bmatrix}  + \lambda\trace\left[\matZ(\Sigk)^{-1}\right]\right\}\I\{(\sfK,\Sigk) \in\Clift\} := \flift(\sfK,\Sigk).
\end{align*}
Next, let $\bSigma$ be any other matrix such that $(\sfK,\bSigma) \in \Clift$. Examining $\flift$, it suffices to show that 
\begin{align*}
\text{(a) } \bSigma \succeq \Sigk  \text{\qquad and\qquad (b) } \matZ(\bSigma) \preceq \matZ(\Sigk).
\end{align*}  
We show (a) and (b) hold as follows. 
\paragraph{Proof of point (a). } Recall the matix $\Aclk$ and $\Wclk$
\begin{align*}
\Aclk := \begin{bmatrix} \bA & 0\\
\Bk\bC & \Ak\end{bmatrix}, \quad \Wclk := \begin{bmatrix} \bW_1 & 0\\
0 & \Bk\bW_2\Bk^\top \end{bmatrix} \succeq 0. 
\end{align*} 
Then, $\Sigk$ is the solution to the Lyapunov equation
\begin{align}
\Aclk \Sigk+\Sigk\Aclk ^\top + \Wclk = 0. 
\end{align}
Since $\bSigma \in \Clift$, \Cref{equ:ftil_def_constrain} part $(iii)$ implies 
\begin{align}
\Aclk \bSigma+\bSigma\Aclk ^\top + \Wclk \preceq 0.
\end{align}
Subtracting these equations gives
\begin{align*}
0 \succeq \Aclk(\bSigma - \bSigma_K) + (\bSigma - \bSigma_K)\Aclk^\top.
\end{align*}
In other words, there exists a matrix $\calQ \succeq 0 $ such that
\begin{align}
\Aclk^\top(\bSigma - \bSigma_K) + (\bSigma - \bSigma_K)\Aclk^\top + \calQ = 0. \label{eq:above_sol}
\end{align}
Since it is assumed $\sfK \in \calKexp \subset \calKstab$,  then $\Aclk$ is Hurwitz. Therefore, the uniqueness of solutions to Lyapunov equations with stable matrices shows that the unique solution to \Cref{eq:above_sol} is some matrix $\bSigma - \bSigma_K  = \tilde{\bSigma}\succeq 0$, as needed.

\paragraph{Proof of point (b).} We build on $\bSigma \succeq \Sigk$. Recall that $\Sigk \succ 0$ as noted above, so that we can invert $\bSigma^{-1} \preceq \Sigk^{-1}$. Taking the bottom-right block and using the block inversion formula, 
\begin{align*}
(\Sigone - \bZ(\bSigma))^{-1} \preceq (\Sigone - \bZ(\Sigk))^{-1},
\end{align*}
which is equivalent after inversion to 
\begin{align}
\Sigone - \bZ(\bSigma) \succeq \Sigkone - \bZ(\Sigk). \label{eq:eq_almost_there_part_b_lift}
\end{align}
Next, observe that since $\Sigone = \Sigkone = \Sigonesys$ for $(\sfK,\bSigma) \in \Clift$ (this follows from the uniqueness of solutions to Lyapunov equations with Hurwitz matrices and constraint $(ii)$ of \Cref{equ:ftil_def_constrain}). Therefore, \Cref{eq:eq_almost_there_part_b_lift} simplifies to
$\bZ(\Sigk)\succeq \bZ(\bSigma)$, 
as needed. 
\qed

\subsection{Proof of \Cref{lemma:convex_kalman_par}}
Consider the parametrization $\bnu = (\bL_1,\bL_2,\bL_3,\bM_1,\bM_2) = \Phi(\sfK,\bSigma)$.
We can then write
\begin{align*}
\flift(\sfK, \bSigma)= \left( \underbrace{\trace\left[\begin{bmatrix}
	\sO & -\Ck
\end{bmatrix} \bSigma \begin{bmatrix}
	\sO^\top \\ -\Ck^\top
\end{bmatrix}\right]}_{\ftil_1(\sfK, \bSigma)} + \lambda \cdot\underbrace{\trace\left[\matZ(\bSigma)^{-1}\right]}_{\ftil_2(\sfK, \bSigma)} \right)\cdot\I\{(\sfK,\bSigma) \in \Clift\}.
\end{align*}
We show that
\newcommand{\bCtil}{\tilde{\bC}}
\newcommand{\bXtil}{\tilde{\bX}}

\begin{itemize}
\item[(a)] Whenever $(\sfK,\bSigma) \in \dom(\flift)$ (that is, $(\sfK,\bSigma) \in \Clift$), then there are affine matrix-valued functions $\bCtil(\cdot) \in \R^{p \times 2n}$ and $\bXtil(\cdot) \in \sym{2n}$  with  $\bXtil(\nu) \succ 0$ of $\bnu$  such that 
\begin{align*}
\ftil_1(\sfK,\bSigma) = \trace[\bCtil(\bnu)^\top\bXtil(\bnu)^{-1}\bCtil(\bnu)].
\end{align*}
\item[(b)] Whenever $(\sfK,\bSigma) \in \dom(\flift)$ (that is, $(\sfK,\bSigma) \in \Clift$), then $\bM_1 \succ 0$ and $\bM_2 \succ \bM_1^{-1}$. One can further express $\ftil_2(\sfK, \bSigma) =  \trace[(\bM_{2} - \bM_1^{-1})^{-1}]$. 
\item[(c)] There exists a convex set $\Ccvx$ such that $(\sfK,\bSigma) \in  \Clift$ if and only if $\bnu \in \Ccvx$. 
\end{itemize}
We turn to the verification of points (a)-(c) momentarily. Presently, let us conclude the proof. Points (a)-(c) directly imply that $\flift(\sfK,\bSigma) = \fcvx(\Phi(\sfK,\bSigma))$, where
\begin{align*}
\fcvx(\bnu) :=  \left(\trace[\bCtil(\bnu)^\top\bXtil(\cdot)^{-1}\bCtil(\bnu)] + \lambda\cdot\trace[(\bM_{2} - \bM_1^{-1})^{-1}]\right)\I\{\bnu \in \Ccvx\}. 
\end{align*}
To conclude, it remains to show that $\fcvx(\cdot)$ is convex. Since $\Ccvx$ is convex by point (c), it suffices to show that the functions $(i)$ $\bnu \mapsto \bCtil(\bnu)^\top\bXtil(\bnu)^{-1}\bCtil(\bnu)$ and $(ii)$  that $(\bM_1,\bM_2) \mapsto \trace[(\bM_{2} - \bM_1^{-1})^{-1}]$ are both convex. Since $\bCtil(\cdot)$ and $\bXtil(\cdot)$ are affine in $\bnu$ (and affine composition preserves convexity), point (i) follows from the following lemma:
\begin{restatable}{lemma}{cvxone}\label{lem:convex_function_one} The  function $g(\tilde\bC,\tilde\bX) = \trace[\tilde\bC ^\top\tilde{\bX}^{-1}\tilde\bC]$ is convex on the domain $(\tilde{\bC},\tilde{\bX}) \in \R^{\tilde{p}\times \tilde{n}}\times \pd{\tilde{n}}$. 
\end{restatable}
Point $(ii)$ follows from the following lemma:
\begin{restatable}{lemma}{cvxtwo}\label{lem:convex_function_two} The function $h(\bM_1,\bM_2) = \trace[(\bM_{2} - \bM_1^{-1})^{-1}]$ is convex on the domain $\{(\bM_1,\bM_2) \in \pd{n} \times \pd{n} : \bM_2 \succ \bM_1^{-1}\}$. 
\end{restatable}
The proof of these lemmas is defered to \Cref{sec:proof_cvx_lemmas}.

\paragraph{Proof of point (a). } Introduce
\begin{align*}
	\tilde\bC(\bnu)\defeq \begin{bmatrix}
		\sO \bM_2-\bL_3 & \bG 
	\end{bmatrix}^\top = \begin{bmatrix}
		\sO \bM_2-\Ck\bV^\top & \bG 
	\end{bmatrix}^\top, \quad  \tilde\bX(\bnu) \defeq \begin{pmatrix} 
\bM_2 & \bI \\
		\bI & \bM_1  
	\end{pmatrix}.
\end{align*}
It is shown in the proof of part (c) below  that $\tilde\bX(\bnu) \succ 0$. We compute (noting that $\bM_2$ is symmetric) that
\begin{align*}
\tilde\bC(\bnu)^\top\tilde\bX(\bnu)^{-1}\tilde\bC(\bnu) &=\begin{bmatrix}
		\sO \bM_2-\Ck\bV^\top & \sO
	\end{bmatrix}\begin{bmatrix} 
\bM_2 & \bI \\
		\bI & \bM_1 
	\end{bmatrix}^{-1}\begin{bmatrix}
		(\sO \bM_2-\Ck\bV^\top)^\top \\ \sO^\top 
	\end{bmatrix}\\
	&=
\begin{bmatrix}
	\sO & -\Ck
\end{bmatrix} \begin{bmatrix}
		\bM_2 & \bV \\ 
		\bI & 0
	\end{bmatrix}^\top \begin{bmatrix} 
\bM_2 & \bI \\
		\bI & \bM_1 
	\end{bmatrix}^{-1} \begin{bmatrix}
		\bM_2 & \bV \\ 
		\bI & 0
	\end{bmatrix}  \begin{bmatrix}
	\sO \\ -\Ck
\end{bmatrix}\\
&\overset{(i)}{=}
\begin{bmatrix}
	\sO & -\Ck
\end{bmatrix} \begin{bmatrix}
\bM_2 & \bV\\
\bV^\top & -\bU^{-1}\bM_1\bV	
\end{bmatrix} \begin{bmatrix}
	\sO^\top \\ -\Ck^\top
\end{bmatrix}\\
&\overset{(ii)}{=}\begin{bmatrix}
	\sO & -\Ck
\end{bmatrix} \bSigma \begin{bmatrix}
	\sO^\top \\ -\Ck^\top
\end{bmatrix} = \ftil_1(\Ak,\Bk, \Ck, \bSigma).
\end{align*}

Here, equality $(i)$ uses the block matrix inversion formula, and the facts that $\bM_1,\bM_2$ are invertible, and $\bI=\bM_1\bM_2+\bU\bV^\top$ as to be shown in \Cref{claim:nu_identities}; 
Equality $(ii)$ is given by the following calculation, whose steps follow from \Cref{claim:nu_identities}. 
\begin{align*}
\bSigma&= \begin{bmatrix}
		\bI & 0 \\ 
		\bM_1 & \bU
	\end{bmatrix}^{-1}\begin{bmatrix}
		\bM_2 & \bV \\ 
		\bI & 0
	\end{bmatrix} \tag*{\Cref{equ:relation_1}}\\
	&= \begin{bmatrix}
		\bI & 0 \\ 
		-\bU^{-1}\bM_1 & \bU^{-1}
	\end{bmatrix}\begin{bmatrix}
		\bM_2 & \bV \\ 
		\bI & 0
	\end{bmatrix}=\begin{bmatrix}
\bM_2 & \bV\\
-\bU^{-1}(\bM_1\bM_2 - \eye) & -\bU^{-1}\bM_1\bV	
\end{bmatrix}\\
	&=\begin{bmatrix}
\bM_2 & \bV\\
\bV^\top & -\bU^{-1}\bM_1\bV
\end{bmatrix} \tag*{\Cref{equ:decomp_UV}}.
\end{align*}

\paragraph{Proof of point (b).}
To show point (b), we have
\begin{align*}
\trace\left[\matZ(\bSigma)^{-1}\right] &= \trace[(\bSigma_{12} \bSigma_{22}^{-1}\bSigma_{12}^\top)^{-1}] \\
&= \trace[\left(-(\bSigma_{11}-\bSigma_{12} \bSigma_{22}^{-1}\bSigma_{12}^\top) + \bSigma_{11}\right)^{-1}] \\
&= \trace[\left(-(\bSigma^{-1})_{11}^{-1} + \bSigma_{11}\right)^{-1}] = \trace[(\bM_{2} - \bM_1^{-1})^{-1}].
\end{align*}
In addition, $\bM_1 = -(\bSigma^{-1})_{11} \succ 0$ since $\bSigma \succ 0$ (see the definition of the constraint set $\Clift$ in \Cref{equ:ftil_def_constrain}). Lastly, since $\matZ(\bSigma) \succ 0$ from the definition of $\Clift$, it must be the case that $\bM_2 \succ \bM_1^{-1}$. 

\paragraph{Proof of point (c).}  We first remark that, as in the specification of the lifted constraint set $\Clift$, specification of the convex constraint set $\Ccvx$ does not invole the parameter $\Ck$ at all.  We first show that $\bnu = \Phi(\sfK,\bSigma)$ satisfies some useful identities, using the convex parameterization in  \citep{scherer97lmi,masubuchi1998lmi}. 
\begin{claim}\label{claim:nu_identities}
$\bnu = \Phi(\Ak,\Bk,\Ck,\bSigma)$ satisfies the identities
\begin{subequations}
\begin{align}
	\bX & = \begin{pmatrix}
		\bM_2 & \bV \\ 
		\bI & 0
	\end{pmatrix}^{-1} \begin{pmatrix}
		\bI & 0 \\ 
		\bM_1 & \bU
	\end{pmatrix}\qquad \bSigma = \bX ^{-1}=\begin{pmatrix}
		\bI & 0 \\ 
		\bM_1 & \bU
	\end{pmatrix}^{-1}\begin{pmatrix}
		\bM_2 & \bV \\ 
		\bI & 0
	\end{pmatrix}\label{equ:relation_1}\\
	\begin{pmatrix}
		\Ak & \Bk 
	\end{pmatrix} & = \begin{pmatrix}
		\bU^{-1} & 0
	\end{pmatrix}
	\begin{pmatrix}
		\bL_1 - \bM_1 \bA \bM_2 & \bL_2 \\
	0 & 0
	\end{pmatrix}
	\begin{pmatrix}
		\bV^\top & 0 \\ 
		\bC \bM_2 & \bI
	\end{pmatrix}^{-1},\label{equ:relation_2}\\
	\label{equ:decomp_UV}
\bI&=\bM_1\bM_2+\bU\bV^\top. 
\end{align}
\end{subequations}
\end{claim}
\begin{proof}[Proof of \Cref{claim:nu_identities}]
To satisfy \Cref{equ:relation_1,equ:decomp_UV}, one uses the variables (written in terms of $\bX$)
\begin{align}
\begin{pmatrix}
\bM_1 & \bM_2\\
\bU & \bV \end{pmatrix} = 
\begin{pmatrix} (\bX)_{11} & (\bX^{-1})_{11}\\
(\bX)_{12} & (\bX^{-1})_{12}
\end{pmatrix} = \begin{pmatrix} (\bSigma^{-1})_{11} & (\bSigma)_{11}\\
(\bSigma^{-1})_{12} & (\bSigma)_{12}
\end{pmatrix}. 
\label{eq:rel1_intermediate}
\end{align} 
Next, by \Cref{equ:relation_2} we have
\begin{align*}
\begin{pmatrix}
		\bL_1  \\ \bL_2 
	\end{pmatrix} = \begin{pmatrix} \bU(\Ak \bV^\top + \Bk \bC \bM_2) \\
	\bU\Bk
	\end{pmatrix} +  \begin{pmatrix}
		\bM_1 \bA \bM_2 \\ 0
	\end{pmatrix}.
\end{align*}
Hence, combining with \Cref{eq:rel1_intermediate} and setting  $\bL_3=\Ck\bV^\top$,  these identities are satisfied for
\begin{align*}
\bnu^\top=&\begin{pmatrix}
		\bL_1  \\ \bL_2 \\ \bL_3 \\ \bM_1 \\ \bM_2
	\end{pmatrix} = \begin{pmatrix} \bU(\Ak \bV^\top + \Bk \bC (\bSigma)_{11})+ (\bSigma^{-1})_{11} \bA (\bSigma)_{11} \\
	\bU\Bk\\ \Ck\bV^\top \\ 
	(\bSigma^{-1})_{11} \\ (\bSigma)_{11}
	\end{pmatrix} , \quad \text{where } \begin{pmatrix}
	\bU \\ \bV \end{pmatrix} := \begin{pmatrix}(\bSigma^{-1})_{12} \\ (\bSigma)_{12}\end{pmatrix}.
\end{align*} 
\end{proof}
To conclude, we use \Cref{claim:nu_identities} to check that $(\Ak, \Bk, \Ck, \bSigma_{\sfK}) \in \Clift$ if and only if $\Phi(\Ak, \Bk,\Ck,\bSigma_{\sfK}) \in \Ccvx$ for some convex constraint set $\Ccvx$. Recall the definition of $\Clift$ in \Cref{equ:ftil_def_constrain}.  
Via a Schur complement argument, we can express $(\Ak,\Bk,\Ck,\bSigma)$ in $\Clift$ if and only if $\Ak,\Bk$ and $\bX := \bSigma^{-1}$ satisfy
\begin{subequations}
\begin{align}
&\begin{bmatrix}
\bX\begin{bmatrix} \bA & 0\\
\Bk\bC & \Ak\end{bmatrix}+\begin{bmatrix} \bA & 0\\
\Bk\bC & \Ak\end{bmatrix}^\top\bX & \bX\begin{bmatrix}
\bI & 0\\
0 & \Bk
\end{bmatrix} \label{equ:ftil_def_constrain2_a} \\
\begin{bmatrix}
\bI & 0\\
0 & \Bk
\end{bmatrix}^\top \bX & -\begin{bmatrix} \bW_1^{-1} & 0\\
0 & \bW_2^{-1}\end{bmatrix}\end{bmatrix}\preceq 0\\
&\bX\succ 0 \label{equ:ftil_def_constrain2_b}\\
 &\bA(\bX^{-1})_{11} +  (\bX^{-1})_{11}\bA^\top + \bW_1 = 0. \label{equ:ftil_def_constrain2_c}
\end{align}
\end{subequations}
Substituting \Cref{equ:relation_1}-\Cref{equ:relation_2}, we see that \Cref{equ:ftil_def_constrain2_a,equ:ftil_def_constrain2_b,equ:ftil_def_constrain2_c} are (respectively) equivalent to the following constraints
\begin{subequations}
\begin{align}
&\begin{bmatrix}
\tilde\bA(\bnu)^\top +\tilde\bA(\bnu) & \tilde\bB(\bnu)\\
\tilde\bB(\bnu)^\top & -\begin{bmatrix} \bW_1^{-1} & 0\\
0 & \bW_2^{-1}\end{bmatrix}\end{bmatrix}\preceq 0,  \label{eq:Ccvx_constraint_a} \\
&\qquad\qquad\text{where }\tilde\bA(\bnu)\defeq \begin{pmatrix}
		\bA \bM_2 & \bA  \\
		\bL_1 & \bM_1\bA + \bL_2\bC
	\end{pmatrix},\quad \tilde\bB(\bnu)\defeq \begin{pmatrix}
		\bI & 0 \\
		\bM_1 & \bL_2
	\end{pmatrix}, \nonumber\\
	&\tilde\bX(\bnu) \succ 0, \text{ where } \tilde\bX(\bnu) \defeq \begin{pmatrix} \label{eq:Ccvx_constraint_b}
\bM_2 & \bI \\
		\bI & \bM_1  
	\end{pmatrix}\\
	&\bA\bM_{2} +\bM_2\bA^\top +\bW_1 = 0.  \label{eq:Ccvx_constraint_c}
\end{align}
\end{subequations}
Here, the equivalence between \Cref{equ:ftil_def_constrain2_a} and \Cref{eq:Ccvx_constraint_a} invokes the following identity, derived similarly  to the expression for $\bSigma$ derived in part (a) above: 
\begin{align*}
\bX=\bSigma^{-1}= \begin{bmatrix}
		\bM_2 & \bV \\ 
		\bI & 0
	\end{bmatrix}^{-1} \begin{bmatrix}
		\bI & 0 \\ 
		\bM_1 & \bU
	\end{bmatrix}= \begin{bmatrix}
\bM_1 & \bU\\
\bU^\top &  -\bV^{-1}\bM_2\bU 	\end{bmatrix}.
 \end{align*}
The equivalence between \Cref{equ:ftil_def_constrain2_b} and \Cref{eq:Ccvx_constraint_b} can be verified via the Schur complement. It is clear that \Cref{eq:Ccvx_constraint_a,eq:Ccvx_constraint_b,eq:Ccvx_constraint_c} determine a convex constraint set.
\qed

\subsection{Proof of \Cref{lem:compact_level_sets}}
	Fix $(\sfK,\Sigk)$ for $\sfK \in \calKexp$, and  $\bnu = (\bL_1,\bL_2,\bL_3,\bM_1,\bM_2)$ be the associated convex parameter, $\bnu = \Phi(\sfK,\Sigk)$ defines the matrix $\bLambda$  as in \Cref{eq:Ccvx_constraint_a}
	\begin{align*}
	\bLambda := &\begin{bmatrix}
	\tilde\bA^\top +\tilde\bA & \tilde\bB\\
	\tilde\bB^\top & -\begin{bmatrix} \bW_1^{-1} & 0\\
	0 & \bW_2^{-1}\end{bmatrix}\end{bmatrix},  ~~\text{where }\tilde\bA\defeq \begin{pmatrix}
			\bA \bM_2 & \bA  \\
			\bL_1 & \bM_1\bA + \bL_2\bC
		\end{pmatrix},\quad \tilde\bB\defeq \begin{pmatrix}
			\bI & 0 \\
			\bM_1 & \bL_2
		\end{pmatrix}.
	\end{align*}
	Since $(\sfK,\Sigk) \in \dom(\flift)$, $\bnu \in \dom(\fcvx)$, and hence \Cref{eq:Ccvx_constraint_a} implies $\bLambda  \preceq 0$. 

	We begin our argument by bounding the operator norms of the matrices $\bL_1$ and $\bL_2$, which we ultimately translate into bounds on $\Ak$ and $\Bk$. Our arguments use the following Schur complement test for negative semidefinite matrices:
	\begin{lemma}\label{lem:neg_schur_com} Let $\bX = \begin{bmatrix} \bX_{11} &\bX_{12}\\\bX_{12}^\top &\bX_{22}\end{bmatrix}$ satisfy $\bX \preceq 0$ and $\bX_{22} \prec 0$. Then, $\|\bX_{12}\|^2/\|\bX_{22}\| \le \|\bX_{11}\|$.
 	\end{lemma}
 	\begin{proof} Since $\bX \preceq 0$, $-\bX \succeq 0$. By the PSD Schur complement test applied to $-\bX$,
 	\begin{align*}
 	0 \preceq -\bX_{11} - (-\bX_{12}) (-\bX_{22}^{-1}) (-\bX_{12})^\top = -\bX_{11} + (\bX_{12}\bX_{22}^{-1}\bX_{12}^\top).
 	\end{align*}
 	Hence, $-\bX_{12}\bX_{22}^{-1}\bX_{12}^\top \preceq -\bX_{11}$. Now, observe that $\bX \preceq 0$ implies that $-\bX_{11}$, $-\bX_{22}^{-1}$, $-(\bX_{12}\bX_{22}^{-1}\bX_{12}^\top)$  are all PSD. Thus, $\|\bX_{12}\bX_{22}^{-1}\bX_{12}^\top\| \le \|\bX_{11}\|$, so that $\|\bX_{12}\|^2\sigma_{\min}(\bX_{22}^{-1}) \le  \|\bX_{11}\|$ (where $\sigma_{\min}$ denotes minimal singular  value). Noting $\sigma_{\min}(\bX_{22}^{-1}) = 1/\|\bX_{22}\|$ concludes.
 	\end{proof}

	We begin  bounding $\|\bL_2\| $. 
	\begin{claim} We have the bound
	\begin{align*}
	\|\bL_2\| \le  2\|\bC\|\|\bW_2^{-1}\|+\sqrt{2\|\bM_1\|\|\bA\|\|\bW_2^{-1}\|}.
	\end{align*}
	\end{claim}
	\begin{proof}
	Let $\bLambda_{(2,4)}$ denote the submatrix of $\bLambda$ corresponding to the 2nd and 4th rows/columns: 
	\begin{align*}
	\bLambda_{(2,4)} := \begin{bmatrix} \bL_2 \bC + (\bL_2 \bC )^\top + \bM_1 \bA + (\bM_1 \bA)^\top & \bL_2 \\
	\bL_2^\top &  -\bW_2^{-1} 
	\end{bmatrix}.
	\end{align*}
	Since $\bLambda \preceq 0$, $\bLambda_{(2,4)} \preceq 0$.  \Cref{lem:neg_schur_com} gives
	\begin{align*}
	\|\bL_2\|^2/\|\bW_2^{-1}\| \le 2\|\bL_2\|\|\bC\| + 2\|\bM_1\|\|\bA\|.
	\end{align*}
	Hence, $x := \|\bL_2\|^2$ satisfies a quadratic inequality $ax^2 - bx - c \le 0$, $a = 1/\|\bW_2^{-1}\|$,  $b =  2\|\bC\|$ and $c =   2\|\bM_1\|\|\bA\|$. Solving the quadratic equation, using $a,b,c \ge 0$ and taking the positive root, and using $\sqrt{x+y} \le \sqrt{x}+\sqrt{y}$ for $x,y \ge 0$,
	\begin{align*}
	x \le \frac{b + \sqrt{b^2 + 4ac}}{2a} \le \frac{2b + 2\sqrt{ac}}{2a} \le \frac{b}{a} + \sqrt{c/a},
	\end{align*}
	that is,
	\begin{align*}
	\|\bL_2\| \le 2\|\bC\|\|\bW_2^{-1}\|+\sqrt{2\|\bM_1\|\|\bA\|\|\bW_2^{-1}\|}.
	\end{align*}
	\end{proof}

	Next, we bound $\|\bL_1\|$ in terms of $\|\bL_2\|$:
	\begin{claim}
	\begin{align*}
	\|\bL_1\|  \le 2\sqrt{\|\bA\|\|\bM_2\|\left(\|\|\bA\|\|\bM_1\| + \|\bC\| \|\bL_2\|\right)} + \|\bA\|.
	\end{align*}
	\end{claim}
	\begin{proof} Observe that $\bLambda \preceq 0 $ implies $\tilde\bA + \tilde \bA^\top \preceq 0$. That is,
	\begin{align*}
	\begin{pmatrix}
			\bA \bM_2 + (\bA \bM_2)^\top & \bA  + \bL_1^\top \\
			\bL_1 + \bA^\top & \bW_{3}
		\end{pmatrix} \preceq 0, \quad \text{ where } \bW_3 := \bM_1\bA + (\bM_1\bA)^\top + \bL_2\bC + (\bL_2\bC)^\top. 
	\end{align*}
	Now, we know that $\bW_3 \preceq 0$, but to invoke a Schur complement, we need strict inequality. To this end, for some $\lambda > 0$ to be choosen larger, we know that $\bW_3 - \lambda \bI \prec 0 $, and 
	\begin{align*}
	\begin{pmatrix}
			\bA \bM_2 + (\bA \bM_2)^\top & \bA  + \bL_1^\top \\
			\bL_1 + \bA^\top & \bW_{3} - \lambda  \bI 
		\end{pmatrix} \preceq 0.
	\end{align*}
	Using \Cref{lem:neg_schur_com}
	\begin{align*}
	\bA \bM_2 + (\bA \bM_2)^\top - (\bA  + \bL_1^\top) (\bW_{3} - \lambda  \bI )^{-1}(\bL_1 + \bA^\top ) \preceq 0,
	\end{align*}
	\begin{align*}
	 (\bA  + \bL_1^\top)^\top ( \lambda  \bI - \bW_3)^{-1}(\bL_1 + \bA^\top ) \preceq -\big(\bA \bM_2 + (\bA \bM_2)^\top\big),
	\end{align*}
	and hence
	\begin{align*}
	\frac{\| \bA  + \bL_1^\top\|^2}{\|\bW_3 - \lambda \bI\|} \le 2 \|\bA\|\|\bM_2\|.
	\end{align*}
	Since  $\bW_3 \preceq 0 $, $\|\bW_3 - \lambda \bI\| = \lambda + \|\bW\|$. Hence, 
	\begin{align*}
	\| \bA  + \bL_1^\top\|^2 \le 2 \|\bA\|\|\bM_2\| \le (\lambda +\|\bW_3\|)\cdot 2 \|\bA\|\|\bM_2\|. 
	\end{align*}
	Since this is irrespective of $\lambda > 0$,
	\begin{align*}
	\| \bA  + \bL_1^\top\|^2 &
	\le 2 \|\bA\|\|\bM_2\|\|\bW_3\|\\
	&\le 4 \|\bA\|\|\bM_2\|\left(\|\bA\|\|\bM_1\| + \|\bC\| \|\bL_2\|\right).
	\end{align*}
	Hence,
	\begin{align*}
	\|\bL_1\|  \le 2\sqrt{\|\bA\|\|\bM_2\|\left(\|\bA\|\|\bM_1\| + \|\bC\| \|\bL_2\|\right)} + \|\bA\|.
	\end{align*}
	 \end{proof}

	 Lastly, let us bound $\bL_3$.
	 \begin{claim} We have that $\|\Ck\|_{\fro} \le \sqrt{\Loe(\sfK)/\|\Sigk^{-1}\|}$ and $\|\bL_3\|_{\fro} \le \|\Sigk\|\sqrt{ \Loe(\sfK)/\|\Sigk^{-1}\|}$.  
	 \end{claim}
	 \begin{proof} 
	  As follows from the proof of \Cref{lem:minimizer},
	 \begin{align*}
	 \Loe(\sfK) = \trace\left(\begin{bmatrix}
	\sO & -\Ck
\end{bmatrix}  \Sigk\begin{bmatrix}
	\sO^\top \\ -\Ck^\top
\end{bmatrix}\right) \ge \lambda_{\min}(\Sigk) (\|\sO\|_{\fro}^2 + \|\Ck\|_{\fro}^2) \ge \lambda_{\min}(\Sigk) \|\Ck\|_{\fro}^2,
	 \end{align*}
	 which gives the desired bound on $\|\Ck\|_{\fro}$.
	 Since $\bL_3 = \Ck\bV^\top$, and since $\bV$ is a submatrix of $\Sigk$,
	 \begin{align*}
	 \|\bL_3\|_{\fro} \le \|\bV\| \|\Ck\|_{\fro} \le \|\Sigk\|\|\Ck\|_{\fro}.
	 \end{align*}
	 The lemma follows.
	 \end{proof}
Summarizing the previous three claims,
	 \begin{align*}
	\|\bL_2\| &\le \|\bL_2\| \le  2\|\bC\|\|\bW_2^{-1}\|+\sqrt{2\|\bM_1\|\|\bA\|\|\bW_2^{-1}\|} = \polyop(\bM_1,\bA,\bC,\bW_2^{-1})\\
	\|\bL_1\| &  \le 2\sqrt{\|\bA\|\|\bM_2\|\left(\|\|\bA\|\|\bM_1\| + \|\bC\| \|\bL_2\|\right)} + \|\bA\| \\
	&= \polyop(\bM_1,\bM_2,\bA,\bC,\bL_2) = \polyop(\bM_1,\bM_2,\bA,\bC,\bW_2^{-1})\\
	\|\bL_3\|_{\fro} &\le \|\Sigk\|\sqrt{ \Loe(\sfK)/\|\Sigk^{-1}\|} = \polyop(\Sigk^{-1},\Sigk)\sqrt{ \Loe(\sfK)}.
	\end{align*}

	This suffices to bound $\|\bnu\|_{\ell_2}$:
	\begin{align*}
	\|\bnu\|_{\ell_2} &= \sqrt{\sum_{i=1}^2\big(\|\bM_i\|_\fro^2 + \|\bL_i\|_{\fro}^2\big) + \|\bL_3\|_{\fro}} \\
	&\le \|\bL_3\|_{\fro} + \sum_{i=1}^2\big(\|\bM_i\|_\fro + \|\bL_i\|_{\fro}\big)\\
	&\overset{(i)}{\le} \max\{n,\sqrt{nm}\}\left(\sum_{i=1}^2\big(\|\bM_i\| + \|\bL_i\|\big)\right) + \|\bL_3\|_{\fro} \\
	&\overset{(ii)}{\le} \max\{n,\sqrt{nm}\}\cdot\polyop(\bA,\bC,\bW_2^{-1},\bM_1,\bM_2) +  \polyop(\bSigma^{-1},\Sigk^{-1})\sqrt{ \Loe(\sfK)}\\
	&\overset{(iii)}{=} \max\{n,\sqrt{nm}\}\cdot\polyop(\bA,\bC,\bW_2^{-1},\Sigk,\Sigk^{-1}) +  \polyop(\Sigk^{-1},\Sigk)\sqrt{ \Loe(\sfK)}\\
	&=\polyop(\bA,\bC,\bW_2^{-1},\Sigk,\Sigk^{-1})(\max\{n,\sqrt{nm}\} + \sqrt{ \Loe(\sfK)}).
	\end{align*}
	Above, $(i)$ uses $\bM_1,\bM_2,\bL_1 \in \R^{n\times n}$, and $\bL_2 \in \R^{n \times m}$, $(ii)$ uses the bounds on $\|\bL_i\|$ developed above, and $(iii)$ uses $\|\bM_1\| = \|(\Sigk^{-1})_{11}\| \le \|\Sigk^{-1}\|$ and similarly, $\|\bM_2\| \le \|\Sigk\|$. 

	Next, we bound $\|\Ak\|$ and $\|\Bk\|$. From the definition of the transformation $\Phi$, and recalling $\Uk = (\Sigk^{-1})_{12}$ and $\Vk = (\Sigk)_{12}$, we have
	\begin{align}
	\bL_1 &=  \Uk(\Ak \Vk^\top + \Bk \bC (\bSigma)_{11})+ \underbrace{(\bSigma^{-1})_{11}}_{=\bM_1} \bA \underbrace{(\bSigma)_{11}}_{=\|\bM_2\|}, \quad \bL_2 = \Uk\Bk. 
	\end{align} 
	Hence, if $\Uk$ and $\Vk$ are invertible,
	\begin{align*}
	\|\Bk\| &\le \|\Uk^{-1}\|\|\bL_2\| = \polyop(\bM_1,\bA,\bC,\bW_2^{-1},\Uk^{-1}),\\
	\|\Ak\| &\le \|\Vk^{-1}\|\|\Uk^{-1}\|\left(\|\bM_1\| \|\bA\|\|\bM_2\| + \|\bL_1\|\right) + \|\Bk\| \|\bC\|\|\Vk^{-1}\|\|\Uk\|\\
	&= \polyop(\bM_1,\bM_2,\bA,\bC,\bW_2^{-1},\Uk,\Vk^{-1},\Uk^{-1})
	\end{align*}
	Again, we note that $\|\bM_1\| = \|(\bSigma^{-1}_{\sfK})_{11}\| \le \|\bSigma^{-1}_{\sfK}\|$ and similarly, $\|\bM_2\| \le \|\bSigma_{\sfK}\|$. Similarly, $\|\Uk\| = \|(\bSigma^{-1}_{\sfK})_{12}\| \le \|\bSigma^{-1}_{\sfK}\|$, hence, we conclude
	\begin{align} 
	\max\left\{\|\Ak\|,\|\Bk\|\right\} \le \polyop(\bA,\bC,\bW_2^{-1},\Sigk,\Sigk^{-1},\Uk^{-1},\Vk^{-1}),
	\end{align}
	as needed.
\qed

\newcommand{\bMbar}{\bar{\bM}}
\newcommand{\bVbar}{\bar{\bV}}
\newcommand{\bUbar}{\bar{\bU}}
\newcommand{\bSigbar}{\bar{\bSigma}}
\newcommand{\Delcvx}{\bDelta_{\mathtt{cvx}}}
\newcommand{\Dellift}{\bDelta_{\mathtt{lft}}}

\subsection{Proof of \Cref{lem:phi_cond}} 
We establish the differentiability and conditioning of $\Phi$ for any $(\sfK,\bSigma) \in \Clift$; the lemma corresponds to the special case when $\bSigma = \Sigk$. We  let $\bnu = \Phi(\sfK,\Sigk)$, where we recall  $\bnu = (\bL_1,\bL_2,\bL_3,\bM_1,\bM_2)$ is given by 
\begin{subequations}
\begin{align}
\begin{pmatrix}
\bL_1\\
\bL_2\\
\bL_3\\
\bM_1 \\
\bM_2 \\
\end{pmatrix} &:= \begin{pmatrix} \bU(\Ak \bV^\top + \Bk \bC (\bSigma)_{11})+ (\bSigma^{-1})_{11} \bA (\bSigma)_{11} \\
	\bU\Bk\\
	\Ck\bV^\top \\
	(\bSigma^{-1})_{11} 
	\\ (\bSigma)_{11}
	\end{pmatrix} , \\
	&\text{where } \begin{pmatrix}
	\bU \\ \bV \end{pmatrix} := \begin{pmatrix}(\bSigma^{-1})_{12} \\ (\bSigma)_{12}\end{pmatrix}. 
\end{align} 
\end{subequations}
To see that $\Phi$ is differentiable, we see that $Phi$ is a polynomial function in $\Ak,\Bk,\Ck$ and $\bSigma$ and $\bSigma^{-1}$, and is therefore differentiable in an open neighborhood of any $(\sfK,\bSigma)$ for which $\bSigma$ is invertible.

Let's turn to the condition of $\nabla \Phi$. We then fix a target perturbation $\Delcvx := (\bDelta_{\bL_1},\bDelta_{\bL_2},\bDelta_{\bL_3},\bDelta_{\bM_1},\bDelta_{\bM_2})$ such that its $\ell_2$-norm as an Euclidean vector (equivalently, the sum of Frobenius norms of its parameters) is 
\begin{align*}
\|\Delcvx\|_{\ell_2}^2 = \sum_{i=1}^3 \|\bDelta_{\bL_i}\|_{\fro}^2+\sum_{j=1}^2\|\bDelta_{\bM_j}\|_{\fro}^2 = 1.
\end{align*}  
Our strategy is to compute a perturbation $\Dellift = (\bDelta_{\bA},\bDelta_{\bB},\bDelta_{\bC}, \bDelta_{\bSigma})$ of the parameters $(\sfK,\bSigma)$ such that 
\begin{align}
\ddt \Phi((\sfK,\bSigma) + t \Dellift)\big{|}_{t=0} = \Delcvx.  \label{eq:ddt_Phi}
\end{align}
Noting the identity
\begin{align*}
\nabla\, \Phi(\by) \cdot \Dellift = \Delcvx, 
\end{align*}
it thus suffices to compute uniform upper bound on $\|\Dellift\|_{\ell_2}^2 = \|\bDelta_{\bA}\|_{\fro}^2+\|\bDelta_{\bB}\|_{\fro}^2 + \|\bDelta_{\bC}\|^2_{\fro} +  \|\bDelta_{\bSigma}\|_{\fro}^2$ for which \Cref{eq:ddt_Phi} holds.  For convenience, let  $j \in \{1,2\}$ (resp $i \in \{1,2,3\}$ ) $\Phi_{\bM_j}$ (resp $\Phi_{\bL_i}$) denote the restriction of $\Phi$'s image to the $\bM_j$ (resp. $\bL_i$) coordinate. 

\paragraph{Handling the $\bM_j$-blocks.} We proceed to choose $\Dellift$ by first ensuring $\ddt \Phi_{\bM_j}((\sfK,\bSigma) +  t\Dellift)\big{|}_{t=0} = \bDelta_{\bM_j}$ for $j \in \{1,2\}$, and then continue to show the same for the $\bL_i$-coordinates. Since $\Phi_{\bM_j}$ are functions of $\bSigma$, it suffices for now to choose perturbations of $\bSigma$; abusing notation, we shall simply write $\Phi_{\bM_j}(\bSigma)$ to  express this fact. We consider a perturbation of the form 
\begin{align}
\bDelta_{\bSigma}, \quad \text{where } \bDelta_{\bSigma} = \begin{bmatrix} \bDelta_{11} & \bDelta_{12}\\
\bDelta_{12}^\top & 0 \end{bmatrix}.  \label{eq:Del_Sig}
\end{align}
Since $\Phi_{\bM_2}(\bSigma) = \bSigma_{11}$, we have  $\ddt \Phi_{\bM_2}(\bSigma + t\bDelta_{\bSigma})\big{|}_{t=0} = \bDelta_{11}$, so it suffices to choose
\begin{align}
\bDelta_{11} = \bDelta_{\bM_2}. \label{eq:Del11}
\end{align}
Next, we consider the $\bM_1$-block. For convenience,  we define the curve $\bSigbar(t) = \bSigma + t\bDelta_{\bSigma}$. Then
\begin{align*}
\ddt \Phi_{\bM_1}(\bSigma + t\bDelta_{\bSigma})\big{|}_{t=0} &= \ddt \Phi_{\bM_1}(\bSigbar(t))\big{|}_{t=0} = \ddt (\bSigbar(t)^{-1})_{11}\big{|}_{t=0}\\
&= \ddt(\bSigbar_{11} - \bSigbar_{12}\bSigbar_{22}^{-1}\bSigbar_{12}^\top)^{-1}\big{|}_{t=0}\\
 &=\underbrace{(\bSigbar_{11} - \bSigbar_{12}\bSigbar_{22}^{-1}\bSigbar_{12}^\top)^{-1}\big{|}_{t=0}}_{=\bM_1} \cdot \left(\ddt (\bSigbar_{11} - \bSigbar_{12}\bSigbar_{22}^{-1}\bSigbar_{12}^\top) \big{|}_{t=0}\right) \cdot \underbrace{\ldots}_{=\bM_1}\\
 &=\bM_1\left(\bDelta_{11} - \bDelta_{12}\bSigma_{22}^{-1}\bSigma_{12} - (\bDelta_{12}\bSigma_{22}^{-1}\bSigma_{12})^\top\right)\bM_1.
\end{align*}
Hence, we can take
\begin{align}
\bDelta_{12} &= \frac{1}{2} \left(\bDelta_{\bM_2} - \bM_1^{-1}\bDelta_{\bM_1}\bDelta_{11}^{-1}\right)\bSigma_{12}^{-1}\bSigma_{22}\nonumber\\
&= \frac{1}{2} \left(\bDelta_{\bM_2} - \bM_1^{-1}\bDelta_{11}\bDelta_{11}^{-1}\right)\bV^{-1}\bSigma_{22} \label{eq:Del12}
\end{align}
\paragraph{Some directional derivatives.} To handle the $\bL_i$ blocks, we extend the ``bar'' notation to  the variables 
$\bMbar_1(t),\bMbar_2(t),\bUbar(t),\bV(t)$ to denote the matrices corresponding to $\bSigbar(t) = \bSigma + t\bDelta_{\bSigma}$, i.e.
\begin{align*}
\bMbar_{1}(t) = (\bSigbar^{-1}(t))_{11},  \quad\bMbar_2(t) = \bSigbar_{11}(t), \quad \bVbar = \bSigbar_{12}(t), \quad \bUbar = (\bSigbar(t)^{-1})_{12}. 
\end{align*}
Since $\bSigbar(0) = \bSigma$, the above matrices are evaluated to their ``non-barred'' counterparts when $t = 0$. Moreover, by choice of $\bDelta_{\bSigma}$, we have
\begin{align*}
\bMbar_1'(0) = \bDelta_{\bM_1}, \quad \bMbar_2'(0) =\bDelta_{\bM_2}, \quad \bVbar'(0) = \bDelta_{12}. 
\end{align*}
Using the block matrix inversion formula, we have
\begin{align*}
\bUbar = (\bSigbar)^{-1}_{12} = -(\bSigbar^{-1})_{11}\bSigbar_{12}\bSigbar_{22}^{-1} = -\bMbar_1 \bSigbar_{12}\bSigma_{22}^{-1},
\end{align*}
where above we use $\bMbar_1 = (\bSigbar^{-1})_{11}$ and $\bSigbar_{22}^{-1}(t) = \bSigma_{22}^{-1}$ is constant for all $t$. Therefore,
\begin{align}
\bUbar'(0) &= -\bMbar_1'(0)\bSigbar_{12}(0)\bSigma_{22}^{-1} - \bMbar_{1}(0)\bSigbar_{12}'(0) \bSigma_{22}^{-1} \nonumber\\
&= -\bDelta_{\bM_1}\bSigma_{12}\bSigma_{22}^{-1} - \bM_1\bDelta_{12}\bSigma_{22}^{-1}\nonumber\\
&= -\bDelta_{\bM_1}\bV(\bSigma_{22})^{-1} - \bM_1\bDelta_{12}\bSigma_{22}^{-1}. \label{eq:bUbar}
\end{align}
\newcommand{\bAbar}{\bar{\bA}}
\newcommand{\bBbar}{\bar{\bB}}
\paragraph{Handling the $\bL_i$-blocks.} Let us also define $\bBbar_{\sfK}(t) = \bB + t\bDelta_{\bB}$ and $\bAbar_{\sfK}(t) = \bA + t\bDelta_{\bA}$. Using the ``bar''-notation, we can compute 
\begin{align*}
\ddt \Phi_{\bL_2}(\bnu + t\Delcvx)\big{|}_{t=0} = \ddt (\bBbar_{\sfK} \bUbar)\big{|}_{t=0} = \bBbar_{\sfK}'(0)\bU + \Bk \bUbar'(0) = \bDelta_{\bB} \bU + \Bk \bUbar'(0).
\end{align*}
Hence, we set
\begin{align}
\bDelta_{\bB} = (\bDelta_{\bL_2}-\Bk \bUbar'(0))\bU^{-1}
\end{align}
Similarly, we can select 
\begin{align}
\bDelta_{\bC} = \bV^{-\top}(\bDelta_{\bL_3} - \Ck (\bVbar'(0))^\top).
\end{align}
Finally, we compute
\begin{align*}
\ddt \Phi_{\bL_1}(\bnu + t\Delcvx)\big{|}_{t=0} &= \ddt \left(\bUbar(\bAbar_{\sfK} \bVbar^\top + \bBbar_{\sfK} \bC \bMbar_2)+ \bMbar_1 \bA \bMbar_2\right)\big{|}_{t=0}\\
&= \bU \bAbar_{\sfK}'(0)\bV + \bUbar'(0)(\Ak \bV + \bK_{\sfK} \bC \bM_2)+ \bU\left(\bA \bVbar'(0) + \bBbar_{\sfK}'(0) \bC \bM_2 + \Bk \bC \bMbar_2'(0)\right)\\
&\qquad + \bMbar_1'(0) \bA \bM_2 + \bM_1(0) \bA \bMbar_2'(0)\\
&= \bU \bDelta_{\bA}\bV + \bUbar'(0)(\Ak \bV + \Bk \bC \bM_2) + \bU\left(\bA \bDelta_{12} + \bDelta_{\bB} \bC \bM_2 + \Bk \bC \bDelta_{\bM_2}\right)\\
&\qquad + \bDelta_{\bM_1} \bA \bM_2 + \bM_1 \bA \bDelta_{\bM_2}.
\end{align*}
Hence, we select 
\begin{equation}
\begin{aligned}
\bDelta_{\bA} &= \bU^{-1}\bDelta_{\bL_1}\bV^{-1}  - \bU^{-1}\bUbar'(0)\left(\Ak +\Bk \bC \bM_2\right)\bV^{-1}\\
&\qquad-\left(\bA \bDelta_{12} + \bDelta_{\bB} \bC \bM_2 + \Bk \bC \bDelta_{\bM_2}\right)\bV^{-1} - \bU^{-1}\left(\bDelta_{\bM_1} \bA \bM_2 + \bM_1 \bA \bDelta_{\bM_2}\right)\bV^{-1}.
\end{aligned}
\end{equation}
\paragraph{Bounding the norm of $\Delcvx$.}
We begin with some useful bounds:
\begin{subequations}
\begin{align}
&\max\{\|\bM_1\|,\|\bU\|,\|(\bSigma_{22})^{-1}\|,\|\bM_2^{-1}\|\} \le \|\bSigma^{-1}\|\label{eq:useful_b_bound}\\
&\max\{\|\bSigma_{22}\|, \|\bM_1^{-1}\|, \|\bM_2\|,\|\bV\| \} \le \|\bSigma\|. \label{eq:useful_c_bound}
\end{align}
\end{subequations}
\Cref{eq:useful_b_bound} uses the  fact that $\bM_1$ and $\bU$ are submatrices of $\bSigma^{-1}$, and the fact that for any positive-definite matrix, $\|\bSigma_{11}^{-1}\|$ (which is just $\|\bM_2^{-1}\|$) and $\|\bSigma_{22}^{-1}\|$ are both at most $\|\bSigma^{-1}\|$ (as can be verified by the block-matrix inverse formula). Finally,  \Cref{eq:useful_c_bound} follows from similar reasoning. 

\emph{Notational aside. } In what follows, we apply our $\polyop(\cdot)$ notation, which denotes a universal polynomial in the operator norms of its matrix arguments, and in the values of its scalar arguments. We let $\polyop(\cdot)$ include universal constant terms (e.g. $1 + \|\bX\|$ is $\polyop(\bX)$). 

From \Cref{eq:Del12}, we can bound
\begin{align*}
\|\bDelta_{12}\|_{\fro} &= \polyop(\bM_1^{-1},\bSigma_{22},\bV^{-1})\cdot(\|\bDelta_{\bM_1}\|_{\fro}+\|\bDelta_{\bM_2}\|_{\fro}).
\end{align*}
In light of \Cref{eq:useful_c_bound},
\begin{align}
\|\bDelta_{12}\|_{\fro} = \polyop(\bSigma,\bV^{-1})\cdot(\|\bDelta_{\bM_1}\|_{\fro}+\|\bDelta_{\bM_2}\|_{\fro})\label{eq:del_one_two_bd},
\end{align}
which means that 
\begin{align*}
\|\bDelta_{\bSigma}\|_{\fro} \overset{(i)}{=} \sqrt{\|\bDelta_{\bM_2}\|_{\fro}^2 + 2\|\bDelta_{12}\|_{\fro}^2}\le \polyop(\bSigma,\bV^{-1})\cdot(\|\bDelta_{\bM_1}\|_{\fro}+\|\bDelta_{\bM_2}\|_{\fro}),
\end{align*}
where $(i)$ uses \Cref{eq:Del_Sig} and \Cref{eq:Del11}, and the second inequality calls \Cref{eq:del_one_two_bd}. 

Next, from \Cref{eq:bUbar} and \Cref{eq:useful_b_bound,eq:useful_c_bound},
\begin{align*}
\|\bUbar'(0)\|_{\fro} &\le \|\bDelta_{\bM_1}\|_{\fro}\|\bV\|\|\bSigma_{22}^{-1}\|- \|\bM_1\|\|\bDelta_{12}\|\|(\bSigma_{22})^{-1}\|\\
&= \polyop(\bSigma,\bSigma^{-1})\left(\|\bDelta_{\bM_1}\|_{\fro} + \|\bDelta_{12}\|_{\fro}\right)\\
&= \polyop(\bSigma,\bSigma^{-1},\bV^{-1})\left(\sum_{j=1}^2\|\bDelta_{\bM_j}\|_{\fro} \right).
\end{align*}
Continuing, 
\begin{align*}
\|\bDelta_{\bB}\|_{\fro} &= \|\bU^{-1}\|\left(\|\bDelta_{\bL_2}\|_{\fro}-\|\Bk\| \|\bUbar'(0)\|_{\fro}\right) \\
&\le \polyop(\bSigma^{-1},\bSigma,\bV^{-1},\bU^{-1},\Bk)\left(\sum_{j=1}^2\|\bDelta_{\bM_j}\|_{\fro} + \|\bDelta_{\bL_2}\|_{\fro} \right),
\end{align*}
and similarly, since $\bVbar'(0) = \bDelta_{12}$ bounded as in \Cref{eq:del_one_two_bd},
\begin{align*}
\|\bDelta_{\bC}\|_{\fro} &= \|\bV^{-1}\|\left(\|\bDelta_{\bL_3}\|_{\fro}-\|\Ck\| \|\bVbar'(0)\|_{\fro}\right) \\
&\le \polyop(\bSigma^{-1},\bSigma,\bV^{-1},\bU^{-1},\Ck)\left(\sum_{j=1}^2\|\bDelta_{\bM_j}\|_{\fro} + \|\bDelta_{\bL_3}\|_{\fro} \right).
\end{align*}

Finally,
\begin{align*}
\|\bDelta_{\bA}\|_{\fro} &= \polyop\left(\bA,\bC,\bU,\bV,\bM_1,\bM_2,\Ak,\Bk,\bU^{-1},\bV^{-1}\right)\\
&\quad \times \left(\|\bUbar'(0)\|_{\fro}+\|\bDelta_{\bL_1}\|_\fro + \|\bDelta_{\bB}\|_\fro +\sum_{j=1}^2\|\bDelta_{\bM_j}\|_{\fro}\right)\\
&= \polyop\left(\bA,\bC,\bSigma,\bSigma^{-1},\Ak,\Bk,\bU^{-1},\bV^{-1}\right)\left(\sum_{j=1}^2\|\bDelta_{\bM_j}\|_{\fro} + \|\bDelta_{\bL_1}\|_{\fro} \right).
\end{align*}
In sum, 
\begin{align*}
\|\Dellift\|_{\ell_2}^2 &= \|\bDelta_{\bA}\|_{\fro}^2+ \|\bDelta_{\bC}\|_{\fro}^2 + \|\bDelta_{\bC}\|_{\fro}^2 + \|\bDelta_{\bSigma}\|_{\fro}^2 \\
&=\polyop\left(\bA,\bC,\bSigma,\bSigma^{-1},\Ak,\Bk,\Ck,\bU^{-1},\bV^{-1}\right)\cdot \left(\sum_{j=1}^2\|\bDelta_{\bM_j}\|_{\fro}^2 + \sum_{i=1}^3\|\bDelta_{\bL_i}\|_{\fro}^2 \right).
\end{align*}
The bound follows.
\qed

\subsection{Proof of \Cref{lem:UV_inverse}}
Let $\bZ = \bSigma_{12}\bSigma_{22}^{-1}\bSigma_{12}^\top$. Then, 
\begin{align*}
\|\bZ^{-1}\| &= \|\bSigma_{12}^{-\top}\bSigma_{22}\bSigma_{12}^{-1}\| \ge \|\bSigma_{12}^{-1}\|^2 \lambda_{\min}(\bSigma_{22})\\
&\ge\|\bSigma_{12}^{-1}\|^2 \lambda_{\min}(\bSigma_{22}) \\
&\ge\|\bSigma_{12}^{-1}\|^2 \lambda_{\min}(\bSigma).
\end{align*}
Hence,
\begin{align*}
\|\bV^{-1}\| = \|\bSigma_{12}^{-1}\| \leq \sqrt{\frac{\|\bZ^{-1}\|}{\lambda_{\min}(\bSigma)}} = \sqrt{\|\bZ^{-1}\| \|\bSigma^{-1}\|}. 
\end{align*}
Next, from the block matrix inversion identity, we have
\begin{align*}
\bU := (\bSigma^{-1})_{12} = -(\bSigma^{-1})_{11}\bSigma_{12}\bSigma_{22}^{-1} =-(\bSigma^{-1})_{11}\bV\bSigma_{22}^{-1}.
\end{align*}
Hence, 
\begin{align*}
\|\bU^{-1}\| = \|\bSigma_{22}\bV^{-1}(\bSigma^{-1})_{11}^{-1}\| &\le \frac{\|\bSigma_{22}\|}{\lambda_{\min}((\bSigma^{-1})_{11})}\|\bV^{-1}\|\\
&\le \frac{\|\bSigma\|}{\lambda_{\min}(\bSigma^{-1})}\|\bV^{-1}\|= \|\bSigma\|\|\bSigma^{-1}\|\|\bV^{-1}\|\\
&\le \|\bSigma\|\sqrt{\|\bSigma^{-1}\|^3\|\|\bZ^{-1}\|}.
\end{align*}\
The conclusion invokes \Cref{lem:phi_cond}.
\qed

\subsection{Proof of convexity lemmas \label{sec:proof_cvx_lemmas}}
Here we prove \Cref{lem:convex_function_two,lem:convex_function_one}, restated below for convenience. Both proofs use the fact that convexity is preserved under partial minimization.
\begin{fact}[Chapter  3.2.5 of \cite{boyd2004convex}]\label{fact:boyd_fact} Let $\tilde{\phi}(\bx,\by)$ be a convex function in two arguments such that $\phi(\bx) := \min_{\by} \tilde{\phi}(\bx,\by)$ is finite and attained for each $\bx$. Then $\phi(\bx)$ is convex.
\end{fact}
\cvxone*
\begin{proof}
Observe that we can express
\begin{align}
g(\bCtil,\bXtil) = \min_{\bE \in \sym{n}} \tilde{g}(\bCtil,\bXtil,\bE), \quad \tilde{g}(\bCtil,\bXtil,\bE) = \left(\trace(\bE) \cdot \I\left\{\bE \succeq 0, \quad \begin{bmatrix} \bE & \bCtil \\
\bCtil^\top & \bXtil\end{bmatrix} \succeq 0 \right\}\right). \label{eq:rep_of_g}
\end{align} 
Indeed, since $\bXtil \succ 0$ on the domain of $g$, the Schur complement test implies that $\begin{bmatrix} \bE & \bCtil^\top \\
\bCtil & \bXtil\end{bmatrix} \succeq 0$ if and only if $\bE \succeq \bCtil \bXtil^{-1}\bCtil^\top$. Hence, the minimal value of $\trace(\bE)$ is attained with $\bE = \bCtil \bXtil^{-1}\bCtil^\top$. Observing that $\tilde{g}(\bCtil,\bXtil,\bE)$ is convex (affine function with a convex constraint), \Cref{fact:boyd_fact} implies that its partial minimization $g$ is convex. 
\end{proof}

\cvxtwo*
\begin{proof}
Introduce the function
\begin{align*}
\tilde{h}(\matM_1,\matM_2,\bE) = \trace[\bE^{-1}] \cdot \I\left\{\begin{bmatrix} \matM_2 - \bE & \eye_n\\
\eye_n & \matM_1
\end{bmatrix} \succeq 0 ,\quad \bE \succ 0, \quad \matM_1 \succ 0\right\}.
\end{align*}
Since the function $\bE \mapsto \trace[\bE^{-1}]$ is convex for $\bE \succ 0$, the function $\tilde{h}$ is also convex. The constraint in $\tilde{h}$ is equivalent to $\matM_1 \succ 0$, $\bE\succ 0$, and $\matM_2 - \bE - \matM_1^{-1} \succeq 0$. Rearranging that is $\bE \preceq \matM_2 - \matM_1^{-1}$, or equivalently, $\bE^{-1} \succeq (\matM_2 - \matM_1^{-1})^{-1}$. Hence, $\tilde{h}$ can be written as
\begin{align*}
\tilde{h}(\matM_1,\matM_2,\bE) = \trace[\bE^{-1}] \cdot \I\left\{\bE^{-1} \succeq (\matM_2 - \matM_1^{-1})^{-1},\quad \bE \succ 0, \quad \matM_1 \succ 0\right\}.
\end{align*}
From the above form, it is clear that $\min_{\bE}\tilde{h}(\matM_1,\matM_2,\bE)  = \tilde{h}(\matM_1,\matM_2)$, which is finite and attained by $\bE = \matM_2 - \matM_1^{-1}$ on the domain of $h$.
\end{proof} 


%% file: appendix/lyapunov.tex

\section{Bounds on Solutions to Closed-Loop Lyapunov Equations (\Cref{prop:clyap_compact_form})}\label{app:Lyap}

\newcommand{\Pk}{\bP_{\sfK}}
\newcommand{\Sigbar}{\bar{\bSigma}}
\newcommand{\Wbar}{\bar{\bW}}
\renewcommand{\rhok}{\rho_{\sfK}}
\newcommand{\covmat}{\bSigma}
\newcommand{\covw}{\covmat_{w}}
\newcommand{\covv}{\covmat_{v}}
\newcommand{\covwone}{\covmat_{11,w}}
\newcommand{\covwonetwo}{\covmat_{12,w}}
\newcommand{\covwtwo}{\covmat_{22,w}}
\newcommand{\Ctk}{C_{[t]}(\sfK)}
\newcommand{\clyapt}{\clyap^{[t]}}

The following proposition gives a more granular statement of \Cref{prop:clyap_compact_form} in the main text. 
\begin{proposition}\label{prop:clyap_formal} Let $\circnorm{\cdot}$ denote either the operator, Frobenius, or nuclear norm, and let $\clyap$ denote the integral in \Cref{eq:clyap}, which corresponds to the continuous Lyapunov opertor when its argument is Hurwitz. Then, for any matrix $\bY \in \sym{2n}$,
\begin{align*}
\circnorm{\clyap(\Aclk,\bY)} \le \conslyapK \cdot \circnorm{\bY},
\end{align*}
where $ \conslyapK  = \poly\left(\|\Sigk\|,\|\Sigk^{-1}\|,\|\Zk^{-1}\|,\|\bW_1^{-1}\|, \|\bW_2^{-1}\|,\|\bC\| \right)$. More precisely,
\begin{align*}
\conslyapK &:=  \frac{ 8 t_{\star}(\sfK)^2 \|\Sigk\|^2\| \Sigonesys\|^2 \|\bC\|^2}{\lambda_{\min}(\Sigk)\lambda_{\min}(\bW_2)\lambda_{\min}(\Sigktwo)\lambda_{\min}(\Zk)} \cdot \max\left\{ 1, \frac{4\|\Sigktwo\|}{\lambda_{\min}(\Sigonesys)}\right\}, \quad \text{ where }\\
t_{\star}(\sfK) &:= \frac{\|\Sigonesys\|}{\lambda_{\min}(\bW_1)}\log \left(\frac{\|\Sigonesys\|^2}{\lambda_{\min}(\bW_1)} \max\left\{\frac{2}{\lambda_{\min}(\Sigonesys)}, \frac{4\|\Sigonesys\|}{\lambda_{\min}(\Sigonesys)\lambda_{\min}(\Zk)}\right\} \right).
 \end{align*}
 \end{proposition}
 The following corollary is also useful for establishing compact level sets.

\subsection{Preliminaries on Lyapunov solutions} 
As a shorthand, we let $\clyap$ denote the following limit, if it converges:\footnote{That is, if $\lim_{t=0}^{\infty}\fronorm{\exp(s\bX)\bY\exp(s\bX)^\top}\,\rmd s$ is finite.}
\begin{align}
\clyap(\bX,\bY) = \lim_{t \to \infty} \int_{0}^t \exp(s\bX)\bY\exp(s\bX)^\top \rmd s =\int_{0}^\infty \exp(s\bX)\bY\exp(s\bX)^\top \rmd s. \label{eq:clyap}
\end{align}
We also define a ``finite-time version'', which is defined for all $\bX \in \R^{d \times d}$ and $\bY \in \sym{d}$: 
\begin{align*}
\clyap^{[t]}(\bX,\bY) = \int_{0}^t \exp(s\bX)\bY\exp(s\bX)^\top \rmd s, \quad \clyap^{[>t]}(\bX,\bY) = \int_{t}^\infty \exp(s\bX)\bY\exp(s\bX)^\top \rmd s.
\end{align*}
The name $\clyap$ is short for ``continuous Lyapunov'', and is motivated by the following lemma:
\begin{lemma}\label{lem:cont_lyap} Suppose that $\bX$ is Hurwitz stable. Then, $\bGamma = \clyap(\bX,\bY)$ exists and is the unique solution to the Lyapunov equation
\begin{align*}
\bX \bGamma + \bGamma \bX^\top + \bY = 0.
\end{align*}
In addition, if there exists a sequence $t_1,t_2,\dots$ such that $\lim_{k \to \infty} \fronorm{\exp(t_k \bX)} \to 0$, then $\bX$ is Hurwitz stable. 
\end{lemma}

\subsection{Proof of \Cref{prop:clyap_formal}}
\subsubsection{Setup.} Recall that $\Sigk = \clyap(\Aclk,\Wclk)$ and $\Sigonesys = \clyap(\bA,\bW_1)$ solve the equations
\begin{align*}
\underbrace{\begin{bmatrix} \bA & 0 \\
\Bk  \bC & \Ak\end{bmatrix}}_{= \Aclk} \Sigk + \Sigk \begin{bmatrix} \bA & 0 \\
\Bk \bC & \Ak\end{bmatrix}^\top + \underbrace{\begin{bmatrix} \bW_1 & 0 \\ 0 & \Bk\bW_2\Bk^\top \end{bmatrix}}_{\Wclk} = 0, \quad 
\bA \Sigonesys + \Sigonesys \bA^\top +\bW_1  = 0.
\end{align*}
Define the matrix $\covw = \clyap(\Aclk,\bW_0)$ and $\covv = \clyap(\Ak,\Bk \bW_2 \Bk^\top)$ as the solutions to the Lyapunov equations 
\begin{align}
\Aclk \covw + \covw \Aclk^\top + \underbrace{\begin{bmatrix} \bW_1 & 0 \\ 0 & 0 \end{bmatrix}}_{:=\bW_0} = 0, \quad 
\Ak \covv + \covv \Ak^\top +\Bk \bW_2 \Bk^\top  = 0.
\end{align}
We recall the following closed-form expression for the solution to Lyapunov equations.

In particular, 
\begin{align*}
\Sigk &= \int_{0}^\infty \exp(\tau\Aclk)\Wclk\exp(\tau\Aclk)^\top\rmd\tau\\
\covw &= \int_{0}^\infty \exp(\tau\Aclk)\bW_0\exp(\tau\Aclk)^\top\rmd\tau\\
\covv &= \int_{0}^\infty \exp(\tau\Ak)\Bk\bW_2\Bk^\top\exp(\tau\Ak)^\top\rmd\tau\\
\Sigonesys &= \int_{0}^\infty \exp(\tau\bA)\bW_1\exp(\tau\bA)^\top\rmd\tau.
\end{align*}
Throughout, we use the following decompositions
\begin{align*}
\Sigk =  \Sigk^{[t]}  + \Sigk^{[>t]}, \quad \covw = \covw^{[t]} + \covw^{[>t]}, \quad \covv  = \covv^{[t]} + \covv^{[>t]}, \quad  \Sigonesys  = \Sigonesys^{[t]} + \Sigonesys^{[>t]},
\end{align*}
where we define
\begin{align*}
\Sigk^{[t]} := \int_{0}^t \exp(\tau\Aclk)\Wclk\exp(\tau\Aclk)^\top\rmd\tau, \quad \Sigk^{[>t]} := \int_{t}^\infty \exp(\tau\Aclk)\Wclk\exp(\tau\Aclk)^\top\rmd\tau,
\end{align*}
and where $\covw^{[t]},\covw^{[>t]}, \covv^{[t]}, \covv^{[>t]}, \Sigonesys^{[t]} ,\Sigonesys^{[>t]}$ are all defined analogously. The following computations are useful.
\begin{lemma}[Computations of exponentials]\label{lem:computation_expon} The following characterizes the  exponentials of $\Aclk$:
\begin{itemize}
	\item[(a)] Defining $\bM(t) = \int_{0}^t \exp((t-s)\Ak)\Bk\bC\exp(s\bA)\rmd s$, one has
\begin{align*}
\exp(t\Aclk) = \begin{bmatrix} \exp(t\bA) & 0 \\
\bM(t) & \exp(t\Ak)
\end{bmatrix}.
\end{align*}
\item[(b)] The following computation holds
\begin{align*}
\exp(t\Aclk) \begin{bmatrix} \bW_1 & 0 \\ 0 & 0 \end{bmatrix}\exp(t\Aclk)^\top =  \begin{bmatrix} \exp(t\bA)\bW_1 \exp(t\bA)^\top & \exp(t\bA)\bW_1 \bM(t)^\top \\
\bM(t) \bW_1 \exp(t\bA)^\top & \bM(t) \bW_1 \bM(t)^\top
\end{bmatrix}. 
\end{align*}
\end{itemize}
\end{lemma}
\begin{proof} Part (b) follows directly from part (a) and a straightforward computation. To prove part (a), we observe that the desired identity holds at time $t = 0$. To prove it holds for all $t$, it suffices to equate derivatives. First, we compute
\begin{align*}
\ddt \bM(t) &= \ddt\int_{0}^t \exp((t-s)\Ak)\Bk\bC\exp(s\bA)\rmd s \\
&= \exp((t-s)\Ak)\Bk\bC\exp(s\bA)\big{|}_{s=t} + \int_{0}^t \ddt \exp((t-s)\Ak)\Bk\bC\exp(s\bA)\rmd s\\
&= \Bk\bC\exp(t\bA)+ \int_{0}^t \Ak \exp((t-s)\Ak)\Bk\bC\exp(s\bA)\rmd s = \Bk\bC\exp(t\bA) + \Ak\bM(t).
\end{align*}
Therefore, 
\begin{align*}
\ddt \begin{bmatrix} \exp(t\bA) & 0 \\
\bM(t) & \exp(t\Ak)
\end{bmatrix} = \begin{bmatrix} \bA \exp(t\bA) & 0 \\
\Bk\bC\exp(t\bA) + \Ak\bM(t) & \Ak \exp(t\Ak)
\end{bmatrix} = \Aclk \begin{bmatrix} \exp(t\bA) & 0 \\
\bM(t) & \exp(t\Ak)
\end{bmatrix}.
\end{align*} 
Similarly, $\ddt \exp(t\Aclk) = \Aclk \exp(t\Aclk)$. The identity follows from uniqueness of solutions to ODEs.
\end{proof}

The following lemma is straightforward to verify using the previous one. 
\begin{lemma}[Useful identitities]\label{lem:useful_idens_sig} The following identities hold:
\begin{itemize}
\item[(a)] One has the decompositions \begin{align}
\Sigk = \covw + \begin{bmatrix} 0 & 0 \\ 0 & \covv\end{bmatrix}, \quad  \Sigk^{[t]} = \covw^{[t]} + \begin{bmatrix} 0 & 0 \\ 0 & \covv^{[t]}\end{bmatrix}\label{eq:covdecomp}
\end{align}
\item[(b)] $\Sigkone^{[t]} = \covwone^{[t]} =  \Sigonesys^{[t]}$ and similarly, $\Sigkone^{[>t]} = \covwone^{[>t]} =  \Sigonesys^{[>t]} ,$ and $\Sigkone= \covwone =  \Sigonesys$.
	\end{itemize}
\end{lemma}
As a consequence, we find that $\covwone^{[t]}$ is invertible for all $t$. Lastly, we show $\covwone^{[t]} \succ 0$.
\begin{lemma}\label{lem:sigone_psd} $\Sigkone^{[t]} = \covwone^{[t]} =  \Sigonesys^{[t]} \succ 0$ for all $t > 0$.
\end{lemma}
\begin{proof} The equivalence $\Sigkone^{[t]} = \covwone^{[t]} =  \Sigonesys^{[t]}$ is given by \Cref{lem:useful_idens_sig}. Using the formula $\Sigonesys^{[t]} = \int_{0}^t \exp(\tau\bA)\bW_1\exp(\tau\bA)^\top\rmd\tau$, we see that we can $\Sigonesys^{[t]} = \int_{0}^t \bN(\tau)\rmd \tau$, where $\bN(\cdot)$ is a continuous matrix valued function with $\bN(0) = \bW_1 \succ 0$. Hence, for all vectors $\bv \ne 0$, the function $f(\cdot;\bv) = \bv^\top \bN(\cdot) \bv $ is continous and has $f(0;\bv) = 0$. Thus, $\bv^\top\Sigonesys^{[t]}\bv = \int_{0}^t f(\tau;\bv) \rmd \tau > 0$ for all nonzero $\bv$. 
\end{proof}

\subsubsection{A Lyapunov argument \label{sec:lyap_arg_sec}}
\newcommand{\Pnot}{\bP_0}
\newcommand{\Ynot}{\bY_0}
\newcommand{\Gamnot}{\bGamma_0}
\newcommand{\rhonot}{\rho_0}
In this section, we show that if there is a finite $t$ for which $\lambda_{\min}(\Sigk^{[t]})$ is strictly positive, then one can bound the solutions to $\clyap(\Aclk,\bY)$ in terms of this $t$ and other problem-dependent quantities. We begin with a general lemma that bounds the decay of matrix exponentials, with their finite-time Gramians.
\begin{lemma}\label{lem:lyap_xy} Fix a matrix $\bX \in \R^{d \times d}$, and matrix $\bY_0 \in \sym{d}$, and suppose that the solution $\Gamnot = \clyap(\bX,\bY_0)$ exists. Define $\bGamma_0^{[t]} = \clyapt(\bX,\bY_0)$, and $\Gamnot^{[>t]}$ analogously.  Then, for all $s,t \ge 0$, 
\begin{itemize}
\item[(a)]  $\Pnot(t) = \Gamnot^{[>t]}$, where $\Pnot(t) := \exp(t\bX) \bGamma_0 \exp(t\bX)^\top$.
\item[(b)] $\Pnot(s+t) \preceq \rhonot(s)\cdot\Pnot(t)$, where $\rhonot(s):= 1 - \frac{\lambda_{\min}(\Gamnot^{[s]})}{\|\Gamnot\|}$.
\item[(c)] In particular, if $\Gamnot \succ 0$, then $\bX$ is Hurwitz stable. 
\end{itemize}
\end{lemma}
\begin{proof}
	\textbf{Part (a).} We see that 
	\begin{align*}
	\Gamnot^{[>t]} &:= \int_{t}^\infty \exp(\tau\bX)\Ynot\exp(\tau\bX)^\top\rmd\tau\\
	&= \exp(t\bX)\left(\int_{t}^\infty \exp((\tau-t)\bX)\Ynot\exp((\tau-t)\bX)^\top\rmd\tau\right)\exp(t\bX)^\top\\
	&= \exp(t\bX)\left(\int_{0}^\infty \exp(\tau\bX)\Ynot\exp(\tau\bX)^\top\rmd\tau\right)\exp(t\bX)^\top\\
	&= \exp(t\bX)\Gamnot\exp(t\bX)^\top := \Pnot(t).
	\end{align*}
	\textbf{Part (b). } We use the decomposition
	\begin{align*}
	\Gamnot = \Gamnot^{[t]}  + \Gamnot^{[>t]} = \Gamnot^{[t]} + \Pnot(t).
	\end{align*}
	For a fixed $t$ and $s \ge 0$, we have
	\begin{align*}
	\Pnot(s+t) &= \exp(t\bX)\cdot \exp(s\bX) \Gamnot \exp(s\bX)^\top \cdot \exp(t\bX)^\top\\
	&= \exp(t\bX)\cdot (\Gamnot - \Gamnot^{[s]}) \cdot \exp(t\bX)^\top\\
	&= \exp(t\bX)\cdot \Gamnot^{1/2}(\eye_n - \Gamnot^{-1/2}\Gamnot^{[s]}\Gamnot^{-1/2})\Gamnot^{1/2} \cdot \exp(t\bX)^\top\\
	&\le \lambda_{\max}(\eye_n - \Gamnot^{-1/2}\Gamnot^{[s]}\Gamnot^{-1/2}) \cdot \underbrace{\exp(t\bX)\cdot \Gamnot^{1/2}\Gamnot^{1/2} \cdot \exp(t\bX)^\top}_{=\Pnot(t)}\\
	&\le \underbrace{\left(1-\frac{\lambda_{\min}(\Gamnot^{[s]})}{\|\Gamnot\|}\right) }_{=\rhonot(s)} \cdot \Pnot(t).
	\end{align*}
	\textbf{Part (c).} Suppose that $\Gamnot \succ 0$. Then, since $\Gamnot^{[s]}$ is monotone, there exists a finite $s$ such that $\Gamnot^{[s]} \succ 0$. Thus, $\rhonot(s) < 1$.  Then, by iterating part $(b)$, we have that for any finite $k \in \N$
	\begin{align*}
	\Pnot(ks+t) \preceq \rhonot(s)^k\cdot\Pnot(t),
	\end{align*}
	so that $\lim_{k \to \infty} \Pnot(ks+t) = \lim_{k\to\infty}\exp((ks+t)\bX)\bGamma_0\exp((ks+t)\bX)^\top = 0$. Since $\bGamma_0 \succ 0$, this implies that $\lim_{k\to\infty}\|\exp((ks+t)\bX)\| = 0$. By \Cref{lem:cont_lyap}, this can only occur if $\bX$ is Hurwitz stable. 
\end{proof}
By integrating \Cref{lem:lyap_xy}, we bound $\circnorm{\clyap(\bX,\bY)}$ in terms of $\lambda_{\min}(\Gamnot^{[t]})$.
\begin{lemma}\label{lem:ctk_bound_xy} Consider the setup of \Cref{lem:lyap_xy}, and suppose that $t > 0$ is such that $\lambda_{\min}(\Gamnot^{[t]}) > 0$. Then, for any $\bY \in \sym{d}$, and for $\circnorm{\cdot}$ denoting either operator, Frobenius, or nuclear norm,
\begin{align*}
\circnorm{\clyap(\bX,\bY)}  \le \frac{t\|\Gamnot\|^2}{\lambda_{\min}(\Gamnot)\lambda_{\min}(\Gamnot^{[t]})}\cdot \circnorm{\bY}.  
\end{align*}
\end{lemma}
	\begin{proof} Using \Cref{lem:cont_lyap}, we write $\Sigbar$ explicitly and bound  it as follows
	\begin{align*}
	\circnorm{\clyap(\bX,\bY)} &= \left\|\int_{0}^\infty \exp(\tau\bX)\bY\exp(\tau\bX)^\top\rmd\tau\right\|_{\circ}\\
	&\le \int_{0}^\infty \left\|\exp(\tau\bX)\bY\exp(\tau\bX)^\top\right\|_{\circ}\rmd\tau\\
	&\overset{(i)}{\le} \circnorm{\bY} \int_{0}^\infty \left\|\exp(\tau\bX)\right\|^2 \rmd\tau\\
	&= \circnorm{\bY} \int_{0}^\infty \left\|\exp(\tau\bX)\exp(\tau\bX)^\top\right\| \rmd\tau\\
	&\le  \circnorm{\bY}\|\Gamnot^{-1}\| \cdot \int_{0}^\infty \|\exp(\tau\bX)\Gamnot\exp(\tau\bX)^\top\|\rmd\tau\\
	&=\circnorm{\bY}\|\Gamnot^{-1}\| \cdot\sum_{k=0}^{\infty}\int_{tk}^{t(k+1)}\|\exp(\tau\bX)\Gamnot\exp(\tau\bX)^\top\|\rmd\tau\\
	&=\circnorm{\bY}\|\Gamnot^{-1}\| \cdot\sum_{k=0}^{\infty}\int_{tk}^{t(k+1)}\|\Pnot(\tau)\|\rmd\tau.
	\end{align*}
	{Here, $(i)$ uses that $\|\bX_1 \bX_2\| \le \min\{\|\bX_1\|\|\bX_2\|_{\circ},\,\|\bX_1\|_{\circ}\|\bX_2\|\}$ for any $\circ$ denoting either the operator, Frobenius, or trace norms (or more generally, any Schatten norm).}
	From \Cref{lem:lyap_xy}, $\Pnot(\tau)$ is non-increasing in the PSD order and $\|\Pnot(tk)\| \le \|\Pnot(0)\|\rhonot(t)^k$.  Hence, noting that $\Pnot(0) = \Gamnot$,
	\begin{align*}
	\int_{tk}^{t(k+1)}\|\Pnot(\tau)\|\rmd\tau  \leq  t \|\Pnot(tk)\| \le t\|\Pnot(0)\|\rhonot(t)^k = t \|\Gamnot\| \rhonot(t)^k. 
	\end{align*}
	Thus, if $\lambda_{\min}(\Gamnot^{[t]}) > 0$, then $\rhonot(t) < 1$, so we can sum
	\begin{align*}
	\circnorm{\clyap(\bX,\bY)} &\le \circnorm{\bY}\|\Gamnot^{-1}\| \cdot t\|\Gamnot\| \cdot\sum_{k=0}^{\infty}\rhonot(t)^k\\
	&= \circnorm{\bY}\|\Gamnot^{-1}\| \cdot t\|\Gamnot\|\cdot\frac{1}{1-\rhonot(t)}\\
	&= \circnorm{\bY}\|\Gamnot^{-1}\| \cdot t\|\Gamnot\|\cdot\frac{\|\Gamnot\|}{\lambda_{\min}(\Gamnot^{[t]})}. 
	\end{align*}
	Hence, 
	\begin{align*}
	\circnorm{\clyap(\bX,\bY)} \le \circnorm{\bY}\cdot t\frac{\|\Gamnot^{-1}\|\|\Gamnot\|^2}{\lambda_{\min}(\Gamnot^{[t]})} = \circnorm{\bY}\cdot \frac{t\|\Gamnot\|^2}{\lambda_{\min}(\Gamnot)\lambda_{\min}(\Gamnot^{[t]})}. 
	\end{align*}
\end{proof}
Specializing with $\bX \gets \Aclk$, $\bY_0 \gets \Wclk$, $\Gamnot \gets \Sigk$, we arrive at the following lemma:
\begin{lemma}\label{lem:ctk_bound} Suppose that $t > 0$ is such that $\lambda_{\min}(\Sigk^{[t]}) > 0$. Then, for any $\bY \in \sym{2n}$, and for $\circnorm{\cdot}$ denoting either operator, Frobenius, or nuclear norm,
\begin{align*}
\circnorm{\clyap(\Aclk,\bY)}  \le  \Ctk\cdot \circnorm{\bY}, \quad \text{ where } \Ctk := \frac{t\|\Sigk\|^2}{\lambda_{\min}(\Sigk)\lambda_{\min}(\Sigk^{[t]})}.
\end{align*}
\end{lemma}

Thus, it remains to show that, for any appropriate choice of $t$ 
\begin{align}
\Ctk \le \conslyapK. \label{eq:ct_K_bound}
\end{align}

\subsubsection{Lower bounding finite-time covariance in terms of diagonal blocks}
In order to upper bound $\Ctk$ from \Cref{lem:ctk_bound}, we must lower bound $\lambda_{\min}(\bSigma_{\sfK}^{[t]})$. Recall that from \Cref{lem:useful_idens_sig},
\begin{align}
\Sigk^{[t]} = \covw^{[t]} + \begin{bmatrix} 0 & 0 \\ 0 & \covv^{[t]}\end{bmatrix}.\label{eq:covdecomp_two}
\end{align}
Leveraging this form, we show that it suffices to lower bound $\lambda_{\min}(\covv^{[t]})$, that is, the finite-time covariance introduced by the observation noise into the policy. 
	\begin{lemma}\label{lem:Sigbound_to_block} Suppose that $t$ is large enough such that  $\|\Sigonesys^{[<t]}\| \le \frac{1}{2}\lambda_{\min}(\Sigonesys)$, then 
	\begin{align*}
	\lambda_{\min}\left( \Sigk^{[t]}\right) &\ge \frac{1}{2}\lambda_{\min}(\covv^{[t]})\min\left\{ 1, \frac{\lambda_{\min}(\Sigonesys)}{4\|\Sigktwo\|}\right\}.
	\end{align*}
	\end{lemma}
	\begin{proof}
	Applying \Cref{lem:PSD_block_lam_min} to the decomposition \Cref{eq:covdecomp_two}, we have
\begin{align*}
	\lambda_{\min}\left( \Sigk^{[t]}\right) &\ge \frac{1}{2}\lambda_{\min}(\covv^{[t]})\min\left\{1, \frac{\lambda_{\min}(\covwone^{[t]})}{2\|\covwtwo^{[t]} + \covv^{[t]}\|}\right\}.
	\end{align*}
	Substituing in $\covwone^{[t]} = \Sigonesys^{[t]}$, and bounding $\|\covwtwo^{[t]}+ \covv^{[t]}\| = \|\Sigktwo^{[t]}\| \le \|\Sigktwo\|$ in view of \Cref{lem:useful_idens_sig}. Moreover, we have $\lambda_{\min}(\covwone^{[t]}) \ge \lambda_{\min}(\Sigonesys) - \|\Sigonesys^{[>t]}\| \ge \frac{1}{2}\lambda_{\min}(\Sigonesys)$, where the last step holds under the assumption on $t$ in the lemma. With these simplifications, we arrive at the desired bound:
	\begin{align*}
	\lambda_{\min}\left( \Sigk^{[t]}\right) &\ge \frac{1}{2}\lambda_{\min}(\covv^{[t]})\min\left\{1, \frac{\lambda_{\min}(\covwone)}{4\|\Sigktwo\|}\right\}.
	\end{align*}
	\end{proof}

\subsubsection{Lower bounding the contribution of output noise}
From \Cref{lem:Sigbound_to_block}, we have to lower bound $\lambda_{\min}\left(\covv^{[t]}\right)$. This step involves the two most original insights of the proof:
\begin{itemize}
	\item First, $\covv^{[t]} \succeq \frac{C}{t} \covwtwo^{[t]}$, for some system-dependent constant $C$. Here, $\covv^{[t]} $ represents the part of the internal-state covariance excited by the full-rank observation noise $\bW_2$, and $\covwtwo^{[t]}$ the part of the covariance excited by the observations  $\tilde{\by}(t) = \bC \bx(t)$. Essentially, we argue that the covariance excited by any stochastic process $\tilde{\by}(t)$ cannot be much greater than the excitation by Gaussian noise $\bv(t)$. 
	\item Second, if $\Zk \succ 0$ and if $\Sigonesys^{[>t]}$ is small, then we can lower bound $\lambda_{\min}(\covwtwo^{[t]})$ in terms of $\Zk$. Intuitively, $\covwtwo^{[t]}$ describes how much of the process noise $\bw(t)$ excites  the internal filter state $\fx(t)$, and $\Zk$ measures the correlation between $\fx(t)$ and $\sx(t)$. This argument therefore uses the insight that, if $\fx(t)$ and $\sx(t)$ have nontrivial correlation, some of the process noise $\bw(t)$ must be exciting the filter state $\fx(t)$.
	
\end{itemize}

\begin{lemma}\label{lem:covwtocovv}  For all $t$, we have
\begin{align*}
\covwtwo^{[t]}  \preceq \frac{t\|\bC\|^2\|\Sigonesys\|}{\lambda_{\min}(\bW_2)} \cdot \covv^{[t]}.
\end{align*}
In particular, $\Sigktwo^{[t]} \succ 0$ if and only if $\covv^{[t]}\succ 0 $.\footnote{Note that this lemma and its conclusion only requires \Cref{asm:pd}. It does not even require stability of $\Ak$.}
\end{lemma}
\begin{proof}
From \Cref{lem:computation_expon}, we have that
\begin{align*}
&\covwtwo^{[t]} \\
&= \int_{0}^t \left(\bM(\tau) \bW_1 \bM(\tau)^\top\right)\rmd \tau\\
&= \int_{0}^{t} \left( \int_{0}^\tau \exp((\tau-s_1)\Ak)\Bk\bC\exp(s\bA)\rmd s_1\right) \bW_1 \left( \int_{0}^\tau \exp((\tau-s_2)\Ak)\Bk\bC\exp(s\bA)\rmd s_2\right)^\top\rmd \tau \\
&= \int_{0}^{t}  \tau^2 \left( \frac{1}{\tau}\int_{0}^\tau \exp((\tau-s_1)\Ak)\Bk\bC\exp(s\bA)\bW_1^{1/2}\rmd s_1\right) \left( \frac{1}{\tau}\int_{0}^\tau \exp((\tau-s_2)\Ak)\Bk\bC\exp(s\bA)\bW_1^{1/2}\rmd s_2\right)^\top\rmd \tau\\
&\overset{(i)}{\preceq} \int_{0}^{t}  \tau^2 \cdot\frac{1}{\tau}\int_{0}^{\tau}  \left(\exp((\tau-s)\Ak)\Bk\bC\exp(s\bA)\bW_1^{1/2}\right)\left(\exp((\tau-s)\Ak)\Bk\bC\exp(s\bA)\bW_1^{1/2}\right)^\top \rmd s \rmd \tau\\
&= \int_{0}^t \int_{0}^{\tau} \tau \exp((\tau-s)\Ak)\Bk\bC\exp(s\bA)\bW_1\exp(s\bA)^\top \bC^\top \Bk^\top \exp((\tau-s)\Ak)^\top \rmd s \rmd \tau\\
&\preceq t \int_{0}^t \int_{0}^{\tau}  \exp((\tau-s)\Ak)\Bk\bC\exp(s\bA)\bW_1\exp(s\bA)^\top \bC^\top \Bk^\top \exp((\tau-s)\Ak)^\top \rmd s \rmd \tau.
\end{align*}
Here, inequality $(i)$ invoked \Cref{lem:matrix_rv_Cauchy}.  Using the integral re-arrangment 
\begin{align*}
\int_{\tau = 0}^t \int_{s= 0}^\tau \bN(\tau-s,s)\rmd s\rmd \tau &= \int_{s=0}^t\int_{\tau=s}^t \bN(\tau-s,s)\rmd \tau \rmd s\\
 &= \int_{s=0}^t\int_{\tau=0}^{t-s} \bN(\tau,s)\rmd \tau \rmd s\\
& \preceq \int_{s=0}^t\int_{\tau=0}^t \bN(\tau,s)\rmd \tau \rmd s
\end{align*}
for any PSD-matrix valued function $\bN(\cdot,\cdot):[0,t]^2 \to \psd{n}$, we obtain that
\begin{align*}
\covwtwo^{[t]} &\preceq t \int_{0}^t \int_{0}^{t}  \exp(\tau \Ak)\Bk\bC\exp(s\bA)\bW_1\exp(s\bA)^\top \bC^\top \Bk^\top \exp(\tau\Ak)^\top \rmd s \rmd \tau\\
&= t\int_{0}^t  \exp(\tau \Ak) \Bk \left(\bC \left(\int_{0}^{t}\exp(s \bA)\bW_1\exp(s\bA)^\top\rmd s\right) \bC^\top \right) \Bk^\top \exp(\tau\Ak)^\top \rmd \\
&=t \int_{0}^t  \exp(\tau \Ak) \Bk \left(\bC  \Sigonesys^{[t]} \bC^\top \right) \Bk^\top \exp(\tau\Ak)^\top \rmd \tau\\
&=t \int_{0}^t  \exp(\tau \Ak) \Bk \bW_{2}^{1/2}\left(\bW_{2}^{-1/2}\bC \Sigonesys^{[t]} \bC^\top \bW_{2}^{-1/2}\right) \bW_{2}^{1/2}\Bk^\top \exp(\tau\Ak)^\top \rmd \tau.
\end{align*}

We render the above integral as 
\begin{align*}
t \int_{s_2=0}^t\int_{s_1=0}^t  \exp(s_1\Ak)\Bk\bC\exp(s_2\bA)\bW_1\exp(s_2\bA)^\top \bC^\top \Bk^\top \exp(s_1\Ak)^\top \rmd s_1 \rmd s_2. 
\end{align*} 

Bounding $\|\bW_{2}^{-1/2}\bC \Sigonesys^{[t]} \bC^\top\bW_{2}^{-1/2}\| \le \frac{\|\bC\|^2\|\Sigonesys\|}{\lambda_{\min}(\bW_2)} \le \frac{\|\bC\|^2\|\Sigonesys\|}{\lambda_{\min}(\bW_2)} $, we have 
\begin{align*}
\covwtwo^{[t]} &\preceq \frac{t\|\bC\|^2\|\Sigonesys\|}{\lambda_{\min}(\bW_2)} \cdot \int_{0}^t  \exp(\tau \Ak) \Bk \bW_2 \Bk^\top \exp(\tau\Ak)^\top \rmd \tau\\
&= \frac{t\|\bC\|^2\|\Sigonesys\|}{\lambda_{\min}(\bW_2)} \cdot \covv^{[t]}.
\end{align*}
The last point follows from $\Sigktwo^{[t]} = \covwtwo^{[t]} + \covv^{[t]}$ by \Cref{lem:useful_idens_sig}.
\end{proof}

\begin{lemma}\label{lem:covwtwo_lb_lem} Suppose that  $t$ is sufficiently large such that 
\begin{align*}
\|\Sigonesys^{[>t]}\| \le \frac{\lambda_{\min}(\Sigktwo)\lambda_{\min}(\Zk)}{4\|\Sigktwo\|}. 
\end{align*}
Then, it holds that
\begin{align*}
\lambda_{\min}\left(\covwtwo^{[t]}\right) \ge \frac{\lambda_{\min}(\Sigktwo)\lambda_{\min}(\Zk)}{4\| \Sigonesys\|}.
\end{align*}
 Therefore, in view of \Cref{lem:covwtocovv},
 \begin{align*}
 \lambda_{\min}\left(\covv^{[t]}\right) \ge \frac{\lambda_{\min}(\bW_2)\lambda_{\min}(\Sigktwo)\lambda_{\min}(\Zk)}{ 4 t\| \Sigonesys\|^2 \|\bC\|^2}.
 \end{align*}
\end{lemma}
\begin{proof} We assume that $\lambda_{\min}(\Zk) \ge 0$ for otherwise the lemma is vacuous. We compute
\begin{align}
\sigma_{\min}(\Sigkonetwo)^2 = \lambda_{\min}(\Sigkonetwo \Sigkonetwo^\top) \ge \frac{1}{\|\Sigktwo^{-1}\|}\lambda_{\min}\left(\Sigkonetwo \Sigktwo^{-1}\Sigkonetwo^\top\right) = \lambda_{\min}(\Sigktwo)\lambda_{\min}(\Zk),  \label{eq:sig_lb_konetwo}
\end{align}
and take $t$ sufficiently large that 
\begin{align}
\|\covwonetwo^{[>t]}\| \le \frac{1}{2}\sqrt{\lambda_{\min}(\Sigktwo)\lambda_{\min}(\Zk)} \le \frac{1}{2}\sigma_{\min}(\Sigkonetwo).  \label{eq:t_cond_first}
\end{align}
Now invoking \Cref{eq:covdecomp} on the $(1,2)$-block of $\Sigk$, then if \Cref{eq:t_cond_first} holds,
\begin{align*}
\Sigkonetwo = \covwonetwo &=  \covwonetwo^{[t]} +  \covwonetwo^{[>t]}, \quad \text{ so that } 
\sigma_{\min}\left(\covwonetwo^{[t]}\right)  \ge \sigma_{\min}(\Sigkonetwo) - \|\covwonetwo^{[>t]}\| \ge \frac{1}{2}\sigma_{\min}(\Sigkonetwo).
\end{align*}
Next, since $\covw^{[t]} \succeq 0$, the Schur complement test implies that
\begin{align*}
\covwtwo^{[t]} &\succeq \covwonetwo^{[t]\top}\left(\covwone^{[t]}\right)^{-1}\covwonetwo^{[t]}\\
&\succeq \frac{1}{\| \covwone^{[t]}\|} \covwonetwo^{[t]\top}\covwonetwo^{[t]}\\
&\succeq \frac{1}{\| \covwone\|} \covwonetwo^{[t]\top}\covwonetwo^{[t]} = \frac{1}{\| \Sigonesys\|} \covwonetwo^{[t]\top}\covwonetwo^{[t]},
\end{align*}
where above we use $\covwone^{[t]} \preceq \covwone$, and $\covwone = \Sigonesys$ in view of \Cref{eq:covdecomp}. Here, invertibility of $\covwone^{[t]}$ is guaranteed by \Cref{lem:sigone_psd}. Therefore, if  \Cref{eq:t_cond_first}  holds,
\begin{align*}
\lambda_{\min}\left(\covwtwo^{[t]}\right) \ge \frac{1}{\| \Sigonesys\|} \sigma_{\min}(\covwonetwo^{[t]})^2 \ge \frac{1}{4\| \Sigonesys\|}\sigma_{\min}(\Sigkonetwo)^2 \ge  \frac{\lambda_{\min}(\Sigktwo)\lambda_{\min}(\Zk)}{4\| \Sigonesys\| },
\end{align*}
where the last inequality applies \Cref{eq:sig_lb_konetwo}. 
Lastly, we simplify the condition $\|\covwonetwo^{[>t]}\| \le \frac{1}{2}\sqrt{\lambda_{\min}(\Sigktwo)\lambda_{\min}(\Zk)}$ in \Cref{eq:t_cond_first}. We have
\begin{align*}
\|\covwonetwo^{[>t]}\|^2 \overset{(i)}{\le} \| \covwone^{[>t]}\|\cdot \|\covwtwo^{[>t]}\|\overset{(ii)}{\le} \| \covwone^{[>t]}\|\cdot \|\covwtwo\| \overset{(iii)}{\le} \| \Sigonesys^{[>t]}\|\cdot \|\Sigktwo\|, 
\end{align*}
where $(i)$ uses \Cref{lem:PSD_cauchy_schwartz}, 
$(ii)$ uses $\covwtwo^{[>t]} \preceq \covwtwo$, and $(iii)$ uses both that  $\covwtwo \preceq \Sigktwo$ by \Cref{eq:covdecomp} and that $ \covwone^{[>t]} = \Sigonesys^{[>t]}$ by \Cref{lem:useful_idens_sig} part (b). Therefore, the condition $\|\covwonetwo^{[>t]}\| \le \frac{1}{2}\sqrt{\lambda_{\min}(\Sigktwo)\lambda_{\min}(\Zk)} $ is met as soon as
\begin{align*}
\|\Sigonesys^{[>t]}\| \le \frac{\lambda_{\min}(\Sigktwo)\lambda_{\min}(\Zk)}{4\|\Sigktwo\|}. 
\end{align*}
\end{proof}

\subsubsection{Bounding the decay of the true system}
Recall that both \Cref{lem:Sigbound_to_block}  and \Cref{lem:covwtwo_lb_lem} require us to bound the decay of $\|\Sigonesys^{[>t]}\|$. This is achieved in the following lemma.
\begin{lemma}\label{lem:Sigsys_decay} For any $t > 0$, we have
\begin{align*}
  \|\Sigonesys^{[>t]}\| \le  e^{\frac{-t\lambda_{\min}(\bW_1)}{\|\Sigonesys\|}} \cdot\frac{\|\Sigonesys\|^2}{\lambda_{\min}(\bW_1)}.
 \end{align*}
 \end{lemma}
\begin{proof}
Define $\bN(s) = \exp(t\bA)\Sigonesys\exp(t\bA)^\top$.
We compute
\begin{align*}
 \ddt \bN(s) &= \exp(t\bA)\bA \Sigonesys\exp(t\bA)^\top + \exp(t\bA)\Sigonesys\bA^\top\exp(t\bA)^\top\\
 &= \exp(t\bA)\left(\bA \Sigonesys + \Sigonesys\bA^\top\right)\exp(t\bA)^\top \\
 &= \exp(t\bA)\left(-\bW_1\right)\exp(t\bA)^\top \\
 &\preceq \frac{-\lambda_{\min}(\bW_1)}{\|\Sigonesys\|} \cdot \exp(t\bA)\Sigonesys\exp(t\bA)^\top = \frac{-\lambda_{\min}(\bW_1)}{\|\Sigonesys\|} \cdot \bN(s).
\end{align*}
Applying \Cref{lem:psd_ode_comparison},  
\begin{align*}
 \exp(t\bA)\Sigonesys\exp(t\bA)^\top = \bN(t) \preceq \bN(0)\cdot e^{\frac{-t\lambda_{\min}(\bW_1)}{\|\Sigonesys\|}} = \Sigonesys \cdot  e^{\frac{-t\lambda_{\min}(\bW_1)}{\|\Sigonesys\|}}.
\end{align*}
As a consequence,
\begin{align*}
 \Sigonesys^{[>t]} &= \int_{s=t}^{\infty}\exp(s\bA)\Sigonesys\exp(s\bA)^\top\rmd s \preceq  \int_{s=t}^{\infty}\left(\Sigonesys  \cdot e^{\frac{-s\lambda_{\min}(\bW_1)}{\|\Sigonesys\|}} \right)\rmd s\\
 &= \Sigonesys \cdot  \frac{\|\Sigonesys\|}{\lambda_{\min}(\bW_1)}  \cdot e^{\frac{-t\lambda_{\min}(\bW_1)}{\|\Sigonesys\|}}.
\end{align*}
The bound follows by taking the operator norm of both sides.
\end{proof}

\subsubsection{Concluding the argument}
It remains to bound $\Ctk \le \conslyapK$, where $\Ctk :=\frac{t\|\Sigk\|^2}{\lambda_{\min}(\Sigk)\lambda_{\min}(\Sigk^{[t]})}$ was given in \Cref{lem:ctk_bound}. 
Consolidating \Cref{lem:Sigbound_to_block} and \Cref{lem:covwtwo_lb_lem}, we have that if
\begin{align}
\|\Sigonesys^{[<t]}\| \le \min\left\{\frac{1}{2}\lambda_{\min}(\Sigonesys), \frac{\lambda_{\min}(\Sigktwo)\lambda_{\min}(\Zk)}{4\|\Sigonesys\|}\right\}.  \label{eq:min_cond}
\end{align} 
Then, 
\begin{align*}
	\lambda_{\min}\left( \Sigk^{[t]}\right) &\ge \frac{\lambda_{\min}(\bW_2)\lambda_{\min}(\Sigktwo)\lambda_{\min}(\Zk)}{ 8 t\| \Sigonesys\|^2 \|\bC\|^2} \cdot \min\left\{ 1, \frac{\lambda_{\min}(\Sigonesys)}{4\|\Sigktwo\|}\right\}.
\end{align*}
Or by inverting,
\begin{align*}
	\frac{1}{\lambda_{\min}\left( \Sigk^{[t]}\right)} &\le \frac{ 8 t\| \Sigonesys\|^2 \|\bC\|^2}{\lambda_{\min}(\bW_2)\lambda_{\min}(\Sigktwo)\lambda_{\min}(\Zk)} \cdot \max\left\{ 1, \frac{4\|\Sigktwo\|}{\lambda_{\min}(\Sigonesys)}\right\}.
\end{align*}
Hence,  if $t$ satisfies \Cref{eq:min_cond}, then 
\begin{align}
\Ctk &\le  \frac{ 8 t^2 \|\Sigk\|^2\| \Sigonesys\|^2 \|\bC\|^2}{\lambda_{\min}(\Sigk)\lambda_{\min}(\bW_2)\lambda_{\min}(\Sigktwo)\lambda_{\min}(\Zk)} \cdot \max\left\{ 1, \frac{4\|\Sigktwo\|}{\lambda_{\min}(\Sigonesys)}\right\}.  \label{eq:Sigbar_bound}
\end{align}
It remains to select $t$ large enough to satisfy \Cref{eq:min_cond}. From \Cref{lem:Sigsys_decay}, it is enough to take
\begin{align*}
t = t_{\star}(\sfK) = \frac{\|\Sigonesys\|}{\lambda_{\min}(\bW_1)}\log \left(\frac{\|\Sigonesys\|^2}{\lambda_{\min}(\bW_1)} \max\left\{\frac{2}{\lambda_{\min}(\Sigonesys)}, \frac{4\|\Sigonesys\|}{\lambda_{\min}(\Sigonesys)\lambda_{\min}(\Zk)}\right\} \right).
 \end{align*}
Substituing $ t_{\star}(\sfK) $ into \Cref{eq:Sigbar_bound} yields the desired upper bound $\conslyapK$. To see that $\conslyapK$ is at most polynomial in $\left(\|\Sigk\|,\|\Sigk^{-1}\|,\|\Zk^{-1}\|,\|\bW_1^{-1}\|, \|\bW_2^{-1}\|,\|\bC\| \right)$, we observe that $\Ctk$ and $t_{\star}(\sfK)$ are at most polynomial in 
\begin{equation}
\begin{aligned}
&\|\Sigonesys\|, \lambda_{\min}(\bW_1)^{-1},\|\Sigonesys\|,\lambda_{\min}(\Sigonesys)^{-1}, \|\Sigonesys\|, \\
&\lambda_{\min}(\Sigktwo)^{-1},\lambda_{\min}(\Zk)^{-1},\|\Sigk\|,\|\bC\|,\lambda_{\min}(\Sigk)^{-1},\lambda_{\min}(\bW_2)^{-1}.
\end{aligned}
\end{equation}
Noting that $\Sigonesys$ and $\Sigktwo$ are submatrices of $\Sigk$, so that $\lambda_{\min}(\Sigonesys), \lambda_{\min}(\Sigktwo) \ge \lambda_{\min}(\Sigk)$ and $\|\Sigonesys\|,\|\Sigktwo\| \le \|\Sigk\|$, $\Ctk$ and $t_{\star}(\sfK)$ are polynomial in 
\begin{align*}
\|\Sigonesys\|, \lambda_{\min}(\bW_1)^{-1},\lambda_{\min}(\Sigk)^{-1},\lambda_{\min}(\Zk)^{-1},\|\bC\|,\lambda_{\min}(\Sigk)^{-1},\lambda_{\min}(\bW_2)^{-1},
\end{align*}
Replacing $\lambda_{\min}(\cdot)^{-1}$ with $\|(\cdot)^{-1}\|$ verifies the simplification.

\subsection{Supporting technical tools}
We first begin with two linear algebra matrices, both of which pertain to partitioned matrices 
\begin{align}
\bLambda = \begin{bmatrix} \bLambda_{11} & \bLambda_{12} \\ \bLambda_{12}^\top & \bLambda_{22}\end{bmatrix} \in \psd{2n}. \label{eq:partitioned_lambda}
\end{align}
\begin{lemma}\label{lem:PSD_cauchy_schwartz} Let $\bLambda \in \psd{2n}$ be PSD be partioned as in \Cref{eq:partitioned_lambda}. Then  $\|\bLambda_{12}\| \le \sqrt{\|\bLambda_{11}\|\|\bLambda_{22}\|}$.
\end{lemma}
\begin{proof}  Let $\bv = (\bv_1,\bv_2) \in \R^{2n}$, then
\begin{align*}
\bv^\top \bLambda \bv =  \bv_1^\top \bLambda_{11}\bv_1 + \bv_2^\top \bLambda_{22}\bv_2 + 2\bv_1^\top \bLambda_{12}\bv_2 \ge 0.
\end{align*}
By considering the same inequality with the vector $\tilde{\bv} = (\bv_1,-\bv_2)$, we have that for all $\bv = (\bv_1,\bv_2) \in \R^{2n}$,
\begin{align*}
|\bv_1^\top \bLambda_{12}\bv_2|  \le \frac{1}{2}\left(\bv_1^\top \bLambda_{11}\bv_1 + \bv_2^\top \bLambda_{22}\bv_2\right). 
\end{align*} 
By considering scalings $\bv_{\alpha} := (\alpha^{1/2} \bv_1,\alpha^{-1/2}\bv_2) \in \R^{2n}$ for $\alpha > 0$, 
we have 
\begin{align*}
|\bv_1^\top \bLambda_{12}\bv_2|  &\le \frac{1}{2}\inf_{\alpha>0}\left(\alpha \bv_1^\top \bLambda_{11}\bv_1 + \alpha^{-1}\bv_2^\top \bLambda_{22}\bv_2\right)\\
&= \sqrt{\bv_1^\top \bLambda_{11}\bv_1 \cdot \bv_2^\top \bLambda_{22}\bv_2}\\
&\le \sqrt{\|\bLambda_{11}\|\bLambda_{22}\|}\|\bv_1\|\|\bv_2\|,
\end{align*} 
which completes the proof.
\end{proof}

\begin{lemma}\label{lem:PSD_block_lam_min} Let $\bLambda \in \psd{2n}$ be PSD and be partitioned  as in \Cref{eq:partitioned_lambda}. Then any  given $\bLambda_{0} \in \psd{n}$, we have
	\begin{align*}
	\lambda_{\min}\left( \bLambda + \begin{bmatrix}0 & 0 \\ 0 & \bLambda_0 \end{bmatrix}\right) \ge \frac{1}{2}\lambda_{\min}(\bLambda_{0})\min\left\{1, \frac{\lambda_{\min}(\bLambda_{11})}{2\|\bLambda_{22} + \bLambda_{0}\|}\right\}.
	\end{align*}
\end{lemma}
\begin{proof} Without loss of generality, may assume $\bLambda_0,\bLambda_{11} \succ 0$ since otherwise the lemma is vacuous.  For compactness, denote
\begin{align*}
\bar{\bLambda} := \bLambda + \begin{bmatrix}0 & 0 \\ 0 & \bLambda_0 \end{bmatrix}. 
\end{align*} It suffices to exhibit $\lambda$ such that $\lambda_{\min}(\bar{\bLambda}) \ge \lambda$. From the Schur complement test, we have that $\lambda_{\min}(\bar{\bLambda}) \ge \lambda$ as long as $\bar{\bLambda}_{11} \succeq \lambda \eye_n$ and 
\begin{align*}
\bar{\bLambda}_{22} - \lambda \eye_n \succeq \bar{\bLambda}_{12}^\top (\bar{\bLambda}_{11} - \lambda \eye_n)^{-1}\bar{\bLambda}_{12},
\end{align*}
so, substituing in the form of $\bar{\bLambda}$, we have $\bLambda_{11} \succeq \lambda \eye_n$ and 
\begin{align*}
\bLambda_{22} + \bLambda_0 - \lambda \eye_n \succeq \bLambda_{12}^\top (\bLambda_{11} - \lambda \eye_n)^{-1}\bLambda_{12}. 
\end{align*}
If we take $\lambda  = \alpha \lambda_{\min}(\bLambda)$ for some $\alpha < 1$, then
\begin{align*}
\bLambda_{12}^\top (\bLambda_{11} - \lambda \eye_n)^{-1}\bLambda_{12}\preceq (1-\alpha)^{-1}\bLambda_{12}^\top (\bLambda_{11})^{-1}\bLambda_{12}^\top \preceq (1-\alpha)^{-1}\bLambda_{22},
\end{align*}
where the last step applies the Schur complement test to $\bLambda$. Thus, it is enough
\begin{align*}
\bLambda_{22} + \bLambda_0 - \lambda \eye_n \succeq (1-\alpha)^{-1}\bLambda_{22},
\end{align*}
so that, with rearranging and substituing in the definition of $\lambda$, it suffices to choose $\alpha \le 1/2$ and 
\begin{align*}
\bLambda_0  \succeq \frac{\alpha}{1-\alpha}\bLambda_{22} + \alpha \lambda_{\min}(\bLambda_{11}).
\end{align*}
Thus, it is enough that $\alpha \ge 0$ satisfies
\begin{align*}
\lambda_{\min}(\bLambda_0) \succeq \alpha \left(2\|\bLambda_{22}\| + \lambda_{\min}(\bLambda_{11})\right), \quad \alpha \le 1/2.
\end{align*}
Hence, choosing the maximal $\alpha$ which satisfies the above display, 
\begin{align*}
\lambda_{\min}(\bar{\bLambda}) &\ge \alpha \lambda_{\min}(\bLambda_{11}) = \min\left\{\frac{1}{2}\lambda_{\min}(\bLambda_{11}), \frac{\lambda_{\min}(\bLambda_0)\lambda_{\min}(\bLambda_{11})}{\left(2\|\bLambda_{22}\| + \lambda_{\min}(\bLambda_{11})\right)} \right\}\\
&\overset{(i)}{\ge} \min\left\{\frac{1}{2}\lambda_{\min}(\bLambda_{11}), \frac{1}{2}\lambda_{\min}(\bLambda_{0}), \frac{\lambda_{\min}(\bLambda_0)\lambda_{\min}(\bLambda_{11})}{4\|\bLambda_{22}\|}\right\}\\
&\overset{(ii)}{\ge} \min\left\{\frac{1}{2}\lambda_{\min}(\bLambda_{11}), \frac{1}{2}\lambda_{\min}(\bLambda_{0}), \frac{\lambda_{\min}(\bLambda_0)\lambda_{\min}(\bLambda_{11})}{4\|\bLambda_{22}+\bLambda_0\|}\right\}\\
&\overset{(iii)}{=} \min\left\{\frac{1}{2}\lambda_{\min}(\bLambda_{0}), \frac{\lambda_{\min}(\bLambda_0)\lambda_{\min}(\bLambda_{11})}{4\|\bLambda_{22}+\bLambda_0\|}\right\}.
\end{align*}
Here $(i)$ used that  $\frac{a}{b+c} \ge \min\{\frac{a}{2b},\frac{a}{2c}\}$, $(ii)$ that since $\bLambda_{0},\bLambda_{22} \succeq 0$, we can replace $\|\bLambda_{0}  + \bLambda_{22}\| \ge \|\bLambda_{22}\|$, and $(iii)$ that $\frac{\lambda_{\min}(\bLambda_0)}{\|\bLambda_{22}+\bLambda_0\|} \le 1$ (again, for $\bLambda_{22},\bLambda_0 \succeq 0$). The bound follows by factoring.
\end{proof}

\begin{lemma}\label{lem:matrix_rv_Cauchy}  For any continuous matrix valued function $\bX(s) \in \R^{n \times n}$, $\left(\int_{0}^1 \bX(s_1)\rmd s_1\right) \left(\int_{0}^1 \bX(s_1)\rmd s_1\right)^\top \preceq \int_{0}^1 \bX(s)\bX(s)^\top\rmd s$. 
\end{lemma}
\begin{proof} It suffices to show that for any vector $\bv_0 \in \R^n$, the function $\bv(s) = \bX(s)^\top \bv_0$ satisfies
\begin{align*}
\left\|\int_{0}^1 \bv(s) \rmd s \right\|^2 \leq \int_{0}^1 \|\bv(s)\|^2 \rmd s.
\end{align*}
We can view both integrals as expectations over a random vector $\tilde{\bv} = \bv(s)$, where $s$ is drawn uniformly on $[0,1]$. With this interpretation, it suffices that $\|\Exp[\tilde{\bv}]\|^2 \le \Exp[\|\tilde{\bv}\|^2]$, which is precisely Jensen's inequality. 
\end{proof}

\begin{lemma}\label{lem:psd_ode_comparison} Let $\bN(\cdot):[0,\infty) \to \psd{n}$ be continuously differentiable PSD-matrix-valued function satisfying
\begin{align*}
\ddt \bN(t) \preceq - \alpha \bN(s), \text{ for some } \alpha > 0.
\end{align*}
Then, $\bN(t) \preceq e^{-\alpha t} \bN(0)$ for all $t > 0$.
\end{lemma}
\begin{proof} For fixed $\bv \ne 0$, define $f(\cdot;\bv) = \bv^\top \bN(\cdot) \bv$. Then, $f(\cdot;\bv) \ge 0$ and $\ddt f(t;\bv) \le - \alpha f(t;\bv)$ for all $t$. Hence, by a scalar ODE comparison inequality, $\bv^\top \bN(t) \bv = f(t;\bv) \le e^{-\alpha t} f(0;\bv) = e^{-\alpha t} \cdot \bv^\top \bN(\cdot) \bv $. The lemma follows. 
\end{proof}

%% file: appendix/smoothness_max.tex
\section{Smoothness (Proof of \Cref{prop:der_bounds})}\label{app:max_smooth}


 This section bounds the first and second derivatives of $\cL_{\lambda}(\sfK)$, and of $\sfK \mapsto \Sigk$, for $\sfK \in \calKexp$. 
\paragraph{Specification of Derivative Norms.} To prove \Cref{prop:der_bounds}, we formally define the norms of the relevant derivatives. Let $\DelK = (\DelA,\DelB,\DelC)$ denote a perturbation of filter $\sfK = (\Ak,\Bk,\Ck)$, with 
\begin{align*}
\ltwonorm{\DelK} &= \sqrt{\fronorm{\DelA}^2 + \fronorm{\DelB}^2 + \fronorm{\DelC}^2}.
\end{align*}
\begin{definition}[Euclidean Norm of Derivatives]\label{defn:Euclidean_norms} We define Euclidean norms of the gradient $\nabla \cL_{\lambda}(\sfK)$,  operator-norm of the Hessian $\nablatwo \cL_{\lambda}(\sfK)$, and $\ell_2 \to \op$-norm of the gradients of $\Sigk$ as
\begin{align*}
\|\nablatwo \cL_{\lambda}(\sfK)\|_{\ell_2 \to \ell_2} &:= \sup_{\DelK: \|\DelK\|_{\ell_2} = 1} \langle \DelK,  \nablatwo \cL_{\lambda}(\sfK) \cdot \DelK \rangle \\
\|\nabla \cL_{\lambda}(\sfK)\|_{\ell_2} &= \sup_{\DelK: \|\DelK\|_{\ell_2} = 1} \langle \DelK,  \nabla \cL_{\lambda}(\sfK)  \rangle\\
\|\nabla \,\bSigma_{\sfK}\|_{\ell_2 \to \op} &:= \sup_{\DelK: \|\DelK\|_{\ell_2} = 1}\|\nabla \,\bSigma_{\sfK} \cdot \DelK\|_{\op} \\
\end{align*}
\end{definition}
We shall compute these bounds by considering directional derivatives, using that 
\begin{align*}
\|\nablatwo \cL_{\lambda}(\sfK)\|_{\ell_2 \to \ell_2} &:= \sup_{\DelK: \|\DelK\|_{\ell_2} = 1}  \left|\ddtsq \cL_{\lambda}(\sfK + t\DelK)\big{|}_{t=0}\right| \\
\|\nabla \cL_{\lambda}(\sfK)\|_{\ell_2} &= \sup_{\DelK: \|\DelK\|_{\ell_2} = 1}  \left|\ddt \cL_{\lambda}(\sfK + t\DelK)\big{|}_{t=0}\right|\\
\|\nabla \,\bSigma_{\sfK}\|_{\ell_2 \to \op} &:= \sup_{\DelK: \|\DelK\|_{\ell_2} = 1}  \left\|\ddt \bSigma_{(\sfK + t\DelK)}\big{|}_{t=0}\right\|_{\op} \\
\end{align*}

\paragraph{Stability preliminaries.} For any $\sfK \in \calKexp$ (and thus $\sfK \in \calKstab$), $\Aclk$ is Hurwitz stable, and the solution to the Lyapunov equation $\Aclk\bSigma_{\sfK,\bY} + \bSigma_{\sfK,\bY}\Aclk^\top + \bY = 0$ for any $\bY \in \psd{2n}$ can be written as 
\begin{align*}
\int_{0}^\infty \exp(s\Aclk)\bY\exp(s\Aclk)^\top \rmd s,
\end{align*}
which recall from \Cref{prop:clyap_compact_form} that satisfies
\begin{align*}
\circnorm{\bSigma_{\sfK,\bY}} \le \conslyapK \cdot \circnorm{\bY},
\end{align*}
for $\|\cdot\|_{\circ}$ denoting either operator, Frobenius, or nuclear norm.  
The explicit form of $\conslyapK$ is given in \Cref{prop:clyap_formal}.

%

\paragraph{Covariance derivatives. } We start with derivatives of $\Sigk$. Define 
\begin{align*}
\Sigk'[\DelK]:= \ddt \bSigma_{\sfK + t \DelK}\bigggiven_{t=0} , \quad   \Sigk''[\DelK] := \ddtsq \bSigma_{\sfK + t \DelK}\bigggiven_{t=0} .
\end{align*}
We first compute these derivatives. In what follows, given a symmetric matrix $\bY \in \sym{n}$, we define its nuclear norm as $\|\bY\|_{\nuc} := \sum_{i=1}^n |\lambda_i(\bY)|$. 
\begin{lemma}[Bounding derivatives of $\Sigk$]\label{lem:SigKDer} For any $\sfK \in \calKstab$, we have that $\sfK \to \Sigk$ is $\cctwo$ in a neighorhood containing $\sfK$, and $\Sigk'[\DelK]$ and $\Sigk''[\DelK] $ solve the Lyapunov equations 
\begin{align*}
 \Aclk \Sigk'[\DelK] + \Sigk'[\DelK]\Aclk^\top +  \bY_1[\DelK]  = 0, \quad \Aclk \Sigk''[\DelK] + \Sigk''[\DelK]\Aclk^\top +  \bY_2[\DelK] = 0,
\end{align*}
where 
\begin{align*}
\bY_1[\DelK] &= \begin{bmatrix} 0 & 0 \\ 
\DelB\bC  & \DelA \end{bmatrix} \Sigk +   \Sigk \begin{bmatrix} 0 & 0 \\ 
\DelB\bC & \DelA \end{bmatrix}^\top  + \begin{bmatrix} 0 & 0 \\ 
0 & \DelB \bW_2 \Bk^\top + \Bk \bW_2 \DelB^\top 
\end{bmatrix}\\
\bY_2[\DelK]  &= \begin{bmatrix} 0 & 0 \\ 
\DelB\bC & \DelA \end{bmatrix} \Sigk'[\DelK] +   \Sigk'[\DelK] \begin{bmatrix} 0 & 0 \\ 
\DelB\bC & \DelA \end{bmatrix}^\top  + \begin{bmatrix} 0 & 0 \\ 
0 & \DelB \bW_2 \DelB^\top 
\end{bmatrix}.
\end{align*}
Hence,
\begin{align*}
\|\Sigk'[\DelK]\|_{\fro} &\le \conslyapK \cdot \poly(\|\Sigk\|,\|\Bk\|, \|\bC\|,\|\bW_2\|) \cdot \|\DelK\|_{\ell_2}\\
\|\Sigk''[\DelK]\|_{\nuc} &\le \conslyapK^2 \cdot \poly(\|\Sigk\|,\|\Bk\|, \|\bC\|,\|\bW_2\|) \cdot \|\DelK\|_{\ell_2}^2.
\end{align*}
\end{lemma}
\begin{proof} The existence of the derivatives $\Sigk'[\DelK]$ and $\Sigk''[\DelK]$ in open neighbrhoods is standard (see , e.g. \cite[Lemma B.1]{tang2021analysis}). We compute the derivatives by implicit differentiation. 
\begin{align*}
\bSigma_{\sfK + t \DelK} =  \Aclkdel \bSigma_{\sfK + t \DelK} +  \bSigma_{\sfK + t \DelK} \Aclkdel^\top  + \begin{bmatrix} \bW_1 & 0 \\ 
0 & (\Bk + t\DelB) \bW_2 (\Bk + t\DelB)^\top\end{bmatrix}.
\end{align*}
Differentiating both sides with respect to $t$ and evaluating at $t = 0$, 
we have 
\begin{align*}
\Sigk'[\DelK] &= \Aclk \Sigk'[\DelK] +  \Sigk'[\DelK]\Aclk^\top\\
&\quad +  \underbrace{\begin{bmatrix} 0 & 0 \\ 
\DelB\bC & \DelA \end{bmatrix} \Sigk +   \Sigk \begin{bmatrix} 0 & 0 \\ 
\DelB\bC & \DelA \end{bmatrix}^\top  + \begin{bmatrix} 0 & 0 \\ 
0 & \DelB \bW_2 \Bk^\top + \Bk \bW_2 \DelB^\top 
\end{bmatrix}}_{:= \bY_1[\DelK]}. 
\end{align*}
 Differentiating twice (and notice that $\ddtsq \Aclkdel \big{|}_{t=0} = 0$), and evaluating at $t=0$, we have
\begin{align*}
\Sigk''[\DelK]  &= \Aclk \Sigk''[\DelK] +  \Sigk''[\DelK]\Aclk^\top\\
&\qquad +\underbrace{\begin{bmatrix} 0 & 0 \\ 
\DelB\bC  & \DelA \end{bmatrix} \Sigk'[\DelK] +   \Sigk'[\DelK] \begin{bmatrix} 0 & 0 \\ 
\DelB\bC & \DelA \end{bmatrix}^\top  + \begin{bmatrix} 0 & 0 \\ 
0 & \DelB \bW_2 \DelB^\top 
\end{bmatrix}}_{:= \bY_{2}[\DelK]}. 
\end{align*}

To prove the second part of the lemma, we use \Cref{prop:clyap_formal}. Since $\sfK \in \calKexp$ and thus $\sfK \in \calKstab$, we know the solutions to the Lyapunov equations  above, i.e., $\Sigk'[\DelK]$ and $\Sigk''[\DelK]$, can be written as $\circnorm{\clyap(\Aclk,\bY_1[\DelK])}$ and $\circnorm{\clyap(\Aclk,\bY_2[\DelK])}$, and can be bounded by  
\begin{align*}
\|\Sigk'[\DelK]\|_{\fro} &\le \conslyapK \fronorm{\bY_1[\DelK]} \\
 &\le \conslyapK \cdot \left(2\|\Sigk\| \left\|\begin{bmatrix} 0 & 0 \\ 
\bC \DelB & \DelA \end{bmatrix}\right\|_{\fro} + 2 \|\DelB\|_{\fro}\|\Bk\|\|\bW_2\|\right) \\
&\le \conslyapK \cdot \poly(\|\Sigk\|,\|\bC\|,\|\Bk\|,\|\bW_2\|)\cdot \ltwonorm{\DelK}. 
\end{align*}
Using the above computation and recall that $\conslyapK \ge 1$,  
\begin{align*}
\|\Sigk''[\DelK]\|_{\nuc} &\le \conslyapK \cdot \fronorm{\bY_2[\DelK]} \\
 &\le \conslyapK \cdot \left(2\|\Sigk'[\DelK]\|_{\fro} \left\|\begin{bmatrix} 0 & 0 \\ 
\bC \DelB & \DelA \end{bmatrix}\right\|_{\fro} + 2 \|\DelB\|_{\fro}^2\|\bW_2\|\right) \\
&\le \conslyapK (1+\|\bC\|)\|\Sigk'[\DelK]\|_{\fro} \ltwonorm{\DelK} + \conslyapK \|\bW_2\|\ltwonorm{\DelK}^2\\
&\le \conslyapK^2 \poly(\|\Sigk\|,\|\bC\|,\|\Bk\|,\|\bW_2\|)\cdot \ltwonorm{\DelK}^2 + \conslyapK  \|\bW_2\|\ltwonorm{\DelK}^2\\
&\le \conslyapK^2 \cdot \poly(\|\Sigk\|,\|\bC\|,\|\Bk\|,\|\bW_2\|)\cdot \ltwonorm{\DelK}^2, 
\end{align*}
which completes the proof. 
\end{proof}

\paragraph{Derivatives of \OE{} loss and regularizer.} Next, we compute the derivatives of $\Loe(\sfK)$ and $\trace[\Zk^{-1}]$ in terms of the above derivatives.
\begin{lemma}[Bounding derivatives of $\Loe(\sfK)$]\label{lem:loe_bounds} We have that $\Loe(\cdot)$ is $\cctwo$ in the neighborhood of any $\sfK \in \calKstab$ and 
\begin{align*}
\left| \ddt \Loe(\sfK + t\DelK)\Biggiven_{t=0}\right| &\le 
\sqrt{n}\cdot \conslyapK \cdot \poly(\|\sO\|,\|\Ck\|, \|\Sigk\|,\|\Bk\|, \|\bC\|,\|\bW_2\|) \cdot \|\DelK\|_{\ell_2} \\
  \left|\ddtsq \Loe(\sfK + t\DelK)\Biggiven_{t=0}\right|  &\le  \conslyapK^2 \cdot \poly(\|\sO\|,\|\Ck\|, \|\Sigk\|,\|\Bk\|, \|\bC\|,\|\bW_2\|) \cdot \|\DelK\|_{\ell_2}^2. 
 \end{align*}
\end{lemma}
\begin{proof} Recall from the computaton in \Cref{eq:llam_simplify} that
\begin{align*}
 \Loe(\sfK) = \trace\left[\begin{bmatrix}
	\sO & -\Ck
\end{bmatrix}  \Sigk \begin{bmatrix}
	\sO^\top \\ -\Ck^\top
\end{bmatrix}\right] = \trace\left[\begin{bmatrix}
	\sO^\top \sO & -\sO^\top \Ck\\
-\Ck^\top \sO & \Ck^\top \Ck
\end{bmatrix} \cdot \Sigk \right].
\end{align*}
Since \Cref{lem:SigKDer} verifies $\sfK \mapsto \Sigk$ is $\cctwo$ in an open neighborhood around $\sfK$, we readily see $\Loe(\sfK)$ is as well. Thus, 
\begin{align*}
 \ddt \Loe(\sfK + t\DelK)\Biggiven_{t=0} = \trace\left[\begin{bmatrix}
	\sO^\top \sO & -\sO^\top \Ck\\
		-\Ck^\top \sO & \Ck^\top \Ck
\end{bmatrix} \cdot \Sigk'[\DelK] +  \begin{bmatrix}
	0 & -\sO^\top \DelC\\
		-\DelC^\top \sO & \DelC^\top \Ck  + \Ck^\top \DelC
\end{bmatrix} \cdot \Sigk \right].
\end{align*}
Thus, 
\begin{align*}
\left| \ddt \Loe(\sfK + t\DelK)\Biggiven_{t=0}\right| &\le \poly(\|\sO\|,\|\Ck\|)\nucnorm{\Sigk'[\DelK]} + \poly(\|\sO\|,\|\Ck\|)\ltwonorm{\DelK} \fronorm{\Sigk[\DelK]}\\
 &\le \sqrt{n}\poly(\|\sO\|,\|\Ck\|)\left(\fronorm{\Sigk'[\DelK]} + \ltwonorm{\DelK} \|\Sigk[\DelK]\|\right)\\
 &\overset{(i)}{\le} \sqrt{n}\cdot \poly(\|\sO\|,\|\Ck\|)\Big(\conslyapK \cdot \poly(\|\Sigk\|,\|\Bk\|, \|\bC\|,\|\bW_2\|) \\
 &\qquad\cdot (\|\DelK\|_{\ell_2} + \ltwonorm{\DelK} \|\Sigk[\DelK]\|)\Big)\\
  &\overset{(ii)}{\le} \sqrt{n}\cdot \conslyapK \cdot \poly(\|\sO\|,\|\Ck\|, \|\Sigk\|,\|\Bk\|, \|\bC\|,\|\bW_2\|) \cdot \|\DelK\|_{\ell_2} 
\end{align*}
where $(i)$ uses \Cref{lem:SigKDer}, and $(ii)$ uses $\conslyapK \ge 1$. Next, 
\begin{align*}
 \ddtsq \Loe(\sfK + t\DelK)\Biggiven_{t=0} &= \trace\Big{[}\begin{bmatrix}
	\sO^\top \sO & -\sO^\top \Ck\\
	-\Ck^\top \sO & \Ck^\top \Ck
\end{bmatrix} \cdot \Sigk''[\DelK]  \\
&\qquad+ \begin{bmatrix}
	0 & -\sO^\top \DelC\\
-\DelC^\top \sO & \DelC^\top \Ck + \Ck^\top \DelC
\end{bmatrix} \cdot \Sigk'[\DelK] + \begin{bmatrix}
	0 & 0\\
	0 & \DelC^\top  \DelC
\end{bmatrix} \cdot \Sigk \Big{]}.
\end{align*}
Using Matrix Holder's inequality, it follows that 
\begin{align*}
 \left|\ddtsq \Loe(\sfK + t\DelK)\Biggiven_{t=0}\right| &\le \left(\poly(\|\sO\|,\|\Ck\|)\nucnorm{\Sigk''[\DelK]} +  \poly(\|\sO\|,\|\Ck\|)\ltwonorm{\DelK}\fronorm{\Sigk'[\DelK]} + \|\Sigk\|\|\DelC\|_{\fro}^2  \right).
\end{align*}
Again, invoking \Cref{lem:SigKDer} and appropriate simplifications, we have 
\begin{align*}
 \left|\ddtsq \Loe(\sfK + t\DelK)\Biggiven_{t=0}\right|  \le  \conslyapK^2 \cdot \poly(\|\sO\|,\|\Ck\|, \|\Sigk\|,\|\Bk\|, \|\bC\|,\|\bW_2\|) \cdot \|\DelK\|_{\ell_2}^2. 
 \end{align*}
\end{proof}
Next, we turn to controlling the derivatives of the regularizer. Here, we require that $\sfK \in \calKexp$, not just $\sfK \in \calKstab$ as above. Introduce $\Zk'[\DelK] = \ddt \bZ_{\Bk + t\DelK}\Biggiven_{t=0}$, 
and define $\Zk''[\DelK]$ analogously. 
\begin{lemma}[Bounding  derivatives of $\Zk$]\label{lem:ZK_der_comp} $\Zk$ is $\cctwo$ in a neighborhood of any $\sfK \in \calKexp$, and 
\begin{align*}
\fronorm{\Zk'[\DelK]} &\le \conslyapK \cdot \poly(\|\Sigktwo^{-1}\|,\|\Sigk\|,\|\Bk\|, \|\bC\|,\|\bW_2\|) \cdot \|\DelK\|_{\ell_2} \\
\nucnorm{\Zk''[\DelK]} &\le \conslyapK^2 \cdot \poly(\|\Sigktwo^{-1}\|,\|\Sigk\|,\|\Bk\|, \|\bC\|,\|\bW_2\|) \cdot \|\DelK\|_{\ell_2}^2.
\end{align*}
\end{lemma}
\begin{proof}
Using $(\Sigkonetwo \Sigktwo^{-1} \Sigkonetwo^\top)$ and the facts that (a) $\sfK \mapsto \Sigk$ is $\cctwo$ on some neighborhood, and $\bX \mapsto \bX^{-1}$ is $\cctwo$ on $\pd{n}$, we see $\Zk$ is $\cctwo$.

To compute derivatives, let us partition the derivatives $\Sigk'[\DelK]$ and $\Sigk''[\DelK]$ into two-by-two blocks $\bSigma_{ij,\sfK}'[\DelK]$ and  $\bSigma_{ij,\sfK}''[\DelK]$  in the obvious way. We have
\begin{align*}
\Zk'[\DelK] = \ddt \bZ_{\Bk + t\DelK}\big{|}_{t=0} &= \ddt (\Sigkonetwodel \Sigktwodel^{-1} \Sigkonetwodel^\top)\big{|}_{t=0}\\
&= \Sigkonetwo'[\DelK] \Sigktwo^{-1} \Sigkonetwo + \Sigkonetwo \Sigktwodel^{-1} (\Sigkonetwo '[\DelK])^\top \\
&\qquad +  \Sigkonetwo \Sigktwo^{-1} \Sigktwo'[\DelK] \Sigktwodel^{-1}  \Sigkonetwo^\top. 
\end{align*}
Thus, 
\begin{align*}
\fronorm{\Zk'[\DelK]} &\le \poly(\|\Sigkonetwo\|,\|\Sigktwo^{-1}\| )(\|\Sigkonetwo '[\DelK]\|_{\fro} + \|\Sigktwo'[\DelK]\|_{\fro})\\
&\le \poly(\|\Sigkonetwo\|,\|\Sigktwo^{-1}\|)\|\Sigk'[\DelK]\|_{\fro}\\
&\le \poly(\|\Sigk\|,\|\Sigktwo^{-1}\|)\|\Sigk'[\DelK]\|_{\fro}. 
\end{align*}
Thus, the intended bound on $\fronorm{\Zk'[\DelK]}$ follows from \Cref{lem:SigKDer}. By the same token, more tedious computations reveal,
\begin{align*}
&\nucnorm{\Zk''[\DelK]}  = \nucnorm{\ddtsq (\Sigkonetwodel \Sigktwodel^{-1} \Sigkonetwodel^\top)\big{|}_{t=0}}\\
&\quad =\poly(\|\Sigkonetwo\|,\|\Sigktwo^{-1}\|) \left(\|\Sigkonetwo'[\DelK]\|_{\fro}^2 + \|\Sigkonetwo'[\DelK]\|_{\fro}\|\Sigktwo'[\DelK]\|_{\fro}   + \|\Sigkonetwo''[\DelK]\|_{\nuc} +  \|\Sigktwo''[\DelK]\|_{\nuc}\right)\\
&\quad \le \poly(\|\Sigk\|,\|\Sigktwo^{-1}\|) \left(\|\Sigk'[\DelK]\|_{\fro}^2 + \|\Sigk''[\DelK]\|_{\nuc}\right).
\end{align*}
Thus, the intended bound on $\nucnorm{\Zk''[\DelK]}$ follows from \Cref{lem:SigKDer}.
\end{proof}
\begin{lemma}[Bounding derivatives of $\regexp(\sfK)$]\label{lem:regexp_comp} Recall $\regexp(\sfK) := \trace[\Zk^{-1}]$. We have
\begin{align*}
\left|\ddt \regexp(\sfK + t\DelK)\Big{|}_{t=0}\right| &\le  \sqrt{n}\conslyapK \cdot \poly(\|\Zk^{-1}\|, \|\Sigktwo^{-1}\|,\|\Sigk\|,\|\Bk\|, \|\bC\|,\|\bW_2\|) \cdot \|\DelK\|_{\ell_2}\\
\left|\ddtsq \regexp(\sfK +  t\DelK)\Big{|}_{t=0}\right| &\le \conslyapK^2 \cdot \poly(\|\Zk^{-1}\|, \|\Sigktwo^{-1}\|,\|\Sigk\|,\|\Bk\|, \|\bC\|,\|\bW_2\|) \cdot \|\DelK\|_{\ell_2}^2. 
\end{align*}
\end{lemma}
\begin{proof} We compute
\begin{align*}
\ddt \regexp(\sfK + t\DelK)\Big{|}_{t=0} = \ddt \trace[\bZ_{\sfK + t\DelK}^{-1}]\Big{|}_{t=0} =  -\trace[\Zk^{-1}\Zk'[\DelK]\Zk^{-1}].
\end{align*}
Thus, invoking \Cref{lem:ZK_der_comp},
\begin{align*}
\left|\ddt \regexp(\sfK + t\DelK)\big{|}_{t=0}\right| &\le \|\Zk^{-1}\|^2 \nucnorm{\Zk'[\DelK]} \le \sqrt{n}\|\Zk^{-1}\|^2 \fronorm{\Zk'[\DelK]}\\
&\le \sqrt{n}\conslyapK \cdot \poly(\|\Zk^{-1}\|, \|\Sigktwo^{-1}\|,\|\Sigk\|,\|\Bk\|, \|\bC\|,\|\bW_2\|) \cdot \|\DelK\|_{\ell_2}.
\end{align*}
Next, 
\begin{align*}
\ddtsq \regexp(\sfK + t\DelK)\big{|}_{t=0} =  \trace[2\Zk^{-1}\Zk'[\DelK]\Zk^{-1}\Zk'[\DelK]\Zk^{-1} + \Zk^{-1}\Zk''[\DelK]\Zk^{-1}],
\end{align*}
so again applying \Cref{lem:ZK_der_comp},
\begin{align*}
\left|\ddtsq \regexp(\sfK + t\DelK)\big{|}_{t=0}\right| &\le 2\|\Zk^{-1}\|^3 \fronorm{\Zk'[\DelK]}^2 + \|\Zk^{-1}\|^2 \nucnorm{\Zk''[\DelK]}\\
&\le \conslyapK^2 \cdot \poly(\|\Zk^{-1}\|, \|\Sigktwo^{-1}\|,\|\Sigk\|,\|\Bk\|, \|\bC\|,\|\bW_2\|) \cdot \|\DelK\|_{\ell_2}^2, 
\end{align*}
we complete the proof. 
\end{proof}

\paragraph{Concluding the proof}
We now turn to the proof of \Cref{prop:der_bounds}.
\begin{proof}[Proof of \Cref{prop:der_bounds}]
Combining \Cref{lem:loe_bounds,lem:regexp_comp}, 
\begin{align*}
\left| \ddt \cL_{\lambda}(\sfK + t\DelK)\Biggiven_{t=0}\right| &\le \left| \ddt \Loe(\sfK + t\DelK)\Biggiven_{t=0}\right| + \lambda\cdot\left| \ddt \regexp(\sfK + t\DelK)\Biggiven_{t=0}\right| \\
 &\quad\le (1+\lambda)\sqrt{n}\cdot \conslyapK \cdot \poly(\|\sO\|,\|\Ck\|, \|\Sigk\|,\|\Bk\|, \|\bC\|,\|\bW_2\|,\|\Zk^{-1}\|) \cdot \|\DelK\|_{\ell_2} 
 \end{align*}
 and
 \begin{align*}
\left| \ddtsq \cL_{\lambda}(\sfK + t\DelK)\Biggiven_{t=0}\right| &\le \left| \ddtsq \Loe(\sfK + t\DelK)\Biggiven_{t=0}\right| + \lambda\cdot\left| \ddtsq \regexp(\sfK + t\DelK)\Biggiven_{t=0}\right| \\
 &\quad\le  (1+\lambda)\conslyapK^2 \cdot \poly(\|\sO\|,\|\Ck\|, \|\Sigk\|,\|\Bk\|, \|\bC\|,\|\bW_2\|,\|\Zk^{-1}\|) \cdot \|\DelK\|_{\ell_2}^2. 
 \end{align*}
 These verify the first two bounds of the proposition. The derivative bound for $\Sigk$ is proven in \Cref{lem:SigKDer},noting that 
 \begin{align*}
 \sup_{\DelK:\|\DelK\|_{\ell_2} =1}\|\Sigk'[\DelK]\|_{\fro} \ge \sup_{\DelK:\|\DelK\|_{\ell_2} =1}\|\Sigk'[\DelK]\|_{\op} \ge \sup_{\DelK:\|\DelK\|_{\ell_2} =1}\|\Sigktwo'[\DelK]\|_{\op}  = \|\nabla\, \Sigktwo\|_{\ell_2 \to \op}.
 \end{align*}

 Lastly, we have shown above that $\Loe(\sfK)$ and $\Zk$ is $\cctwo$ in a neighborhood of any $\sfK \in \calKexp$.  Since $\Zk$ is invertible on $\sfK \in \calKexp$, this implies that $\cL_{\lambda} = \Loe(\sfK) + \lambda \trace[\Zk^{-1}]$ is $\cctwo$ in a neighborhood of any $\sfK \in \calKexp$.
 \end{proof}

%% file: main.bbl
\begin{thebibliography}{45}
\providecommand{\natexlab}[1]{#1}
\providecommand{\url}[1]{\texttt{#1}}
\expandafter\ifx\csname urlstyle\endcsname\relax
  \providecommand{\doi}[1]{doi: #1}\else
  \providecommand{\doi}{doi: \begingroup \urlstyle{rm}\Url}\fi

\bibitem[Agarwal et~al.(2021)Agarwal, Kakade, Lee, and
  Mahajan]{agarwal2021theory}
Alekh Agarwal, Sham~M Kakade, Jason~D Lee, and Gaurav Mahajan.
\newblock On the theory of policy gradient methods: Optimality, approximation,
  and distribution shift.
\newblock \emph{Journal of Machine Learning Research}, 22\penalty0
  (98):\penalty0 1--76, 2021.

\bibitem[Agarwal et~al.(2017)Agarwal, Allen-Zhu, Bullins, Hazan, and
  Ma]{agarwal2017finding}
Naman Agarwal, Zeyuan Allen-Zhu, Brian Bullins, Elad Hazan, and Tengyu Ma.
\newblock Finding approximate local minima faster than gradient descent.
\newblock In \emph{Proceedings of the 49th Annual ACM SIGACT Symposium on
  Theory of Computing}, pages 1195--1199, 2017.

\bibitem[Andrychowicz et~al.(2020)Andrychowicz, Baker, Chociej, Józefowicz,
  McGrew, Pachocki, Petron, Plappert, Powell, Ray, Schneider, Sidor, Tobin,
  Welinder, Weng, and Zaremba]{openai2020manipulation}
OpenAI:~Marcin Andrychowicz, Bowen Baker, Maciek Chociej, Rafal Józefowicz,
  Bob McGrew, Jakub Pachocki, Arthur Petron, Matthias Plappert, Glenn Powell,
  Alex Ray, Jonas Schneider, Szymon Sidor, Josh Tobin, Peter Welinder, Lilian
  Weng, and Wojciech Zaremba.
\newblock Learning dexterous in-hand manipulation.
\newblock \emph{The International Journal of Robotics Research}, 39\penalty0
  (1):\penalty0 3--20, 2020.

\bibitem[Athans(1974)]{athans1974importance}
Michael Athans.
\newblock The importance of {K}alman filtering methods for economic systems.
\newblock In \emph{Annals of Economic and Social Measurement, Volume 3, number
  1}, pages 49--64. NBER, 1974.

\bibitem[Bernussou et~al.(1989)Bernussou, Peres, and
  Geromel]{bernussou1989robust}
J~Bernussou, PLD Peres, and JC~Geromel.
\newblock Robust decentralized regulation: a linear programming approach.
\newblock \emph{IFAC Proceedings Volumes}, 22\penalty0 (10):\penalty0 133--136,
  1989.

\bibitem[Bhandari and Russo(2019)]{bhandari2019global}
Jalaj Bhandari and Daniel Russo.
\newblock Global optimality guarantees for policy gradient methods.
\newblock \emph{arXiv preprint arXiv:1906.01786}, 2019.

\bibitem[Boyd et~al.(2004)Boyd, Boyd, and Vandenberghe]{boyd2004convex}
Stephen Boyd, Stephen~P Boyd, and Lieven Vandenberghe.
\newblock \emph{Convex Optimization}.
\newblock Cambridge University Press, 2004.

\bibitem[Brockett(1976)]{brockett76geometry}
Roger Brockett.
\newblock Some geometric questions in the theory of linear systems.
\newblock \emph{IEEE Transactions on Automatic Control}, 21\penalty0
  (4):\penalty0 449--455, 1976.
\newblock \doi{10.1109/TAC.1976.1101301}.

\bibitem[Bu et~al.(2019)Bu, Mesbahi, Fazel, and Mesbahi]{bu2019lqr}
Jingjing Bu, Afshin Mesbahi, Maryam Fazel, and Mehran Mesbahi.
\newblock {LQR} through the lens of first order methods: Discrete-time case.
\newblock \emph{arXiv preprint arXiv:1907.08921}, 2019.

\bibitem[Bubeck(2014)]{bubeck2014convex}
S{\'e}bastien Bubeck.
\newblock Convex optimization: Algorithms and complexity.
\newblock \emph{arXiv preprint arXiv:1405.4980}, 2014.

\bibitem[Carmon et~al.(2018)Carmon, Duchi, Hinder, and
  Sidford]{carmon2018accelerated}
Yair Carmon, John~C Duchi, Oliver Hinder, and Aaron Sidford.
\newblock Accelerated methods for nonconvex optimization.
\newblock \emph{SIAM Journal on Optimization}, 28\penalty0 (2):\penalty0
  1751--1772, 2018.

\bibitem[Doyle et~al.(1989)Doyle, Glover, Khargonekar, and Francis]{dgkf89}
J.C. Doyle, K.~Glover, P.P. Khargonekar, and B.A. Francis.
\newblock State-space solutions to standard $h_2$ and $h_{\infty}$ control
  problems.
\newblock \emph{IEEE Transactions on Automatic Control}, 34\penalty0
  (8):\penalty0 831--847, 1989.
\newblock \doi{10.1109/9.29425}.

\bibitem[Duan and Patton(1998)]{duan1998note}
Guang-Ren Duan and Ron~J Patton.
\newblock A note on hurwitz stability of matrices.
\newblock \emph{Automatica}, 34\penalty0 (4):\penalty0 509--511, 1998.

\bibitem[Fatkhullin and Polyak(2021)]{fatkhullin2021optimizing}
Ilyas Fatkhullin and Boris Polyak.
\newblock Optimizing static linear feedback: Gradient method.
\newblock \emph{SIAM Journal on Control and Optimization}, 59\penalty0
  (5):\penalty0 3887--3911, 2021.

\bibitem[Fazel et~al.(2018)Fazel, Ge, Kakade, and Mesbahi]{fazel18lqr}
Maryam Fazel, Rong Ge, Sham Kakade, and Mehran Mesbahi.
\newblock Global convergence of policy gradient methods for the linear
  quadratic regulator.
\newblock In Jennifer Dy and Andreas Krause, editors, \emph{Proceedings of the
  35th International Conference on Machine Learning}, volume~80 of
  \emph{Proceedings of Machine Learning Research}, pages 1467--1476. PMLR,
  10--15 Jul 2018.

\bibitem[Flaxman et~al.(2005)Flaxman, Kalai, and McMahan]{flaxman2005}
Abraham~D. Flaxman, Adam~Tauman Kalai, and H.~Brendan McMahan.
\newblock Online convex optimization in the bandit setting: gradient descent
  without a gradient.
\newblock In \emph{Proceedings of the Sixteenth Annual ACM-SIAM Symposium on
  Discrete Algorithms}, SODA '05, page 385–394, USA, 2005. Society for
  Industrial and Applied Mathematics.

\bibitem[Furieri et~al.(2020)Furieri, Zheng, and
  Kamgarpour]{furieri2020learning}
Luca Furieri, Yang Zheng, and Maryam Kamgarpour.
\newblock Learning the globally optimal distributed {LQ} regulator.
\newblock In \emph{Learning for Dynamics and Control}, pages 287--297. PMLR,
  2020.

\bibitem[Gautier and Poignet(2001)]{gautier2001extended}
Maxime Gautier and Ph~Poignet.
\newblock Extended kalman filtering and weighted least squares dynamic
  identification of robot.
\newblock \emph{Control Engineering Practice}, 9\penalty0 (12):\penalty0
  1361--1372, 2001.

\bibitem[Hardt et~al.(2016)Hardt, Recht, and Singer]{hardt2016train}
Moritz Hardt, Ben Recht, and Yoram Singer.
\newblock Train faster, generalize better: Stability of stochastic gradient
  descent.
\newblock In \emph{International Conference on Machine Learning}, pages
  1225--1234. PMLR, 2016.

\bibitem[Iserles(2008)]{iserles_2008}
Arieh Iserles.
\newblock \emph{Ordinary differential equations}, page 1–2.
\newblock Cambridge Texts in Applied Mathematics. Cambridge University Press, 2
  edition, 2008.

\bibitem[Jin et~al.(2017)Jin, Ge, Netrapalli, Kakade, and
  Jordan]{jin2017escape}
Chi Jin, Rong Ge, Praneeth Netrapalli, Sham~M. Kakade, and Michael~I. Jordan.
\newblock How to escape saddle points efficiently.
\newblock In Doina Precup and Yee~Whye Teh, editors, \emph{Proceedings of the
  34th International Conference on Machine Learning}, volume~70 of
  \emph{Proceedings of Machine Learning Research}, pages 1724--1732. PMLR,
  06--11 Aug 2017.

\bibitem[Jin et~al.(2018)Jin, Netrapalli, and Jordan]{jin2018accelerated}
Chi Jin, Praneeth Netrapalli, and Michael~I. Jordan.
\newblock Accelerated gradient descent escapes saddle points faster than
  gradient descent.
\newblock In Sébastien Bubeck, Vianney Perchet, and Philippe Rigollet,
  editors, \emph{Proceedings of the 31st Conference On Learning Theory},
  volume~75 of \emph{Proceedings of Machine Learning Research}, pages
  1042--1085. PMLR, 06--09 Jul 2018.

\bibitem[Kalman(1960)]{kalman60new}
R.~E. Kalman.
\newblock {A New Approach to Linear Filtering and Prediction Problems}.
\newblock \emph{Journal of Basic Engineering}, 82\penalty0 (1):\penalty0
  35--45, 03 1960.
\newblock ISSN 0021-9223.

\bibitem[Ku{\v{c}}era(1975)]{kuvcera1975stability}
Vladim{\'\i}r Ku{\v{c}}era.
\newblock Stability of discrete linear feedback systems.
\newblock \emph{IFAC Proceedings Volumes}, 8\penalty0 (1):\penalty0 573--578,
  1975.

\bibitem[Lemar{\'e}chal(2012)]{lemarechal2012cauchy}
Claude Lemar{\'e}chal.
\newblock Cauchy and the gradient method.
\newblock \emph{Doc Math Extra}, 251\penalty0 (254):\penalty0 10, 2012.

\bibitem[Li et~al.(2021)Li, Tang, Zhang, and Li]{li2021distributed}
Yingying Li, Yujie Tang, Runyu Zhang, and Na~Li.
\newblock Distributed reinforcement learning for decentralized linear quadratic
  control: {A} derivative-free policy optimization approach.
\newblock \emph{IEEE Transactions on Automatic Control}, 2021.

\bibitem[Lillacci and Khammash(2010)]{lillacci2010parameter}
Gabriele Lillacci and Mustafa Khammash.
\newblock Parameter estimation and model selection in computational biology.
\newblock \emph{PLoS computational biology}, 6\penalty0 (3):\penalty0 e1000696,
  2010.

\bibitem[Malik et~al.(2019)Malik, Pananjady, Bhatia, Khamaru, Bartlett, and
  Wainwright]{malik2019derivative}
Dhruv Malik, Ashwin Pananjady, Kush Bhatia, Koulik Khamaru, Peter Bartlett, and
  Martin Wainwright.
\newblock Derivative-free methods for policy optimization: Guarantees for
  linear quadratic systems.
\newblock In \emph{The 22nd International Conference on Artificial Intelligence
  and Statistics}, pages 2916--2925. PMLR, 2019.

\bibitem[Masubuchi et~al.(1998)Masubuchi, Ohara, and Suda]{masubuchi1998lmi}
Izumi Masubuchi, Atsumi Ohara, and Nobuhide Suda.
\newblock {LMI}-based controller synthesis: A unified formulation and solution.
\newblock \emph{International Journal of Robust and Nonlinear Control},
  8\penalty0 (8):\penalty0 669--686, 1998.
\newblock
  \doi{https://doi.org/10.1002/(SICI)1099-1239(19980715)8:8<669::AID-RNC337>3.0.CO;2-W}.

\bibitem[Mohammadi et~al.(2021)Mohammadi, Zare, Soltanolkotabi, and
  Jovanovic]{mohammadi2021convergence}
Hesameddin Mohammadi, Armin Zare, Mahdi Soltanolkotabi, and Mihailo~R
  Jovanovic.
\newblock Convergence and sample complexity of gradient methods for the
  model-free linear quadratic regulator problem.
\newblock \emph{IEEE Transactions on Automatic Control}, 2021.

\bibitem[Scaman and Malherbe(2020)]{scaman2020robustness}
Kevin Scaman and Cedric Malherbe.
\newblock Robustness analysis of non-convex stochastic gradient descent using
  biased expectations.
\newblock \emph{Advances in Neural Information Processing Systems}, 33, 2020.

\bibitem[Scherer et~al.(1997)Scherer, Gahinet, and Chilali]{scherer97lmi}
C.~Scherer, P.~Gahinet, and M.~Chilali.
\newblock Multiobjective output-feedback control via {LMI} optimization.
\newblock \emph{IEEE Transactions on Automatic Control}, 42\penalty0
  (7):\penalty0 896--911, 1997.
\newblock \doi{10.1109/9.599969}.

\bibitem[Scherer(1995)]{scherer1995mixed}
Carsten Scherer.
\newblock Mixed {H}2 /{H}$\infty$ control.
\newblock \emph{Trends in control}, pages 173--216, 1995.

\bibitem[Schulman et~al.(2015)Schulman, Levine, Abbeel, Jordan, and
  Moritz]{schulman2015trust}
John Schulman, Sergey Levine, Pieter Abbeel, Michael Jordan, and Philipp
  Moritz.
\newblock Trust region policy optimization.
\newblock In \emph{International Conference on Machine Learning}, pages
  1889--1897. PMLR, 2015.

\bibitem[Schulman et~al.(2017)Schulman, Wolski, Dhariwal, Radford, and
  Klimov]{schulman2017proximal}
John Schulman, Filip Wolski, Prafulla Dhariwal, Alec Radford, and Oleg Klimov.
\newblock Proximal policy optimization algorithms.
\newblock \emph{arXiv preprint arXiv:1707.06347}, 2017.

\bibitem[Sun and Fazel(2021)]{sun2021learning}
Yue Sun and Maryam Fazel.
\newblock Learning optimal controllers by policy gradient: Global optimality
  via convex parameterization.
\newblock In \emph{2021 60th IEEE Conference on Decision and Control (CDC)},
  pages 4576--4581. IEEE, 2021.

\bibitem[Tang et~al.(2021)Tang, Zheng, and Li]{tang2021analysis}
Yujie Tang, Yang Zheng, and Na~Li.
\newblock Analysis of the optimization landscape of linear quadratic {Gaussian}
  {(LQG)} control.
\newblock In \emph{Learning for Dynamics and Control}, pages 599--610. PMLR,
  2021.

\bibitem[Thorp and Barmish(1981)]{thorp1981guaranteed}
JS~Thorp and BR~Barmish.
\newblock On guaranteed stability of uncertain linear systems via linear
  control.
\newblock \emph{Journal of Optimization Theory and Applications}, 35\penalty0
  (4):\penalty0 559--579, 1981.

\bibitem[Tropp(2015)]{tropp2015introduction}
Joel~A Tropp.
\newblock An introduction to matrix concentration inequalities.
\newblock \emph{arXiv preprint arXiv:1501.01571}, 2015.

\bibitem[Wright et~al.(1999)Wright, Nocedal, et~al.]{wright1999numerical}
Stephen Wright, Jorge Nocedal, et~al.
\newblock Numerical optimization.
\newblock \emph{Springer Science}, 35\penalty0 (67-68):\penalty0 7, 1999.

\bibitem[Youla et~al.(1976{\natexlab{a}})Youla, Jabr, and
  Bongiorno]{youla1976partii}
Dante Youla, Hamid Jabr, and Jr~Bongiorno.
\newblock Modern {W}iener-{H}opf design of optimal controllers--part ii: The
  multivariable case.
\newblock \emph{IEEE Transactions on Automatic Control}, 21\penalty0
  (3):\penalty0 319--338, 1976{\natexlab{a}}.

\bibitem[Youla et~al.(1976{\natexlab{b}})Youla, Bongiorno, and
  Jabr]{youla1976parti}
DC~Youla, J~d Bongiorno, and Hamid Jabr.
\newblock Modern {W}iener--{H}opf design of optimal controllers part i: The
  single-input-output case.
\newblock \emph{IEEE Transactions on Automatic Control}, 21\penalty0
  (1):\penalty0 3--13, 1976{\natexlab{b}}.

\bibitem[Yu(1994)]{yu1994rates}
Bin Yu.
\newblock Rates of convergence for empirical processes of stationary mixing
  sequences.
\newblock \emph{The Annals of Probability}, pages 94--116, 1994.

\bibitem[Zhang et~al.(2020)Zhang, Hu, and Basar]{zhang20mixed}
Kaiqing Zhang, Bin Hu, and Tamer Basar.
\newblock Policy optimization for $\mathcal{H}_{2}$ linear control with
  $\mathcal{H}_{\infty}$ robustness guarantee: Implicit regularization and
  global convergence.
\newblock In Alexandre~M. Bayen, Ali Jadbabaie, George Pappas, Pablo~A.
  Parrilo, Benjamin Recht, Claire Tomlin, and Melanie Zeilinger, editors,
  \emph{Proceedings of the 2nd Conference on Learning for Dynamics and
  Control}, volume 120 of \emph{Proceedings of Machine Learning Research},
  pages 179--190. PMLR, 10--11 Jun 2020.

\bibitem[Zhou et~al.(1996)Zhou, Doyle, Glover, et~al.]{zhou1996robust}
Kemin Zhou, John~Comstock Doyle, Keith Glover, et~al.
\newblock \emph{Robust and Optimal Control}, volume~40.
\newblock Prentice hall New Jersey, 1996.

\end{thebibliography}
